\documentclass[11pt, a4paper]{article}
\usepackage{amsmath, amssymb, amsthm, amsfonts}
\allowdisplaybreaks[4]

\usepackage{geometry}
\usepackage{mathtools}
\usepackage{mathrsfs} 
\usepackage{bm}
\usepackage{esint}
\usepackage{verbatim}
\usepackage{xcolor}

\definecolor{darkgreen}{RGB}{0,120,0}
\usepackage{comment}
\usepackage{hyperref}

\usepackage[style=alphabetic,maxalphanames=5,maxnames=5]{biblatex}
\usepackage{titlesec}
\titlespacing*{\section}{0pt}{1.5ex plus 0.5ex minus 0.2ex}{0.8ex plus 0.2ex}
\titlespacing*{\subsection}{0pt}{1.2ex plus 0.3ex minus 0.2ex}{0.5ex plus 0.2ex}
\titlespacing*{\subsubsection}{0pt}{1ex plus 0.2ex minus 0.2ex}{0.3ex plus 0.2ex}

\titleformat{\section}{\normalfont\large\bfseries}{\thesection}{1em}{}
\titleformat{\subsection}{\normalfont\normalsize\bfseries}{\thesubsection}{1em}{}
\titleformat{\subsubsection}{\normalfont\normalsize\bfseries}{\thesubsubsection}{1em}{}

\usepackage{enumitem}
\setlist{itemsep=0.2ex, parsep=0pt, topsep=0.5ex, partopsep=0pt}

\AtBeginDocument{
  \setlength{\abovedisplayskip}{1ex plus 0.3ex minus 0.3ex}
  \setlength{\belowdisplayskip}{1ex plus 0.3ex minus 0.3ex}
  \setlength{\abovedisplayshortskip}{0.3ex plus 0.2ex}
  \setlength{\belowdisplayshortskip}{0.8ex plus 0.2ex minus 0.2ex}
}

\newtheorem{theorem}{Theorem}[section]
\newtheorem{definition}[theorem]{Definition}
\newtheorem{definitiontheorem}[theorem]{Definition and Theorem}
\newtheorem{proposition}[theorem]{Proposition}
\newtheorem{corollary}[theorem]{Corollary}
\newtheorem{lemma}[theorem]{Lemma}
\newtheorem{example}[theorem]{Example}
\newtheorem{remark}[theorem]{Remark}

\newcommand{\R}{\mathbb{R}}
\newcommand{\Hn}{\mathbb{H}^n}

\newcommand{\V}{\mathbb{V}} 
\newcommand{\W}{\mathbb{W}} 
\newcommand{\norm}[1]{\lVert#1\rVert}
\newcommand{\knorm}[1]{\lVert#1\rVert
} 
\newcommand{\inner}[2]{\langle #1, #2 \rangle}
\newcommand{\gr}{\mathrm{gr}}

\author{
  Yibo Chen\thanks{Department of Mathematics and Statistics, University of Jyv\"askyl\"a, P.O. Box 35 (MaD), FI-40014 University of Jyv\"askyl\"a, Finland. \texttt{yibo.y.chen@jyu.fi}}
  \and
  Katrin F\"assler\thanks{Department of Mathematics and Statistics, University of Jyv\"askyl\"a, P.O. Box 35 (MaD), FI-40014 University of Jyv\"askyl\"a, Finland. \texttt{katrin.s.fassler@jyu.fi}}
  \and
  Kilian Zambanini\thanks{Dipartimento di Matematica, Universit\`a di Trento, Via Sommarive, 14, 38123 Povo TN, Italy. \texttt{kilian.zambanini@unitn.it}}
\thanks{Part of this work was carried out while K.Z. visited the Department of Mathematics and Statistics at the University of Jyv\"askyl\"a. The hospitality of the department is gratefully acknowledged. K.Z. is member of the {\it Gruppo Nazionale per l’Analisi Matematica, la Probabilità e le loro Applicazioni} (GNAMPA), of the {\it Istituto Nazionale di Alta Matematica} (INdAM), and was supported by the University of Trento and the INdAM-GNAMPA 2026 Project \emph{Variational, Geometric, and Analytic Perspectives on Regularity}, CUP E53C25002010001. K.Z. is also grateful to Andrea Pinamonti and Francesco Serra Cassano for useful discussions and suggestions.
}}

\title{On low-dimensional uniform rectifiability in Heisenberg groups - Part 2}

\date{\today}

\begin{document}
\maketitle
\numberwithin{equation}{section}

\begin{abstract}
Let $1\leq k\leq n$. We prove that $k$-dimensional intrinsic Lipschitz graphs in the Heisenberg group $\mathbb{H}^n$ satisfy a geometric lemma $\mathrm{GLem}(\beta_{2,\mathcal{V}_k},p)$ for horizontal $\beta$-numbers with an exponent $p=p(k)$. Previously, this result was  known only in the case $k=1$; our proof recovers  the sharp exponent $p=4$ in this setting. For $k>1$, we adapt an integral geometric approach originally developed by Orponen for Euclidean and parabolic Lipschitz functions. In addition, for $k=n$, we show how to deduce a geometric lemma with $p=4$ directly from an isotropic Dorronsoro theorem in $\mathbb{R}^{2n}$ using a Poincar\'e inequality. Building on the new geometric lemmas, we establish
a  necessary condition for $k$-regular sets in $\mathbb{H}^n$ to admit corona decompositions by intrinsic Lipschitz graphs. The condition is known to be sufficient by earlier work of the last two authors together with Pinamonti. It involves additional flatness coefficients besides $\beta_{2,\mathcal{V}_k}$. Along the way, we therefore extend the known stability results for geometric lemmas under the ``big pieces'' functor to a larger class of coefficients.
\end{abstract}

\setcounter{tocdepth}{1}
\tableofcontents

\section{Introduction} 
\subsection{Quantitative differentiation for intrinsic Lipschitz graphs} This paper studies properties of low-dimensional intrinsic Lipschitz graphs in Heisenberg groups $\mathbb{H}^n=(\mathbb{R}^{2n+1},\cdot)$ equipped with a left-invariant non-Euclidean distance $d$. Intrinsic Lipschitz graphs (also of large dimensions) are defined through a geometric cone condition adapted to the group and dilation structures of $\mathbb{H}^n$; see Sections \ref{s:Heis} -- \ref{s:iLG} for details. They were introduced by Franchi, Serapioni and Serra Cassano \cite{MR2287539} as a surrogate for Lipschitz graphs in Euclidean spaces, and they  play a pivotal role in geometric measure theory on Heisenberg groups (e.g., \cite{MR2789472,MR3587666,MR3682744,MR4342997,MR4377000}) and applications beyond \cite{MR3678211,MR3815462}. Intrinsic Lipschitz functions constitute a rich class of mappings;  see Theorem \ref{t:Lift} for a way to construct low-dimensional intrinsic Lipschitz graphs in $\mathbb{H}^n$.

One theme of research concerns approximation of intrinsic Lipschitz graphs in $\mathbb{H}^n$ by affine homogeneous subgroups. \emph{How well can an intrinsic Lipschitz graph be approximated by planes at different places and scales?} Infinitesimally, the approximation is guaranteed in almost all points of the graph by versions of Rademacher's differentiability theorem \cite{MR2836591,MR4277829,MR4377000}, but here we are interested in quantitative multi-scale properties.

In the classical Euclidean setting, quantitative tools have proven useful to establish the boundedness of certain singular integrals on Lipschitz graphs or, more generally, sets with big pieces of Lipschitz graphs \cite{zbMATH00051391}. Quantitative approximation results 
in this spirit are provided by Dorronsoro's quantitative differentiation theorem \cite{MR796440} applied to Lipschitz functions, and by Carleson-type geometric lemmas for Lipschitz graphs \cite{MR1251061}: if $E\subset \mathbb{R}^n$ is a $k$-dimensional Lipschitz graph, then there exists a constant $C>0$ such that 
 \begin{equation}\label{eq:EuclGLEm}\int_0^R\int_{E\cap B(x,R)}{\beta_{2,A(n,k)}^E}(y,r)^2\,d\mathcal H^k(y)\frac{dr}{r}\leq CR^k,\quad x\in E,\,0<R<\mathrm{diam}E\end{equation}
 where 
 \begin{equation}\label{eq:EuclBeta}
{\beta_{2,A(n,k)}^E}(x,r) = \inf_{V\in A(n,k)}\left(\fint_{B(x,r)\cap E} \left(\frac{d(y,V)}{r}\right)^2\,d\mathcal{H}^k(y)\right)^{1/2}
 \end{equation}
and the infimum is taken over all affine $k$-dimensional subspaces in $\mathbb{R}^n$.
The study of the multifaceted relations between singular integrals, geometric lemmas, Lipschitz graphs and images is nowadays known as the theory of \emph{uniform} or \emph{quantitative rectifiability}.

In Heisenberg groups, geometric lemmas for intrinsic Lipschitz graphs in terms of homogeneous subgroups are known in the case of dimension and co-dimension $1$. Such graphs 
$E\subset \mathbb{H}^n$ satisfy a condition as in the following Definition \ref{def:qgeomlemballs}, for suitable exponents $p$ and  with  coefficients $h(y,r)$ that measure how well $E$ is approximated by $1$-dimensional `horizontal' (or $1$-codimensional `vertical') left-translated homogeneous subgroups at point $y\in E$ and scale $r>0$, similarly to the Euclidean $\beta$-numbers \eqref{eq:EuclBeta}.

\begin{definition}[Geometric lemma]\label{def:qgeomlemballs}
Let $h:E\times \R^+\to\R^+$ be a $\mathcal{H}^k \otimes \mathcal{L}^1$-measurable map.     Let $E\subset \mathbb{H}^n$ be $k$-regular, $1\leq p<\infty$.  We say that \textit{$E$ satisfies the $p$-geometric lemma with respect to $h$}, and we write  $E\in{\rm GLem}(h,p,C)$ (or simply $E\in{\rm GLem}(h,p)$) if there exists $C>0$ such that for every $0<R<\infty$ with $R\leq \mathrm{diam}(E)$ and every $x\in E$ it holds
    \begin{equation*}\int_0^R\int_{E\cap B(x,R)}h(y,r)^p\,d\mathcal H^k(y)\frac{dr}{r}\leq CR^k.\end{equation*}
    \end{definition}

 Geometric lemmas for $1$-codimensional intrinsic Lipschitz graphs and vertical $\beta$-numbers are deep results that were proven with different methods in $\mathbb{H}^1$ (with exponent $p=4$; \cite{MR4922674}) and $\mathbb{H}^n$ for $n>1$  (\cite{MR4388340}; with exponent $p=2$); and it was shown \cite{2022arXiv220703013C} that $p=4$ cannot be improved to $p<4$ in $\mathbb{H}^1$. Recent work \cite{2025arXiv251026934H} has further clarified the role of geometric lemmas with exponent $p=2$ in the theory of $1$-codimensional uniform rectifiability on $\mathbb{H}^1$. 
 
 For $1$-dimensional intrinsic Lipschitz graphs, geometric lemmas with exponent $p=4$ and horizontal $\beta$-numbers follow from the quantitative methods in the study of traveling salesman theorems \cite{MR3456155,MR4477201}; see also the discussion in \cite{FV2,arXiv:2601.03837}. Yet, quoting \cite[p.4]{MR4922674}, ``\emph{Little is known, however,
about quantitative rectifiability for higher-dimensional subsets, even surfaces in $\mathbb{H}^n$ with topological
dimension between $2$ and $2n-1$}''. We aim to fill this gap for dimensions $k\in \{2,\ldots,n\}$.

It is a feature of the Heisenberg group that the study of quantitative $k$-rectifiability for the cases $k\leq n$ and $k>n$ often requires entirely different techniques. This distinction is  ultimately related to the fact that homogeneous subgroups of $\mathbb{H}^n$ appear in two flavors: \emph{horizontal} subgroups, which exist only in dimensions $k\leq n$, and \emph{vertical} subgroups. The present paper concerns the low-dimensional \emph{horizontal} regime and resumes where \cite{arXiv:2601.03837} left off. In this case, the graph maps of intrinsic Lipschitz functions are also metrically Lipschitz, and therefore our work contributes to the recent study of metric rectifiability modeled on Euclidean spaces \cite{2023arXiv230612933B}. We prove:

\begin{theorem}[Geometric lemma for low-dimensional intrinsic Lipschitz graphs]\label{t:SumiLG}
    Let $n\in \mathbb{N}$ and $k\in \{1,\cdots,n\}$. Then there exists an exponent $p=p(k)$ such that every entire $k$-dimensional intrinsic Lipschitz graph $E \subset \mathbb{H}^n$ satisfies 
    \begin{equation}\label{eq:MainGLem}
    E\in \mathrm{GLem}(\beta_{2,\mathcal{V}_k}^E,p)
   \quad\text{with}\quad
    \beta^E_{2,\mathcal{V}_k}
(x, r) = \inf_{\V\in\mathcal{V}_k} \left( \fint_{B(x,r) \cap E} \left( \frac{d(y, \V)}{r} \right)^2 d\mathcal{H}^k(y) \right)^{1/2},
\end{equation}
where the infimum is taken over all  $k$-dimensional horizontal planes. 
For all $n\in \mathbb{N}$, if $k=1$ or $k=n$, then \eqref{eq:MainGLem} holds with $p=4$.
\end{theorem}

Theorem \ref{t:SumiLG} is a simplified and combined version of our results.  
For $k=1$, Theorem \ref{t:GlemStratBetanILGLinftyk=2} gives the geometric lemma with the $L^{\infty}$-based \emph{stratified $\widehat{\beta}_{\infty,\mathcal{V}_1}$-numbers}  from Li's work \cite{MR4477201} on the traveling salesman theorem, with an independent proof. For the case $k=n$, we obtain in Theorem \ref{t:GlemStratBetanILG}  a 
 geometric lemma for stratified $\widehat{\beta}_{2,\mathcal{V}_n}$-numbers; see Definition \ref{d:StratifBeta}. In both cases, the proofs actually yield geometric lemmas for potentially larger coefficients defined using  parametrizations of the intrinsic Lipschitz graph $E$ and the approximating planes $\mathbb{V}$ over a common horizontal plane $\mathbb{V}_0$. To be more precise, consider the horizontal subgroup $\mathbb{V}_0=\mathbb{R}^k \times \{0\}$ with complementary vertical subgroup $\mathbb{W}_0=\{0\}\times \mathbb{R}^{2n+1-k}$ in $\mathbb{H}^n$, and denote the coordinate projections
 \begin{displaymath}
     \Pi(w_1,\ldots,w_{2n},t):=(w_1,\ldots,w_k)\quad \text{and}\quad \Pi^{\bot}(w_1,\ldots,w_{2n},t):=(w_{k+1},\ldots,w_{2n}),
 \end{displaymath}
 for $(w_1,\ldots,w_{2n},t)\in \mathbb{H}^n$.
These projections isolate the horizontal components of an intrinsic Lipschitz map $\varphi: \mathbb{V}_0 \to \mathbb{W}_0$. The following theorem bounds the graph's geometric $\beta$-numbers using affine (for $k=1$) or isotropic affine (for $k=n$) approximations of $\Pi^\perp(\varphi)$.

 \begin{theorem}[Bounding $\beta$-numbers by parametrized coefficients]\label{t:ParamGLem}
     Let $n\in \mathbb{N}$. 
  
\noindent   If $k=1$ and $\varphi:\mathbb{V}_0\to \mathbb{W}_0$ is intrinsic $L$-Lipschitz, 
         then, for all $x\in \mathrm{gr}(\varphi)$ and $r>0$, it holds
       $$ 
       \beta^{\mathrm{gr}(\varphi) }_{\infty,\mathcal{V}_1}(x,r) ^4 \leq  \widehat{\beta}^{\mathrm{gr}(\varphi) }_{\infty,\mathcal{V}_1}(x,r) ^4 \lesssim \widehat{\Omega}_{\infty,\varphi}(\Pi(x),r)^4 \lesssim_L
       \Omega_{\infty,\Pi^{\bot}(\varphi)}(\Pi(x),r)^2,
        $$
where 
\begin{align*}
     \widehat{\Omega}_{\infty,\varphi}(v,r)^4:=
     \inf_{\substack{(A,A_{2n}):\V_0 \to \W_0\\\text{intrinsic affine}}}\sup_{z\in B_{\mathbb{R}}(v,r)}\tfrac{|\Pi^{\bot}(\varphi)(z)-A(z)|^2}{r^2}+\sup_{z\in B_{\mathbb{R}}(v,r)}\tfrac{d(\varphi(z),(A,A_{2n})(z))^4}{r^4},
\end{align*}
\begin{displaymath}
     \Omega_{\infty,\Pi^{\bot}(\varphi)}(v,r)^2:= \inf_{\substack{A:\mathbb{R} \to \mathbb{R}^{2n-1}\\\text{affine}}} \sup_{z\in B_{\mathbb{R}}(v,r)}\tfrac{|\Pi^{\bot}(\varphi)(z)-A(z)|^2}{r^2}.
\end{displaymath}
Moreover, one has $\mathrm{gr}(\varphi)\in \mathrm{GLem}( \Omega_{\infty,\Pi^{\bot}(\varphi)}(\Pi(\cdot),\cdot),2)$.

       \medskip
       
 \noindent        If $k=n$ and $\varphi:\mathbb{V}_0\to \mathbb{W}_0$ is intrinsic $L$-Lipschitz, then for all $x\in \mathrm{gr}(\varphi)$ and $r>0$, it holds

          $$ 
       \beta^{\mathrm{gr}(\varphi) }_{2,\mathcal{V}_n}(x,r) ^4 \leq  \widehat{\beta}^{\mathrm{gr}(\varphi) }_{2,\mathcal{V}_n}(x,r) ^4 \lesssim_{n,L} \widehat{\Omega}_{2,\varphi}(\Pi(x),r)^4\lesssim_{n,L}
       \Omega_{2,\Pi^{\bot}(\varphi)}^{iso}(\Pi(x),r)^2,
        $$
where 
\begin{align*}
     \widehat{\Omega}_{2,\varphi}(v,r)^4:=
     \inf_{\substack{(A,A_{n+1}):\V_0 \to \W_0\\\text{intrinsic affine}}}\fint_{B_{\mathbb{R}^n}(v,r)}\tfrac{|\Pi^{\bot}(\varphi)(z)-A(z)|^2}{r^2}\,dz+\left(\fint_{B_{\mathbb{R}^n}(v,r)}\tfrac{d(\varphi(z),(A,A_{n+1})(z))^2}{r^2}\,dz\right)^2,
\end{align*}
\begin{displaymath}
     \Omega_{2,\Pi^{\bot}(\varphi)}^{iso}(v,r)^2:= \inf_{\substack{A:\mathbb{R}^n \to \mathbb{R}^{n}\\\text{isotropic affine}}} \fint_{B_{\mathbb{R}^n}(v,r)}\tfrac{|\Pi^{\bot}(\varphi)(z)-A(z)|^2}{r^2}\,dz.
\end{displaymath}
Moreover, one has $\mathrm{gr}(\varphi)\in \mathrm{GLem}( \Omega_{2,\Pi^{\bot}(\varphi)}^{iso}(\Pi(\cdot),\cdot),2)$.
 \end{theorem}

 We recall that the case $k=1$ of Theorem \ref{t:SumiLG} also follows from  \cite{MR3456155} (for $n=1$)  and \cite{MR4477201} (for $n\geq 1$) since entire $1$-dimensional intrinsic Lipschitz graphs are regular curves, but  the case $k=1$ of Theorem \ref{t:ParamGLem} does not.  It involves specific parametrizations of intrinsic Lipschitz graphs. We crucially rely on these finer 1-dimensional estimates from Theorem \ref{t:ParamGLem} to prove Theorem \ref{t:SumiLG} for the intermediate dimensions ($1 < k < n$). For $k>1$, the results  in Theorems  \ref{t:SumiLG}--\ref{t:ParamGLem} are, to the best of our knowledge, new.

We emphasize that the geometric lemmas  for $\Omega_{\infty,\Pi^{\bot}(\varphi)}$ and $\Omega_{2,\Pi^{\bot}(\varphi)}^{iso}$ 
in Theorem \ref{t:ParamGLem}
hold with exponent $p=2$, the same as for the Euclidean geometric lemma \ref{eq:EuclGLEm}, while we obtain worse exponents in Theorem \ref{t:SumiLG}. In our proofs, this difference appears since we are able to bound the $\widehat{\beta}$-coefficients by the $\Omega$-coefficients only when the former are raised to a large enough power as stated in Theorem \ref{t:ParamGLem}. However, in the case $k=1$, we also know that the exponent $p=4$ in Theorem  \ref{t:SumiLG} is sharp, at least if we require the constants in the geometric lemmas to be quantitatively controlled in terms of the Lipschitz constants and the dimensions. This can be seen as follows:

Essentially due to the construction by Juillet \cite{MR2789375}, it is known \cite{arXiv:2601.03837}
that there exists a $1$-regular curve $\Gamma \subset \mathbb{H}^1$ which does not satisfy $\mathrm{GLem}(\beta_{2,\mathcal{V}_1}^E,p)$ for any $p<4$. By \cite[Theorem 6.24]{MR4299821}, the curve $\Gamma$ has \emph{big pieces of $1$-dimensional intrinsic Lipschitz graphs}. Now if Theorem  \ref{t:SumiLG} were true, quantitatively, with an exponent $p<4$ for all $1$-dimensional Lipschitz graphs, this would imply by the ``big pieces'' stability, cf.\ Section \ref{s:StabAbstrGlem}, that $\Gamma$ should also satisfy $\mathrm{GLem}(\beta_{2,\mathcal{V}_1}^E,p)$, a contradiction.

The exponents we obtain in Theorem \ref{t:SumiLG} for the geometric lemmas in the cases $1<k<n$ are likely not sharp, 
nonetheless combined with other results,  Theorem \ref{t:SumiLG} yields a characterization of a notion of low-dimensional quantitative rectifiability in Heisenberg groups for all $1\leq k\leq n$, which we will next discuss.

\subsection{Characterizations of corona decompositions and other applications}

In \cite{arXiv:2601.03837}, the last two authors, together with Pinamonti, 
gave a new sufficient condition for a $k$-regular set $E\subset \mathbb{H}^n$, $1\leq k\leq n$, to admit a \emph{corona decomposition by intrinsic Lipschitz graphs (ILG-C)}, and they showed that the existence of such a corona decomposition implies that the set has \emph{big pieces of Lipschitz images (BPLI) of subsets in $\mathbb{R}^k$}, in a quantitative way.
The mentioned sufficient condition was a combination of a strong and a weak geometric lemma, where ``weak geometric lemma'' refers to a condition of the following form.

\begin{definition}[Weak geometric lemma]\label{def:WGLballs}
 Let $h:E\times \R^+\to\R^+$ be a $\mathcal{H}^k \otimes \mathcal{L}^1$-measurable map.    Let $E\subset \mathbb{H}^n$ be $k$-regular.  We say that \textit{$E$ satisfies the weak geometric lemma with respect to $h$}, and we write $E\in{\rm WGL}(h,C(\cdot))$ (or simply $E\in{\rm WGL}(h)$) if for every $\varepsilon>0$ there exists $C=C(\varepsilon)>0$ such that for every $R>0$ and every $x\in E$ it holds
    \begin{equation*}\int_0^R\int_{E\cap B(x,R)}\chi_{\{h>\varepsilon\}}(y,r)\,d\mathcal H^k(y)\frac{dr}{r}\leq CR^k.\end{equation*}
    \end{definition}
The corona decomposition (ILG-C) for a set $E\subset \mathbb{H}^n$ is somewhat technical to state, see Definition \ref{d:ILG-C}, but informally it means that a system of dyadic cubes on $E$ can be divided into a family of bad cubes, of which there are not too many, and the remaining good cubes can be further partitioned into a forest $\mathcal{F}$ of -- again not too many -- trees, so that the set $E$ is well approximated by intrinsic Lipschitz graphs with small constants from the perspective of each tree  $\mathcal{S}\in \mathcal{F}$.
We obtain the following characterization.

\begin{corollary}\label{c:introChar}
    Let $n\in \mathbb{N}$ and $k\in \{1,\cdots,n\}$.
    Let $E\subset\mathbb H^n$ be $k$-regular. Then \[E\in\mathrm{(ILG-C)}\Longleftrightarrow E\in\mathrm{GLem}(\beta_{2,\pi,A(2n,k)},2) \text{  and } E\in\mathrm{WGL}(\beta_{\infty,\mathcal V_k}),\]
    with the coefficient functions $\beta_{\infty,\mathcal V_k}$ and  $\beta_{2,\pi,A(2n,k)}$ defined in Definitions \ref{d:HorizBeta} and  \ref{def_projbetas}, respectively.
\end{corollary}
The direction ``$\Leftarrow$'' follows from \cite{arXiv:2601.03837}; our new contribution is the converse implication. This makes crucial use of the geometric lemmas which we prove for $k$-dimensional intrinsic Lipschitz graphs in Theorem \ref{t:SumiLG}. Once this is established, standard arguments imply that such graphs also satisfy the weak geometric lemma $\mathrm{WGL}(\beta_{\infty,\mathcal V_k})$, while $\mathrm{GLem}(\beta_{2,\pi,A(2n,k)},2)$ follows rather directly by considering projections of the graphs to the horizontal plane $\mathbb{R}^{2n}\times \{0\}$. The final step requires to show that the same strong and weak geometric lemmas hold if a set $E$ merely admits a corona decomposition by intrinsic Lipschitz graphs, but is not necessarily a graph itself. This can be achieved  with an abstract procedure to obtain ``big pieces squared'' from corona decompositions due to
Bortz, Hoffman, Hofmann, 
              Luna-Garcia,  and Nystr\"{o}m,
\cite{MR4485846}, and by stability of geometric lemmas under the ``big pieces functor''. The stability result from 
 \cite{MR4485846} does not directly apply to the coefficients $\beta_{2,\pi,A(2n,k)}$, even if the methods work with minor changes.

We took the opportunity to formulate in Theorem \ref{thm: BHHGNgeneralizedAbstract} and Lemma \ref{lemmaA2} a generalized framework in metric spaces
that provides sufficient conditions for coefficient functions $h$ so that if a set $E$ has big pieces of sets in a class $\mathcal{E}$, each of which satisfies $\mathrm{GLem}(h,p)$ with uniform constants, then also $E$ itself satisfies $\mathrm{GLem}(h,p)$. Our condition applies to the coefficient functions in  \cite{MR4485846}, of which $\beta_{q,A(n,k)}$ and $\beta_{q,\mathcal V_k}$ are examples, but it also covers $\beta_{q,\pi,A(2n,k)}$ and other types of flatness coefficients studied in \cite{MR2297880,MR2337487,2007arXiv0706.2517S,FV2,arXiv:2601.03837}, for a given range of exponents.

We close this section with two other applications of our main results  motivated by  the study of singular integrals on low-dimensional intrinsic Lipschitz graphs in Heisenberg groups.

First, by Corollary \ref{c:introChar} and the established (weak) geometric lemmas, $k$-dimensional intrinsic Lipschitz graphs in $\mathbb{H}^n$ for $k\in \{1,\cdots,n\}$ satisfy (ILG-C). It is probably easier to appreciate this statement after having read the technical comments in Section \ref{ss:ConseqGLem}, but in essence, (ILG-C) yields approximations by intrinsic Lipschitz graphs \emph{with small constants}. Such information has been useful to prove $L^2$-boundedness of certain singular integral operators on regular curves in $\mathbb{H}^1$, see \cite{MR4299821}, and it was not previously known beyond the case $k=1\leq n$ studied in \cite{MR4375018}.

The second application of Theorem \ref{t:SumiLG} is related to sets with big pieces of Lipschitz graphs. According to a classical result by David and Semmes \cite{MR1132876}, a $k$-regular set in $\mathbb{R}^n$ has big pieces of $k$-dimensional Lipschitz graphs if and only if it satisfies a weak geometric lemma and it has big orthogonal projections onto $k$-planes.
Theorem \ref{t:SumiLG} allows us to extend the ``only if'' part to our setting; see Section \ref{ss:BigPieceAppl} for the relevant definitions.
\begin{theorem}\label{t:BPiLGIntro} Let $n\in \mathbb{N}$ and $k\in \{1,\cdots,n\}$.
    Let $E\subset\mathbb H^n$ be $k$-regular. Assume that $E$ has big pieces of intrinsic Lipschitz graphs. Then $E\in \mathrm{WGL}(\beta_{\infty,\mathcal V_k})$ and $E$ has big horizontal projections.
\end{theorem}
The converse implication was proven for $k=1=n$ in  \cite[Section 6.2]{MR4299821}. In the codimension $1$ case,  $(2n+1)$-regular sets in $\mathbb{H}^n$ have big pieces of intrinsic Lipschitz graphs over vertical hyperplanes if and only if they admit big vertical projections and satisfy a weak geometric lemma for vertical $\beta$-numbers, see Remark \ref{r:BPiLG}.

\subsection{Proof strategies for the geometric lemmas}

We now describe the main ideas behind the proof of the geometric lemmas for intrinsic Lipschitz graphs stated in Theorem \ref{t:SumiLG}. Let $n\in \mathbb{N}$, $k\in \{1,\ldots,n\}$, and assume that we are given a function $$\varphi:\mathbb{V}_0:=\mathbb{R}^k \times \{0\}\to \mathbb{W}_0:=\{0\}\times \mathbb{R}^{2n+1-k}.$$ Slightly abusing notation, we write in coordinates $\varphi=(\varphi_1,\ldots,\varphi_{2n+1-k})$.

Our starting point is the following observation. If $\varphi$ is intrinsic Lipschitz, then its horizontal components $\varphi_j$, $j=1,\ldots,2n-k$, are Euclidean Lipschitz and satisfy a certain PDE condition that we call \emph{weakly graph isotropic} (Theorem \ref{horcon}). Conversely, given a weakly graph isotropic Euclidean Lipschitz function $(\varphi_1,\ldots,\varphi_{2n-k}):\mathbb{R}^k \to \mathbb{R}^{2n-k}$, we can find $\varphi_{2n+1-k}:\mathbb{R}^k\to \mathbb{R}$ such that $\varphi:=(\varphi_1,\ldots,\varphi_{2n+1-k})$ is intrinsic Lipschitz (Theorem \ref{t:Lift}). The heuristic explanation is that $\varphi$ is intrinsic Lipschitz if and only if its intrinsic graph map $\Phi:\mathbb{V}_0 \to \mathbb{H}^n$ is metrically Lipschitz, which means that the intrinsic graph of $\varphi$ is tangential to the horizontal distribution in a weak sense ($\varphi$ is \emph{weakly graph contact} in the terminology of Definition \ref{d:WeakGraphContact}).

For $k=1$, the weak graph isotropic condition is trivial. Indeed, any Euclidean Lipschitz map $(\varphi_1,\ldots,\varphi_{2n-1}):\mathbb{R}^1 \to \mathbb{R}^{2n-1}$ can be lifted to an intrinsic Lipschitz function with a vertical component $\varphi_{2n}$ defined through integration. On the other hand, if $k=n$, then the weak graph isotropic 
condition reduces to the following system of linear PDEs
\begin{displaymath}
    \frac{\partial \varphi_l}{\partial x_j}=\frac{\partial \varphi_j}{\partial x_l},\quad\text{a.e.,}\quad \text{for all }j,l\in \{1,\ldots,n\}
\end{displaymath}
while for $1<k<n$, it becomes a non-linear condition; see
Definition \ref{d:WeaklyGraphIsotro}. Similarly, all lines in $\mathbb{R}^{2n}$ arise as projections of horizontal lines in $\mathbb{H}^n$ under the map $(w,t)\to w$, while for $k>1$, only the affine \emph{isotropic} $k$-planes arise as projections of horizontal $k$-planes; cf.\ Proposition \ref{ahp}.
These distinctions are the reason that our proof of Theorem \ref{t:SumiLG} is divided in three cases: $k=1$, $k=n$ and $1<k<n$.

\subsubsection{Proof of Theorem \ref{t:SumiLG} for $k=1$}
Dorronsoro's theorem for $1$-dimensional Euclidean Lipschitz graphs immediately implies that  $$\mathrm{gr}(\varphi)\in \mathrm{GLem}( \Omega_{\infty,\Pi^{\bot}(\varphi)}(\Pi(\cdot),\cdot),2).$$ Thus, once the comparison between the coefficients stated in Theorem \ref{t:ParamGLem} is established, the case $k=1$  of Theorem \ref{t:SumiLG}  will follow. The main task is therefore to verify that the `horizontal information'  encoded by  $\Omega_{\infty,\Pi^{\bot}(\varphi)}$ yields sufficient `vertical information' to control $ \widehat{\beta}^{\mathrm{gr}(\varphi) }_{\infty,\mathcal{V}_1}$. This is the case since both $\varphi$ and the approximating intrinsic affine functions $(A,A_{2n})$ are weakly graph contact, so the derivative of their last, vertical, component is controlled by their horizontal components. However, the precise computations are subtle due to the non-linearity of the weak graph contact condition in the case $k=1$. We address this by considering also the weak graph contact lift of $\Pi^{\bot}(\varphi)-A$, which is emphatically not the same as the difference of the individual weak graph contact lifts of $\Pi^{\bot}(\varphi)$ and $A$, respectively; see the proof of Theorem \ref{t:ControlStratifBetaByOmegaInftyk=1}.

\subsubsection{Proof of Theorem \ref{t:SumiLG} for $k=n$}
The case of $n$-dimensional intrinsic Lipschitz graphs in $\mathbb{H}^n$ is easier than the $1$-dimensional case in the sense that the weak graph contact condition becomes linear for $k=n$. The dimension $k=n$ is also the only one for which one can write $\mathbb{R}^{2n}$ as a direct sum of a $k$-dimensional subspace $V$ and its orthogonal complement $V^{\bot}$ so that both $V$ and $V^{\bot}$ are isotropic.
On the other hand, unlike for $k=1$, being weakly graph isotropic is a non-trivial additional requirement in the case $k=n$, and the statement
 \begin{equation}\label{eq:GlemIsotr}\mathrm{gr}(\varphi)\in \mathrm{GLem}( \Omega_{2,\Pi^{\bot}(\varphi)}^{iso}(\Pi(\cdot),\cdot),2)
 \end{equation}
 in Theorem \ref{t:ParamGLem}  does not directly follow from an existing result. Instead we note that a Dorronsoro-type theorem holds for weakly graph isotropic Lipschitz functions $\mathbb{R}^n \to \mathbb{R}^n$, where the family of admissible approximations does not include all affine functions  $\mathbb{R}^n \to \mathbb{R}^n$ but only the strict subclass of isotropic affine functions (Theorem \ref{aiso}). We believe that this observation, which is a purely Euclidean result, might be of independent interest also beyond analysis on the Heisenberg group. It can  be interpreted as a quantitative differentiation  result for the gradient of a real-valued function $u\in C^{1,1}(\mathbb{R}^n)$ that describes how well $\nabla u$ can be approximated by affine isotropic maps at different places and scales.

 With \eqref{eq:GlemIsotr} at hand, the case $k=n$ of Theorem \ref{t:SumiLG} will follow by proving the inequality 
 \begin{displaymath}
 \widehat{\Omega}_{2,\varphi}(\Pi(x),r)^4\lesssim_{n,L}
       \Omega_{2,\Pi^{\bot}(\varphi)}^{iso}(\Pi(x),r)^2,\quad x\in \mathrm{gr}(\varphi),\,r>0.
       \end{displaymath}
        This, once again, amounts to understanding how  a given approximation of the horizontal components $\Pi^{\bot}(\varphi)$ can be used to construct a good approximation of the full map $\varphi=(\Pi^{\bot}(\varphi),\varphi_{n+1})$. This time we use the weak graph contact condition $\nabla \varphi_{n+1}=-\Pi^{\bot}(\varphi)$ and a Poincar\'e inequality to obtain the desired estimates,
        see the proof of Theorem \ref{t:ControlStratifBetaByOmega}. 

\subsubsection{Proof of Theorem \ref{t:SumiLG} for $1<k<n$}
The case $1<k<n$ is the most difficult to handle as it combines challenges from both previous cases: the weak graph isotropic condition is non-trivial (like in the case $k=n$) and the weak graph contact condition is non-linear (like in the case $k=1$). It is not clear to us how to prove a counterpart for Theorem \ref{t:ParamGLem} in this case, but we were still able to obtain the case $1<k<n$ of Theorem \ref{t:SumiLG} with a different inductive argument, adapting an integral-geometric proof of Dorronsoro's theorem for Euclidean and parabolic Lipschitz functions due to Orponen \cite{MR4345824}. The idea is to consider lower-dimensional slices of intrinsic Lipschitz graphs.
The approximating $k$-planes produced by this approach do not a priori need to be horizontal planes in the sense of the Heisenberg group, so a main challenge is to verify that, with a small perturbation, they can be made horizontal, see Theorem \ref{l:almostHoriz_k}. 

\subsection{Comparison to previous work}
Low-dimensional intrinsic Lipschitz graphs in $\mathbb{H}^n$ are Euclidean rectifiable subsets of $\mathbb{R}^{2n+1}$, but our results do not follow from the Euclidean theory of uniform rectifiability since the horizontal $\beta_{\mathcal{V}_k}$-numbers require us to study approximation by \emph{horizontal} $k$-planes. 

Dorronsoro-type theorems for Lipschitz or Sobolev functions have been proven in a variety of settings. Without claiming to be complete, in addition to Dorronsoro's original theorem and the cited versions for $1$-codimensional intrinsic Lipschitz graphs in $\mathbb{H}^n$,
 we mention \cite{MR4171381,MR4345824,2023arXiv230613017A,MR5062411,hyde2026quantitativeharmonicapproximationsdorronsoros}. Our results are rather different as they concern mappings \emph{into} the metric space $\mathbb{H}^n$ rather than real-valued functions \emph{from} a metric space. The idea of considering lower-dimensional slices to prove a geometric lemma for intrinsic Lipschitz graphs has been applied in \cite{MR4388340} for the codimension-$1$ case in $\mathbb{H}^n$, $n>1$. In this setting, the slices
 lead the authors to consider 
 metric Lipschitz functions $\mathbb{H}^{n-1}\to \mathbb{R}$, thus again, real-valued functions that are quite different from our situation.
 Unlike for Euclidean Lipschitz mappings, it is also not possible to deduce Theorems \ref{t:SumiLG}--\ref{t:ParamGLem} by treating each component of the intrinsic Lipschitz function $\varphi$ separately. This difficulty essentially stems from the non-commutativity of the Heisenberg group product, and the resulting expressions for the Heisenberg metric and the cone condition.

\subsection{Structure of the paper} In Section \ref{s:Prelim}, we recall relevant terminology and facts related to Heisenberg groups, intrinsic Lipschitz graphs, and intrinsic linearity. Section \ref{s:LowDimiLG} focuses on low-dimensional intrinsic Lipschitz graphs in $\mathbb{H}^n$ and their projections to $\mathbb{R}^{2n}$. We claim little novelty for this section, as many of the results follow from the theory of metric Lipschitz images of $\mathbb{R}^k$, $1\leq k\leq n$, in $\mathbb{H}^n$. However, our presentation offers a new perspective which is useful for the main body of the paper and potentially for other applications as well. Section \ref{s:IsotrDorronsoro} discusses the isotropic Dorronsoro theorem, a Euclidean result. In Section \ref{s:GLem} we return to the Heisenberg group and prove the main results, Theorems  \ref{t:SumiLG}--\ref{t:ParamGLem}. Section \ref{s:BP} is mainly written in the setting of metric spaces, where we discuss stability of geometric lemmas under the big pieces functor, along with 
applications in $\mathbb{H}^n$ such as Corollary \ref{c:introChar}.

\subsection{AI disclosure} 
No AI tools were used to generate 
the mathematical content of the first version \cite{2026arXiv260830269C} of this article. 
In that first version, 
we had  applied a Morrey-Sobolev inequality in the proof of Theorem \ref{t:ControlStratifBetaByOmega}, which for $n>1$ resulted in a worse exponent $q>2n$ compared to $q=4$ in the present statement. The application of the Poincar\'e inequality to prove the current version of Theorem \ref{t:ControlStratifBetaByOmega} was suggested by OpenAI’s ChatGPT 6 Astra, which was also used to assist in searching for typographical errors and other minor mistakes.

\section{Preliminaries}\label{s:Prelim}

\subsection{Notation}\label{s:Not}
 For nonnegative numbers or functions $A$ and $B$, we will write $A \lesssim B$ to mean $A \leq CB$ where $C$ is a constant, and $A \lesssim_{s,t} B$ if $C$ depends on some parameters $s$ and $t$. Similarly, we will write $A
\sim B$ if $A \lesssim B \lesssim A$ and $A \sim_t B$ if $A \lesssim_t B \lesssim_t A$.

The Lipschitz constant of a Euclidean Lipschitz function $f$ is denoted by $\mathrm{Lip}(f)$. Average integrals are written as $\fint_B h \,d \mu = \frac{1}{\mu(B)}\int_B h \, d\mu$.

Balls in a metric space $(X,d)$ are denoted by $B(x,r)$, and Hausdorff measures by $\mathcal{H}^s$. The dependence on the metric will be indicated where needed. For a  subset $E\subset X$ and $0<s<\infty$, we write $E\in \mathrm{Reg}_s(C)$ if $E$ is AD-$s$-regular, or simply \emph{$s$-regular} with constant $C$, meaning that $E$
 is closed and 
 \begin{equation}\label{def_Ahlfors}
       C^{-1}r^s\leq\mathcal H^s(E\,\cap\,B(x,r) )\leq Cr^s\quad\text{ for every }x\in E,\, 0<r\leq\mathrm{diam}(E),\,r<\infty. \end{equation}

\subsection{The Heisenberg group}\label{s:Heis}
The $2n$-dimensional Heisenberg group is the manifold $\Hn = \R^{2n+1}$ endowed with the group product
\begin{equation}\label{sym}
    (z, t) \cdot (\zeta, \tau) = \left(z + \zeta, t + \tau + \tfrac{1}{2} \omega(z, \zeta) \right)
\end{equation}
where $z, \zeta \in \R^{2n}$, and $\omega(z, \zeta) = \inner{Jz}{\zeta}=\sum_{i=1}^n z_{i}\zeta_{n+i}-z_{n+i}\zeta_{i}$. Here $J = \begin{bmatrix} 0 & -I_n \\ I_n & 0 \end{bmatrix}_{2n \times 2n}$ is the \emph{standard complex structure}.

We will sometimes write simply $pq$ instead of $p\cdot q$ for $p,q\in \mathbb{H}^n$. Let $p = (z, t) \in \Hn$. The inverse element of $p$ is $p^{-1} = (-z, -t)$. The identity element is $e = (0, 0)$.

The \emph{Kor\'{a}nyi norm} $\knorm{\cdot}$ on $\Hn$ is
\begin{equation*}
    \knorm{(z, t)} = \left( |z|^4 + 16t^2 \right)^{1/4},
\end{equation*}
where $|\cdot|$ denotes the Euclidean norm on $\mathbb{R}^{2n}$.
The Euclidean distance is sometimes denoted as $d_{\mathrm{Eucl}}(z,z')=|z-z'|$.
The Kor\'{a}nyi metric is $d(p, p') = \knorm{p^{-1} \cdot p'}$ for all $p,p'\in \mathbb{H}^n$. The homogeneous dilations $\delta_\lambda : \Hn \to \Hn$ for $\lambda>0$
are given by
\begin{align*}
    \delta_\lambda(z, t) &= (\lambda z, \lambda^2 t),\quad (z,t)\in \mathbb{H}^n.
\end{align*}

The \emph{horizontal subgroups} of $\mathbb{H}^n$ are the sets of the form $\V = V \times \{0\}$, where $V \subset \R^{2n}$ is a subspace such that $\omega(v, v') = 0$ for all $v, v' \in V$.
We call $\W = V^\perp \times \R$ the \emph{complementary vertical subgroup} of $\mathbb{V}$. In particular, in this paper, the complementary vertical subgroup $\mathbb{W}$  of $\mathbb{V}$ is always assumed to be an orthogonal complement of $\mathbb{V}$ with respect to the underlying Euclidean structure of $\mathbb{R}^{2n+1}$. The sets
$\V$ and $\W$ are homogeneous subgroups of $\Hn$, that is, they are subgroups invariant under dilations. We have $\Hn = \V \ltimes \W$. This induces for every $p\in\mathbb{H}^n$ a unique decomposition $p =p_{\mathbb{V}}\cdot p_{\mathbb{W}}$ such that $p_{\mathbb{V}}\in\mathbb{V}$ and $p_{\mathbb{W}}\in\mathbb{W}$.

\begin{definition}[Isotropic subspaces and (affine) horizontal planes]\label{hp}
    A vector subspace $V \subset \mathbb{R}^{2n}$ is said to be \textit{isotropic} if $\omega(z, z') = 0$ for all $z, z' \in V$, where $\omega : \mathbb{R}^{2n} \times \mathbb{R}^{2n} \to \mathbb{R}$ is the form defined in (\ref{sym}).

For $1 \le k \le n$, we denote by $\mathcal{V}_k^0$ the family of $k$-dimensional \textit{horizontal subgroups} of $\mathbb{H}^n$. Every $\mathbb{V} \in \mathcal{V}_k^0$ is of the form $\mathbb{V} = V \times \{0\}$ for a $k$-dimensional isotropic subspace $V$ of $\mathbb{R}^{2n}$.

We call a subset $\mathbb{V}$ of $\mathbb{H}^n$ an \textit{(affine) $k$-dimensional horizontal plane} if it can be written as $x \cdot \mathbb{V}_0$ for some $x \in \mathbb{H}^n$ and $\mathbb{V}_0 \in \mathcal{V}_k^0$. We denote by $\mathcal{V}_k$ the collection of all affine horizontal $k$-dimensional planes of $\mathbb{H}^n$. The elements of $\mathcal{V}_1$ are also called \emph{horizontal lines}. A \emph{horizontal line segment} is simply a segment on a horizontal line.
\end{definition}

\begin{remark}\label{r:Line}
 If $\mathbb{V}\in \mathcal{V}_k$ and $\ell$ is an arbitrary line contained in $\mathbb{V}$, then $\ell\in \mathcal{V}_1$.    
\end{remark}
\begin{definition}[Projection onto horizontal subgroups]\label{d:HorizProj}
Let $\mathbb V=V\times\{0\}\in \mathcal V_k^0$ be a $k$-dimensional horizontal subgroup of $\mathbb H^n$, where $V$ is a $k$-dimensional isotropic subspace of $\R^{2n}$. The \textit{horizontal projection} onto $\V$ is defined by 
\[\pi_\V:\mathbb H^n\to \V,\quad \pi_\V(z,t):=(\pi_{V}(z),0),\]
where $\pi_{V}:\R^{2n}\to V$ stands for the standard Euclidean orthogonal projection onto $V$.
\end{definition}
When considering a splitting $\mathbb H^n=\V\ltimes\W$, where $\W$ is the complementary vertical subgroup associated to $\V$, it turns out that $\pi_\V (p)=p_\V$ for every point $p\in \mathbb H^n$.

The $k$-dimensional Hausdorff measure $\mathcal{H}^k$ with respect to the metric $d$ is a Haar measure on every horizontal subgroup $\mathbb{V}$ of topological dimension $k$.  ``Almost everywhere on $\mathbb{V}$''  should always be understood with respect to this measure, unless otherwise specified.

\subsection{Intrinsic Lipschitz functions}\label{s:iLG}
The notion of intrinsic Lipschitz continuity was introduced by B. Franchi, R. Serapioni and F. Serra Cassano \cite{MR2287539}. It can be formulated for maps between complementary homogeneous subgroups of Carnot groups, but we state it here only in the context relevant for the present paper. 
\begin{definition}[Intrinsic graph and intrinsic graph map]\label{d:graph}
The \emph{intrinsic graph} of a function
 $\varphi: A\subset  \V \to \W$ between complementary homogeneous subgroups of $\mathbb{H}^n$ is
\begin{equation*}
    \gr(\varphi) := \{ v \cdot \varphi(v) : v \in 
   A\}.
\end{equation*}
We also denote the \emph{graph map} of $\varphi$ by $\Phi:A \to \mathbb{H}^n,\,
\Phi(v):=v \cdot \varphi(v)$.
\end{definition}
The term ``intrinsic" stems from the fact that the class of intrinsic graphs is preserved under group translations of $\mathbb H^n$. Specifically, if we consider any $p\in \mathbb{H}^n$ and $\varphi: A\subset \V\to\W$ between complementary horizontal and vertical subgroups, then $p\cdot \mathrm{gr}(\varphi)$ is the intrinsic graph of $\varphi_p:A_p\subset \mathbb{V} \to \mathbb{W}$ given by 
\begin{displaymath}
    \varphi_p(v):= v^{-1}\cdot p \cdot \pi_{\mathbb{V}}(p)^{-1}\cdot v\cdot \varphi\left(\pi_{\mathbb{V}}(p)^{-1}\cdot v\right),\quad v\in A_p:=\pi_\V(p)\cdot A.
\end{displaymath}
Indeed, $\varphi_p$ is well-defined since $\W$ is always a normal subgroup of $\mathbb H^n$ and one has
\begin{align*}\mathrm{gr}(\varphi_p)=\{v\cdot \varphi_p(v): v\in A_p\}&=p\cdot\{ \pi_{\mathbb{V}}(p)^{-1}\cdot v\cdot \varphi(\pi_{\mathbb{V}}(p)^{-1}\cdot v):v\in \pi_\V(p)\cdot A\}\\ &=p\cdot \{v'\cdot \varphi(v'):v'\in A\}=p\cdot \gr(\varphi).\end{align*}
In particular, 
if $p = v \cdot \varphi(v) \in \gr(\varphi)$, then $\varphi_{p^{-1}}$ is the function whose intrinsic graph represents the shifting of the original graph $\gr(\varphi)$ so that the point $p$ is moved to the origin $e$. In this case, the explicit form of $\varphi_{p^{-1}}$ is given by $\varphi_{p^{-1}}(\eta) := \eta^{-1}\cdot \varphi(v)^{-1}\cdot \eta \cdot \varphi(v\cdot\eta)$ for $\eta\in v^{-1}\cdot A$.

\begin{definition}[Intrinsic Lipschitz]\label{d:IntrLip} Let  $\Hn = \V \ltimes \W$ be a splitting into complementary horizontal and vertical subgroups, and let $L\geq 0$.
We say that a function $\varphi: A\subset \V \to \W$ 
is \emph{intrinsic $L$-Lipschitz}, or simply \emph{intrinsic Lipschitz},
if for all $p = v \cdot \varphi(v) \in \gr(\varphi)$, we have:
\begin{equation}\label{def_intLip}
    \knorm{\varphi_{p^{-1}}(\eta)} \le L \knorm{\eta},\quad \eta \in v^{-1}\cdot A.
\end{equation}
In this case, $\mathrm{gr}(\varphi)$ is said to be an \emph{intrinsic ($L$-)Lipschitz graph}.
The \emph{intrinsic Lipschitz constant} of $\varphi$ is the infimum of the numbers $L$ such that the inequality holds for all $p\in \mathrm{gr}(\varphi)$ and all $\eta$ in the domain of $\varphi_{p^{-1}}$. It will be denoted by $\mathrm{Lip}(\varphi)$. If $\mathbb{V}\in \mathcal{V}_k^0$, we also say that $\mathrm{gr}(\varphi)$ is a \emph{$k$-dimensional} intrinsic Lipschitz graph.
\end{definition}

The symbol  $\mathrm{Lip}(\cdot)$ is used in this paper also to denote the Lipschitz constant of a (often real-valued) function that is Lipschitz with respect to the Euclidean metric. The intended meaning should always be clear from the context.

Definition \ref{d:IntrLip} has also been stated in the literature with the Kor\'{a}nyi norm $\|\cdot\|$ in \eqref{def_intLip} replaced by a different homogeneous gauge function. Since the relevant gauges are bi-Lipschitz equivalent, the resulting class of intrinsic Lipschitz functions and graphs is the same, although the value of the intrinsic Lipschitz constant may change (in a quantitatively controlled way).

Intrinsic Lipschitz continuity can be characterized through a cone property for $\gr(\varphi)$ \cite{MR2287539}, in analogy with the cone property of Euclidean Lipschitz graphs.
Let $\Hn = \V \ltimes \W$ be a splitting into horizontal and complementary vertical subgroup. The \textit{intrinsic cone} $\mathscr{C}_{\alpha}$ of aperture $\alpha > 0$ with axis $\mathbb{W}$ is
$$\mathscr{C}_{\alpha} := \{p \in \mathbb{H}^n : \|p_{\mathbb{V}}\| \leq \alpha\|p_{\mathbb{W}}\|\}.
$$
Observe that $\mathscr{C}_{\alpha}$ is homogeneous (invariant under dilations) and that $\mathbb{W} \subset \mathscr{C}_{\alpha}$. For $q \in \mathbb{H}^n$ we also introduce the cone $\mathscr{C}_{\alpha}(q) := q\cdot \mathscr{C}_{\alpha}$ with vertex $q$. Directly from the definition, one obtains the following characterization:

\begin{proposition}\label{p:iLipChar} Let  $\Hn = \V \ltimes \W$ be a splitting into complementary horizontal and vertical subgroups.
A function $\varphi : A \subset \mathbb{V} \to \mathbb{W}$ is intrinsic $L$-Lipschitz if and only if for all $ p \in \gr(\varphi)$ and for all $ 0 \leq \alpha < 1/L$,
\begin{equation*}
\gr(\varphi) \cap \mathscr{C}_{\alpha}(p) = \{p\}.
\end{equation*}
\end{proposition} 
An explicit expression for the intrinsic Lipschitz condition for maps from horizontal to complementary vertical subgroups in $\mathbb{H}^n$ is the following, cf. \cite[Remark 4.57  (ii)]{MR3587666}:

\begin{proposition}\label{1.2} Let  $\Hn = \V \ltimes \W$ be a splitting into complementary horizontal and vertical subgroups, and let $L\geq 0$.
A function $\varphi: A\subset \V \to \W$ is intrinsic $L$-Lipschitz if and only if
\begin{equation*}
\|{v'}^{-1}\cdot v \cdot \Phi(v)^{-1}\cdot \Phi(v')\|=  \knorm{v'^{-1} \cdot v \cdot \varphi(v)^{-1}\cdot v^{-1} \cdot v' \cdot \varphi(v')} \le L \knorm{v^{-1} \cdot v'}, \quad v,v'\in A.
\end{equation*}
\end{proposition}

\subsection{Intrinsic linearity and intrinsic differentiability}
Next, we review  intrinsic notions  of linearity and differentiability in Heisenberg group.
\begin{definition}[Intrinsic linearity]\label{linearity} Let  $\Hn = \V \ltimes \W$ be a splitting into complementary horizontal and vertical subgroups. A function
$L: \V \to \W$ is 
\emph{intrinsic linear} if for all $v,v' \in \V$ and $\lambda > 0$:
\begin{align*}
    L(\delta_\lambda v) &= \delta_\lambda(L v) \\
    L(v \cdot v') &= v'^{-1} \cdot L(v) \cdot v' \cdot L(v').
\end{align*}
\end{definition}

The geometric intuition here is that a function 
 $L:\mathbb{V}\to \mathbb{W}$ is intrinsic linear if and only if 
its intrinsic graph $\text{gr}(L) = \{v \cdot L(v) : v \in \mathbb{V}\}$ forms a homogeneous subgroup of $\mathbb{H}^n$, see for instance \cite[Proposition 4.68]{MR3587666}.
Indeed, the first equation in Definition \ref{linearity} guarantees closure under dilations. To ensure closure under group operation, we need $$(v \cdot v') \cdot L(v \cdot v')=(v \cdot L(v)) \cdot (v' \cdot L(v')).$$
Left-multiplying by $(v \cdot v')^{-1}$ gives the second equation. 

\emph{Intrinsic differentiable} functions are those which are, infinitesimally, well approximated by intrinsic linear functions in the sense of the following definition.

\begin{definition}[Intrinsic differentiability]\label{1.4}  Let  $\Hn = \V \ltimes \W$ be a splitting into complementary horizontal and vertical subgroups. A function
$\varphi:U\subset \V\to \W$, with $U$ Borel,  is 
\emph{intrinsically differentiable} at a density point $v \in U$ of $U$ if there exists an 
intrinsic linear map $d\varphi_v: \V \to \W$ such that
\begin{equation}\label{eq:IntrDiff}
    \knorm{d\varphi_v(\eta)^{-1} \cdot \eta^{-1} \cdot \varphi(v)^{-1} \cdot \eta \cdot \varphi(v \cdot \eta)} = o(\knorm{\eta}) \quad \text{as } \knorm{\eta} \to 0.
\end{equation}
The map $d\varphi_v$ is called the intrinsic differential of $\varphi$ at $v$.
\end{definition}

With the notation introduced below Definition \ref{d:graph}, condition \eqref{eq:IntrDiff} can be rephrased as
\begin{equation}\label{eq:IntrDiffTransl}
    \|d\varphi_v(\eta)^{-1} \cdot \varphi_{p^{-1}}(\eta)\|=o(\|\eta\|),\quad
    \text{as } \knorm{\eta} \to 0
\end{equation}
with $p=v\cdot \varphi(v)$. In \eqref{eq:IntrDiff} and 
\eqref{eq:IntrDiffTransl}, $\eta$ ranges in the domain of $\varphi_{p^{-1}}$.

The following is a special instance of a more general Rademacher-type theorem for intrinsic Lipschitz functions with normal targets in Carnot groups, \cite[Theorem 1.1]{MR4277829}.

\begin{theorem}[Intrinsic Rademacher's Theorem]\label{1.5}  Let  $\Hn = \V \ltimes \W$ be a splitting into complementary horizontal and vertical subgroups. 
If  $U\subset\mathbb{V}$ is a Borel set and $\varphi: U \subset \V \to \W$ is intrinsic Lipschitz, then $\varphi$ is 
intrinsically differentiable
almost everywhere in $U$.
\end{theorem}

We will also use the second part of \cite[Theorem 1.1]{MR4277829}, which provides an area formula:

\begin{theorem}[Area Formula]\label{af}  Let  $\Hn = \V \ltimes \W$ be a splitting into complementary horizontal and vertical subgroups. 
Let $\varphi: U \subset \V \to \W$ be an intrinsic Lipschitz map on a Borel set, and let $\Phi: U \subset \V \to \Hn$ be its graph map. If $\dim_H(\mathbb{V}) = k\leq n$, then
\begin{equation*}
    \mathcal{H}^k(\Phi(V)) = \int_V J(d\Phi_v) d\mathcal{H}^k(v), \quad V\subset U\text{ Borel},
\end{equation*}
where $J(d\Phi_v)$ is the Jacobian given by
\begin{equation}\label{eq:JacFormula}
    J(d\Phi_v) = \frac{\mathcal{H}^k(d\Phi_v(B_\V(0,1)))}{\mathcal{H}^k(B_\V(0,1))}
\end{equation}
at points of intrinsic differentiability.
Here $B_\V(0,1)$ is the unit ball in $\V$, and $d\Phi_v(x)=x \cdot d\varphi_v(x)$.
\end{theorem}

The area formula in \cite[Theorem 1.1]{MR4277829} is stated in terms of the \emph{Pansu differential} of $\Phi$ but the version formulated here with the intrinsic differential is an immediate consequence as explained in \cite[Remark 3.1]{MR4277829}. By the standard argument using characteristic functions, we obtain the following corollary:

\begin{corollary}\label{cov} 
If $f: \mathbb{H}^n\to \R$ is a Borel function, and $\Phi$ is the graph of $\varphi$, $ U\subset\V$ as in Theorem \ref{af}, then the following change of variables formula holds:
\begin{equation*}
    \int_{\Phi(V)} f(y) d\mathcal{H}^k(y) = \int_V f(\Phi(v)) J(d\Phi_v) d\mathcal{H}^k(v),\quad V\subset U \text{ Borel}.
\end{equation*}
\end{corollary}

Partially defined intrinsic Lipschitz functions from horizontal to vertical subgroups in $\mathbb{H}^n$ admit intrinsic Lipschitz extensions according to \cite[Theorem 1.2]{MR4375018}, which we recall here.

\begin{theorem}[Intrinsic Lipschitz extension]\label{3.2}
Let $n \in \mathbb{N}$, $k\in \{1,\dots,n\}$. Assume that $\V$ is a $k$-dimensional horizontal subgroup of $\Hn$ with complementary vertical subgroup $\W$. Then for every $L \ge 0$, there exists a constant $L'(L,k,n)>0$ such that every intrinsic $L$-Lipschitz function $\varphi: E \subset \V \to \W$ defined on a set $E\subset\V$ can be extended to an intrinsic $L'$-Lipschitz function $\bar{\varphi}: \V \to \W$.
\end{theorem}

A Euclidean Lipschitz function can easily be modified so that it vanishes outside a ball $B(0,2r)$ and agrees with the original function inside $B(0,r)$ while maintaining the Lipschitz property. This ``truncation approach'' is useful, for instance when deducing a Dorronsoro theorem for Lipschitz functions  from a corresponding result for $W^{1,2}(\mathbb{R}^n)$ Sobolev functions. Due to the non-commutativity of the group law, analogous arguments in the Heisenberg group are more subtle. Nonetheless, the next proposition enables us to use the truncation approach for intrinsic Lipschitz functions on horizontal subgroups.

\begin{proposition}\label{loc} Let $n \in \mathbb{N}$, $k\in \{1,\dots,n\}$. Assume that $\V$ is a $k$-dimensional horizontal subgroup of $\Hn$ with complementary vertical subgroup $\W$.
Let $\varphi:
\V\to\W$ be an intrinsic $L$-Lipschitz function. 
If $\varphi(0)=0$,
then,
for all $
r > 0$, there exists $\tilde{\varphi}_{r}: \V \to \W$ intrinsically $L'$-Lipschitz such that
\[
    \tilde{\varphi}_{r} = \begin{cases}
        \varphi & \text{on } B(0,r) \\
        0 & \text{on } \V \setminus B(0, 2r)
    \end{cases}
\]
where $B(x,r)$ is the ball centered in $x$ with radius $r$ in $\V$ and $L'(L,k,n)>0$ is a constant depending on $L$, $k$, and $n$. 
\end{proposition}

\begin{proof}
Let $E = B(0,r) \cup (\V \setminus B(0, 2r))$. Define a function $\hat{\varphi}_r:E\to\W$
\[
    \hat{\varphi}_{r} := \begin{cases}
        \varphi & \text{on } B(0,r) \\
        0 & \text{on } \V \setminus B(0, 2r)
    \end{cases}
\]
For simplicity, we denote $\hat{\varphi}_{r}$ by $\hat{\varphi}$.
We first show that $\hat{\varphi}$ is intrinsically $\hat{L}$-Lipschitz on $E$ for some $\hat{L}(L)>0$. 
Fix $x, y \in E = B(0,r) \cup (\mathbb{V} \setminus B(0,2r))$. We need to verify that
\begin{displaymath}
\|y^{-1} \cdot x \cdot \hat{\varphi}(x)^{-1} \cdot x^{-1} \cdot y\cdot\hat{\varphi}(y)\|\leq \hat L \|x^{-1}\cdot y\|,
   \end{displaymath}
   and an analogous inequality with the roles of $x$ and $y$ interchanged.

The cases where both $x,y\in B(0,r)$ or $x,y\in \mathbb{V} \setminus B(0,2r)$ hold trivially as $\hat{\varphi}$ coincides with either the intrinsic $L$-Lipschitz function $\varphi$ or the zero function.

Assume $x \in B(0,r)$ and $y \in \mathbb{V} \setminus B(0,2r)$, then $\hat{\varphi}(x) = \varphi(x)$ and $\hat{\varphi}(y) = 0$. 
By the triangle inequality, $\|y^{-1} \cdot x\| \ge \|y\| - \|x\| \ge 2r - r = r$, which implies
$$ \|x\| \le \|y^{-1} \cdot x\| \quad \text{and} \quad \|y\| \le \|x\| + \|y^{-1} \cdot x\| \le 2\|y^{-1} \cdot x\|. $$
Since $\varphi$ is intrinsic $L$-Lipschitz and $\varphi(0)=0$ 
then 
$$
\|x \cdot \varphi(x)^{-1} \cdot x^{-1}\|=\|0^{-1} \cdot x \cdot \varphi(x)^{-1} \cdot x^{-1} \cdot 0 \cdot \varphi(0)\| \le L \|0^{-1} \cdot x\| = L\|x\|.
$$ 
Applying the triangle inequality, we get
\begin{align*}
\|y^{-1} \cdot x \cdot \hat{\varphi}(x)^{-1} \cdot x^{-1} \cdot y\cdot\hat{\varphi}(y)\| &=\|y^{-1} \cdot x \cdot \varphi(x)^{-1} \cdot x^{-1} \cdot y\|\\
&\le \|y^{-1}\| + \|x \cdot \varphi(x)^{-1} \cdot x^{-1}\| + \|y\| \\
&\le 2\|y\| + L\|x\| \\
&\le 4\|y^{-1} \cdot x\| + L\|y^{-1} \cdot x\| = (L+4)\|y^{-1} \cdot x\|.
\end{align*}
To verify the intrinsic Lipschitz property of $\hat \varphi$ on $E$, we need to verify an analogous inequality where $x$ and  $y$ are swapped (but still $x\in B(0,r)$ and $y\in \mathbb{V}\setminus B(0,2r)$). In this case we have
\begin{align*}
\|x^{-1} \cdot y \cdot \hat{\varphi}(y)^{-1} \cdot y^{-1} \cdot x\cdot\hat{\varphi}(x)\| &= \|x^{-1}\cdot x \cdot  \varphi(x)\|\\
&=\|\varphi(x)\|\leq L \|x\|\leq L \|y^{-1}\cdot x\|= L\|x^{-1}\cdot y\|,
\end{align*}
where the bound for $\|\varphi(x)\|$ is a direct consequence of the intrinsic Lipschitz property of $\varphi$ and the assumption $\varphi(0)=0$ (see \eqref{def_intLip}).
Therefore, by Proposition \ref{1.2}, we conclude that $\hat \varphi$ is intrinsic $\hat L$-Lipschitz with $\hat{L} \le L+4$.

Finally, by Theorem \ref{3.2}, $\hat{\varphi}$ can be extended to an intrinsic $L'$-Lipschitz function $\tilde{\varphi}:\V\to\W$ with $L'$ depending only on $k$, $n$, and $\hat L$. This completes the proof.
\end{proof}

\section{Low-dimensional intrinsic Lipschitz graphs and their horizontal projections}\label{s:LowDimiLG}

In this section we show that low-dimensional intrinsic Lipschitz graphs in $\mathbb{H}^n$ project to Euclidean Lipschitz graphs with certain additional properties in $\mathbb{R}^{2n}$ (Theorem \ref{horcon}), and conversely, all Euclidean Lipschitz graphs with this property can be lifted to intrinsic Lipschitz graphs (Theorem \ref{t:Lift}). Various characterizations of low-dimensional intrinsic Lipschitz graphs in $\mathbb{H}^n$ have already appeared in \cite{MR3587666,MR4277829,MR4375018,MR4377000}, but here we focus on a characterization purely in terms of the first, horizontal, components of the function. As a related, but simpler problem, we make explicit the correspondence between \emph{intrinsic affine} functions in $\mathbb{H}^n$, and \emph{isotropic affine} functions in $\mathbb{R}^{2n}$, see Proposition \ref{ahp}. This provides also an alternative proof for the main part of Theorem \ref{horcon}, as explained in Remark \ref{r:AlternativeProof}.

\subsection{Reduction to the standard setting}\label{ss:RedStd}

The conditions for Euclidean Lipschitz graphs in $\mathbb{R}^{2n}$ to arise as projections of intrinsic Lipschitz graphs in $\mathbb{H}^n$ are particularly easy to state for graphs over the plane $\mathbb{R}^k \times \{0\}$, where $1\leq k\leq n$. We first justify that, thanks to rotations induced by unitary maps of $\mathbb{R}^{2n}$, it is not restrictive to consider graphs over this specific plane. Recall that the unitary group is $U(n)=\text{Sp}(n)\cap O(2n,\R)$, where $\text{Sp}(n)$ denotes the group of symplectic matrices in $\R^{2n}$ (namely, matrices preserving the symplectic form $\omega$), while $O(2n,\R)$ is the orthogonal group of matrices in $\R^{2n}$ (preserving the standard inner product).
See for instance \cite[Section 2]{MR2955184} for some relevant properties of~$U(n)$.

For $U\in U(n)$, denote $R_U(z,t):= (Uz,t)$. 
It is well known that, for every $U\in U(n)$, the map $R_U$ is a group isomorphism and an isometry of $\Hn$. Indeed, 
\begin{align*}R_U((z,t)\cdot(z',t'))&=R_U(z+z',t+t'+\tfrac{1}{2}\omega(z,z'))=(U(z+z'),t+t'+\tfrac{1}{2}\omega(z,z'))\\&= (Uz+Uz',t+t'+\tfrac{1}{2}\omega(Uz,Uz'))=(Uz, t)\cdot(Uz',t')=R_U(z,t)\cdot R_U(z',t'),\end{align*}
where we used the fact that $U\in U(n)\subset \text{Sp}(n)$ preserves the symplectic form $\omega$. Clearly, $R_U$ is also a bijection, since the same holds for $U$. In addition, by definition of \emph{Kor\'{a}nyi norm} $\knorm{\cdot}$, it is immediate that \[\|R_U(z,t)\|=\|(Uz,t)\|=\|(z,t)\|,\] since $U\in U(n)\subset O(2n,\R)$ preserves the Euclidean norm on $\R^{2n}$. It follows that $R_U$ is an isometry of $\Hn$.
\begin{lemma}\label{l:RotRed}
    Let $1\leq k\leq n$ and let $\Gamma$ be an intrinsic $L$-Lipschitz graph over a $k$-dimensional horizontal subgroup. Then there exists $U\in U(n)$ such that $R_U(\Gamma)$ is the intrinsic graph of an intrinsic $L$-Lipschitz function $\mathbb{V}_0\to \mathbb{W}_0$, where
    $
    \mathbb{V}_0:=\{(x,0):\,x\in \mathbb{R}^k\}$ and $\mathbb{W}_0=\{(0,y,t):\, y\in \mathbb{R}^{2n-k},\,t\in\mathbb{R}\}.
$
\end{lemma}

\begin{proof}
By assumption, there exists $\mathbb{V}\in \mathcal{V}_k^0$ with complementary vertical subgroup $\mathbb{W}$, and an intrinsic $L$-Lipschitz map $\varphi:\mathbb{V}\to \mathbb{W}$ such that $\Gamma$ is the intrinsic graph of $\varphi$, that is
\begin{displaymath}
    \Gamma=\gr(\varphi)= \{\Phi(v)=v\cdot \varphi(v):\,v\in \mathbb{V}\}.
\end{displaymath}
We have $\mathbb{V}=V \times\{0\}$, for a $k$-dimensional isotropic subspace $V$ of $\mathbb{R}^{2n}$. Now there exists $U\in U(n)$ such that $U(V)=\mathbb{R}^k \times \{0\}\subset\mathbb{R}^{2n}$, see for instance \cite[Lemma 2.1]{MR2955184}. 
Since $R_U$ is a group isomorphism
, we have, for all $v\in \mathbb{V}$, 
\begin{displaymath}
    R_U(\Phi(v))= R_U(v) \cdot R_U(\varphi(v))= R_U(v) \cdot [R_U \circ \varphi \circ R_U^{-1}](R_U(v))=:
    R_U(v) \cdot \varphi_U(R_U(v))=:\Phi_U(R_U(v)).
\end{displaymath}
Therefore,
\begin{displaymath}
\Phi_U(v)= R_U(\Phi(R_U^{-1}(v))),\quad v\in \mathbb{V}_0.
\end{displaymath}
It remains to check that $\varphi_U= R_U \circ \varphi \circ R_U^{-1}$ is intrinsic $L$-Lipschitz. To this end, consider, for $v,v'\in \mathbb{V}_0$, the expression
\begin{align*}
    \|{v'}^{-1}\cdot v \cdot \Phi_U(v)^{-1}\cdot \Phi_U(v')\|&= \|{v'}^{-1}\cdot v \cdot R_U(\Phi(R_U^{-1}(v)))^{-1}\cdot R_U(\Phi(R_U^{-1}(v')))\|\\
    &=\|{R_U^{-1}(v')}^{-1}\cdot R_U^{-1}(v) \cdot \Phi(R_U^{-1}(v))^{-1}\cdot \Phi(R_U^{-1}(v'))\|
    \\&\leq L \|{R_U^{-1}(v')}^{-1}\cdot R_U^{-1}(v)\|= L \|{v'}^{-1}\cdot v\|,
\end{align*}
where we used the definition of $\Phi_U$, the fact that $R_U$ is an isometry and a group isomorphism of $\Hn$ and the Lipschitz property of $\Gamma$, concluding the proof.
\end{proof}

By the above lemma, we can now fix our setting. We choose the horizontal subgroup $\V_0 = \R^k \times \{0\}^{2n-k+1}$ and the vertical subgroup  $\W_0 = \{0\}^k \times \R^{2n-k+1}$.  

\begin{remark}\label{r:FixedGroups}
The subgroup $\V_0$ with the Kor\'anyi metric $d$ is isometrically isomorphic to $(\mathbb{R}^k,+)$ with the Euclidean distance via the mapping $(x,0)\mapsto x$. We will sometimes abuse notation and identify $\mathbb{V}_0$ with $\mathbb{R}^k$. For all $x = \left(x_1,\cdots, x_k, 0, \cdots, 0\right) \in \V_0$, we also write $x=(x_1,\cdots, x_k)\in \mathbb{R}^k$. A mapping $\varphi: \V_0 \to \W_0$  will be written in coordinates as
\begin{equation}\label{ilg}
    \varphi(x) = \left(0,\cdots, 0, \varphi_1(x),\cdots, \varphi_{2n-k}(x), \varphi_{2n-k+1}(x)\right),
\end{equation}
or simply $\varphi(x) = \left(\varphi_1(x),\cdots, \varphi_{2n-k}(x), \varphi_{2n-k+1}(x)\right)$. We sometimes refer to the components $\varphi_i$ for $i=1,\ldots,2n-k$ as the \emph{horizontal components} and we call $\varphi_{2n-k+1}$ the \emph{vertical component}.
Finally, the graph map $\Phi(x)=x\cdot \varphi(x)$ of $\varphi$ is:
\begin{equation*}
    \Phi(x) = \left(x_1,\cdots, x_k, \varphi_1(x),\cdots, \varphi_{2n-k}(x), \varphi_{2n-k+1}(x) + \tfrac{1}{2} \sum_{i=1}^k x_i \varphi_{n-k+i}(x)\right)
\end{equation*}
where $\varphi_i: \V_0\cong\R^k \to \R$, $1\le i\le 2n-k+1$. 
\end{remark}


\subsection{Projections and liftings of planes}\label{s:PlaneProjLift}

Horizontal planes in $\mathbb{H}^n$, as in Definition \ref{hp}, are nothing but flat, low-dimensional intrinsic Lipschitz graphs. The characterization we will give for general intrinsic Lipschitz graphs takes a particularly simple form in the flat case. Here, the characterization (Proposition \ref{ahp}) follows directly from the relation between horizontal planes in $\mathbb{H}^n$ and (translates of) isotropic subspaces of $\mathbb{R}^{2n}$. This also has applications to intrinsic differentials, see Proposition \ref{p:IntrDiffForm}.

\begin{definition}[Isotropic affine]\label{defisoa} 
Let $1\leq k\leq n$ and let $A = (A_1, \dots, A_{2n-k}) : \mathbb{R}^k \to \mathbb{R}^{2n-k}$ be an affine map. We say that $A$ is \emph{isotropic affine} 
if its graph is a translation of an isotropic subspace, that is 
\[
\{(x,A(x)):x\in\R^k\}=V+c
\]
for some $k$-dimensional isotropic subspace $V\subset\R^{2n}$ and $c\in\R^{2n}$. If $c=0$, we also call $A$ \emph{isotropic linear}.
\end{definition}

\begin{proposition}\label{isoa}
Let $1\leq k\leq n$ and $A = (A_1, \dots, A_{2n-k}) : \mathbb{R}^k \to \mathbb{R}^{2n-k}$ be an affine map, that is,
\begin{equation}
    A(x) = 
    \begin{pmatrix}
        a_{1,1} & \cdots & a_{1,k} \\
        \vdots & \ddots & \vdots \\
        a_{2n-k,1} & \cdots & a_{2n-k,k}
    \end{pmatrix}
    \begin{pmatrix}
        x_1 \\ \vdots \\ x_k
    \end{pmatrix}
    + 
    \begin{pmatrix}
        c_1 \\ \vdots \\ c_{2n-k}
    \end{pmatrix}.
\end{equation} 
Then, it is isotropic affine if and only if for all $j, l \in \{1, \dots, k\}$,
\begin{equation}\label{iso}
    a_{n-k+l, j} - a_{n-k+j, l} + \sum_{i=1}^{n-k} (a_{n+i, j} \cdot a_{i, l} - a_{n+i, l} \cdot a_{i, j}) = 0.
\end{equation}
For $k = n$, the summation is zero, and we simply have $a_{l,j} = a_{j,l}$. If instead $k=1$, then \eqref{iso} is trivially
 satisfied: thus, any affine map $A:\R \to \R^{2n-1}$ is isotropic affine. \end{proposition}

\begin{proof}
Let $A(x) = A_0(x) + c$ be an affine map, where $A_0: \mathbb{R}^k \to \mathbb{R}^{2n-k}$ is its linear part given by $A_{0, p}(x) = \sum_{m=1}^k a_{p,m} x_m$ for $1 \le p \le 2n-k$, and $c \in \mathbb{R}^{2n-k}$ is a constant vector.

The graph of $A$ is a translation of the linear subspace $V = \{(x, A_0(x)) : x \in \mathbb{R}^k\} \subset \mathbb{R}^{2n}$. Thus, by Definition \ref{defisoa}, $A$ is isotropic affine if and only if $V$ is an isotropic subspace, which means that the symplectic form vanishes for any two points in $V$:
\begin{equation*}
    \omega(z_1, z_2) = 0 \quad \text{for all } z_1, z_2 \in V.
\end{equation*}
For any $x, y \in \mathbb{R}^k$, let $z_1=(x,A_0(x)), z_2=(y,A_0(y)) \in V$ be the points corresponding to $x$ and $y$. 
Evaluating the standard symplectic form on $z_1$ and $z_2$, we split the sum into the first $k$ terms and the remaining $n-k$ terms and get
\begin{align*}
    \omega(z_1, z_2) 
    &= \sum_{j=1}^k \left[ A_{0, n-k+j}(y) x_j - y_j A_{0, n-k+j}(x) \right] + \sum_{i=1}^{n-k} \left[ A_{0, n+i}(y) A_{0, i}(x) - A_{0, i}(y) A_{0, n+i}(x) \right]\\
    &= \sum_{j=1}^k \sum_{l=1}^k \left[ a_{n-k+l, j} - a_{n-k+j, l} + \sum_{i=1}^{n-k} (a_{n+i, j} a_{i, l} - a_{i, j} a_{n+i, l}) \right] y_j x_l.
\end{align*}
Then the condition $\omega(z_1, z_2) = 0$ holds for all $x, y \in \mathbb{R}^k$ if and only if for all $j, l \in \{1, \dots, k\}$
\begin{equation*}
    a_{n-k+l, j} - a_{n-k+j, l} + \sum_{i=1}^{n-k} (a_{n+i, j} a_{i, l} - a_{i, j} a_{n+i, l}) = 0.
\end{equation*}
This precisely matches the algebraic condition (\ref{iso}).
\end{proof}

We define the following generalization of intrinsic linearity (Definition \ref{linearity}):

\begin{definition}[Intrinsic affine]\label{intrinsicaffine}
    Let $\mathbb{V}$ be a $k$-dimensional horizontal subgroup of $\mathbb{H}^n$ with complementary vertical subgroup $\mathbb{W}$. We say that $A:\mathbb{V}\to \mathbb{W}$ is \emph{intrinsic affine} if its intrinsic graph is a $k$-dimensional horizontal plane.
\end{definition}

\begin{remark}\label{r:VertTranslIntrAff}
   If $A:\mathbb{V}\to \mathbb{W}$ is intrinsic affine, it remains intrinsic affine if a constant $c\in \mathbb{R}$ is added to its last, vertical, component. This is because the family of $k$-dimensional horizontal planes is preserved under left translates and left translating the intrinsic graph of $A$ by an element  $(z,t)=(0,c)\in \mathbb{R}^{2n}\times \mathbb{R}$ corresponds to adding constant $c$ to the vertical component of $A$. 
 \end{remark}

Specializing the definition of intrinsic affine to $\mathbb{V}_0=\mathbb{R}^k \times \{0\}$ and $\mathbb{W}_0=\{0\}\times \mathbb{R}^{2n+1-k}$ as in Remark \ref{r:FixedGroups}, we obtain:

\begin{proposition}[Lifting isotropic affine maps and projecting intrinsic affine maps]\label{ahp} Let $1\leq k\leq n$.
\begin{enumerate}
    \item If $A = (A_1, \dots, A_{2n-k}): \mathbb{R}^k \to \mathbb{R}^{2n-k}$ is an isotropic affine map, then there exists a map $A_{2n-k+1} : \mathbb{R}^k \to \mathbb{R}$ such that the map $L = (A, A_{2n-k+1}): \mathbb{V}_0 \to \mathbb{W}_0$ is intrinsic affine. Moreover, if $A$ is an isotropic linear map, then $A_{2n-k+1}$ can be taken equal to
    \begin{equation}\label{2n-k+1}
    A_{2n-k+1}(x) = -\tfrac{1}{2} \sum_{j=1}^k x_j A_{n-k+j}(x),
\end{equation}
and this is the uniquely defined function $A_{2n-k+1}$ which makes $L=(A,A_{2n-k+1})$  intrinsic linear.

    \item Conversely, if $L = (A_1, \dots, A_{2n-k}, A_{2n-k+1}): \mathbb{V}_0 \to \mathbb{W}_0$ is an intrinsic affine map, then the projection of $L$ onto its first $2n-k$ components, $A = (A_1, \dots, A_{2n-k}): \mathbb{R}^k \to \mathbb{R}^{2n-k}$, is an isotropic affine map. Moreover, if $L$ is intrinsic linear, then $A$ is isotropic linear. 
\end{enumerate}
\end{proposition}

\begin{proof}\quad

 \textbf{(1)} First assume that $A=(A_1,\ldots,A_{2n-k}):\mathbb{R}^k \to \mathbb{R}^{2n-k}$ is isotropic linear, that is, there exists a $k$-dimensional isotropic subspace $V$ such that
\begin{displaymath}
    \{(x,A(x)):\,x\in \mathbb{R}^k\}=V. 
\end{displaymath}
Now we define
 \begin{equation}\label{eq:A_{2n-k+1}}
     A_{2n-k+1}(x):=-\tfrac{1}{2}\sum_{j=1}^k x_j A_{n-k+j}(x)
 \end{equation}
 and consider $L:=(A,A_{2n-k+1}):\mathbb{V}_0\to \mathbb{W}_0$. We claim that the intrinsic graph $\mathrm{gr}(L)$ is a $k$-dimensional horizontal subgroup, and $L$ is therefore intrinsic linear.  Indeed, 
 \begin{align}\mathrm{gr}(L)
 &=\{x\cdot L(x):x\in\mathbb{V}_0\}
 \notag\\ 
 &=\left\{\left(x, A(x), A_{2n-k+1}(x)+\tfrac{1}{2}\omega\left(\begin{pmatrix}x\\0\end{pmatrix}, \begin{pmatrix}0\\A(x)\end{pmatrix}\right)\right):x\in\R^k\right\}\notag \\ 
 &=\left\{\left(x, A(x), A_{2n-k+1}(x)+\tfrac{1}{2}\sum_{j=1}^k x_j A_{n-k+j}(x)\right):x\in\R^k\right\}\notag\\
 &=\left\{\left(x, A(x), 0\right):x\in\R^k\right\} =V\times \{0\}. \label{eq:grL}
 \end{align}
Since $V$ is isotropic, this is a horizontal subgroup. Moreover, the definition of $A_{2n-k+1}$ in \eqref{eq:A_{2n-k+1}} is the only possible choice for which $\mathrm{gr}(L)$ becomes a horizontal subgroup.
This concludes the proof of (1) in the case $A$ is isotropic linear.

\medskip
If $A$ is only isotropic affine, then we have the existence of an isotropic subspace $V$ and a vector $c$ in $\mathbb{R}^{2n}$ such that
\begin{displaymath}
    \{(x,A(x)):\,x\in \mathbb{R}^k\}=V+c. 
\end{displaymath}
Without loss of generality we may assume that $c=(0,c_0)$ for some $c_0 \in \mathbb{R}^{2n-k}$. Then $A_0:= A-c_0$ is isotropic linear. By the previous discussion, there exists an intrinsic linear map $L_0:\mathbb{V}_0 \to \mathbb{W}_0$ whose horizontal components agree with $A_0$ and whose  intrinsic graph $\mathrm{gr}(L_0)$ equals the horizontal subgroup $V \times \{0\}$. The translated set $(0,c_0,0)\cdot [V\times \{0\}]$ is again a  $k$-dimensional horizontal plane. Moreover, as a left-translate of an intrinsic graph over $\mathbb{V}_0$ it is the intrinsic graph of a function $L:\mathbb{V}_0\to \mathbb{W}_0$, the horizontal components of which agree with $A$ as desired.

This concludes the proof of part (1) of Proposition \ref{ahp}.

\medskip

\textbf{(2)} Conversely, assume that $L = (A_1, \dots, A_{2n-k}, A_{2n-k+1}): \mathbb{V}_0 \to \mathbb{W}_0$ is an intrinsic affine map. We want to show that its projection $A = (A_1, \dots, A_{2n-k})$ is an isotropic affine map. Since $\mathrm{gr}(L)$ is a $k$-dimensional horizontal plane, it can be written as $p\cdot \mathbb{V}$, with $\mathbb{V}$ a horizontal subgroup and $p=(c,t)\in \mathbb{R}^{2n}\times \mathbb{R}$. Then
\[\mathbb{V}=p^{-1}\cdot \mathrm{gr}(L)=\{(x_1-c_1,\dots, x_k-c_k,A_1(x)-c_{k+1},\dots,A_{2n-k}(x)-c_{2n}, \ast): x\in\R^k\}.\]
Since $\mathbb{V}$ is a horizontal subgroup, we must have  $\ast=0$. Moreover $A_1,\dots,A_{2n-k}$ are forced to be affine. Finally, $A=(A_1,\dots, A_{2n-k})$ is isotropic affine by Definition \ref{defisoa} since $\mathbb{V}$ is of the form $\mathbb{V}=V \times \{0\}$ for a $k$-dimensional isotropic subspace $V\subset \mathbb{R}^{2n}$.
If $L$ is intrinsic linear, then this argument with $p=(0,0)$ shows that $A$ is isotropic linear.
\end{proof}

 Applied to intrinsic differentials, which are intrinsic linear, Proposition \ref{ahp} yields the following.

\begin{proposition}\label{p:IntrDiffForm} Let $\mathbb{V}_0=\{(x,0,\ldots,0):x\in \mathbb{R}^k\}$ be the $k$-dimensional horizontal subgroup with complementary vertical subgroup $\mathbb{W}_0=\{(0,y,t):\, y\in \mathbb{R}^{2n-k},t\in\mathbb{R}\}$, and let $U\subset \mathbb{V}_0$ be a Borel set. Assume that $\varphi=(\varphi_1,\dots,\varphi_{2n-k+1}):U\subset \V_0 \to \W_0$ is intrinsically differentiable at a density point $v\in U$ of $U$. Then $\varphi_i:U\subset \mathbb{R}^k \to \mathbb{R}$ is differentiable at $v$ for every $i=1,\dots, 2n-k$. Moreover, the intrinsic differential of $\varphi$ at $v$ is given by \begin{equation*}
        d\varphi_v(x) = \left(0,\cdots 0, L_v^1(x),\cdots, L_v^{2n-k}(x), L_v^{2n-k+1}(x)\right),\qquad x=(x_1,\cdots,x_k)\in \V_0
    \end{equation*}
    where
    \[
   \left( \begin{array}{c} L_v^1(x)\\ \vdots \\ L_v^{2n-k}(x) \end{array} \right)
   =   \begin{pmatrix}
        \partial_1 \varphi_1(v) & \ldots & \partial_{k} \varphi_1(v) \\
        \vdots & \ddots & \vdots \\
        \partial_1 \varphi_{2n-k}(v) & \ldots & \partial_{k} \varphi_{2n-k}(v)
        \end{pmatrix}                   \left( \begin{array}{c} x_1\\ \vdots \\ x_k \end{array} \right)
   \] 
and
\[
L_v^{2n-k+1}(x)=-\tfrac{1}{2} \sum_{j=1}^k x_j L_v^{n-k+j}(x)=-\tfrac{1}{2} \sum_{j=1}^k\sum_{i=1}^k\partial_i\varphi_{n-k+j}(v)x_i x_j.
\]
\end{proposition}

\begin{proof} 
We write the intrinsic differential of $\varphi$ at $v \in \mathbb{V}_0$ as an intrinsic linear map $d\varphi_v : \mathbb{V}_0 \to \mathbb{W}_0$. For any $x = (x_1, \dots, x_k, 0, \dots, 0) \in \mathbb{V}_0$, its coordinate representation is given by:
\begin{equation*}
    d\varphi_v(x) = \left(0, \dots, 0, L_v^1(x), \dots, L_v^{2n-k}(x), L_v^{2n-k+1}(x)\right).
\end{equation*}

We begin by computing the first $2n-k$ components. By Definition \ref{1.4}, the intrinsic differential satisfies:
\begin{equation*}
    \| d\varphi_v(x)^{-1} \cdot x^{-1} \cdot \varphi(v)^{-1} \cdot x \cdot \varphi(v \cdot x) \| = o(\|x\|) \quad\text{ as }  x\to 0.
\end{equation*}
Since $\omega(v, x) = 0$ for $v, x \in \mathbb{V}_0$, we have $v \cdot x = v + x$. The first $2n$ components of the group product in $\mathbb{H}^n$ are simply vector addition. Extracting the $(k+i)$-th component of the expression inside the norm (which corresponds to the $i$-th component of $\mathbb{W}_0$, where $1 \le i \le 2n-k$), we obtain
\begin{equation*}
    -L_v^i(x) - x_{k+i} - \varphi_i(v) + x_{k+i} + \varphi_i(v+x) = \varphi_i(v+x) - \varphi_i(v) - L_v^i(x).
\end{equation*}
Then, the Kor\'{a}nyi norm gives us
\begin{equation}\label{diffofcomponents}
    |\varphi_i(v+x) - \varphi_i(v) - L_v^i(x)| = o(\|x\|)=o(|x|).
\end{equation}
Furthermore, by the intrinsic linearity of $d\varphi_v$ (Definition \ref{linearity}), $d\varphi_v(\delta_\lambda x) = \delta_\lambda (d\varphi_v(x))$, which implies $L_v^i(\lambda x) = \lambda L_v^i(x)$. A similar argument shows that $L_v^i(x+y)=L_v^i(x)+L_v^i(y)$ for every $x,y\in\V_0$, $1\leq i\leq 2n-k$. Hence $L_v^i$ is a linear map and \eqref{diffofcomponents} implies that $\varphi_i$ is differentiable at $v$ with
\begin{equation}
    L_v^i(x) = \sum_{j=1}^k \partial_j \varphi_i(v) x_j \quad \text{for } 1 \le i \le 2n-k.
\end{equation}
Finally, the formula for the last component $L_v^{2n-k+1}(x)$ follows from Proposition \ref{ahp} and the fact that $d\varphi_v$ is intrinsic linear.
\end{proof}

Proposition 
\ref{p:IntrDiffForm} can be used to control the Jacobian  in the area formula for an intrinsic Lipschitz function $\varphi:U\subset \mathbb{V}_0 \to \mathbb{W}_0$, Theorem \ref{af}. We recall from \eqref{eq:JacFormula} that
\begin{equation*}
    J(d\Phi_v) = \frac{\mathcal{H}^k(d\Phi_v(B_{\V_0}(0,1)))}{\mathcal{H}^k(B_{\V_0}(0,1))}
\end{equation*}
with $d\Phi_v(x)=x\cdot d\varphi_v(x)$. Moreover, the horizontal components of $\varphi$ are Euclidean Lipschitz, see for instance Proposition \ref{horcon}. 

\begin{lemma}[Bound for the Jacobian]\label{rem_jacobian} 
Let $1\leq k\leq n$, and let $U\subset \mathbb{V}_0$ and $\mathbb{W}_0$ be as in Proposition
\ref{p:IntrDiffForm}.
Assume that $\varphi=(\varphi_1,\dots,\varphi_{2n-k+1}):U\subset \V_0 \to \W_0$ is intrinsic Lipschitz and $v\in U$ is a point of intrinsic differentiability. Then there exists a constant $C$, depending only on $k$, $n$, and the Euclidean Lipschitz constants $\mathrm{Lip}(\varphi_1),\ldots,\mathrm{Lip}(\varphi_{2n-k})$ such that
\begin{displaymath}
    J(d\Phi_v)\leq C.
\end{displaymath}
\end{lemma}

\begin{proof}
We have by Proposition
\ref{p:IntrDiffForm} that
\[
d\Phi_v(x) = x\cdot d\varphi_v(x) = (x_1, \dots, x_k, L_v^1(x), \dots, L_v^{2n-k}(x), 0)
\]
with $L_v^j(x)=\sum_{i=1}^k \partial_i \varphi_j(v) x_i$. Let $F_v := d\Phi_v : \mathbb{R}^k \to \mathbb{R}^{2n+1}$, then
\begin{equation}\label{eq:PrepJacComp}
\mathcal{H}^k(d\Phi_v(B_{\mathbb{V}_0}(0,1))) = \mathcal{H}^k_{\mathrm{Eucl}}(DF_v(B_{\mathbb{R}^k}(0,1))) = 
\sqrt{\det(DF_v^T \cdot DF_v)}\, \mathcal{H}^k_{\mathrm{Eucl}}(B_{\mathbb{R}^k}(0,1)),
\end{equation}
where $DF_v$ denotes the Jacobi matrix of $F_v$ in the usual sense.
Here we used for the first identity that 
$d\Phi_v(B_{\V_0}(0,1)))$ is contained in the $k$-dimensional horizontal subgroup  $\V:=d\Phi_v(\V_0)$, and $d(v,v')=|v-v'|$ for all $v,v'\in \V$. Moreover, we also used that 
$(\mathbb{V}_0,d)$ and the Euclidean space $\mathbb{R}^k$ are isometric under the identification of $(x_1,\ldots,x_k,0,\ldots,0)\in \mathbb{H}^n$ with $(x_1,\ldots,x_k)\in \mathbb{R}^k$.
Since
\[
DF_v = \begin{pmatrix}\mathbb{I}_k\\D(\varphi_1,\ldots,\varphi_{2n-k})(v)\\0\end{pmatrix},
\] with $\mathbb{I}_k$ being the $(k\times k)$ identity matrix, the claim in the lemma follows from \eqref{eq:PrepJacComp}.
\end{proof}


\subsection{Projections and liftings of Lipschitz graphs}\label{s:LipGraphProjLift}

The goal of this section is a version of Proposition \ref{ahp} for general $k$-dimensional Lipschitz graphs instead of $k$-dimensional planes. 

\subsubsection{Isotropic and contact maps} The properties of low-dimensional intrinsic Lipschitz graphs in $\mathbb{H}^n$ discussed in later sections fit into the general theory of \emph{isotropic} and \emph{weakly contact} maps, so we first recall the relevant terminology.

\begin{definition}[Weakly contact]\label{d:WeaklyContact}  Let $1\leq k\leq n$.
A map $F = (f_1, \dots, f_n, g_1, \dots, g_n, h) \in
W_{\rm loc}^{1,1}(\mathbb{R}^k,\mathbb{R}^{2n+1})
$  is \emph{weakly contact} if
    \begin{equation}\label{dh}
\frac{\partial h}{\partial x_m}=        \frac{1}{2} \sum_{i=1}^n \left( f_i \frac{\partial g_i}{\partial x_m} - g_i \frac{\partial f_i}{\partial x_m} \right),\quad \text{a.e.}, \quad \text{for all } 1 \le m \le k.
    \end{equation}
\end{definition}

\begin{remark}\label{r:ContactEq}
    Equation \eqref{dh} is known as the \emph{weak contact equation}. If $F$ is differentiable at a point $x$,  it means that the image of the differential of $F$ at $x$ is contained in the horizontal tangent space of $\mathbb{H}^n$ at $x$.
      The horizontal tangent space of $\mathbb{H}^n$ is the kernel of the standard contact form, which is given by $\alpha = dt + \frac{1}{2}\sum_{i=1}^n (y_i dx_i - x_i dy_i)$. The condition stated above is equivalent to $F^*\alpha = 0$. By pulling back $\alpha$ via $F$, we get
    $$F^*\alpha = dh + \tfrac{1}{2} \sum_{i=1}^n (g_i df_i - f_i dg_i) = 0, $$
    which is exactly \eqref{dh}. For Sobolev mappings, this geometric constraint becomes a differential equation that holds almost everywhere.  For more details about the weak contact equation, see for instance \cite[Remark 2.16]{MR2659687}, \cite{MR3298724}, and \cite[(1.7)]{MR3358059}.
\end{remark}

\begin{definition}[Weakly isotropic]\label{d:WeaklyIsotropic}
Let $1\leq k\leq n$. A map $(f_1, \dots, f_n, g_1, \dots, g_n) \in 
W_{\rm loc}^{1,1}(\mathbb{R}^k,\mathbb{R}^{2n})
$
is \emph{weakly isotropic} if
\begin{equation}\label{dfdg}
\sum_{i=1}^n \left[ \frac{\partial g_i}{\partial x_j}\frac{\partial f_i}{\partial x_l} - \frac{\partial g_i}{\partial x_l}\frac{\partial f_i}{\partial x_j} \right] = 0,\quad \text{a.e.}, \quad \text{for all } j,l\in \{1,\ldots,k\}.
\end{equation}
\end{definition}
The term ``isotropic'' refers to the fact that 
    equation \eqref{dfdg} can be phrased as 
     \begin{equation}\label{dfdg1}
         \sum_{i=1}^n df_i\wedge dg_i=0,\quad  \text{a.e.}.
  \end{equation}

  For a smooth map $(f_1,\ldots,f_n,g_1,\ldots,g_n)$, this condition means precisely that the map pulls back the standard symplectic form to the zero form.
For the study of intrinsic Lipschitz graphs, it is convenient to specialize the conditions in Definitions \ref{d:WeaklyContact} and \ref{d:WeaklyIsotropic} to mappings of the form $\Phi(x)=x\cdot \varphi(x)$
as in Remark \ref{r:FixedGroups}, and to their horizontal components. This is the purpose of the following two definitions;  Lemma \ref{l:PDEEquiv} will justify the terminology.

\begin{definition}[Weakly graph contact]\label{d:WeakGraphContact}
A map $\varphi = (\varphi_1, \ldots, \varphi_{2n-k+1}) \in W_{\rm loc}^{1,1}(\mathbb{R}^k,\mathbb{R}^{2n+1-k})
$
is \emph{weakly graph contact} if:
\begin{itemize}
\item
In case $1\leq k < n$,  we have
\begin{equation}\label{DF22} 
\begin{aligned}
\nabla  \varphi_{2n-k+1}= -(\varphi_{n-k+1},\ldots,\varphi_n)^T - \frac{1}{2} \sum_{i=1}^{n-k} \left[ \varphi_{n+i} \nabla  \varphi_i - \varphi_i \nabla\varphi_{n+i} \right], \quad \text{a.e.} 
\end{aligned} 
\end{equation}
\item In case $k=n$, we have
\begin{equation}\label{eq:GradEq}
    \nabla\varphi_{n+1}=-( \varphi_1,\cdots, \varphi_n)^T, \quad \text{a.e.}\end{equation}
    \end{itemize}
\end{definition}

\begin{definition}[Weakly graph isotropic]\label{d:WeaklyGraphIsotro}
A map $ (\varphi_1, \ldots, \varphi_{2n-k}) \in
W_{\rm loc}^{1,1}(\mathbb{R}^k,\mathbb{R}^{2n-k})
$
is \emph{weakly graph isotropic} if:
\begin{itemize}
\item
In case $1\leq k < n$,  we have
\begin{equation}\label{HorizEq2}
\frac{\partial \varphi_{n-k+l}}{\partial x_j} - \frac{\partial \varphi_{n-k+j}}{\partial x_l} + \sum_{i=1}^{n-k} \left[ \frac{\partial \varphi_{n+i}}{\partial x_j} \frac{\partial \varphi_i}{\partial x_l} - \frac{\partial \varphi_{n+i}}{\partial x_l} \frac{\partial \varphi_i}{\partial x_j} \right] = 0,\quad \text{a.e., }\quad \text{for all } j,l\in \{1,\ldots,k\}.
\end{equation}
\item In case $k=n$, we have 
\begin{equation}\label{HorizEq2_k=1}
\frac{\partial \varphi_l}{\partial x_j}=\frac{\partial \varphi_j}{\partial x_l},\quad \text{a.e.,}\quad \text{for all }j,l\in \{1,\ldots,n\}.
\end{equation}
\end{itemize}
\end{definition}

\begin{remark}[The case $k=1$]\label{eq:Casek=1_1} For $k=1$ there is only one coordinate $x_1$ ($j=l=1$) and equation \eqref{HorizEq2} (or \eqref{HorizEq2_k=1} if $k=n=1$)  is trivially satisfied. On the other hand, \eqref{DF22} (or \eqref{eq:GradEq}) is non-trivial also when $k=1$; in that case, the gradient reduces to the usual derivative in one variable. 
\end{remark}

\begin{remark}[Condition for affine maps]
An affine function $(\varphi_1,\ldots,\varphi_{2n-k})$ is weakly graph isotropic if and only if it satisfies condition  \eqref{iso} from Theorem \ref{isoa}, that is, if and only if it is isotropic affine.
\end{remark}

\subsubsection{Projections of low-dimensional intrinsic Lipschitz graphs}

The projection $\pi:\mathbb{H}^n \to \mathbb{R}^{2n}$, $\pi(z,t)=z$, maps the graph of an intrinsic Lipschitz function $\varphi:\mathbb{V}_0 \to \mathbb{W}_0$ to the graph of a Euclidean Lipschitz function $\mathbb{R}^k \to \mathbb{R}^{2n-k}$ with a PDE constraint, the \emph{weak graph isotropicity condition}. This is a consequence of the weak contact condition satisfied by the graph map $\Phi:\mathbb{V}_0 \to \mathbb{H}^n$. One  motivation for explicitly spelling out the PDE condition for $\varphi$, and its horizontal components, is to highlight the conceptual differences between the cases $k=1$, $k=n$, and $1<k<n$ for the study of $k$-dimensional intrinsic Lipschitz graphs in $\mathbb{H}^n$, see Remarks \ref{r:k=1}--\ref{r:k=n}.

We first show that substituting the components of $\Phi$ into the weak contact condition  \eqref{dh} algebraically recovers the PDE system in Definition \eqref{d:WeakGraphContact}  for $\varphi$, and vice-versa. We also establish an analogous relation for the isotropicity conditions of the horizontal components.

\begin{lemma}\label{l:PDEEquiv}
    Let $1\leq k\leq n$, let $\mathbb{V}_0$ be the $k$-dimensional subgroup with complementary $\mathbb{W}_0$ as in Remark \ref{r:FixedGroups}. Then the following equivalences hold:
\begin{enumerate}
    \item[(i)]  A map $\varphi: \V_0 \to \W_0$ 
   in $W^{1,1}_{\rm loc}(\mathbb{R}^k,\mathbb{R}^{2n+1-k})$ 
  is weakly graph contact if and only if its graph map
   $\Phi:\V_0\to \mathbb{H}^n$, given by
    \begin{equation*}
    \Phi(x) = \left(x_1,\cdots, x_k, \varphi_1(x),\cdots, \varphi_{2n-k}(x), \varphi_{2n-k+1}(x) + \tfrac{1}{2} \sum_{i=1}^k x_i \varphi_{n-k+i}(x)\right),
\end{equation*} is weakly contact.
    \item[(ii)] A map $(\varphi_1,\ldots,\varphi_{2n-k}):\mathbb{R}^k \to \mathbb{R}^{2n-k}$   in $W^{1,1}_{\rm loc}(\mathbb{R}^k,\mathbb{R}^{2n-k})$  is weakly graph isotropic if and only if 
    \begin{displaymath}
     u:   \mathbb{R}^k \to \mathbb{R}^{2n},\quad u(x):= (x_1,\ldots,x_k,\varphi_1(x),\ldots,\varphi_{2n-k}(x))
    \end{displaymath}
    is weakly isotropic.
\end{enumerate}
\end{lemma}
\begin{proof}
    (i) We consider first the case $1\leq k<n$.  Let $\varphi: \V_0 \to \W_0$ be a map as in \eqref{ilg}, we write
 its graph map as 
 $\Phi = (f_1, \dots, f_n, g_1, \dots, g_n, h)$, where
\begin{align}
f_i(x) &= x_i \quad i=1, \dots, k \label{eq:Deff_i1}\\
f_i(x) &= \varphi_{i-k}(x) \quad i=k+1, \dots, n \label{eq:Deff_i2}\\
g_i(x) &= \varphi_{n-k+i}(x) \quad i=1, \dots, n \label{eq:Defg_i}\\
h(x) &= \varphi_{2n-k+1}(x) + \tfrac{1}{2} \sum_{i=1}^k x_i \varphi_{n-k+i}(x).\label{eq:Defh}
\end{align}
We assume that $\varphi$ is weakly graph contact, that is, \eqref{DF22} holds.
 Differentiating $h$ defined in \eqref{eq:Defh} with respect to $x_m$ (for $1 \le m \le k$) yields:
    \begin{align*}
        \frac{\partial h}{\partial x_m} &= \frac{\partial \varphi_{2n-k+1}}{\partial x_m} + \frac{1}{2} \sum_{i=1}^k \left[ \frac{\partial x_i}{\partial x_m} \varphi_{n-k+i} + \frac{\partial \varphi_{n-k+i}}{\partial x_m} x_i \right].
    \end{align*}
    Assuming \eqref{DF22} holds and substituting it into the first term, we obtain
    \begin{align*}
        \frac{\partial h}{\partial x_m} &= -\varphi_{n-k+m} - \frac{1}{2} \sum_{i=1}^{n-k} \left[ \varphi_{n+i}\frac{\partial \varphi_i}{\partial x_m} - \varphi_i \frac{\partial \varphi_{n+i}}{\partial x_m} \right] + \frac{1}{2} \sum_{i=1}^k \left[ \frac{\partial x_i}{\partial x_m} \varphi_{n-k+i} + \frac{\partial \varphi_{n-k+i}}{\partial x_m} x_i \right]\\
        &= \frac{1}{2}\sum_{i=1}^{n-k} \left[ \varphi_i \frac{\partial \varphi_{n+i}}{\partial x_m} - \varphi_{n+i}\frac{\partial \varphi_i}{\partial x_m} \right] + \frac{1}{2} \sum_{i=1}^k \left[ \frac{\partial \varphi_{n-k+i}}{\partial x_m} x_i - \frac{\partial x_i}{\partial x_m} \varphi_{n-k+i} \right]\\
        &= \frac{1}{2} \sum_{i=1}^n \left( f_i \frac{\partial g_i}{\partial x_m} - g_i \frac{\partial f_i}{\partial x_m} \right),
    \end{align*} 
   where the final identity is a direct substitution using the component definitions \eqref{eq:Deff_i1}--\eqref{eq:Defg_i}. This proves the weak contact equation \eqref{dh} for $\Phi$. Conversely, starting from equation \eqref{dh} for $\Phi$, reversing the same algebraic steps directly yields \eqref{DF22} and thus shows that $\varphi$ is weakly graph contact.

   The case $k=n$ is proven in essentially the same way, except that the sum $\sum_{i=1}^{n-k}$ does not appear in the computation of $\frac{\partial h}{\partial x_m}$.

(ii)
    An analogous substitution using \eqref{eq:Deff_i1}--\eqref{eq:Defg_i} verifies that $u$ satisfies condition \eqref{dfdg} for weak isotropicity exactly if $(\varphi_1,\ldots,\varphi_{2n-k})$ is weakly graph isotropic in the sense of Definition \ref{d:WeaklyGraphIsotro} since
    \begin{align*}
        \sum_{i=1}^n \left[ \frac{\partial g_i}{\partial x_j}\frac{\partial f_i}{\partial x_l} - \frac{\partial g_i}{\partial x_l}\frac{\partial f_i}{\partial x_j} \right]=& \sum_{i=1}^k\left[\frac{\partial \varphi_{n-k+i}}{\partial x_j}\frac{\partial x_i}{\partial x_l}-\frac{\partial \varphi_{n-k+i}}{\partial x_l}\frac{\partial x_i}{\partial x_j}\right]\\&+\sum_{i=k+1}^{n}\left[\frac{\partial \varphi_{n-k+i}}{\partial x_j}\frac{\partial \varphi_{i-k}}{\partial x_l}-\frac{\partial \varphi_{n-k+i}}{\partial x_l}\frac{\partial \varphi_{i-k}}{\partial x_j}\right]\\
        =&\frac{\partial \varphi_{n-k+l}}{\partial x_j} - \frac{\partial \varphi_{n-k+j}}{\partial x_l} + \sum_{i=1}^{n-k} \left[ \frac{\partial \varphi_{n+i}}{\partial x_j} \frac{\partial \varphi_i}{\partial x_l} - \frac{\partial \varphi_{n+i}}{\partial x_l} \frac{\partial \varphi_i}{\partial x_j} \right]. \qedhere
    \end{align*}
\end{proof}

\begin{theorem}[Projections of intrinsic Lipschitz graphs]\label{horcon}\label{diff}
   Let $1\leq k\leq n$, let $\mathbb{V}_0$ be the $k$-dimensional subgroup with complementary $\mathbb{W}_0$ as in Remark \ref{r:FixedGroups}. Then the following properties hold for an intrinsic $L$-Lipschitz function
    $\varphi: \mathbb{V}_0 \to \mathbb{W}_0$:
    \begin{enumerate}
\item[(i)]  The components $\varphi_i:\mathbb{R}^k \to \mathbb{R}$, $i=1,\ldots,2n-k$, are Euclidean $L$-Lipschitz, and $\varphi_{2n-k+1}$ is locally Euclidean Lipschitz. \item[(ii)] The map $\varphi$ is weakly graph contact.\\ Moreover, if $k=n$, then $\nabla \varphi_{n+1}=-(\varphi_1,\ldots,\varphi_n)^T$ holds everywhere.
\item[(iii)] The horizontal components $(\varphi_1,\ldots,\varphi_{2n-k})$ are weakly graph isotropic.
\end{enumerate}
\end{theorem}

\begin{remark}[The case $k=1$]\label{r:k=1} We recall from Remark \ref{eq:Casek=1_1} that the weak graph isotropic condition is trivial when $k=1$, while the weak graph contact condition is not. This reflects the well-known fact that \emph{every} Euclidean Lipschitz curve in $\mathbb{R}^{2n}$ can be lifted to a horizontal curve in $\mathbb{H}^n$, but not every Euclidean Lipschitz curve in $\mathbb{R}^{2n+1}$ is horizontal and Lipschitz with respect the Kor\'{a}nyi metric.
\end{remark}

\begin{remark}[The case $k=n$]\label{r:k=n} 
    The case $k=n$ in Theorem \ref{horcon} is special for two reasons. First,  (i) and (ii) imply together that the vertical component $\varphi_{n+1}$ of an intrinsic Lipschitz function on the $n$-dimensional subgroup  $\mathbb{V}_0 \subset \mathbb{H}^n$ is Euclidean $C^{1,1}(\mathbb{R}^n)$. Second, conditions \eqref{eq:GradEq} and \eqref{HorizEq2_k=1} are linear (in particular, preserved under addition of functions), unlike the conditions \eqref{DF22} and \eqref{HorizEq2}.
\end{remark}

\begin{proof}[Proof of Theorem \ref{horcon}]
(i) Since $\mathbb{V}_0$ is isometric to $\mathbb{R}^k$,  the definition of intrinsic Lipschitz continuity implies, for instance via Proposition \ref{1.2}, that the horizontal components of $\varphi$ are Euclidean Lipschitz. See also, e.g., the proof of \cite[Proposition 3.3]{MR4375018}. The local Euclidean Lipschitz continuity of the vertical component $\varphi_{2n-k+1}$ was shown in \cite[Remark 2.9]{MR4375018}.

(ii) The weak graph contact conditions \eqref{DF22} and \eqref{eq:GradEq} were already proven in \cite[Remark 3.14]{MR4375018} and 
\cite[Proposition 3.12]{MR4375018} (for $k=n$ pointwise everywhere), respectively. There they were stated  for so-called tame maps 
 (see \cite[Definition 2.5]{MR4375018}), which differ from intrinsic Lipschitz functions only by a sign change in the last component. The components were also labeled in a different way, but it is easily checked that the PDEs agree up to relabeling.
 
(iii) This is the new part of the theorem compared to the results in \cite{MR4375018}. If $\varphi$ was sufficiently smooth, (iii) would easily follow by taking partial derivatives of the equations in (ii). As we do not have such strong regularity assumptions,  we will use a result in \cite{MR3358059} to deduce the weak graph isotropic equations \eqref{HorizEq2} and \eqref{HorizEq2_k=1}. 
Theorem 5.1 in \cite{MR3358059}  states  -- with our choice of coordinates -- for $k\geq 2$, $\Omega \subset \mathbb{R}^k$ open and $F=(f_1,\ldots,f_n,g_1,\ldots,g_n,h)\in W_{\rm loc}^{1,1}(\Omega,\mathbb{R}^{2n+1})$ that if 
\begin{equation}\label{eq:MMM1.7}
  dh = \frac{1}{2} \sum_{i=1}^n (f_i dg_i - g_i df_i),\quad \text{a.e. on }\Omega,
\end{equation}then also 
    \begin{equation}\label{MMM15_5.1}
  \sum_{i=1}^n df_i\wedge dg_i=0,\quad  \text{a.e. on }\Omega.
\end{equation}
According to parts (i) and (ii) of the theorem, $\Phi \in W_{\rm loc}^{1,1}(\mathbb{R}^k,\mathbb{R}^{2n+1})$ and $\varphi$ is weakly graph contact.  Lemma \ref{l:PDEEquiv} (i) then shows that $F(x):=\Phi(x)=x\cdot \varphi(x)$ is weakly contact and thus satisfies $\eqref{eq:MMM1.7}$.
By \cite[Theorem 5.1]{MR3358059}, this implies the weak isotropicity condition $\eqref{MMM15_5.1}$. Applying Lemma \ref{l:PDEEquiv} (ii), we then obtain that $\varphi$ is weakly graph isotropic.
\end{proof}

\begin{remark}\label{r:AlternativeProof} The intrinsic Rademacher theorem (Theorem \ref{1.5}) provides an alternative way to show that the horizontal components of an intrinsic Lipschitz function $\varphi:\mathbb{V}_0\to \mathbb{W}_0$ form a weak graph isotropic map. Indeed, by  Theorem \ref{1.5} such $\varphi$ is almost everywhere intrinsically differentiable. Proposition \ref{p:IntrDiffForm} yields an explicit expression for the intrinsic differential $d\varphi_v$ for a.e.\ $v\in \mathbb{V}_0$ in terms of the partial derivatives of the horizontal components of $\varphi$. Since $d\varphi_v$ is intrinsic linear, its horizontal components form an isotropic affine map according to Proposition \ref{ahp}. Finally the algebraic condition \eqref{iso} stated  in Proposition \ref{isoa}  for isotropic affine maps implies the PDE \eqref{HorizEq2} (respectively \eqref{HorizEq2_k=1}).
\end{remark}

\begin{remark}
    Parts (i)-(ii) of  Theorem \ref{horcon} can also be proven via the metric Lipschitz continuity of the graph map $\Phi:\mathbb{V}_0\to \mathbb{H}^n$, but we wish to highlight here that these properties follow easily directly from the intrinsic Lipschitz continuity of $\varphi$.
\end{remark}

\subsubsection{Low-dimensional intrinsic Lipschitz graphs as lifted Euclidean Lipschitz graphs}
Theorem \ref{horcon} describes the projections of low-dimensional intrinsic Lipschitz graphs to the horizontal plane $\mathbb{R}^{2n}\times \{0\}$. The next result, on the other hand, states a sufficient condition under which  a Euclidean Lipschitz graph in $\mathbb{R}^{2n}$ can be lifted to an intrinsic Lipschitz graph. Together, the two results provide a characterization of low-dimensional (entire) intrinsic Lipschitz graphs in terms of their horizontal projections.

The proof of the lifting theorem fits into the general theory of lifting isotropic to contact mappings. Such liftings have been studied in various contexts and under various regularity assumptions for instance in \cite{zbMATH01156777,zbMATH01782689,fassler2007extending,MR2659687, 2025arXiv250311506H}, and we apply here the same approach as in the proof of \cite[Theorem 7.8]{MR2659687}. However, as the argument in \cite{MR2659687} (below formula (65) therein) is formulated for maps defined on the disk in $\mathbb{R}^2$ respecting certain boundary conditions, we decided to include the full details for the precise statement needed in our application to intrinsic Lipschitz functions.

\begin{theorem}[Lifting to intrinsic Lipschitz graphs]\label{t:Lift}
Let $1\le k\le n$. Let $(\varphi_1, \varphi_2, \dots, \varphi_{2n-k}): \mathbb{R}^k \to \mathbb{R}^{2n-k}$ be a Euclidean Lipschitz map  which is weakly graph isotropic.
Then there exists a function $\varphi_{2n-k+1}: \mathbb{R}^k \to \mathbb{R}$, such that the map 
$$\varphi := (\varphi_1, \dots, \varphi_{2n-k}, \varphi_{2n-k+1}): \mathbb{V}_0 \to \mathbb{W}_0$$ 
is an intrinsic Lipschitz map with Lipschitz constant $L = L(n,k, \operatorname{Lip} \varphi_1,\ldots,\operatorname{Lip} \varphi_{2n-k})$.
If $k=n$, then $$\nabla\varphi_{n+1} = -(\varphi_1, \dots, \varphi_n)^T \quad\text{everywhere.}$$ 
Moreover, for every $c\in \mathbb{R}$, one can choose $\varphi_{2n-k+1}$ such that $\varphi_{2n-k+1}(0)=c$.
\end{theorem}

\begin{proof}
We will first show that there exists a locally Lipschitz continuous function $\varphi_{2n-k+1}$ such that $\varphi:=(\varphi_1,\ldots,\varphi_{2n-k},\varphi_{2n-k+1})$ is weakly graph contact; in other words, the given map can be lifted to a weakly graph contact map whose horizontal components are Lipschitz and the vertical component locally Lipschitz. The arguments for $k=1$ and $k=n$ are easier than for $1<k<n$, so we discuss those first.

 \textbf{Case 1: $k=1$.} \\When $k=1$, the domain is $\mathbb{R}$, and  $(\varphi_1,\ldots,\varphi_{2n-1})$ is an arbitrary Euclidean Lipschitz map. We define a function $F: \mathbb{R} \to \mathbb{R}$ by the right-hand side of the desired PDE (\ref{DF22})
\begin{equation}\label{eq:HorizLift}
F(x) := -\varphi_n(x) - \frac{1}{2}\sum_{i=1}^{n-1} \left[ \varphi_{n+i}(x)\varphi'_i(x) - \varphi_i(x)\varphi'_{n+i}(x) \right].
\end{equation}
Since $\varphi_i$ and $\varphi_{n+i}$ are Lipschitz continuous, they are locally bounded. By Rademacher's theorem, their classical derivatives $\varphi'_i$ and $\varphi'_{n+i}$ exist almost everywhere and belong to $L^\infty_{\rm loc}(\mathbb{R})$. Therefore, the product of these functions yields $F \in L^\infty_{\rm loc}(\mathbb{R})$. 
We can directly define $\varphi_{2n}$ by the integration
\begin{equation}\label{eq:IntLift}
\varphi_{2n}(x) := \int_0^x F(t) dt + C,
\end{equation}
for some constant $C \in \mathbb{R}$. By the fundamental theorem of calculus, $\varphi_{2n}$ is absolutely continuous, its derivative exists almost everywhere, and $\varphi'_{2n}(x) = F(x)$ almost everywhere. Furthermore, since $F \in L^\infty_{\rm loc}(\mathbb{R})$, $\varphi_{2n}$ is locally Lipschitz continuous.

\textbf{Case 2: $k=n$.} \\
It suffices to show that there exists $\varphi_{n+1}$ with $\nabla \varphi_{n+1}=-(\varphi_1,\ldots,\varphi_n)^T$ everywhere. Let $\mathcal{D}'(\mathbb{R}^n)$ be the space of all distributions over $\mathbb{R}^n$.
Since we know for $i,j = 1, 2, \dots, n$, that $\varphi_i$ is Lipschitz with $\partial_i \varphi_j = \partial_j \varphi_i$ almost everywhere, then 
\begin{equation}\label{eq:MixedDistrPartial}\partial_i \varphi_j = \partial_j \varphi_i \quad \text{in } \mathcal{D}'(\mathbb{R}^n).\end{equation}
By  a distributional version of the Poincar\'{e} lemma \cite{MR209834}, see also
\cite[Theorem 2.1]{MR2474500}, there exists a distribution
$\varphi_{n+1} \in \mathcal{D}'(\mathbb{R}^n)$
such that
\begin{equation}\label{eq:DistrGradEq}\nabla \varphi_{n+1} = -(\varphi_1, \varphi_2, \dots, \varphi_n)^T \quad \text{in } \mathcal{D}'(\mathbb{R}^n).\end{equation}
As the components $\varphi_i$ are continuous for $i=1,\ldots,n$, we have $\varphi_{n+1}\in C^1(\R^n)$ and the classical gradient
$$\nabla \varphi_{n+1} = -(\varphi_1, \varphi_2, \dots, \varphi_n) \quad \text{everywhere},$$
see, for instance, \cite[Theorem 6.10]{MR1817225}.

 \textbf{Case 3: $1<k<n$.} \\
 By assumption,  $(\varphi_1, \varphi_2, \dots, \varphi_{2n-k}): \mathbb{R}^k \to \mathbb{R}^{2n-k}$ is weakly graph isotropic, which means that 
for
all $j,l\in \{1,\ldots,k\}$,
\begin{equation}\label{3.19}
    \frac{\partial\varphi_{n-k+l}}{\partial x_j} - \frac{\partial\varphi_{n-k+j}}{\partial x_l} + \sum_{i=1}^{n-k} \left[ \frac{\partial\varphi_{n+i}}{\partial x_j}\frac{\partial\varphi_i}{\partial x_l} - \frac{\partial\varphi_{n+i}}{\partial x_l}\frac{\partial\varphi_i}{\partial x_j} \right] = 0 \quad \text{a.e. on } \mathbb{R}^k.
\end{equation}
Since each $\varphi_i$ is Lipschitz, its classical partial derivatives exist almost everywhere and are equal to its distributional derivatives. Let $\frac{\partial\varphi_i}{\partial x_m} \in L^\infty(\mathbb{R}^k) \subset \mathcal{D}'(\mathbb{R}^k)$ denote its distributional derivative. 
For $1 \le m \le k$, let us define the distribution $F_m \in \mathcal{D}'(\mathbb{R}^k)$ by
$$F_m := -\varphi_{n-k+m} - \frac{1}{2}\sum_{i=1}^{n-k} \left[ \varphi_{n+i} \frac{\partial\varphi_i}{\partial x_m} - \varphi_i \frac{\partial\varphi_{n+i}}{\partial x_m} \right].$$

Let $\eta_\varepsilon$ be the standard mollifier, for $1\le h\le 2n-k$ and $1\le m\le k$, we define the smooth approximations of $\varphi_h$ and $F_m$ by
\begin{equation*}
    \varphi_h^\varepsilon := \varphi_h * \eta_\varepsilon \quad\text{and}\quad
F^\varepsilon_m := -\varphi^\varepsilon_{n-k+m} - \frac{1}{2}\sum_{i=1}^{n-k} \left[ \varphi^\varepsilon_{n+i} \frac{\partial\varphi^\varepsilon_i}{\partial x_m} - \varphi^\varepsilon_i \frac{\partial\varphi^\varepsilon_{n+i}}{\partial x_m} \right].
\end{equation*}
Now, both $\varphi^\varepsilon_h$ and $F^\varepsilon_m$ are smooth on $\mathbb{R}^k$. Taking the derivative of $F^\varepsilon_j$ with respect to $x_l$, we obtain
$$ \frac{\partial F^\varepsilon_j}{\partial x_l} = -\frac{\partial\varphi^\varepsilon_{n-k+j}}{\partial x_l} - \frac{1}{2}\sum_{i=1}^{n-k} \left[ \frac{\partial\varphi^\varepsilon_{n+i}}{\partial x_l} \frac{\partial\varphi^\varepsilon_i}{\partial x_j} + \varphi^\varepsilon_{n+i} \frac{\partial^2\varphi^\varepsilon_i}{\partial x_l \partial x_j} - \frac{\partial\varphi^\varepsilon_i}{\partial x_l} \frac{\partial\varphi^\varepsilon_{n+i}}{\partial x_j} - \varphi^\varepsilon_i \frac{\partial^2\varphi^\varepsilon_{n+i}}{\partial x_l \partial x_j} \right].$$
Note that  $\frac{\partial^2\varphi^\varepsilon_h}{\partial x_j \partial x_l} = \frac{\partial^2\varphi^\varepsilon_h}{\partial x_l \partial x_j}$, thus for all $1\le j,l\le k$, we have
\begin{equation}\label{eq:EpsilonIsotropic}\frac{\partial F^\varepsilon_j}{\partial x_l} - \frac{\partial F^\varepsilon_l}{\partial x_j} = \left( \frac{\partial\varphi^\varepsilon_{n-k+l}}{\partial x_j} - \frac{\partial\varphi^\varepsilon_{n-k+j}}{\partial x_l} \right) + \sum_{i=1}^{n-k} \left[ \frac{\partial\varphi^\varepsilon_{n+i}}{\partial x_j}\frac{\partial\varphi^\varepsilon_i}{\partial x_l} - \frac{\partial\varphi^\varepsilon_{n+i}}{\partial x_l}\frac{\partial\varphi^\varepsilon_i}{\partial x_j} \right]. \end{equation}
Letting $\varepsilon\to 0$, we have $\varphi^\varepsilon_h\to \varphi_h$  and $\frac{\partial\varphi^{\varepsilon}_h}{\partial x_j}\to \frac{\partial\varphi_h}{\partial x_j}$ almost everywhere on $\mathbb{R}^k$. Since $F_m^\varepsilon$, restricted to compact sets, has uniformly bounded $L^{\infty}$-norm with respect to $\varepsilon$, by the dominated convergence theorem, we have $F_m^\varepsilon \to F_m$ in $L_{\rm loc}^1(\R^k)$. By the same reasoning, as the right-hand side of \eqref{eq:EpsilonIsotropic} 
 has uniformly bounded $L^{\infty}$-norm with respect to $\varepsilon$, we obtain
\begin{equation}\label{convergencetoDF22}
     \frac{\partial F^\varepsilon_j}{\partial x_l} - \frac{\partial F^\varepsilon_l}{\partial x_j}\to \left(\frac{\partial \varphi_{n-k+l}}{\partial x_j} - \frac{\partial \varphi_{n-k+j}}{\partial x_l}\right) + \sum_{i=1}^{n-k} \left[ \frac{\partial \varphi_{n+i}}{\partial x_j}\frac{\partial \varphi_i}{\partial x_l} - \frac{\partial \varphi_{n+i}}{\partial x_l}\frac{\partial \varphi_i}{\partial x_j} \right] \quad\text{in } L^1_{\rm loc}(\R^k).
\end{equation}
By our assumption \eqref{3.19}, the right-hand side of \eqref{convergencetoDF22} is equal to $0$ almost everywhere. 
Thus
\begin{equation}\label{eq:ConvEpsDeriv}
 \frac{\partial F^\varepsilon_j}{\partial x_l} - \frac{\partial F^\varepsilon_l}{\partial x_j}\to 0  \quad\text{in } L^1_{\rm loc}(\R^k).
\end{equation}
We aim to deduce that $ \frac{\partial F_j}{\partial x_l}=\frac{\partial F_l}{\partial x_j}$ in the sense of distributions. Since $F_j$ and $F_l$ are in $L^1_{\rm loc}(\mathbb{R}^k)$, this means verifying 
\begin{equation}\label{eq:spelledOutGoal}
    \int_{\mathbb{R}^k} F_j \frac{\partial \psi}{\partial x_l} dx=\int_{\mathbb{R}^k} F_l \frac{\partial \psi}{\partial x_j}   dx,\quad \psi\in C_c^{\infty}(\mathbb{R}^k).
\end{equation}
Now, let $\psi\in C_c^{\infty}(\R^k)$ be an arbitrary test function. We have
$$\int_{\mathbb{R}^k} \frac{\partial F_j^\varepsilon}{\partial x_l} \psi dx = - \int_{\mathbb{R}^k} F_j^\varepsilon \frac{\partial \psi}{\partial x_l}  dx.$$
Letting $\varepsilon\to 0$, we obtain
$$\lim_{\varepsilon \to 0}\int_{\mathbb{R}^k}\frac{\partial F_j^\varepsilon}{\partial x_l} \psi dx = \lim_{\varepsilon \to 0} - \int_{\mathbb{R}^k}F_j^\varepsilon \frac{\partial \psi}{\partial x_l} dx  =-\int_{\mathbb{R}^k} F_j \frac{\partial \psi}{\partial x_l} dx.$$
Therefore, our goal \eqref{eq:spelledOutGoal} is equivalent to
\begin{displaymath}
\lim_{\varepsilon \to 0}   \int_{\mathbb{R}^k} \frac{\partial F_j^\varepsilon}{\partial x_l} \psi dx=\lim_{\varepsilon \to 0}   \int_{\mathbb{R}^k} \frac{\partial F_l^\varepsilon}{\partial x_j} \psi dx,\quad \psi\in C_c^{\infty}(\mathbb{R}^k),
\end{displaymath}
but this clearly holds because of \eqref{eq:ConvEpsDeriv}.
Thus we get
$$ \frac{\partial F_j}{\partial x_l} = \frac{\partial F_l}{\partial x_j} \quad \text{in } \mathcal{D}'(\mathbb{R}^k). $$
Again by the distributional version of Poincar\'e's lemma, there exists a distribution $\varphi_{2n-k+1} \in \mathcal{D}'(\mathbb{R}^k)$ such that
\begin{equation}\label{eq:GoalDistrForm}
\nabla \varphi_{2n-k+1} = (F_1,\ldots,F_k) \quad \text{in } \mathcal{D}'(\mathbb{R}^k). 
\end{equation}
A priori, $\varphi_{2n-k+1}$ is just a distribution. However, as $F_m \in L^1_{\rm loc}(\mathbb{R}^k)$, the distributional solution $\varphi_{2n-k+1}$ of \eqref{eq:GoalDistrForm} belongs
to $W^{1,1}_{\rm loc}(\mathbb{R}^k)$; see for instance \cite[Theorem 6.74]{MR2895178}. Finally, as its distributional derivative is in $L^{\infty}_{\rm loc}(\mathbb{R}^k)$, the solution $\varphi_{2n-k+1}$ admits a locally Lipschitz representative by standard arguments as in \cite[5.8.b]{MR2597943}.
 By Rademacher's theorem, its classical partial derivatives exist almost everywhere. By expressing the distributional equality as an integral against a test function and applying the fundamental lemma of calculus of variation, we conclude that the above equation \eqref{eq:GoalDistrForm}, and thus the weak graph contact equation \eqref{DF22}, 
 are satisfied almost everywhere in $\mathbb{R}^k$.

\textbf{From weakly graph contact to intrinsic Lipschitz.}
Finally, we show that the map $\varphi = (\varphi_1, \dots, \varphi_{2n-k}, \varphi_{2n-k+1}): \mathbb{V}_0 \to \mathbb{W}_0$ is intrinsically Lipschitz since it is weakly graph contact with Lipschitz continuous horizontal components and locally Lipschitz continuous vertical component.

We consider its graph map $\Phi: \mathbb{V}_0 \to \mathbb{H}^n$ defined by $\Phi(x)=x\cdot \varphi(x)$. Since $\varphi$ is weakly graph contact,  $\Phi$ satisfies the weak contact equation \eqref{dh} almost everywhere on $\mathbb{R}^k$ by 
 Lemma \ref{l:PDEEquiv} (i). For the Heisenberg group (which has step $\nu = 2$), this equation is exactly the abstract contact equation \cite[Equation (26)]{MR2659687}. 
 
 We aim to apply \cite[Theorem 4.5]{MR2659687} to deduce that $\Phi:\mathbb{R}^k \to (\mathbb{H}^n,d)$ is metrically Lipschitz. This result is formulated in terms of Lipschitz maps from open geodetically convex subsets of Riemannian manifolds to graded Lie algebras. Let $\mathcal{M}$ denote the graded Lie algebra of $\mathbb{H}^n$, and define the algebra-valued map $F = \exp^{-1} \circ \Phi: \mathbb{V}_0 \to \mathcal{M}$. In the standard exponential coordinates of the Heisenberg group, $F$ shares the exact same coordinate representation as $\Phi$.

The vertical component $h$ of $\Phi$, as defined in \eqref{eq:Defh}, is not globally Euclidean Lipschitz. However, for any  radius $R>0$, the restriction $\Phi|_{B(0,R)}$ to  the open ball $B(0,R) \subset \mathbb{R}^k$ is  Euclidean Lipschitz. Therefore, we can apply \cite[Theorem 4.5]{MR2659687} on the restricted domain $B(0,R)$, which establishes that $\Phi|_{B(0,R)}$ is metrically Lipschitz with respect to the sub-Riemannian Carnot-Carathéodory distance. Moreover, \cite[Theorem 4.5]{MR2659687} also guarantees that the metric Lipschitz constant of $\Phi|_{B(0,R)}$ is bounded by a geometric constant times the Euclidean Lipschitz constant of its horizontal components, denoted as ${F_1}|_{B(0,R)}$. Since  $F_1(x) = (x_1, \dots, x_k, \varphi_1(x), \dots, \varphi_{2n-k}(x))$ is globally Euclidean Lipschitz on $\mathbb{R}^k$ by assumption, it follows that the metric Lipschitz constant of $\Phi|_{B(0,R)}$ is uniformly bounded by a global constant dependent only on $n,k$ and $\operatorname{Lip} \varphi_i$ for $i=1,\ldots,2n-k$, regardless of the radius $R$. This implies that $\Phi$ is globally metrically Lipschitz on the entire $\mathbb{R}^k$.

To conclude, according to \cite[Corollary 4.62 (i)]{MR3587666}, since the complementary vertical subgroup $\mathbb{W}_0$ is a normal subgroup of $\mathbb{H}^n$, a base map $\varphi: \mathbb{V}_0 \to \mathbb{W}_0$ is intrinsically Lipschitz if and only if its associated graph map $\Phi$ is globally metrically Lipschitz. Therefore, $\varphi$ is an intrinsic Lipschitz map with intrinsic Lipschitz constant $L = L(n,k, \operatorname{Lip} \varphi_1,\ldots,\operatorname{Lip} \varphi_{2n-k})$.

Finally, we can assume that $\varphi_{2n-k+1}(0)=c$ by simply adding a suitable constant to $\varphi_{2n-k+1}$, which will not change the property of being weakly graph contact.
\end{proof}

\begin{remark}[Weakly graph contact lift]\label{r:HorizLift} Let $I\subset \mathbb{R}$ be an interval. We call  $(\varphi_1,\ldots,\varphi_{2n-1},\varphi_{2n}):I\to \mathbb{R}^{2n}$  a \emph{weakly graph contact lift} of a given Euclidean Lipschitz function  $(\varphi_1,\ldots,\varphi_{2n-1}):I \to \mathbb{R}^{2n-1}$ if the last component is locally Euclidean Lipschitz and satisfies
\eqref{eq:HorizLift}, that is,
\begin{equation}\label{eq:HorizCondCurve}
\varphi_{2n}'(s) := -\varphi_n(s) - \frac{1}{2}\sum_{i=1}^{n-1} \left[ \varphi_{n+i}(s)\varphi'_i(s) - \varphi_i(s)\varphi'_{n+i}(s) \right],\quad \text{a.e. }s,
\end{equation}
or $\varphi_{2}':= -\varphi_1$ almost everywhere if $n=1$.

The terminology refers to the fact that the curve 
\begin{displaymath}
\gamma:I\to \mathbb{H}^n,  \quad \gamma(s):= (s,0\ldots,0)\cdot (0,\varphi_1(s),\ldots,\varphi_{2n}(s))
\end{displaymath} is tangential to the horizontal distribution, cf. Remark \ref{r:ContactEq}, and $(\varphi_1,\ldots,\varphi_{2n})$ is weakly graph contact in the sense of Definition \ref{d:WeakGraphContact}. Sometimes $\gamma$ is called the \emph{horizontal lift} of the curve $s\mapsto (s,\varphi_1(s),\ldots,\varphi_{2n-1}(s))$ from $\mathbb{R}^{2n}$ to $\mathbb{H}^n$.
For any Lipschitz  $(\varphi_1,\ldots,\varphi_{2n-1}):I\to \mathbb{R}^{2n-1}$ and any $c\in \mathbb{R}$ and $s_0 \in I$, there exists a unique weakly graph contact lift with the property $\varphi_{2n}(s_0)=c$, which can be obtained simply by integration as in \eqref{eq:IntLift}.
\end{remark}

\begin{remark}[$1$- and $n$-dimensional graphs]
    If $k=1$ or $k=n$, the last part of the proof of Theorem \ref{t:Lift} can also be carried out without reference to \cite[Theorem 4.5]{MR2659687}. For $k=n$, it was already shown in \cite[Proposition 3.6 and Proposition 3.12]{MR4375018}, with a direct computation, that in order to prove the intrinsic Lipschitz continuity of
     $\varphi:\mathbb{V}_0\to \mathbb{W}_0$, it suffices to show that there exists $\varphi_{n+1}$ with $\nabla \varphi_{n+1}=-(\varphi_1,\ldots,\varphi_n)^T$ everywhere. In the case $k=1$, the curve $\Phi:\mathbb{R}\to \mathbb{H}^n$ is horizontal and hence Lipschitz with respect to the Kor\'anyi distance, e.g., by \cite[Proposition 1.1]{MR3417082}.
\end{remark}

\begin{remark}[PDE characterization of intrinsic Lipschitz] The last part of the proof of Theorem \ref{t:Lift} combined with Theorem \ref{horcon} yields the following characterization. A function $\varphi:\mathbb{V}_0\to \mathbb{W}_0$ is intrinsic Lipschitz if and only if is weakly graph contact and its horizontal components are Euclidean Lipschitz.
\end{remark}

\section{Isotropic Dorronsoro Theorem}\label{s:IsotrDorronsoro}
We begin by recalling in Section \ref{ss:ClassicalDorronsoro} Dorronsoro's quantitative differentiation theorem in $\mathbb{R}^n$. In Section \ref{ss:IsotrDorronsoro} we prove a Dorronsoro-type theorem for weakly graph isotropic Lipschitz functions $\mathbb{R}^n \to \mathbb{R}^n$ in terms of approximations by isotropic affine maps.  This \emph{isotropic Dorronsoro theorem} will play a crucial role for our application to $n$-dimensional intrinsic Lipschitz graphs in $\mathbb{H}^n$ in Section~\ref{ss:k=n}.

\subsection{The classical Dorronsoro theorem for Sobolev functions}\label{ss:ClassicalDorronsoro}
In \cite[Theorem 2, Theorem 6]{MR796440}, Dorronsoro proved a quantitative differentiation result that expresses how well a  function in the Bessel potential space $L_\alpha^p(\mathbb{R}^n)$ can be approximated by specific polynomials of degree $ [\alpha]$ at different places and scales. For positive integer values of $\alpha$, the space $L_\alpha^p(\mathbb{R}^n)$ coincides with the classical Sobolev space $W^{\alpha,p}(\mathbb{R}^n)$ (see Section 9.4 of \cite{MR1681462} and Section 1.3 of \cite{MR3243741}). 
By specializing \cite[Theorem 6]{MR796440} to $\alpha=1$ and $p=2$, and evaluating the polynomial approximations over balls $B(x,r)$ (see \cite[p.23]{MR796440}), one obtains the following corollary for $W^{1,2}(\mathbb{R}^n)$; see also the different proofs in \cite{MR2299766,MR3512428,MR4012342}.

\begin{theorem}[Dorronsoro's Theorem for $W^{1,2}(\R^n)$]\label{corollaryofdorr}
Let $f \in L^2(\mathbb{R}^n)$. For $x \in \mathbb{R}^n$ and $r > 0$, define
$$ \Psi_{2,f}(x,r) := \left( \fint_{B(x,r)} \left( \frac{f(y) - A_{x,r,f}(y)}{r} \right)^2 dy \right)^{\frac{1}{2}}, $$
where $A_{x,r,f}(y) $ is the unique polynomial of degree $1$ such that
$$ \int_{B(x,r)} (f(y) - A_{x,r,f}(y)) y^\gamma dy = 0 $$
for every $n$-tuple $\gamma=(\gamma_1,\cdots,\gamma_n)\in\mathbb{N}^n$ with $|\gamma| =\gamma_1+\cdots+\gamma_n\le 1$. Then $f \in W^{1,2}(\mathbb{R}^n)$ if and only if
$$ G_1 f(x) := \left( \int_0^\infty \Psi_{2,f}(x,r)^2 \frac{dr}{r} \right)^{\frac{1}{2}} \in L^2(\mathbb{R}^n). $$
Furthermore, we have the norm equivalence $\|f\|_{W^{1,2}} \sim \|f\|_{L^2} + \|\nabla f\|_{L^2} \sim_n \|f\|_{L^2} + \|G_1 f\|_{L^2}$.
\end{theorem}

We now rewrite an instance of Dorronsoro's Theorem in a form that is particularly useful for geometric applications. Specifically, this version explicitly quantifies how well a $W^{1,2}$ function $f$ can be approximated by affine functions almost everywhere and at all scales, with the quantification controlled solely in terms of the gradient of $f$. This formulation is well-known, but for the convenience of the reader, we explain in detail how to derive it from Theorem \ref{corollaryofdorr}.

\begin{theorem}[Corollary of Dorronsoro's Theorem]\label{t:DorronsoroSpecial}
Let $n\in \mathbb{N}$ and $f\in L^2_{\rm loc}(\mathbb{R}^n)$. 
For $x \in \R^n$ and $r>0$, define
\begin{equation*}
    \Omega_{2, f}(x, r) := \inf_{A} \left(\fint_{B(x,r)} \left(\frac{|f(y) - A(y)|}{r}\right)^2 dy \right)^{1/2},
\end{equation*}
where the infimum is over all affine maps $A: \R^n \to \R$. If $f\in W^{1,2}(\R^n)$, then
\begin{equation}\label{app}
    \Omega_2(f) := \int_{\R^n} \int_0^\infty \Omega_{2, f}(x, r)^2 \frac{dr}{r} dx\le\int_{\R^n} (G_1f(x))^2 dx\lesssim_n \norm{\nabla f}_{L^2}^2,
\end{equation}
where the affine maps  $A_{x,r,f}$ appearing in the definition of $G_1 f$ in Theorem \ref{corollaryofdorr} are 
explicitly given by
\begin{equation}\label{aff}
   A_{x,r,f}(y) = \fint_{B(x,r)} f(z) dz + \frac{n+2}{r^2} \sum_{i=1}^n (y_i - x_i) \fint_{B(x,r)} f(z)(z_i - x_i) dz.
\end{equation}
\end{theorem}

\begin{proof}
First, we verify the explicit expression for the affine maps $A_{x,r,f}$ given in (\ref{aff}). 
By Theorem \ref{corollaryofdorr}, $A_{x,r,f}$ is the unique polynomial of degree 1 such that 
$$ 
\int_{B(x,r)} (f(z) - A_{x,r,f}(z)) A(z) dz = 0 
$$
for any affine map $A$.

Let $M$ be the space of affine functions on $B(x,r)$, then $M\subset L^2(B(x,r))$ is a subspace of dimension $n+1$, and the above equation tells us $(f-A_{x,r,f})\in M^{\perp}$ with respect to the inner product $\langle \cdot, \cdot \rangle$ on $L^2$. We choose an orthogonal basis of $M$ 
$$
\{e_0,e_1,e_2,\cdots,e_{n}\}=
\{1, y_1-x_1, y_2-x_2, \dots, y_n-x_n\}.
$$ 
Thus, $A_{x,r,f}$ can be written as:
$$ A_{x,r,f}(y) = c_0 + \sum_{i=1}^n c_i (y_i - x_i).$$
For all $i=0,1,\cdots,n$ we have $\langle f-A_{x,r,f},e_i\rangle=0$,
which means that
$$\langle f, e_k \rangle = \langle A_{x,r,f},e_k\rangle=\left\langle \sum_{j=0}^n c_j e_j, e_k \right\rangle = \sum_{j=0}^n c_j \langle e_j, e_k \rangle=c_k\langle e_k,e_k\rangle. $$
Hence, the coefficients are given by $c_k = \frac{\langle f, e_k \rangle}{\langle e_k, e_k \rangle},$ thus
$$ c_0 = \frac{\fint_{B(x,r)} f(z) dz}{\fint_{B(x,r)} 1 dz} = \fint_{B(x,r)} f(z) dz. $$
For $i\ge1$, we have
$$ \fint_{B(x,r)} (z_i - x_i)^2 dz = \frac{1}{n} \fint_{B(x,r)} |z-x|^2 dz = \frac{1}{n \omega_n r^n} \int_0^r \rho^2 \cdot n \omega_n \rho^{n-1} d\rho = \frac{r^2}{n+2}, $$
where $\omega_n$ is the volume of the unit ball in $\mathbb{R}^n$, therefore
$$ c_i = \frac{\fint_{B(x,r)} f(z)(z_i - x_i) dz}{\fint_{B(x,r)} (z_i - x_i)^2 dz} = \frac{n+2}{r^2} \fint_{B(x,r)} f(z)(z_i - x_i) dz. $$
Then, we obtain the explicit expression (\ref{aff}):
$$ A_{x,r,f}(y) = \fint_{B(x,r)} f(z) dz + \frac{n+2}{r^2} \sum_{i=1}^n (y_i-x_i) \fint_{B(x,r)} f(z)(z_i - x_i) dz. $$

Next, we prove the estimates in (\ref{app}). By definition, $\Omega_{2,f}(x,r)$ is defined as an infimum over \textit{all} affine maps $A$. Since $A_{x,r,f}$ is a specific affine map, we trivially have
$$ \Omega_{2,f}(x,r)^2 \le \fint_{B(x,r)} \left( \frac{|f(y) - A_{x,r,f}(y)|}{r} \right)^2 dy = \Psi_{2,f}(x,r)^2. $$
Integrating both sides over $r \in (0, \infty)$ with measure $\frac{dr}{r}$ and over $x \in \mathbb{R}^n$, we get
$$ \Omega_2(f) \le \int_{\mathbb{R}^n} \int_0^\infty \Psi_{2,f}(x,r)^2 \frac{dr}{r} dx = \int_{\mathbb{R}^n} (G_1 f(x))^2 dx =\|G_1 f\|^2_{L^2}. 
$$
Finally, we need to show that $\|G_1 f\|_{L^2} \lesssim_n \|\nabla f\|_{L^2}$. From the norm equivalence in Theorem \ref{corollaryofdorr}, we know that there exists a constant $C=C(n) > 0$ such that for any $f \in W^{1,2}(\mathbb{R}^n)$
$$
\|G_1 f\|_{L^2} \le C \left( \|f\|_{L^2} + \|\nabla f\|_{L^2} \right). 
$$
To remove the term $\|f\|_{L^2}$, we use a standard homogeneity argument; see for instance \cite[Lemma 2.6]{MR4171381}. 
More precisely, for any $s > 0$, define the dilated function $f_s(x) := f(sx)$. Then we have 
\begin{equation*}
    \|f_s\|_{L^2} = s^{-n/2}\|f\|_{L^2}\quad\text{and}\quad\|\nabla f_s\|_{L^2} = s^{1-n/2}\|\nabla f\|_{L^2}.
\end{equation*}
We also observe that $A_{x,r,f_s}(y) = A_{sx,sr,f}(sy)$, then $\Psi_{2, f_s}(x,r)^2 = s^2 \Psi_{2,f}(sx, sr)^2$. Consequently, we obtain
$$ G_1 f_s(x)^2 = \int_0^\infty s^2 \Psi_{2,f}(sx, sr)^2 \frac{d(sr)}{sr} = s^2 G_1 f(sx)^2. $$Integrating over $\mathbb{R}^n$ yields
$$ \|G_1 f_s\|_{L^2} = s^{1-n/2} \|G_1 f\|_{L^2}. $$
Now we have $\|G_1 f_s\|_{L^2} \le C (\|f_s\|_{L^2} + \|\nabla f_s\|_{L^2})$, that is
$$ s^{1-n/2} \|G_1 f\|_{L^2} \le C \left( s^{-n/2} \|f\|_{L^2} + s^{1-n/2} \|\nabla f\|_{L^2} \right). $$
Dividing both sides by $s^{1-n/2}$, we obtain
$$ \|G_1 f\|_{L^2} \le C \left( s^{-1} \|f\|_{L^2} + \|\nabla f\|_{L^2} \right). $$Letting $s \to \infty$, the term $s^{-1} \|f\|_{L^2}$ vanishes, and we conclude that $\|G_1 f\|_{L^2} \le C \|\nabla f\|_{L^2}$. Combining this with $\Omega_2(f) \le \|G_1 f\|_{L^2}^2$, we finally get
$$ \Omega_2(f) \le\|G_1 f\|_{L^2}^2 \lesssim_n \|\nabla f\|_{L^2}^2. $$
This completes the proof of Theorem \ref{t:DorronsoroSpecial}.
\end{proof}

\begin{remark}\label{t:DorronsoroAzzam}
In \cite[Theorem 1.2, Appendix (Section 7.3)]{MR3512428}, Azzam gave an alternative proof of Dorronsoro's estimate
\begin{equation}\label{eq:FinalDorronsoroEst}
    \Omega_2(f)\lesssim_n \|\nabla f\|_{L^2}^2,\quad f\in W^{1,2}(\mathbb{R}^n).
\end{equation}
His argument produced a different way of assigning to $x,r$ and $f$ an affine map $\widetilde A_{x,r,f}$ which yields the bound \eqref{eq:FinalDorronsoroEst}. Namely,
let $\phi$ be a radially symmetric nonnegative function on $\R^n$ supported in $B(0,1)$ such that $\int_{\R^n} \phi =1$, and let $\phi_r(x):=r^{-n}\phi(r^{-1}x)$. Define an affine map 
\begin{equation}\label{affAZZ}
    \widetilde{A}_{x,r,f}(y) := (\phi_r\ast\nabla f)(x) \cdot (y-x)+(f\ast\phi_r)(x)
\end{equation}
and let
$$
\widetilde{\Psi}_{2,f}(x,r):=\left(\fint_{B(x,r)} \left(\frac{|f(y) - \widetilde{A}_{x,r,f}(y)|}{r}\right)^2 dy\right)^\frac{1}{2}.
$$
Then one still has the same result as in Theorem \ref{t:DorronsoroSpecial}:
\begin{equation}\label{appAZZ}
    \Omega_2(f)\le \int_{\R^n} \int_0^\infty \widetilde{\Psi}_{2, f}(x, r)^2 \frac{dr}{r} dx\lesssim_n \norm{\nabla f}_{L^2}^2.
\end{equation}
\end{remark}

\subsection{Isotropic version of Dorronsoro's theorem}\label{ss:IsotrDorronsoro}

Dorronsoro's theorem, as stated in Theorem \ref{t:DorronsoroSpecial}, implies a geometric lemma for Lipschitz graphs in Euclidean spaces. In this section, we will prove an ``isotropic variant'' of Dorronsoro's theorem. This provides a method for quantifying how well  the graphs of weakly graph isotropic Lipschitz functions $\mathbb{R}^{n}\to \mathbb{R}^n$ can be approximated by affine isotropic planes at different locations and scales.

By the classical Rademacher's theorem, Euclidean Lipschitz functions  are differentiable almost everywhere. If a Lipschitz function  $f:\mathbb{R}^k \to \mathbb{R}^{2n-k}$ is additionally assumed to be weakly graph isotropic, then at almost every point its graph has a tangent plane that is isotropic. When $k=n$, this isotropicity condition propagates to the affine approximations at \emph{all} scales. This is related to how the condition in Definition \ref{d:WeaklyGraphIsotro} for a function $f=(f_1,\ldots,f_{2n-k})$ interacts in the case $k=n$ with the form of the affine approximations $A_{x,r,f_i}$ in Theorem \ref{t:DorronsoroSpecial} (or $\widetilde{A}_{x,r,f_i}$ in Remark \ref{t:DorronsoroAzzam}) for all points $x$ and $r>0$.  We will apply the construction from Theorem \ref{t:DorronsoroSpecial} to each component of $f$ and verify that the thus produced $n$-tuple of affine functions yields an isotropic affine map.

\begin{theorem}[Isotropic Dorronsoro's Theorem]\label{appr}
Let $f = (f_1, \dots, f_n):\R^n\to\R^n$ be a map with $f_i \in L^2_{\rm loc}(\mathbb{R}^n)$ for all $i=1,\cdots,n$. Define
\begin{equation}\label{3.4}
    \Omega^{iso}_{2,f}(x,r) := \inf_{A} \left( \fint_{B(x,r)} \left( \frac{|f(y) - A(y)|}{r} \right)^2 dy \right)^{1/2},\quad x\in \mathbb{R}^n,r>0,
\end{equation}
where the infimum is over all \textbf{isotropic} affine maps $A: \R^n \to \R^n$ and
\begin{equation}\label{3.5}
    |f(y) - A(y)|=\left(\sum_{i=1}^n |f_i(y)-A_i(y)|^2\right)^{1/2}.
\end{equation}
If $f_i\in W^{1,2}(\R^n)$  and $f$ is weakly graph isotropic, that is,
$$\partial_i f_j(z) = \partial_j f_i(z)\qquad 
\text{a.e. }z\in \mathbb{R}^n,$$
then $A_{x,r,f}(w):=(A_{x,r,f_1},\cdots,A_{x,r,f_n})$ is isotropic affine and
\[
\Omega^{iso}_{2}(f) := \int_{\mathbb{R}^n} \int_0^\infty \Omega^{iso}_{2,f}(x,r)^2 \frac{dr}{r} dx \lesssim_{n} \|\nabla f\|_{L^2}^2.
\]
\end{theorem}

Since the class of  admissible  functions in the definition of $\Omega_{2,f}^{iso}$ is restricted to isotropic ones,  $\Omega_{2,f}(x,r)\leq \Omega_{2,f}^{iso}(x,r)$ holds for the comparison with the classical coefficients in Theorem \ref{t:DorronsoroSpecial}. For $k>1$, the \emph{isotropic Grassmannian} of $k$-dimensional isotropic subspaces of $\mathbb{R}^{2n}$ is a submanifold of codimension $\dim O(k)$ inside the full Grassmannian $G(2n,k)$, see for instance \cite[Remark 2.8]{MR2955184}. It is therefore non-trivial information that the inequality $\Omega_{2}(f)\lesssim_n  \|\nabla f\|_{L^2}^2$ from Theorem \ref{t:DorronsoroSpecial} can be upgraded to $\Omega^{iso}_{2}(f)\lesssim_n  \|\nabla f\|_{L^2}^2$ in the context of Theorem \ref{appr}.

\begin{proof}[Proof of Theorem \ref{appr}] Fix $x$ and $r$, and consider the affine map $
A_{x,r,f}:=(A_{x,r,f_1},\cdots,A_{x,r,f_n})$,
where
\begin{equation*}
    A_{x,r,f_i}(w) :=\fint_{B(x,r)} f_i(z) dz + \frac{n+2}{r^2} \sum_{j=1}^n (w_j - x_j) \fint_{B(x,r)} f_i(z)(z_j - x_j) dz,\qquad w \in\R^n.
\end{equation*}
The $j$-th partial derivative of $A_{x,r,f_i}$ is
$$ \partial_j A_{x,r,f_i}(y) = \frac{n+2}{r^2} \fint_{B(x,r)} f_i(z)(z_j - x_j) dz. $$ 
We know that $\partial_i f_j = \partial_j f_i$ in the distributional sense, and we aim to exploit this information to show that $ \partial_j A_{x,r,f_i}=\partial_i A_{x,r,f_j}$, which will prove that $A_{x,r,f}$ is isotropic affine. Thus, our task is to show for all $i,j\in \{1,\ldots,n\}$ that
\begin{equation}\label{eq:CommIntGoal}
    \int_{B(x,r)} f_i(z)(z_j - x_j) dz = \int_{B(x,r)} f_j(z)(z_i - x_i) dz.
\end{equation}
Assume first that $f_1,\ldots,f_n$ are continuous. 
As in the proof of Theorem \ref{t:Lift}, the distributional version of Poincar\'{e}'s lemma implies that there exists  $f_{n+1}\in \mathcal{D}'(\mathbb{R}^n)$ such that $\nabla f_{n+1} = -(f_1, \dots, f_n)$ in the distributional sense. Since the components $f_1,\ldots,f_n$ are assumed to be continuous, $f_{n+1}$ is $C^1$ and the identity $\nabla f_{n+1} = -(f_1, \dots, f_n)$ holds for the classical derivatives pointwise everywhere. Thus we can write 
$$ \int_{B(x,r)} f_i(z)(z_j - x_j) dz = -\int_{B(x,r)} \partial_i f_{n+1}(z)(z_j - x_j) dz. $$
Applying the divergence theorem (\cite[\S C.2]{MR2597943}) and noting that the $i$-th component of the outward unit normal vector on the boundary $\partial B(x,r)$ is given by $\nu_i = \frac{z_i - x_i}{r}$, we obtain
\begin{align*}
    -\int_{B(x,r)} \partial_i f_{n+1}(z)(z_j - x_j) dz 
    &= \int_{B(x,r)} f_{n+1}(z) \partial_i (z_j - x_j) dz - \int_{\partial B(x,r)} f_{n+1}(z)(z_j - x_j) \frac{z_i - x_i}{r} d\sigma \\
    &= \delta_{ij} \int_{B(x,r)} f_{n+1}(z) dz - \frac{1}{r} \int_{\partial B(x,r)} f_{n+1}(z)(z_j - x_j)(z_i - x_i) d\sigma\\
    &= \delta_{ji} \int_{B(x,r)} f_{n+1}(z) dz - \frac{1}{r} \int_{\partial B(x,r)} f_{n+1}(z)(z_i - x_i)(z_j - x_j) d\sigma\\
    &= -\int_{B(x,r)} \partial_j f_{n+1}(z)(z_i - x_i) dz.
\end{align*}
Therefore
\begin{align*}
    \int_{B(x,r)} f_i(z)(z_j - x_j) dz =-\int_{B(x,r)} \partial_i f_{n+1}(z)(z_j - x_j) dz & = -\int_{B(x,r)} \partial_j f_{n+1}(z)(z_i - x_i) dz \\
    &= \int_{B(x,r)} f_j(z)(z_i - x_i) dz.
\end{align*}
This proves \eqref{eq:CommIntGoal} in the case of \emph{continuous} $f_1,\ldots,f_n$. The general case follows by a standard mollification argument. By what we discussed,  \eqref{eq:CommIntGoal} holds for the mollifications $f_1^{\varepsilon},\ldots,f_n^{\varepsilon}$ of $f_1,\ldots,f_n$. Since $f_i^{\varepsilon}$ converges to $f_i$ in $L^2(B(x,r))$, and $z\mapsto z_j-x_j$ is bounded on $B(x,r)$, we have
\begin{displaymath}
    \lim_{\varepsilon\to 0} \int_{B(x,r)} f_i^{\varepsilon}(z)(z_j - x_j) dz = \int_{B(x,r)} f_i(z)(z_j - x_j) dz,
\end{displaymath}
and analogously with the roles of ``$i$'' and ``$j$'' reverted. This yields the general case of  \eqref{eq:CommIntGoal}. 

So, the affine map $A_{x,r,f}$ satisfies condition (\ref{iso}) for $k=n$, thus, it is an isotropic affine map. By (\ref{app}), (\ref{3.4}) and (\ref{3.5}), we also have
\begin{align*}
    \Omega^{iso}_{2}(f) = \int_{\mathbb{R}^n} \int_0^{\infty}\Omega^{iso}_{2,f}(x,r)^2 \frac{dr}{r} dx
                \le \int_{\mathbb{R}^n} \int_0^{\infty} \sum_{i=1}^n\Psi_{2, f_i}(x, r)^2 \frac{dr}{r} dx\lesssim_n\|\nabla f \|_{L^2}^2.\label{3.11}
\end{align*}
This proves the theorem.
\end{proof}

\begin{remark}
An alternative proof of the second part of Theorem \ref{appr} passes via the affine functions  $\widetilde{A}_{x,r,f_i}$ from Azzam's approach \cite{MR3512428} (recall Remark \ref{t:DorronsoroAzzam}), instead of the canonical affine functions $A_{x,r,f_i}$ from Dorronsoro's original proof \cite{MR796440}. One can easily verify that, likewise,   $ \partial_j\widetilde{A}_{x,r,f_i}=\partial_i \widetilde{A}_{x,r,f_j}$ holds for $i,j\in \{1,\ldots,n\}$ if $f$ is a weakly graph isotropic function. 
\end{remark}

 Theorem \ref{appr} yields an isotropic approximation result for weakly graph isotropic Lipschitz functions from $\mathbb{R}^n$ to $\mathbb{R}^n$ in the spirit of a Carleson-type condition at all places and scales.

\begin{proposition}\label{aiso}
Let $f=(f_1,\ldots,f_n):\mathbb{R}^n \to \mathbb{R}^n$ be a Euclidean Lipschitz map that is weakly graph isotropic, that is, $\partial_i f_j = \partial_j f_i$ almost everywhere, $i,j=1,\ldots,n$. Then, for all $y\in \mathbb{R}^n$ and $R>0$,
\[
    \Omega^{iso}_{2,y,R}(f):= \int_{B(y,R)} \int_0^R \Omega^{iso}_{2,f}(x,r)^2 \frac{dr}{r} dx \lesssim_{n,\mathrm{Lip}(f)} R^n.
\]
\end{proposition}

\begin{proof} 
Let $y\in\mathbb{R}^n$ and $R>0$ be arbitrary.

\noindent \textbf{Case 1: $y=0$.} We first assume $f(0)=0$. Since $\partial_i f_j = \partial_j f_i$ almost everywhere, Theorem \ref{t:Lift} guarantees the existence of $\varphi_{n+1}:\R^n\to\R$ such that $\varphi:=(f,\varphi_{n+1})$ is intrinsic $L$-Lipschitz from $\mathbb{V}_0=\mathbb{R}^n \times \{0\}$ to $\mathbb{W}_0=\{0\}\times \mathbb{R}^{n+1}$, with intrinsic Lipschitz constant $L=L(n,\mathrm{Lip}(f))$. In addition, we may assume that $\varphi_{n+1}(0)=0$, and hence $\varphi(0)=0$. By Proposition \ref{loc}, there exists an intrinsic $L'$-Lipschitz function $\tilde{\varphi}_{2R}:\V_0\to\W_0$ such that
$$
\tilde{\varphi}_{2R}(z) = \begin{cases}
        \varphi(z), & z\in B(0,2R) \\
        0, &  z\in \V_0\setminus B(0, 4R),
    \end{cases}
$$
where $L'=L'(n,L)$.
By Theorem \ref{diff}, the functions  $\tilde{\varphi}_{2R,i}$, $i=1,\ldots,n$, are Euclidean Lipschitz and satisfy
\begin{displaymath}
    \partial_j\tilde{\varphi}_{2R,i}(z)=\partial_i\tilde{\varphi}_{2R,j}(z),\qquad \text{a.e. } z\in \mathbb{R}^n,i,j=1,\cdots,n.
\end{displaymath}
We denote the horizontal components by $\tilde{f}_{2R}:=(\tilde{\varphi}_{2R,1},\ldots,\tilde{\varphi}_{2R,n})$.
Since $\tilde{f}_{2R}$ is a compactly supported Lipschitz function, it follows that $\tilde{f}_{2R}\in W^{1,2}(\mathbb{R}^n,\mathbb{R}^n)$, so that Theorem \ref{appr} is applicable to this function. 
By (i) of Theorem \ref{diff}, we also notice that the Lipschitz constant of the components of $\tilde{f}_{2R}$ is still $L'$. Moreover, for $x\in B(0,R)$, $0<r<R$, we have $B(x,r)\subset B(0,2R)$, and therefore $f=\tilde{f}_{2R}$ on $B(x,r)$, and thus 
\begin{equation}\label{3.15}
\Omega^{iso}_{2,f}(x,r)=\Omega^{iso}_{2,\tilde{f}_{2R}}(x,r), \quad
x\in B(0,R),\, 0<r<R.
 \end{equation}
Then, it follows that
\begin{align*}
\Omega^{iso}_{2,0,R}(f)&:= \int_{B(0,R)} \int_0^R \Omega^{iso}_{2,f}(x,r)^2 \frac{dr}{r} dx\notag \\
&\overset{\eqref{3.15}}{=} \int_{B(0,R)} \int_0^R \Omega^{iso}_{2,\tilde{f}_{2R}}(x,r)^2 \frac{dr}{r} dx \\
&\leq  \int_{\mathbb{R}^n} \int_0^\infty \Omega^{iso}_{2,\tilde{f}_{2R}}(x,r)^2 \frac{dr}{r} dx \notag\\&\lesssim_{n,\mathrm{Lip}(\tilde{f}_{2R})}R^n.\notag
\end{align*}
The last inequality is a direct consequence of  Theorem \ref{appr} (isotropic Dorronsoro's Theorem) and the fact that $\tilde{f}_{2R}$ vanishes outside $B(0,4R)$. Since the constants $L$ and $L'$ are both depending on $\mathrm{Lip}(f)$, we conclude that
$$\Omega^{iso}_{2,0,R}(f)\lesssim_{n,\mathrm{Lip}(f)}R^n.$$

This proves the claim for $y=0$ and $f(0)=0$. If $f(0)\neq 0$, we consider the function $f^0:=f-f(0)$. This is still Euclidean Lipschitz (with the same Lipschitz constant as $f$) and satisfies $\partial_i f^0_j =\partial_j f^0_i$, but now additionally $f^0(0)=0$. Then, by what we have proved so far, we know
\begin{equation}\label{eq:f0_ineq}
\Omega^{iso}_{2,0,R}(f^0)\lesssim_{n,\mathrm{Lip}(f)}R^n.
\end{equation}
On the other hand,
\begin{align*}
    \Omega^{iso}_{2,f^0}(x,r)&=\inf_{A} \left( \fint_{B(x,r)} \left( \frac{|f^0(z) - A(z)|}{r} \right)^2 dz \right)^{1/2}\\
    &=\inf_{A} \left( \fint_{B(x,r)} \left( \frac{\left|f(z) - (A(z)+f(0))\right|}{r} \right)^2 dz \right)^{1/2}
\end{align*}
where the infimum is over all isotropic affine maps $A: \R^n \to \R^n$, but $A+f(0)$ is again an isotropic affine map, 
and all affine isotropic maps can be written in this form.
Thus, $ \Omega^{iso}_{2,f^0}(x,r)=\Omega^{iso}_{2,f}(x,r)$ for all $x$ and $r$. Therefore, the statement \eqref{eq:f0_ineq} holds also for $f^0$ replaced by $f$.

\medskip

\noindent \textbf{Case 2: $y\neq0$.} By definition, we have
\begin{align*}
\Omega^{iso}_{2,y,R}(f)
&= \int_{B(y,R)} \int_0^R \Omega^{iso}_{2,f}(x,r)^2 \frac{dr}{r} dx\notag \\
&=  \int_{B(y,R)} \int_0^R\inf_{A}  \fint_{B(x,r)} \left( \frac{|f(z) - A(z)|}{r} \right)^2 dz  \frac{dr}{r} dx    \\
&\overset{\text{let }v=x-y}{=} \int_{B(0,R)} \int_0^R\inf_{A}  \fint_{B(v+y,r)} \left( \frac{|f(z) - A(z)|}{r} \right)^2 dz  \frac{dr}{r} dv\\
&\overset{\text{let }w=z-y}{=} \int_{B(0,R)} \int_0^R\inf_{A}  \fint_{B(v,r)} \left( \frac{|f^y(w) - A(w+y)|}{r} \right)^2 dw  \frac{dr}{r} dv
\end{align*}
where $f^y(w):=f(w+y)$ and the infimum is over all isotropic affine maps $A: \R^n \to \R^n$. However, $w\mapsto A(w+y)$ is again an isotropic affine map and any isotropic affine map can be written in this form; thus, by Case 1, we have
\begin{align*}
    \Omega^{iso}_{2,y,R}(f)
    =\Omega^{iso}_{2,0,R}(f^y)\lesssim_{n,\mathrm{Lip}(f^y)}R^n\lesssim_{n,\mathrm{Lip}(f)}R^n,
\end{align*}
completing the proof. 
\end{proof}

\section{Geometric lemmas for low-dimensional intrinsic Lipschitz graphs}\label{s:GLem}
The main results of this section are geometric lemmas (Definition \ref{def:qgeomlemballs}) for $k$-di{\-}men{\-}sional intrinsic Lipschitz graphs in $\mathbb{H}^n$. We apply different methods, and obtain results of different strengths, for the cases $k=n$ (Section \ref{ss:k=n}), $k=1$  (Section \ref{ss:Glemk=1}), and $1<k<n$ (Section \ref{s:IntegralGeo}). This distinction is related to differences between the $1$- and higher dimensional cases that are also present in the Euclidean setting, but additionally, we face differences specific for  intrinsic Lipschitz maps in these three cases, recall Remarks \ref{r:k=1} and \ref{r:k=n}. In Section \ref{ss:ConseqGLem}, we collect properties of intrinsic Lipschitz graphs that follow from the geometric lemmas, such as weak geometric lemmas (Definition \ref{def:WGLballs}) and corona decompositions by intrinsic Lipschitz graphs with small constants.

Throughout this section, $B(x,r)$ denotes the ball in $\Hn$ computed with respect to the Kor\'{a}nyi norm, while $B_{\mathbb{R}^k}(v,r)$ denotes a Euclidean ball in $\mathbb{R}^k$.

\subsection{Preliminaries on flatness coefficients}

We recall the definitions of various quantitative coefficients that all measure some aspects of `flatness' for low-dimensional sets in the Heisenberg group. We refer the reader to \cite{arXiv:2601.03837} for a historical account and motivation to study these coefficients. First, we consider the \emph{horizontal $\beta$-numbers}, which are a Heisenberg version of the classical $\beta$-numbers.

\begin{definition}[Horizontal $\beta$-numbers]\label{d:HorizBeta}
Let $E \subset \Hn$ be a $k$-regular set. For $x \in \Hn$ and $r > 0$, the \emph{horizontal} $\beta$-numbers are defined as:
\begin{align*}
    \beta^E_{p,\mathcal{V}_k}
(x, r) &= \inf_{\V\in\mathcal{V}_k} \left( \fint_{B(x,r) \cap E} \left( \frac{d(y, \V)}{r} \right)^p d\mathcal{H}^k(y) \right)^{1/p},\quad 1\leq p<+\infty,\\ \beta^E_{\infty,\mathcal{V}_k}(x,r)&:=\inf_{\V\in\mathcal V_k}
\,\sup_{y\in B(x,r)\cap E} \frac{d(y,\V)}{r},
\end{align*}
where the infimum is over all $k$-dimensional horizontal planes $\V\in\mathcal V_k$.
\end{definition}

The horizontal $\beta$-numbers fall into the general framework studied in \cite{MR4485846}. However, the next definitions are not of this type as they all involve expressions which are of a more complicated form than $d(y,\V)$.

\begin{definition}[Stratified $\beta$-numbers]\label{d:StratifBeta}
Let $E \subset \Hn$ be a $k$-regular set. For $x \in \Hn$ and $r > 0$, the \emph{stratified} $\beta$-numbers are defined as:\begin{small} 
\begin{align*}
   \widehat \beta^E_{p,\mathcal{V}_k}
(x, r) &= \inf_{\V\in\mathcal{V}_k} \left(\left[\fint_{B(x,r) \cap E} \left( \frac{d_{\rm Eucl}(\pi(y), \pi(\V))}{r} \right)^p d\mathcal{H}^k\right]^\frac{2}{p}+ \left[\fint_{B(x,r) \cap E} \left( \frac{d(y, \V)}{r} \right)^p d\mathcal{H}^k\right]^\frac{4}{p} \right)^{1/4},
\end{align*}
for $1\leq p<\infty$, and
\begin{align*}
\widehat{\beta}^E_{\infty,\mathcal{V}_k}(x,r)&:=\inf_{\V\in\mathcal V_k}\,\left(\sup_{y\in B(x,r)\cap E}\left[\frac{d_{\mathrm{Eucl}}(\pi(y),\pi(\V))}{r}\right]^2+\sup_{y\in B(x,r)\cap E}\left[\frac{d(y,\V)}{r}\right]^4\right)^{\frac{1}{4}},
\end{align*}
\end{small}where the infimum is over all $k$-dimensional horizontal planes $\mathbb{V}\in\mathcal V_k$.
\end{definition}

The \emph{horizontal projection $\beta$-numbers}, which we define next, correspond essentially to the first half of the stratified $\beta$-numbers.

\begin{definition}[Horizontal projection $\beta$-numbers]
Let $E \subset \Hn$ be a $k$-regular set. For $x \in \Hn$ and $r > 0$, the \emph{horizontal projection} $\beta$-numbers are defined as:
\begin{align*}
    \beta^E_{p,\pi, \mathcal{V}_k}
(x, r) &= \inf_{\V\in\mathcal{V}_k} \left( \fint_{B(x,r) \cap E} \left( \frac{d_{\rm Eucl}(\pi(y), \pi(\V))}{r} \right)^p d\mathcal{H}^k(y) \right)^{1/p}, \quad p<+\infty,\\ 
\beta^E_{\infty,\pi,\mathcal{V}_k}(x,r)&:=\inf_{\V\in\mathcal V_k}\,  \sup_{y\in B(x,r)\cap E} \frac{d_{\mathrm{Eucl}}(\pi(y),\pi(\V))}{r},
\end{align*}
where the infimum is over all $k$-dimensional horizontal planes $\V\in\mathcal V_k$.
\end{definition}
Equivalently, horizontal projection $\beta$-numbers could be defined by taking the infimum over all isotropic affine $k$-planes $V$ in $\mathbb{R}^{2n}$, instead of considering $\pi(\mathbb{V})$ for horizontal $k$-planes $\mathbb{V}$.
The \emph{projection $\beta$-numbers} are defined similarly as the horizontal projection $\beta$-numbers, only that the infimum is taken over the larger family of \emph{all} affine $k$-planes.
\begin{definition}[Projection $\beta$-numbers]\label{def_projbetas}
Let $E \subset \Hn$ be a $k$-regular set. For $x \in \Hn$ and $r > 0$, the \emph{projection} $\beta$-numbers are defined as:
\begin{align*}
    \beta^E_{p,\pi, A(2n,k)}
(x, r) &= \inf_{W\in A(2n,k)} \left( \fint_{B(x,r) \cap E} \left( \frac{d_{\rm Eucl}(\pi(y), W)}{r} \right)^p d\mathcal{H}^k(y) \right)^{1/p}, \quad p<+\infty,\\ 
\beta^E_{\infty,\pi,A(2n,k)}(x,r)&:=\inf_{W\in A(2n,k)}\,  \sup_{y\in B(x,r)\cap E} \frac{d_{\mathrm{Eucl}}(\pi(y),W)}{r},
\end{align*}
where  $A(2n,k)$ denotes the affine Grassmannian of $k$-dimensional affine planes in $\R^{2n}$.
\end{definition}

The above coefficients behave well under isometries induced by unitary maps as in Section \ref{ss:RedStd}. 

\begin{lemma}\label{l:RotInvCoeff} Let $n\in \mathbb{N}$, $k\in \{1,\ldots,n\}$, and let $h^E:\mathbb{H}^n \times (0,\infty) \to [0,\infty)$ be any of the coefficient functions in Definitions \ref{d:HorizBeta}--\ref{def_projbetas} associated to a $k$-regular set $E\subset \mathbb{H}^n$. 
If $U\in U(n)$ and $R_U(z,t):=(Uz,t)$, then $h^{R_U(E)}(R_U(x),r)=h^E(x,r)$ for all $x\in \mathbb{H}^n$ and $r>0$.
\end{lemma}
\begin{proof}
This follows by arguments similar to \cite[Lemmas 3.12-3.13]{arXiv:2601.03837}.
\end{proof}

In the following sections, we will study these coefficients  for $E$ equal to a $k$-dimensional intrinsic Lipschitz graph, and its horizontal projection. In this case, there is a natural parametrization of the graph, and we will work with `parametric' versions of the $\beta$-numbers. To ensure that the relevant coefficients are uniformly bounded, the following lemma will be useful.
\begin{lemma}\label{lemma_lipconst}
   Let $k>0,m\in \mathbb N$. Assume $f:\R^k\to\R^{m}$ is a Lipschitz map with Lipschitz constant $\mathrm {Lip}(f)$. Assume $A:\R^k\to\R^m$ is an affine map with the property that
    \begin{equation*}
        \fint_{B_{\mathbb{R}^k}(x,r)}\frac{|A(y)-f(y)|^2}{r^2} dy\leq 2\,\mathrm {Lip}(f)^2
    \end{equation*}
    for some $x\in \R^k$, $r>0$. There is some universal constant $N_k$, depending only on $k$, such that $\mathrm{Lip}(A)\leq N_k\mathrm{Lip}(f)$.
\end{lemma}
\begin{proof}
    Let $x=(x_1,\dots, x_k)\in\R^k$, $r>0$ be as in the assumptions. Let $e\in\mathbb S^{k-1}$ be a direction of maximal stretch for $A$ (this means that $|A(y)-A(y+he)|=\mathrm {Lip}(A)|h|$ for every $y\in\R^k$, where $\mathrm {Lip}(A)$ equals the operator norm $\|A\|$). There exists a dimensional constant $c=c(k)$ such that, denoting by $Q(x,r)$ the cube centered at $x$ with side length $r$ and edges parallel to the coordinate axes, it holds $Q(x,cr)\subset B(x,r)$: take for instance $c(k)=2/\sqrt{k}$. Without loss of generality, assume that $e=e_1=(1,0,\dots,0)\in\R^k$ and write $Q(x,cr)=(x_1-cr/2,x_1+cr/2)\times \widehat Q$, where $\widehat Q$ is a cube in $\R^{k-1}$. Assume by contradiction that $\mathrm{Lip}(A)>N_k\mathrm{Lip}(f)$, where the constant $N_k\geq 1$ will be determined later. Fix now $\hat y=(y_2,\dots,y_{k})\in \widehat Q$ and distinguish two cases.\\

\noindent    \textbf{Case 1:} 
  \[|A(s,\hat y)-f(s,\hat y)|>\sqrt{\frac{4\omega_k}{c^k}}\mathrm{Lip}(f)\,r\quad\text{ for all }s\in (x_1-cr/2, x_1+cr/2).\]
    In this case, we can simply estimate
    \[\int_{x_1-cr/2}^{x_1+cr/2}\frac{|A(s,\hat y)-f(s,\hat y)|^2}{r^2}\,ds\geq \frac{4\omega_k}{c^{k-1}}\mathrm{Lip}(f)^2r.\]

\noindent  \textbf{Case 2:}  There exists $s_0=s_0(\hat y)\in(x_1-cr/2,x_1+cr/2)$ such that \[|A(s_0,\hat y)-f(s_0,\hat y)|\leq \sqrt{\frac{4\omega_k}{c^k}}\mathrm{Lip}(f)\,r.\]
    In this case we estimate, by our choice of $e$,
    \begin{align*}
        |A(s,\hat y)-f(s,\hat y)|&\geq |A(s,\hat y)-A(s_0,\hat y)|-|A(s_0,\hat y)-f(s_0,\hat y)|-|f(s_0,\hat y)-f(s,\hat y)|\\ &\geq \mathrm{Lip}(A)|s-s_0|-\sqrt{\frac{4\omega_k}{c^k}}{\mathrm{Lip}}(f)\,r-\mathrm{Lip}(f)|s-s_0|.
    \end{align*}
  Notice that there exists an interval $I\subset(x_1-cr/2,x_1+cr/2)$ of length at least $cr/4$ such that $|s-s_0|\geq cr/4$ for any $s\in I$. In particular, for those values of $s$, we can estimate
  \begin{align*}|A(s,\hat y)-f(s,\hat y)|&\geq [\mathrm {Lip}(A)-\mathrm{Lip}(f)]\frac{cr}{4}-\sqrt{\frac{4\omega_k}{c^k}}{\mathrm{Lip}}(f)\,r\\ &\geq \left[(N_k-1)\frac{c}{4}-\sqrt{\frac{4\omega_k}{c^k}}\right]\mathrm{Lip}(f)r.\end{align*}
  Integrating, this yields
  \[\int_{x_1-cr/2}^{x_1+cr/2}\frac{|A(s,\hat y)-f(s,\hat y)|^2}{r^2}\,ds\geq\int_I\frac{|A(s,\hat y)-f(s,\hat y)|^2}{r^2}\,ds\geq\frac{4\omega _k}{c^{k-1}}\mathrm{Lip}(f)^2r,\]
  if $N_k$ is chosen large enough depending only on $k$. By Fubini's theorem, we can now compute the integral on $Q(x,cr)$:
  \begin{align*}   
  \fint_{B_{\mathbb{R}^k}(x,r)}\frac{|A(y)-f(y)|^2}{r^2}dy&\geq \frac{1}{\omega_kr^k}\int_{Q(x,cr)}\frac{|A(y)-f(y)|^2}{r^2}dy\\ &=\frac{1}{\omega_kr^k}\int_{\widehat Q}\int_{x_1-cr/2}^{x_1+cr/2}\frac{|A(s,\hat y)-f(s,\hat y)|^2}{r^2}dsd\hat y\\ &\geq \frac{1}{\omega_k r^k}\frac{4\omega_k}{c^{k-1}}\mathrm{Lip}(f)^2r(cr)^{k-1}=4\mathrm{Lip}(f)^2.\end{align*}
  This is a contradiction and the proof is concluded. Note that, in the case $\mathrm{Lip}(f)=0$, then the unique map $A$ satisfying the assumption of the lemma is constant and the conclusion follows.
\end{proof}

\subsection{Isotropic approach for the case $k=n$}\label{ss:k=n}
The main goal of this section is to prove Theorem \ref{t:GlemStratBetanILG}, a geometric lemma  with stratified $\beta$-numbers for $n$-dimensional intrinsic Lipschitz graphs in $\mathbb{H}^n$. Our strategy is to apply the isotropic Dorronsoro theorem, in the form of Proposition \ref{aiso}, to the horizontal projection of the considered intrinsic Lipschitz function $\varphi$. Since $\nabla \varphi_{n+1}= -(\varphi_1,\ldots,\varphi_n)$ by Theorem \ref{horcon} (and similarly for intrinsic affine maps),  a Poincar\'e inequality will then allow us to control the vertical errors of the affine approximations in terms of the horizontal errors.

For the remainder of this section, we fix an $n$-dimensional intrinsic Lipschitz graph in $\mathbb{H}^n$. According to Lemma \eqref{l:RotRed}, it is not restrictive to assume that the graph is associated to the splitting $\mathbb{H}^n = \V_0 \ltimes  \W_0$, where $\V_0$ and $\W_0$ are as in Remark \ref{r:FixedGroups}, that is,
 \begin{equation}\label{eq:V0W0}
    \mathbb{V}_0:=\{(x,0):\,x\in \mathbb{R}^n\}\quad \text{and}\quad \mathbb{W}_0=\{(0,y,t):\, y\in \mathbb{R}^n,\,t\in\mathbb{R}\}.
\end{equation}
We will frequently identify $\mathbb{V}_0$ with $\mathbb{R}^n$, and use the projections
\begin{equation*}
 \Pi: \Hn \to \mathbb{R}^n,\quad    \Pi(w_1, \dots, w_{2n}, t) = (w_1, \dots, w_n), 
\end{equation*}
and
\begin{equation*}
 \Pi^{\bot}: \Hn \to \mathbb{R}^n,\quad    \Pi^\bot(w_1, \dots, w_{2n}, t) = (w_{n+1}, \dots, w_{2n}).
\end{equation*}
We also use the projection 
\begin{displaymath}
\pi:\mathbb{H}^n \to \mathbb{R}^{2n},\quad \pi(z,t)=z.
\end{displaymath} With this notation, $\pi(z,t)=(\Pi(z,t),\Pi^{\bot}(z,t))$.

For the intrinsic graph $\Gamma$ of an intrinsic Lipschitz function $\varphi:\mathbb{V}_0\to \W_0$, we aim to bound $ \widehat{\beta}^{\Gamma}_{2,\mathcal{V}_n}(x,r)$ from above by a suitable power of the quantity $\Omega_{2, \Pi^{\bot}(\varphi)}^{iso}(\Pi(x),r)$ in the isotropic Dorronsoro theorem for Lipschitz functions (Proposition \ref{aiso}).

\begin{theorem}\label{t:ControlStratifBetaByOmega}
Assume that $\varphi:\mathbb{V}_0 \to \mathbb{W}_0$ is an intrinsic Lipschitz function. If $\Gamma$ denotes the intrinsic graph of $\varphi$, then 

\begin{displaymath}
\widehat{\beta}^{\Gamma}_{2,\mathcal{V}_n}(x,r)^{4}\lesssim_{n,\mathrm{Lip}(\varphi_1),\dots,\mathrm{Lip}(\varphi_n)} \Omega^{iso}_{2,\Pi^{\bot}(\varphi)}(\Pi(x),r)^2,\quad x\in \Gamma,\,r>0.
\end{displaymath}
\end{theorem}
\begin{proof}
    \noindent \textbf{Step 1: Reduction to parameterized version.}
For fixed $x\in \Gamma$ and $r>0$, consider
\begin{equation*}
   \widehat{\beta}^{\Gamma}_{2,\mathcal{V}_n}
(x, r) = \inf_{V\in\mathcal{V}_n} \left[\fint_{B(x,r) \cap \Gamma} \left(\frac{d_{\rm Eucl}(\pi(y), \pi(V))}{r}\right)^2 d\mathcal{H}^n(y)+ \left(\fint_{B(x,r) \cap \Gamma} \left(\frac{d(y, V)}{r} \right)^2d\mathcal{H}^n(y)\right)^2\right]^{1/4}.
\end{equation*}
We will replace these integrals by integrals over domains in $\mathbb{R}^n$, as in the definition of $\Omega^{iso}_{2,\Pi^{\bot}(\varphi)}$. 
Recall that $\Phi:\V_0\to\Hn$ denotes the graph map associated to $\varphi.$

First, by AD-regularity of the graph of an entire intrinsic Lipschitz map (\cite[Theorem 3.9]{zbMATH06607318}), we have $\mathcal{H}^n(\Gamma \cap B(x,r)) \sim r^n$. Second, since $\Phi^{-1}(\Gamma \cap B(x,r))= \Pi(\Gamma \cap B(x,r))$, by the area formula (Corollary \ref{cov}), and by noting that the Jacobian $J(d\Phi_v) \lesssim_{n,\mathrm{Lip}(\varphi_1),\dots,\mathrm{Lip}(\varphi_n)} 1 $ a.e. $v$ (see Lemma \ref{rem_jacobian}), we have for every nonnegative Borel function $f:\mathbb{H}^n \to \mathbb{R}$ that 
\begin{align}
  \int_{\Gamma \cap B(x,r)}f(y)\,d\mathcal{H}^n(y)=   \int_{\Pi(\Gamma \cap B(x,r))} f(\Phi(v)) J(d\Phi_v) dv\notag \lesssim_{n,\mathrm{Lip}(\varphi_1),\dots,\mathrm{Lip}(\varphi_n)}   \int_{\Pi(\Gamma \cap B(x,r))} f(\Phi(v)) dv. \label{eq:4.3}
\end{align}
Third, one can easily check, by the definition of Kor\'{a}nyi norm, that for $x \in \Hn$:
\begin{equation*}
    \Pi(\Gamma \cap B(x,r)) \subset \Pi(B(x,r)) \subset B_{\mathbb{R}^n}(\Pi(x), r)
\end{equation*}
where $B(x,r)$ is a ball in $\Hn$, and $B_{\mathbb{R}^n}(\Pi(x), r)$ is a ball in $\R^n\equiv \mathbb{V}_0$. Combining these three ingredients, we deduce that, for all $\V\in \mathcal{V}_n$,
\begin{equation}\label{eq:UpperStratifiedBeta2}
 \widehat{\beta}^{\Gamma}_{2,\mathcal{V}_n}
(x, r)\lesssim_{n,\mathrm{Lip}(\varphi_i)}
 \left[\fint_{B_{\mathbb{R}^n}(\Pi(x),r)} \left(\tfrac{d_{\rm Eucl}(\pi(\Phi(v)), \pi(\V))}{r}\right)^2 dv + \left(\fint_{B_{\mathbb{R}^n}(\Pi(x),r)} \left(\tfrac{d(\Phi(v), \V)}{r}\right)^2 dv\right)^2\right]^{1/4}.
\end{equation}
Now let $A:\mathbb{R}^n\to \mathbb{R}^n$ be an arbitrary affine isotropic map. 
Then there exists a function $A_{n+1}:\mathbb{R}^n \to \mathbb{R}$ such that the intrinsic graph of $(A,A_{n+1})$ is an affine horizontal $n$-plane (see Proposition \ref{ahp}). Such $A_{n+1}$ is not unique, but we will choose it so that

\begin{equation}\label{eq:VertCondNew}
 \int_{B_{\R^n}(\Pi(x),r)}(\varphi_{n+1}(v)-A_{n+1}(v))dv=0.
\end{equation}

This is possible by adding a suitable constant, as explained in Remark \ref{r:VertTranslIntrAff}. Thus, given an arbitrary affine isotropic $A$, the choice of $A_{n+1}$ depends additionally on $\varphi_{n+1}|_{B_{\R^n}(\Pi(x),r)}$.

 Let $\mathbb{A}$ be the graph map of $(A,A_{n+1})$. 
Then $\V:=\mathbb{A}(\mathbb{V}_0)$ is an affine horizontal $n$-plane for which we will apply \eqref{eq:UpperStratifiedBeta2} using the estimates
\begin{displaymath}
    d_{\rm Eucl}(\pi(\Phi(v)),\pi(\V))\leq |(v,\Pi^{\bot}(\varphi)(v))-(v,A(v))|=|\Pi^{\bot}(\varphi)(v)-A(v)|,\quad v\in \mathbb{R}^n,
\end{displaymath}
and
\begin{align*}
d(\Phi(v),\V)\leq d(\Phi(v),\mathbb{A}(v))&= d(\varphi(v),(A,A_{n+1})(v))\\&\sim |\Pi^{\bot}(\varphi)(v)-A(v)|+\sqrt{|\varphi_{n+1}(v)-A_{n+1}(v)|},\quad v\in \mathbb{R}^n.
\end{align*}
Using these, we deduce that
\begin{equation}\label{eq:BetaAest2}
     \widehat{\beta}^{\Gamma}_{2,\mathcal{V}_n}
(x, r)\lesssim_{n,\mathrm{Lip}(\varphi_1),\dots,\mathrm{Lip}(\varphi_n)}\inf_{A} [I_{\pi,2}(A) + \left(I_{\pi,2}(A)+ I_{t,2}(A)\right)^2]^{1/4},
\end{equation}
where the infimum runs over all affine isotropic maps $A$, and
\begin{displaymath}
I_{\pi,2}(A):=   \fint_{B_{\mathbb{R}^n}(\Pi(x),r)} \left(\frac{|\Pi^{\bot}(\varphi)(v)-A(v)|}{r}\right)^2 dv
\end{displaymath}
and
\begin{displaymath}
I_{t,2}(A)   := \fint_{B_{\mathbb{R}^n}(\Pi(x),r)} \frac{|\varphi_{n+1}(v)-A_{n+1}(v)|}{r^2} dv,
\end{displaymath}
where $A_{n+1}$ is the lift associated with $A$ as in \eqref{eq:VertCondNew}.

\medskip

\noindent \textbf{Step 2: Comparing with  $\Omega^{iso}_{2,\Pi^{\bot}(\varphi)}(\Pi(x),r)$.} 

Let us now compare the infimum on the right-hand side of \eqref{eq:BetaAest2} with the definition
\begin{equation}\label{eq:Omega_square}
  \left(\Omega^{iso}_{2,\Pi^{\bot}(\varphi)}(\xi,r)\right)^2 := \inf_{A}  \fint_{B_{\mathbb{R}^n}(\xi,r)} \left( \frac{|\Pi^{\bot}(\varphi)(y) - A(y)|}{r} \right)^2 dy, \quad \xi\in \mathbb{R}^n
\end{equation}
where the infimum runs over affine isotropic maps $A:\mathbb{R}^n \to \mathbb{R}^n$. 
Notice that, in the trivial case when $\mathrm{Lip}(\Pi^\perp(\varphi))=0$, then $\Pi^\bot(\varphi)$ is a constant map. In this situation one can choose $A\equiv \Pi^\bot(\varphi)$, which is an affine isotropic map satisfying $I_{\pi,2}(A)=I_{t,2}(A)=0$. Therefore, (\ref{eq:BetaAest2}) implies that $\widehat{\beta}^{\Gamma}_{2,\mathcal{V}_n}
(x, r)=0$ and the conclusion of the theorem is satisfied. Hence we can assume in the following that $\mathrm{Lip}(\Pi^\perp(\varphi))>0$. Fix an arbitrary $0<\varepsilon<\mathrm{Lip}(\Pi^\bot(\varphi))^2$.
Choose $A=A_\varepsilon$ to be an affine isotropic map almost realizing the infimum in  $\Omega^{iso}_{2,\Pi^{\bot}(\varphi)}(\Pi(x),r)$,
namely
\begin{equation}\label{A_eps}
   I_{\pi,2}(A)= \fint_{B_{\mathbb{R}^n}(\Pi(x),r)} \left( \frac{|\Pi^{\bot}(\varphi)(y) - A(y)|}{r} \right)^2 dy <\left(\Omega^{iso}_{2,\Pi^{\bot}(\varphi)}(\Pi(x),r)\right)^2 +\varepsilon.
\end{equation}
Thus, the main task is to control the term $I_{t,2}(A)^{1/2}$ in \eqref{eq:BetaAest2}.
\medskip

\noindent \textbf{Step 3: Controlling the vertical error.} 
We consider the expression
\begin{displaymath}
I_{t,2}(A)   = \fint_{B_{\mathbb{R}^n}(\Pi(x),r)}\frac{|\varphi_{n+1}(v)-A_{n+1}(v)|}{r^2}dv,
\end{displaymath}
 where $A_{n+1}$ is chosen such that $\nabla A_{n+1}=-A$ and \eqref{eq:VertCondNew} holds.
  We denote 
\begin{displaymath}
    u:= \varphi_{n+1}-A_{n+1},
\end{displaymath}
which is in $W^{1,2}_{\rm loc}(\mathbb{R}^n)$. Then, by Cauchy–Schwarz and $(2,2)$-Poincaré's inequality for a ball (see, e.g., \cite[\S 5.8.1, Theorem 2]{MR2597943}), we have\begin{align*}
I_{t,2}^2(A)&\le\frac{1}{r^4}\fint_{B_{\R^n}(\Pi(x),r)}|u(v)|^2dv\\
&\lesssim_n\fint_{B_{\R^n}(\Pi(x),r)}\left(\frac{|\nabla u(v)|}{r}\right)^2dv=I_{\pi,2}(A).
 \end{align*}
 Here we have used for the last identity that
 $\nabla \varphi_{n+1}=-\Pi^{\bot}(\varphi)$ (by Theorem \ref{horcon}). 
Therefore, by \eqref{eq:BetaAest2}
\[ \begin{split} \widehat{\beta}^{\Gamma}_{2,\mathcal{V}_n}
(x, r)^4&\lesssim_{n,\mathrm{Lip}(\varphi_1),\dots,\mathrm{Lip}(\varphi_n)} I_{\pi,2}(A) + \left(I_{\pi,2}(A)+ I_{t,2}(A)\right)^2\\ &\lesssim_{n,\mathrm{Lip}(\varphi_1),\dots,\mathrm{Lip}(\varphi_n)} I_{\pi,2}(A)+I_{\pi,2}(A)^{2}. \end{split}\]
To conclude, we will show that $I_{\pi,2}(A)$ is bounded in terms of the Lipschitz constants of $\varphi_1,\ldots,\varphi_n$, so that $I_{\pi,2}(A)^{2}\lesssim_{n,\mathrm{Lip}(\varphi_1),\dots,\mathrm{Lip}(\varphi_n)} I_{\pi,2}(A)$. This is indeed the case:
Since the map $\bar A:\R^n\to\R^n$, $\bar A(v)\equiv \Pi^\bot(\varphi)(\Pi(x))$ is affine isotropic, then \begin{equation*}
\left(\Omega^{iso}_{2,\Pi^\bot(\varphi)}(\Pi(x),r)\right)^2\leq \fint_{B_{\R^n}(\Pi(x),r)}\left(\frac{|\Pi^\bot(\varphi)(y)-\Pi^\bot(\varphi)(\Pi(x))|}{r}\right)^2dy\leq \mathrm{Lip}(\Pi^\bot(\varphi))^2.\end{equation*}
Therefore, $I_{\pi,2}(A)$ can be bounded as desired, using also \eqref{A_eps}.
 Thus, we finally get \begin{align}
\widehat{\beta}^{\Gamma}_{2,\mathcal{V}_n}(x,r)^{4}&\lesssim_{n,\mathrm{Lip}(\varphi_1),\dots,\mathrm{Lip}(\varphi_n)} I_{\pi,2}(A)\\&\lesssim_{n,\mathrm{Lip}(\varphi_1),\dots,\mathrm{Lip}(\varphi_n)} \Omega^{iso}_{2,\Pi^{\bot}(\varphi)}(\Pi(x),r)^2+\varepsilon,
\end{align}
and we conclude by the arbitrariness of $\varepsilon.$
\end{proof}

We can now state a {geometric lemma} in the sense of Definition \ref{def:qgeomlemballs} for $n$-dimensional intrinsic Lipschitz graphs in $\mathbb{H}^n$ and stratified $\beta$-numbers:

\begin{theorem}\label{t:GlemStratBetanILG}
Let $\mathbb{V}$ be an $n$-dimensional horizontal subgroup of $\mathbb{H}^n$ with complementary vertical subgroup $\mathbb{W}$. If $E\subset \mathbb{H}^n$ is the intrinsic graph of an intrinsic Lipschitz function $\varphi:\mathbb{V}\to \mathbb{W}$, then
\begin{equation*}
    \int_{E\cap B(y,R)} \int_0^R {\widehat\beta {^E_{2,\mathcal{V}_n}}}(x, r)^4 \frac{dr}{r} d\mathcal{H}^n(x) \lesssim_{n, \mathrm{Lip}(\varphi_1),\ldots, \mathrm{Lip}(\varphi_n)} R^n,\qquad y\in E,\,R>0.
\end{equation*}
\end{theorem}
\begin{proof}
By Lemma \eqref{l:RotRed} and Lemma \ref{l:RotInvCoeff}, we may assume that
$\V$ and $\W$ equal the standard subgroups,
$ \V=\V_0$ and $\W=\W_0$, as in \eqref{eq:V0W0}.
    The proof then follows immediately by combining Theorem \ref{t:ControlStratifBetaByOmega} with Proposition \ref{aiso} (isotropic Dorronsoro theorem) and Theorem \ref{diff}, using the area formula as in the first part of the proof of Theorem \ref{t:ControlStratifBetaByOmega}.
\end{proof}

\begin{remark}
    The argument used in the  proof of Theorem \ref{t:GlemStratBetanILG} also implies the case $k=n$ of Theorem \ref{t:ParamGLem}.
\end{remark}

The next result follows directly from  Theorem \ref{t:GlemStratBetanILG}, since $\beta^E_{2,\pi,\mathcal{V}_n}(x, r)^2\leq {\widehat\beta {^E_{2,\mathcal{V}_n}}}(x, r)^4$.

\begin{theorem}\label{t:GlemProjBetanILG}
 Let $\mathbb{V}$ be an $n$-dimensional horizontal subgroup of $\mathbb{H}^n$ with complementary vertical subgroup $\mathbb{W}$. If $E\subset \mathbb{H}^n$ is the intrinsic graph of an intrinsic Lipschitz function $\varphi:\mathbb{V}\to \mathbb{W}$, then
\begin{equation*}
    \int_{E\cap B(y,R)} \int_0^R {\beta^E_{2,\pi,\mathcal{V}_n}}(x, r)^2 \frac{dr}{r} d\mathcal{H}^n(x) \lesssim_{n, \mathrm{Lip}(\varphi_1),\ldots, \mathrm{Lip}(\varphi_n)} R^n,\qquad y\in E,\,R>0.
\end{equation*}
\end{theorem}

\subsection{An $L^{\infty}$-based Dorronsoro theorem for $k=1$}
\label{ss:Glemk=1}

For $1$-dimensional intrinsic Lipschitz graphs in $\mathbb{H}^n$, one can obtain a  stronger version of the geometric lemma with $L^{\infty}$- instead of $L^2$-based coefficients, see Theorem \ref{t:GlemStratBetanILGLinftyk=2} below. This is due to a similar improvement of the classical Dorronsoro theorem in the $1$-dimensional case, which we recall here.
\begin{theorem}[Dorronsoro's theorem, $L^{\infty}$-version for the $1$-dimensional case]\label{t:DorronsoroLinftyk=1}
    Let $f:\mathbb{R}\to \mathbb{R}$ be an $L$-Lipschitz function. Then the following holds 
    \begin{displaymath}
        \int_{B(y,R)}\int_0^R \Omega_{\infty,f}(x,r)^2\frac{dr}{r}dx\lesssim_{\mathrm{Lip}(f)} R,\quad y\in \mathbb{R},\,R>0,
    \end{displaymath}
    where $\Omega_{\infty,f}(x,r):=\inf_A \sup_{y\in B(x,r)}\frac{|f(y)-A(y)|}{r}$ and the infimum runs over all affine functions $A:\mathbb{R}\to \mathbb{R}$.
\end{theorem}
This version of Dorronsoro's theorem is proven in \cite[9, Chapter X, Lemma 2.4]{MR2150803}. We learned it from   \cite[Section 2.3]{MR4345824}. It is a crucial ingredient in the following proof.

\begin{theorem}\label{t:GlemStratBetanILGLinftyk=2}
Let $\mathbb{V}$ be a $1$-dimensional horizontal subgroup of $\mathbb{H}^n$ with complementary vertical subgroup $\mathbb{W}$. If $E\subset \mathbb{H}^n$ is the intrinsic graph of an intrinsic Lipschitz function $\varphi:\mathbb{V}\to \mathbb{W}$, then
\begin{equation*}
    \int_{E\cap B(y,R)} \int_0^R {\widehat\beta {^E_{\infty,\mathcal{V}_1}}}(x, r)^4 \frac{dr}{r} d\mathcal{H}^{1}(x) \lesssim_{n, \mathrm{Lip}(\varphi)} R,\qquad y\in E,\,R>0
\end{equation*}
\end{theorem}
\begin{proof}
 By Lemma \eqref{l:RotRed} and Lemma \ref{l:RotInvCoeff}, we may  assume that
$ \V=\mathbb{R}^1\times \{0\}$ and $\W=\{0\}\times \mathbb{R}^{2n}$. The claim then follows from Theorem \ref{t:DorronsoroLinftyk=1}
 combined with Theorem \ref{t:ControlStratifBetaByOmegaInftyk=1} below and the area formula (Corollary \ref{cov}).
\end{proof}

\begin{remark}
    The argument used in the  proof of Theorem \ref{t:GlemStratBetanILGLinftyk=2} also implies the case $k=1$ of Theorem \ref{t:ParamGLem}.
\end{remark}

To formulate Theorem \ref{t:ControlStratifBetaByOmegaInftyk=1},
 we will use $\Omega_{\infty,f}(x,r)$ also for vector-valued functions $f:\R\to\R^m$, simply taking the infimum over all affine maps $A:\R\to\R^m$ and considering, in the definition, the Euclidean norm on $\R^m$. It is clear that the conclusion of Theorem \ref{t:DorronsoroLinftyk=1} still holds for vector-valued functions (by approximating each component separately).
Since we are now dealing with the case $k=1$, we let \begin{displaymath}
    \mathbb{V}_0:=\{(x,0):\,x\in \mathbb{R}\}\subset\Hn\quad \text{and}\quad \mathbb{W}_0=\{(0,y,t):\, y\in \mathbb{R}^{2n-1},\,t\in\mathbb{R}\}\subset\Hn
\end{displaymath}
and we will identify $\mathbb{V}_0$ with $\mathbb{R}$. We use the projections
\begin{equation*}
 \Pi: \Hn \to \mathbb{R},\quad    \Pi(w_1, \dots, w_{2n}, t) = w_1, 
\end{equation*}
and
\begin{equation}\label{eq:PiPerp1D}
 \Pi^{\bot}: \Hn \to \mathbb{R}^{2n-1},\quad    \Pi^\bot(w_1, \dots, w_{2n}, t) = (w_2, \dots, w_{2n}).
\end{equation}
We adapt the proof of Theorem \ref{t:ControlStratifBetaByOmega} to show the following:
\begin{theorem}\label{t:ControlStratifBetaByOmegaInftyk=1}
Let $n\geq 1$, $\mathbb{V}_0$ be the horizontal subgroup $\mathbb{R}\times \{0\}$ with complementary vertical subgroup $\mathbb{W}_0$ in $\mathbb{H}^n$. 
Assume that $\varphi:\mathbb{V}_0 \to \mathbb{W}_0$ is an intrinsic Lipschitz function. If $\Gamma$ denotes the intrinsic graph of $\varphi$, then 
\begin{displaymath}
    \widehat{\beta}^{\Gamma}_{\infty,\mathcal{V}_1}(x,r)^{4}\lesssim_{\mathrm{Lip}(\Pi^\bot(\varphi))} \Omega_{\infty,\Pi^{\bot}(\varphi)}(\Pi(x),r)^2,\quad x\in \Gamma,\,r>0.
\end{displaymath}
\end{theorem}
\begin{proof}
Fix $x\in \Gamma, r>0$. Using 
\begin{equation*}
    \Pi(\Gamma \cap B(x,r)) \subset \Pi(B(x,r)) \subset B_{\mathbb{R}}(\Pi(x), r),
\end{equation*}
we obtain
that, for all $\V\in \mathcal{V}_1$,
\begin{equation}\label{eq:UpperStratifiedBetainfty}
 \widehat{\beta}^{\Gamma}_{\infty,\mathcal{V}_1}
(x, r)\leq
 \left[\left(\sup_{v\in B_{\mathbb{R}}(\Pi(x),r)} \frac{d_{\rm Eucl}(\pi(\Phi(v)), \pi(\V))}{r}  \right)^2+ \left(\sup_{v\in B_{\mathbb{R}}(\Pi(x),r)} \frac{d(\Phi(v), \V)}{r} \right)^4\right]^{1/4},
\end{equation}
where $\Phi$ denotes the graph map associated to $\varphi$.
Now let $A:\mathbb{R}\to \mathbb{R}^{2n-1}$ be an arbitrary affine map. 
Let $\psi_A:\mathbb{R}\to \mathbb{R}^{2n}$ be the weakly graph contact  lift of $\Pi^{\bot}(\varphi)-A$ such that $(\psi_A)_{2n}(\Pi(x))=0$ (recall Remark \ref{r:HorizLift}). 
We let $\pi_t:\Hn\to\R$ be the standard projection on the last component, namely $\pi_t(z,t)=t$. Define also $A_{2n}:\R\to\R$ to be such that $(A,A_{2n}):\R\to \R^{2n}$ is the weakly graph contact lift of $A$ such that
\begin{equation}\label{eq:VertCond}\pi_t\left[(\mathbb A(\Pi(x))^{-1}\Phi(\Pi(x)))^{-1}\cdot\Psi_A(\Pi(x))\right]=0,
\end{equation}where $\mathbb A:\R\to \Hn$ denotes the graph map of $(A,A_{2n})$ and $\Psi_A:\R\to\Hn$ is defined by $\Psi_A(v):=(0,\psi_A(v)).$

The idea behind the definition of $\Psi_A$ is the following: if the intrinsic Lipschitz condition for $k=1$ were preserved under group multiplication of mappings, then we could simply take $\Psi_A(v)= \mathbb{A}(v)^{-1}\Phi(v)$ and \eqref{eq:VertCond} would hold globally. Since this is not true, $\Psi_A$ is a more complicated mapping and we require \eqref{eq:VertCond} to hold only in the center of $B_{\mathbb{R}}(\Pi(x),r).$ However, we will obtain a good control on the derivative of this expression, see \eqref{eq:DerivControlComplTerm} below, which will allow us to control the error term $\widetilde{I}_t(A)$ in \eqref{eq:It_term}.

Since  $\mathbb A:\R\to \Hn$ is a horizontal curve which projects to a line in $\mathbb{R}^{2n}$, by uniqueness of the weakly graph contact lift (up to vertical translations), we obtain that $\mathbb{A}(\R)=\mathbb{A}(\V_0)$ is a horizontal line. Hence, taking $\V=\mathbb A(\V_0)\in\mathcal V_1$, 
\begin{align*}d(\Phi(v),\V)&\leq d(\Phi(v), \mathbb A(v))=\|\mathbb A(v)^{-1}\Phi(v)\|\\ &\leq \|(\mathbb A(v)^{-1}\Phi(v))^{-1}\Psi_A(v)\|+\|\Psi_A(v)\|\\ &\lesssim \sqrt{|\pi_t((\mathbb A(v)^{-1}\Phi(v))^{-1}\Psi_A(v))|}+|\Pi^\bot(\varphi)(v)-A(v)|+\sqrt{|(\psi_A)_{2n}(v)|},\end{align*}
where in the last estimate we used the explicit form of the Kor\'{a}nyi norm and the observation that $\mathbb A(v)^{-1}\Phi(v)$ and $\Psi_A(v)$ coincide in the first $2n$ coordinates. In addition, as in the proof of Theorem \ref{t:ControlStratifBetaByOmega} for the case $k=n$, we can bound
\[ d_{\rm Eucl}(\pi(\Phi(v)),\pi(\V))\leq |(v,\Pi^{\bot}(\varphi)(v))-(v,A(v))|=|\Pi^{\bot}(\varphi)(v)-A(v)|,\quad v\in \mathbb{R}.\]
Using these estimates, we deduce that 
\begin{equation}\label{eq:BetaAestinfty}
     \widehat{\beta}^{\Gamma}_{\infty,\mathcal{V}_1}
(x, r)\lesssim \inf_{A} \left[I_{\pi}(A)^2 + \left(I_{\pi}(A)+ \widetilde{I}_{t}(A)+\widehat{I}_{t}(A)\right)^4\right]^{1/4},
\end{equation}
where the infimum runs over all affine maps $A:\R\to\R^{2n-1}$ and
\begin{align}\label{eq:It_term}
I_{\pi}(A)&:=   \sup_{v\in B_{\mathbb{R}}(\Pi(x),r)} \frac{|\Pi^{\bot}(\varphi)(v)-A(v)|}{r},\notag \\ 
\widetilde{I}_{t}(A)   &:= \sup_{v \in B_{\mathbb{R}}(\Pi(x),r)} \frac{\sqrt{|\pi_t((\mathbb A(v)^{-1}\Phi(v))^{-1}\Psi_A(v))|}}{r} ,\\
\widehat{I}_{t}(A)  & := \sup_{v \in B_{\mathbb{R}}(\Pi(x),r)} \frac{\sqrt{|(\psi_A)_{2n}(v)|}}{r} .\notag
\end{align}
Let us now consider two cases. If $\mathrm{Lip}(\Pi^\perp(\varphi))=0$, then $\Pi^\bot(\varphi)$ is constant. Let us now fix $A\equiv\Pi^\bot(\varphi)$. It is easy to see, with this choice of $A$, that $I_\pi(A)=\widetilde I_t(A)=\widehat I_t(A)=0$. By \eqref{eq:BetaAestinfty}, it follows that $\widehat\beta_{\infty,\mathcal V_1}^\Gamma(x,r)=0$ and the conclusion of the theorem is trivially satisfied. Assume now that $\mathrm{Lip}(\Pi^\perp(\varphi))>0$ and fix an arbitrary $0<\varepsilon<(\sqrt 2-1)\mathrm{Lip}(\Pi^\perp(\varphi))$.
Let us now fix an affine map $A=A_\varepsilon:\R \to \R^{2n-1}$ such that
\begin{equation}\label{1estimate}I_\pi(A):=\sup_{v\in B_{\mathbb{R}}(\Pi(x),r)} \frac{|\Pi^{\bot}(\varphi)(v)-A(v)|}{r}<\Omega_{\infty,\Pi^{\bot}(\varphi)}(\Pi(x),r)+\varepsilon.\end{equation}
This also implies that 
\[\begin{split}\fint_{B_\R(\Pi(x),r)}\frac{|\Pi^\perp(\varphi)(v)-A(v)|^2}{r^2}\,dv&\leq \left(\sup_{v\in B_\R(\Pi(x),r)}\frac{|\Pi^\perp(\varphi)(v)-A(v)|}{r}\right)^2\\ &\leq \left(\sup_{v\in B_\R(\Pi(x),r)}\frac{|\Pi^\perp(\varphi)(v)-\Pi^\perp(\varphi)(\Pi(x))|}{r}+\varepsilon\right)^2\\&\leq 2\mathrm{Lip}(\Pi^\perp(\varphi))^2.\end{split}\]
Hence we can apply Lemma \ref{lemma_lipconst} and deduce that $\mathrm{Lip}(A)\leq N\,\mathrm{Lip}{\Pi^\perp(\varphi)}$ for some absolute constant $N$.
We then estimate $\widehat I_t(A)$. By the weak graph contact condition \eqref{eq:HorizCondCurve} and the previous control on $\mathrm{Lip}(A)$, 
\[\begin{split}\left|\frac{d}{dv}(\psi_A)_{2n}\right|&=\left|(\varphi_n-A_n)+\frac{1}{2}\sum_{i=1}^{n-1}(\varphi_{n+i}-A_{n+i})\frac{d}{dv}(\varphi_i-A_i)-(\varphi_{i}-A_{i})\frac{d}{dv}(\varphi_{n+i}-A_{n+i})\right|\\ &\lesssim_{\mathrm{Lip}(\Pi^\bot(\varphi))} |\Pi^\bot(\varphi)-A|\qquad\text{a.e. on }\R.\end{split} \]
Hence, if we define $u:=(\psi_A)_{2n}$, then $u$ is locally Euclidean Lipschitz, $u(\Pi(x))=0$ (by the choice of the weakly graph contact lift $\psi_A$) and $
    |\dot{u}|\lesssim_{\mathrm{Lip}(\Pi^{\bot}(\varphi)} |\Pi^{\bot}(\varphi)-A|$ almost everywhere.
By the fundamental theorem of calculus, for every $v\in B_{\mathbb{R}}(\Pi(x),r)$,
\begin{displaymath}
    |u(v)|=|u(v)-u(\Pi(x))|\lesssim_{\mathrm{Lip}(\Pi^{\bot}(\varphi))} \left( \sup_{y\in B_{\mathbb{R}}(\Pi(x),r)}|\Pi^{\bot}(\varphi)(y)-A(y)|\right)|v-\Pi(x)|,
\end{displaymath}
and therefore 
\begin{equation}\label{2estimate}\begin{aligned}
\widehat{I}_{t}(A) ^4  & = \sup_{v \in B_{\mathbb{R}}(\Pi(x),r)} \left( \frac{\sqrt{|(\psi_A)_{2n}(v)|}}{r}  \right)^4\\& \lesssim_{\mathrm{Lip}(\Pi^{\bot}(\varphi))}
\left(\sup_{y \in B_{\mathbb{R}}(\Pi(x),r)}\frac{|\Pi^{\bot}(\varphi)(y)-A(y)|}{r}  \right)^2< (\Omega_{\infty,\Pi^{\bot}(\varphi)}(\Pi(x),r)+\varepsilon)^2.\end{aligned}
\end{equation}
Similarly, we now deal with the expression $\widetilde I_t(A)$ from \eqref{eq:It_term}.
Since $\mathbb A(v)^{-1}\Phi(v)$ and $\Psi_A(v)$ coincide in the first $2n$ coordinates, it follows that 
\begin{equation}\label{eq:LinearVert}\pi_t((\mathbb A(v)^{-1}\Phi(v))^{-1}\Psi_A(v))=\pi_t(\Psi_A(v))-\pi_t(\mathbb A(v)^{-1}\Phi(v)),\quad v\in \mathbb{R}.\end{equation}
By the weak graph contact condition \eqref{eq:HorizCondCurve}, using standard identifications, we get that, almost everywhere on $\R$,
\begin{align*}\tfrac{d}{dv}\pi_t&\circ \Psi_A=-(\varphi_n-A_n)-\tfrac{1}{2}\sum_{i=1}^{n-1}(\varphi_{n+i}-A_{n+i})\tfrac{d}{dv}(\varphi_i-A_i)-(\varphi_{i}-A_{i})\tfrac{d}{dv}(\varphi_{n+i}-A_{n+i})\\ =&-\varphi_n+A_n+\tfrac{1}{2}\omega\left(\Pi^\bot(\varphi)-A, \tfrac{d}{dv}(\Pi^\bot(\varphi)-A)\right)\\ =& -\varphi_n+A_n+\tfrac{1}{2}\left[\omega\left(\Pi^\bot(\varphi), \Pi^\bot(\varphi)'\right)+\omega\left(A,    A'\right)-\omega\left(\Pi^\bot(\varphi),  A'\right)-\omega\left(A,\Pi^\bot(\varphi)'\right)\right],\end{align*}
where we slightly abused notation by identifying $A(v)\in\mathbb{R}^{2n-1}$ with $(0,A(v))\in \mathbb{R}^{2n}$, and $\Pi^{\bot}(\varphi(v))\in\mathbb{R}^{2n-1}$ with $(0,\Pi^{\bot}(\varphi(v))\in \mathbb{R}^{2n}$.

On the other hand, the derivative of the second term in \eqref{eq:LinearVert} can be expressed as follows
\begin{align*}
    \tfrac{d}{dv}(\pi_t(\mathbb A(\cdot)^{-1}\cdot\Phi))&=\tfrac{d}{dv}(\pi_t((A,A_{2n})(\cdot)^{-1}\cdot \varphi)) \\ & = \varphi_{2n}'- A_{2n}'-\tfrac{1}{2}\tfrac{d}{dv}\omega(A,\Pi^\bot(\varphi)) \\ &=\varphi_{2n}'- A_{2n}'-\tfrac{1}{2}\omega(A',\Pi^\bot(\varphi)-\tfrac{1}{2}\omega( A,\Pi^\bot(\varphi)')\quad\text{a.e. on }\R.
\end{align*}
Using again the weak graph contact condition \eqref{eq:HorizCondCurve} to compute $\varphi_{2n}'$ and $ A_{2n}'$, we get 
\begin{align*}
   & \frac{d}{dv}(\pi_t(\mathbb A(v)^{-1}\Phi(v)))\\&=-\varphi_n+A_n+\frac{1}{2}\left[\omega(\Pi^\bot(\varphi), \Pi^\bot(\varphi)')-\omega(A, A')-\omega( A',\Pi^\bot(\varphi))-\omega( A,\Pi^\bot(\varphi)')\right].
\end{align*}
Combining the previous estimates together with \eqref{eq:LinearVert}, we find that, almost everywhere on $\R$,
\begin{equation}\label{eq:DerivControlComplTerm}
    \frac{d}{dv}(\pi_t((\mathbb A(v)^{-1}\Phi(v))^{-1}\Psi_A(v)))=\omega(A, A')-\omega(\Pi^\bot (\varphi), A')=\omega(A-\Pi^\bot(\varphi), A').
\end{equation}
Now define $u_2(v):=\pi_t((\mathbb A(v)^{-1}\Phi(v))^{-1}\Psi_A(v))$. By choice of $A_{2n}$ as in \eqref{eq:VertCond} we know that $u_2(\Pi(x))=0$ and the previous relation, combined with the estimate on $\mathrm{Lip}(A)$, gives 
    \begin{displaymath}
    |u_2'|\lesssim_{\mathrm{Lip}(\Pi^\perp(\varphi))}|\Pi^{\bot}(\varphi)-A|.
    \end{displaymath}
By using the fundamental theorem of calculus ($u_2$ is locally Lipschitz), we conclude analogously as before by showing that
\begin{equation}\label{3estimate}\begin{aligned}
\widetilde{I}_{t}(A) ^4   &= \sup_{v \in B_{\mathbb{R}}(\Pi(x),r)} \left( \frac{\sqrt{|\pi_t((\mathbb A(v)^{-1}\Phi(v))^{-1}\Psi_A(v))|}}{r}  \right)^4 \\ &\lesssim_{\mathrm{Lip}(\Pi^\bot(\varphi))}
\left(\sup_{y \in B_{\mathbb{R}}(\Pi(x),r)}\frac{|\Pi^{\bot}(\varphi)(y)-A(y)|}{r}  \right)^2 <(\Omega_{\infty,\Pi^{\bot}(\varphi)}(\Pi(x),r)+\varepsilon)^2.\end{aligned}
\end{equation}
Finally, note that 
\begin{equation}\label{4estimate}
I_\pi(A)^4\lesssim_{\mathrm{Lip}(\Pi^\bot(\varphi))} I_\pi(A)^2,
\end{equation}
which follows from $I_{\pi}(A)<\Omega_{\infty,\Pi^{\bot}(\varphi)}(\Pi(x),r)+\varepsilon\leq \mathrm{Lip}(\Pi^\bot(\varphi))+\varepsilon\leq \sqrt 2\,\mathrm{Lip}(\Pi^\bot(\varphi))$.
Combining \eqref{eq:BetaAestinfty} with \eqref{1estimate}, \eqref{2estimate}, \eqref{3estimate} and \eqref{4estimate}, we conclude by arbitrariness of $\varepsilon$.
\end{proof}

\subsection{An integral geometric approach for intermediate dimensions}\label{s:IntegralGeo}

Throughout this section we assume that $n\geq 2$. We consider a $k$-dimensional horizontal subgroup $\mathbb{V}_0$ and, by Lemma \ref{l:RotRed}, we identify it with $\mathbb{R}^k$. We use the projection
\begin{equation}\label{projtoRk}
 \Pi: \Hn \to \mathbb{R}^k,\quad    \Pi(w_1, \dots, w_{2n}, t) = (w_1, \dots, w_k).
\end{equation}

Our goal is to prove the following result.

\begin{theorem}[Geometric Lemma for $k$-dimensional intrinsic Lipschitz graphs]\label{t:GlemBetanILGLk=2} Let $1<k<n$. There exists $q^*=q^*(k)\geq 4$ such that the following holds.
Let $\mathbb{V}_0$ be a $k$-dimensional horizontal subgroup of $\mathbb{H}^n$ with complementary vertical subgroup $\mathbb{W}_0$. If $E\subset \mathbb{H}^n$ is the intrinsic graph of an intrinsic Lipschitz function $\varphi:\mathbb{V}_0\to \mathbb{W}_0$, then
\begin{equation*}
    \int_{E\cap B(y,R)} \int_0^R {\beta {^E_{2,\mathcal{V}_k}}}(x, r)^{q^*} \frac{dr}{r} d\mathcal{H}^k(x) \lesssim_{n,k, \text{Lip}(\varphi)} R^k,\qquad y\in E,\,R>0.
\end{equation*}
\end{theorem}

The statement of the theorem is also true for $k=1$ and $k=n$, but in this case, Theorems \ref{t:GlemStratBetanILGLinftyk=2} and \ref{t:GlemStratBetanILG} yield more precise information.

We will prove Theorem \ref{t:GlemBetanILGLk=2} by induction, following the approach in Orponen's paper \cite{MR4345824}. Roughly speaking, if the $(k-1)$-dimensional slices of a $k$-dimensional intrinsic Lipschitz graph $E$ are well approximated by $(k-1)$-dimensional horizontal planes, then $E$ itself is well approximated by $k$-dimensional horizontal planes; the precise statement is more complicated, see Lemma \ref{l:IntGeomBdd}. The main new challenges when dealing with intrinsic Lipschitz graphs are discussed in Section \ref{ss:ConstrHoriz} (constructing horizontal instead of arbitrary planes) and Sections \ref{ss:1Dslice}-\ref{ss:1cDslice} (intrinsic Lipschitz slices of intrinsic Lipschitz graphs). Once these steps are dealt with, the proof of Theorem \ref{t:GlemBetanILGLk=2}  can be concluded similarly as in \cite{MR4345824}, see Section \ref{ss:ProofConclu}, keeping in mind that we always need to work with the full mapping $\varphi:\V_0\to\W_0$, rather than with real-valued component functions.

\subsubsection{Constructing horizontal planes}\label{ss:ConstrHoriz}
In brief,
the approach in Orponen's paper \cite{MR4345824} allows to find approximations by \emph{some} $k$-planes, but approximating by \emph{horizontal} $k$-planes requires more work. 
The following lemma is not needed for the proof of Theorem \ref{t:GlemBetanILGLk=2} but it  illustrates a core idea in our argument that will be made precise in Lemma \ref{l:almostHoriz_k}.

\begin{lemma}\label{l:HorizSimplex} Let $n\geq 2$, $n\in \mathbb{N}$ and $k\in \{2,\ldots,n\}$.
    Let $\mathbb{V} \subset \mathbb{H}^n$ be a topologically $k$-dimensional plane (an affine $k$-plane in $\mathbb{R}^{2n+1}$). Let $\Delta \subset \mathbb{V}$ be a non-degenerate $k$-simplex whose $(k-1)$-dimensional faces are contained in $(k-1)$-dimensional affine horizontal planes. Then $\mathbb{V}$ is an affine horizontal $k$-plane.
\end{lemma}

Lemma \ref{l:HorizSimplex} is only relevant for $n>1$. Indeed, since $\mathbb{H}^1$ contains no horizontal subgroups of dimension $2$, it also cannot contain horizontal $k$-simplices as in the statement of the lemma.

\begin{proof}
    Let $p_0, p_1, \dots, p_k$ be the $k+1$ vertices of the simplex $\Delta$. For $i=0,\cdots,k$, we denote by 
    $F_i$ the $(k-1)$-dimensional face of $\Delta$ not containing $p_i$. 
 By assumption, each face $F_i$ is contained in a $(k-1)$-dimensional affine horizontal plane $\mathbb{V}_i$.
 
 Left translations are bijections which map planes to planes of the same topological dimension and they preserve the family of horizontal $m$-planes for $1\leq m\leq n$.
Without loss of generality, by applying a left translation, we may therefore assume that $p_0 = 0$.  Now observe that, for each $1 \le i \le k$,  
the edge connecting $p_0$ and $p_i$ is contained in one of the (horizontal) $(k-1)$-dimensional faces of $\Delta$. Thus, the edges connecting $p_0$ and $p_i$ ($1 \le i \le k$) are horizontal line segments starting at the origin (recall Remark \ref{r:Line}), and hence they must be contained in $\mathbb{R}^{2n} \times \{0\}$. Thus, we can write $p_i = (z_i, 0)$ for some $z_i \in \mathbb{R}^{2n}$. Now, for any $i \neq j$ and $i, j \ge 1$, the line connecting $p_i$ and $p_j$ is also horizontal by assumption. Hence, by a left translation, the point $p_i^{-1}p_j$ lies on a horizontal line passing through the origin. This implies that the last coordinate of $p_i^{-1}p_j$ is equal to 0, which means that $\omega(z_i,z_j)=0.$ This shows that $\text{span}\{z_1, \dots, z_k\}\subset\R^{2n}$ is a $k-$dimensional isotropic subspace of $\mathbb{R}^{2n}$ and $\mathbb{V}=\mathrm{span}\{p_1,\ldots,p_k\}$ therefore horizontal.
\end{proof}

The main result of this section, Lemma \ref{l:almostHoriz_k}, can be seen as an ``$\varepsilon$-version'' of Lemma \ref{l:HorizSimplex}. To prove it, we will also use the auxiliary result below. The idea is the following: if two points $q_1$ and $q_2$ lie in a  metric $\varepsilon$-tube around a horizontal line, then they either span an `almost horizontal' line themselves, or $d(q_1,q_2)$ must be small.

\begin{lemma}[Points close to a horizontal line]\label{l:VertDiff} Let $C>0$ and $0<\varepsilon<1$.
Assume that $q_1=(z_1,t_1),q_2=(z_2,t_2)\in \mathbb{H}^n$ are two points with $|z_1-z_2|\leq C$ and such that 
\begin{displaymath}
    d(q_1,p_1)\leq \varepsilon \quad \text{and}\quad
    d(q_2,p_2)\leq \varepsilon,
\end{displaymath}
where $p_1,p_2\in\mathbb H^n$ are two points lying on a common horizontal line $\ell$.
Then
\begin{displaymath}
    |t_1-t_2-\tfrac{1}{2}\omega(z_2,z_1)|\lesssim_C \varepsilon.
\end{displaymath}
\end{lemma}
\begin{proof} We denote in coordinates  $p_1=(z_1', t_1')\,,p_2=(z_2',t_2')$.
Since $\ell$ is a horizontal line, $p_1=p\cdot v_1$, $p_2=p\cdot v_2$ for some $p\in\mathbb H^n$, $v_1,v_2\in p^{-1}\cdot\ell\in\mathcal V^0_1$. Hence $p_2^{-1}\cdot p_1=v_2^{-1}\cdot v_1$ and, taking the last coordinate, using the structure of the Heisenberg product and recalling that $v_1,v_2\in\mathcal V_1^0$, we get 
\begin{equation}\label{eq:(i)Horiz}
    t_1'-t_2'-\tfrac{1}{2}\omega(z_2',z_1')=0.
\end{equation}
On the other hand, spelling out  $ d(q_1,p_1)$ and $ d(q_2,p_2)$, we also have
\begin{equation}\label{eq:(v)Dist}
 |z_1-z_1'|,|z_2-z_2'|\leq \varepsilon,\quad |t_1-t_1'-\tfrac{1}{2}\omega(z_1',z_1)|,|t_2-t_2'-\tfrac{1}{2}\omega(z_2',z_2)|\leq \varepsilon^2.
\end{equation}
Now, let us write
\begin{align*}
    t_1-t_2-\tfrac{1}{2}\omega(z_2,z_1)\overset{\eqref{eq:(i)Horiz}}{=}& \left[t_1-t_2-\tfrac{1}{2}\omega(z_2,z_1)\right]-\left[ t_1'-t_2'-\tfrac{1}{2}\omega(z_2',z_1')\right]\\
    =& \left[t_1-t_1'-\tfrac{1}{2}\omega(z_1',z_1) \right]
    -  \left[t_2-t_2'-\tfrac{1}{2}\omega(z_2',z_2) \right]
  \\&  + \tfrac{1}{2}\omega(z_1',z_1)-\tfrac{1}{2}\omega(z_2',z_2)-\tfrac{1}{2}\omega(z_2,z_1)+\tfrac{1}{2}\omega(z_2',z_1')\\
   =& \left[t_1-t_1'-\tfrac{1}{2}\omega(z_1',z_1) \right]
    -  \left[t_2-t_2'-\tfrac{1}{2}\omega(z_2',z_2) \right]
  \\&  +\tfrac{1}{2} \omega(z_1'-z_2,z_1-z_2')\\
   =& \left[t_1-t_1'-\tfrac{1}{2}\omega(z_1',z_1) \right]
    -  \left[t_2-t_2'-\tfrac{1}{2}\omega(z_2',z_2) \right]
  \\&  +\tfrac{1}{2}\omega(z_1'-z_1,z_1-z_2)+\tfrac{1}{2}\omega(z_1'-z_1,z_2-z_2')\\&+\tfrac{1}{2}\omega(z_1-z_2,z_1-z_2)+\tfrac{1}{2}\omega(z_1-z_2,z_2-z_2').
\end{align*}
It finally follows from this expression, the assumption $|z_1-z_2|\leq C$ and \eqref{eq:(v)Dist} that 
\begin{displaymath}
    |t_1-t_2-\tfrac{1}{2}\omega(z_2,z_1)|\lesssim_C \varepsilon,
\end{displaymath}
as desired.
\end{proof}


We now give a sufficient condition under which a $k$-plane in $\mathbb{H}^n$ is close to a horizontal plane, in a suitable sense. By left-translations, this can be reduced to the problem of approximating an `almost isotropic' plane in $\mathbb{R}^{2n}$ by an isotropic plane. The construction of the approximating isotropic plane bears similarities with the Gram-Schmidt orthogonalization, although to a given set of $k$ vectors, we now assign a set of vectors that span a possibly different, but isotropic, plane. 

\begin{lemma}[Finding horizontal $k-$planes]\label{l:almostHoriz_k}
Let $2\leq k\leq n$. Let $0<\varepsilon\ll C$ with $\varepsilon$ sufficiently small depending on $C$ and $k$. 
    Assume that there are $k+1$ points $q_0, q_1, \dots, q_k \in \mathbb{H}^n$ satisfying the following conditions:
    \begin{enumerate}
        \item[(1)] There exists a $k$-dimensional horizontal subgroup $\mathbb{V}_0 = V_0 \times \{0\}$ so that the projected points $\{\pi_{\mathbb{V}_0}(q_0), \pi_{\mathbb{V}_0}(q_1), \dots, \pi_{\mathbb{V}_0}(q_k)\}$ form the vertices of a non-degenerate $k$-simplex in $\mathbb{V}_0$ which contains a $k$-dimensional cube of side-length $\frac{1}{C}$ and is contained in a cube of side-length $C$.
        \item[(2)] $d(q_i, q_j) \in [\frac{1}{C}, C]$ for all $i \neq j$.
        \item[(3)] For each $m \in \{0, 1, \dots, k\}$, there exists an affine $(k-1)$-dimensional horizontal plane $\mathbb{U}_m\in \mathcal{V}_{k-1}$ such that for all $j \in \{0, 1, \dots, k\} \setminus \{m\}$, we have:$$ d(q_j, \mathbb{U}_m) \le \varepsilon $$
        \end{enumerate}
    Then there exists a horizontal $k$-plane $\mathbb{V} \in \mathcal{V}_k$ that is an intrinsic graph of a function $A_{\mathbb{V}}:\mathbb{V}_0\to\mathbb{W}_0$ such that$$ d(q_i, \mathbb{A}_{\mathbb{V}}(\pi_{\mathbb{V}_0}(q_i))) \lesssim_{C,k} \sqrt{\varepsilon}, \quad \text{for all } i=0,1,\dots,k$$
    where $\mathbb{A}_{\mathbb{V}}$ denotes the graph map associated to $A_{\mathbb{V}}$. 
\end{lemma}

Before proving Lemma \ref{l:almostHoriz_k}, we remark that assumption $(1)$ implies that the $k$-simplex with vertices $\{\pi_{\mathbb{V}_0}(q_0), \pi_{\mathbb{V}_0}(q_1), \dots, \pi_{\mathbb{V}_0}(q_k)\}$ has edges of length at least $\frac{1}{C}$. This is a purely geometric observation, which we formulate below in a stronger form.
\begin{lemma}\label{lem:geometric}
    Let $\triangle$ be a non-degenerate $k$-simplex inside $\R^m$, with $k\leq m$. Assume that $\triangle$ contains a $k$-dimensional ball of diameter $D$. Let $h\in\{1,\dots,k\}$, and let $\triangle_h$ be a $h$-face (an element of the $h$-skeleton of $\triangle$). Then $\triangle_h$ contains a $h$-dimensional ball of diameter $D$. In particular, the edges of $\triangle$ (namely, its 1-faces) have length at least $D$.
\end{lemma}
\begin{proof}
    We prove the claim by backward induction on $h$. The case $h=k$ holds by assumption. Let us assume now that the conclusion holds for $h+1$ and let us prove it for $h$. Let $\triangle_h$ be a $h$-face of $\triangle$ and consider a $(h+1)$-face $\triangle_{h+1}$ of $\triangle$ containing $\triangle_h$. By inductive hypothesis, $\triangle_{h+1}$ contains a $(h+1)$-dimensional ball $B_{h+1}$ of diameter $D$. Consider the ``equatorial" section $S_h$ of $B_{h+1}$ which is parallel to $\triangle_h$. Hence, $S_h$ is a $h$-dimensional ball of diameter $D$. Let us also denote by $p_h$ the vertex of $\triangle_{h+1}$ which is not contained in $\triangle_h$. Consider the radial projection \textit{from} $p_h$ \textit{onto} $\triangle_{h}$, which is defined on $\triangle_{h+1}$: by similarity, this maps $S_h$ to a $h$-dimensional ball $B_h$ which is contained in $\triangle_h$. In addition, again by similarity, the diameter of $B_h$ is at least the one of $S_h$, hence at least $D$. 
\end{proof}

\begin{proof}[Proof of Lemma \ref{l:almostHoriz_k}]
 \textbf{Step 1: Reduction to $q_0 = 0$.} 
By applying a left translation, we may without loss of generality assume that $q_0=0$. Indeed, left translations are bijections that preserve the families $\mathcal{V}_m$ of horizontal $m$-planes for $m=1,\ldots,k$. They are also isometries with respect to the Kor\'{a}nyi distance. Moreover, horizontal projections are group homomorphisms. Finally, if $\Gamma=\{v\cdot f(v):\,v\in \mathbb{V}_0\}$ is an intrinsic graph and $q\in \mathbb{H}^n$, then $q\cdot \Gamma$ is the intrinsic graph of $f_q:\mathbb{V}_0 \to \mathbb{W}_0$ given by 
\begin{displaymath}
    f_q(v):= v^{-1}\cdot q \cdot \pi_{\mathbb{V}_0}(q)^{-1}\cdot v\cdot f\left(\pi_{\mathbb{V}_0}(q)^{-1}\cdot v\right),\quad v\in \mathbb{V}_0,
\end{displaymath}
recall the comment below Definition \ref{d:graph}.
Combining all of this, we conclude that we may indeed assume that $q_0=0$

We will show that there exists a horizontal subgroup $\mathbb{V}=V\times \{0\}$ with the desired properties. Let $q_i = (z_i, t_i)$. The isotropic plane $V$ will be constructed as $V=\mathrm{span}\{v_1,\ldots,v_k\}$, where, for each $i=1,\ldots,k$, the vector $v_i$ is a small perturbation of $z_i$. 

\medskip

\textbf{Step 2: $\mathrm{span}\{z_1,\ldots,z_k\}$ is almost isotropic.}
Our goal is to show that 
\begin{equation}\label{eq:OmegaEst}
\left|\omega\left(z_i,z_j\right)\right|\lesssim_C \varepsilon,\quad i,j\in \{1,\ldots,k\}.
\end{equation} 
We observe that
\begin{equation}\label{eq:|z1|,|z2|}
    |z_l|\sim_C 1,\quad l=1,\ldots,k.
\end{equation}
For now, we will only need the upper bound, but the lower bound will be used later.
To see why \eqref{eq:|z1|,|z2|} holds true, recall that, by (2), 
\begin{displaymath}
|z_l|\leq \|q_l\|=d(q_l,q_0)\leq C,\quad l=1,\ldots,k.
\end{displaymath}
On the other hand, denoting by  $\pi_{V_0}: \R^{2n}\to V_0$ the Euclidean orthogonal projection onto $V_0$,
we find by (1) and Lemma \ref{lem:geometric} that
\begin{equation}\label{eq:z1z2est}
    |z_l|\geq |\pi_{V_0}(z_l)|\gtrsim_C 1,\quad l=1,\ldots,k,
\end{equation}
yielding together \eqref{eq:|z1|,|z2|}.
Fix such $i\neq j$. 
Choose $m\in \{0,\ldots,k\}\setminus \{i,j\}$ and let $\mathbb{U}_m$ be the $(k-1)$-dimensional horizontal plane given by condition (3).  Then there exist $p_i^m, p_j^m \in \mathbb{U}_m$ such that 
\begin{equation}\label{eq:CloseEndpoints}
d( q_i,p_i^m) \leq  \varepsilon\quad\text{and}\quad d( q_j,p_j^m) \leq  \varepsilon. 
\end{equation}
As a subset of $\mathbb{U}_m$, the line segment $\ell_{i,j}$ connecting  $p_i^m$ and $p_j^m$ is horizontal (recall Remark \ref{r:Line}). We will apply Lemma \ref{l:VertDiff} with $\ell=\ell_{i,j}$. For this we need the condition \eqref{eq:CloseEndpoints} as well as assumption (2) to ensure that
\begin{displaymath}
    |z_i-z_j|\leq d(q_1,q_2)\leq C.
\end{displaymath}
It then follows by Lemma \ref{l:VertDiff}  that
\begin{equation}\label{eq:VertDiffEst}
    |t_i-t_j-\tfrac{1}{2}\omega(z_j,z_i)|\lesssim_C \varepsilon.
\end{equation}
This will yield the desired estimate for $\omega(z_j,z_i)$ if we can control $t_i,t_j$.  To achieve the latter, we apply again Lemma \ref{l:VertDiff}, but with $\ell=\ell_{0,i}$ and $\ell=\ell_{0,j}$, respectively. This yields
\begin{equation}\label{eq:t1_t2_bdd}
    |t_l|\lesssim_C \varepsilon,\quad l=0,\ldots,k.
\end{equation}
Therefore, by \eqref{eq:VertDiffEst},
we obtain 
    $$ |\omega(z_i, z_j)|\lesssim_C \varepsilon,\quad i,j=1,\ldots,k.$$

\medskip

     \textbf{Step 3: Constructing $\mathbb{V}\in \mathcal{V}_k^0$ to which $\{q_0,q_1,\ldots,q_k\}$ lie close.} 
The plane spanned by $z_1,\ldots,z_k$ need not be isotropic, but with a small perturbation of the vectors $z_l$ it will be. We define
\begin{equation}\label{eq:v_j}
    v_1:=z_1\quad \text{and}\quad v_m := z_m + \sum_{l=1}^{m-1} c_{m,l}\, J z_l\text{ for }1 < m \le k,
\end{equation}
    where $J$ is the standard complex structure as defined below \eqref{sym}. We want to find (small) coefficients $c_{m,l}\in \mathbb{R}$ such that $\omega(v_r, v_m) = 0$ for all $1 \le r < m \le k$. This will ensure that $V := \text{span}\{v_1, \dots, v_k\}$ is isotropic. 
    Note that $\omega(v_r, v_m) = 0$ is equivalent to
    \begin{equation}\label{eq:EquivIsotr} \omega(v_r, z_m) + \sum_{l=1}^{m-1} c_{m,l}\, \omega(v_r, J z_l) = 0.\end{equation}
    Using the property $\omega(v_r, J z_l) = \langle v_r, z_l \rangle$,  condition \eqref{eq:EquivIsotr} is further equivalent to 
    $$ \sum_{l=1}^{m-1} c_{m,l} \langle v_r, z_l \rangle = -\omega(v_r, z_m). $$

 In conclusion, we have $\omega(v_r,v_m)=0$ for all $1\leq r<m\leq k$ if and only if for all $1<m\leq k$, we can find
 $C_m = (c_{m,1}, \dots, c_{m,m-1})^T$ such that 
 \begin{equation}\label{eq:CoeffCond}
\widetilde{G}_{m-1} C_m = -W_m,
 \end{equation}
 where 
 \begin{displaymath}
    \widetilde{ G}_{m-1}:=\begin{pmatrix}\langle v_1,z_1\rangle&& \langle v_1,z_{m-1}\rangle \\&\ddots&\\\langle v_{m-1},z_1\rangle&&\langle v_{m-1},z_{m-1}\rangle\end{pmatrix}
 \end{displaymath}
and $W_m = (\omega(v_1, z_m), \dots, \omega(v_{m-1}, z_m))^T$.

We will now verify by induction on $m$ that the equation \eqref{eq:CoeffCond} can be solved for $C_m$,   and
\begin{equation}\label{eq:IndHyp}\tag{$H_m$}
\|\widetilde{G}_{m-1}\|\lesssim_C 1,\quad |\det \widetilde{G}_{m-1}|\gtrsim_C 1,\quad\text{and}\quad |c_{m,l}|\lesssim_C \varepsilon\quad \text{for } l=1,\ldots,m-1,
\end{equation}
where $\|\widetilde{G}_{m-1}\|$ denotes the operator norm of $\widetilde{G}_{m-1}$. Here the implicit constants are allowed to depend on $m$, which is bounded by $k$, without special mentioning.

To start the induction, consider $m=2$. Then, by \eqref{eq:|z1|,|z2|}, 
\begin{displaymath}
    \widetilde{G}_1 = \langle v_1,z_1\rangle = \langle z_1,z_1 \rangle \sim_C 1,
\end{displaymath}
and
\eqref{eq:CoeffCond} is equivalent to 
\begin{displaymath}
   \langle v_1,z_1 \rangle c_{2,1}=-\omega (v_1,z_2).
\end{displaymath}
The left-hand side equals  $c_{2,1}|z_1|^2=\langle z_1,z_1\rangle c_{2,1}$, while the right-hand side is $-\omega (v_1,z_2)=-\omega(z_1,z_2)$. Thus,  $C_2:=c_{2,1}:=-\omega(z_1,z_2)/|z_1|^2$ solves \eqref{eq:CoeffCond} for $m=2$ with
\begin{displaymath}
    |c_{2,1}|=\frac{|\omega(z_1,z_2)|}{|z_1|^2}\lesssim_C \varepsilon,
\end{displaymath}
where we used \eqref{eq:OmegaEst},\eqref{eq:|z1|,|z2|} in the last estimate. Hence the base case $m=2$ of the induction hypothesis \eqref{eq:IndHyp} holds.

Assume now that $(H_{r})$ holds for all $2\leq r\leq m-1$, and show that \eqref{eq:IndHyp} follows. By definition \eqref{eq:v_j},
the entries of $W_m$ are given by $\omega(v_1,z_m)=\omega(z_1,z_m)$ and
\begin{displaymath}
    \omega(v_r,z_m)= \omega(z_r,z_m)+\sum_{l=1}^{r-1}c_{r,l}\omega(Jz_l,z_m),\quad r=2,\ldots,m-1.
\end{displaymath}
Since $|\omega(z_r,z_m)|\lesssim_C \varepsilon$
by \eqref{eq:OmegaEst}, and moreover $|z_j|\sim_C 1$, $j=1,\ldots,k$ and $|c_{r,l}|\lesssim_C \varepsilon$ by induction hypothesis $(H_r)$, we conclude that $|W_m|\lesssim_C \varepsilon$. 
By a similar computation, we verify that  $\|\widetilde{G}_{m-1}\|\lesssim_C 1$ since $\langle v_1,z_j\rangle = \langle z_1,z_j\rangle$ and, for every $ r=2,\ldots,m-1,\, j=1,\ldots,m-1$,
\begin{align*}
 \langle v_r,z_j \rangle &= \langle z_r,z_j\rangle+\sum_{l=1}^{r-1}c_{r,l}\langle Jz_l,z_j\rangle \\
 & =\langle z_r,z_j \rangle + \langle C_r,(\langle Jz_1,z_j\rangle,\ldots,\langle Jz_{r-1} ,z_j\rangle)^T\rangle.
\end{align*}
This computation also shows that
\begin{displaymath}
    \widetilde{G}_{m-1}=G_{m-1}+ R_{m-1},
\end{displaymath}
where $G_{m-1}=(\langle z_i,z_j\rangle )_{i,j=1,\ldots,m-1}$ is the Gram matrix of $\{z_1,\ldots,z_{m-1}\}$, and $R_{m-1}$ is a matrix whose coefficients are  in absolute value bounded by $\lesssim_C \varepsilon$ thanks to the induction hypothesis and the fact that $|z_j|\sim_C 1$, $j=1,\ldots,k$. Then, for all $|v|=1$,
\begin{equation}\label{eq:l(tilde G)}
|\widetilde{G}_{m-1}v|= |G_{m-1}v+R_{m-1}v|\geq |G_{m-1}v|-|R_{m-1}v|\gtrsim_C 1
\end{equation}
provided that $\varepsilon$ is chosen small enough depending on $C$ and $k$. Here we have used in the last step that $\|G_{m-1}\|\lesssim_C 1$ and 
\begin{equation}\label{eq:DetGram}
|\det G_{m-1}|\gtrsim_C 1.
\end{equation}

Let us justify more carefully why 
the Gram determinant $\det G_{m-1}$ is bounded away from zero. This would be  immediate if the points $z_j$ were replaced by $\pi_{V_0}(z_j)$. Indeed, since $\pi_{\mathbb{V}_0}(q_l)=(\pi_{V_0}(z_l),0)$ by definition of the horizontal Heisenberg projection, it follows from assumption (1) and Lemma \ref{lem:geometric} that the points $\{0,\pi_{V_0}(z_1),\ldots,\pi_{V_0}(z_{m-1})\}$ are the vertices of a non-degenerate $(m-1)$-simplex $\Delta_{m-1}$ in $V_0$ containing a sphere of diameter $1/C$. Therefore, $\pi_{V_0}(z_1),\ldots,\pi_{V_0}(z_{m-1})$ are linearly independent and $\mathcal{H}^{m-1}(\Delta_{m-1})\gtrsim_C 1$. (We record for later that this step also works for $m=k+1$.)

Finally, let $\widetilde{\Delta}_{m-1}$ be the $(m-1)$-simplex in $\mathbb{R}^{2n}$ with vertices $0,z_1,\ldots,z_{m-1}$. By construction $\pi_{V_0}(\widetilde{\Delta}_{m-1})=\Delta_{m-1}$, and therefore $$\mathcal{H}^{m-1}(\widetilde{\Delta}_{m-1})\geq  \mathcal{H}^{m-1}(\pi_{V_0}(\widetilde{\Delta}_{m-1}))\gtrsim_C 1$$ (again, this also holds for $m=k+1$). 
By Gram's formula, the volume of the $(m-1)$-simplex  $\widetilde{\Delta}_{m-1}$ equals $\sqrt{|\det G_{m-1}|}/{(m-1)!}$, yielding \eqref{eq:DetGram}, and thus \eqref{eq:l(tilde G)}.

From \eqref{eq:l(tilde G)} it finally follows that $|\det \widetilde{G}_{m-1}|\geq_C 1$. Recalling \eqref{eq:CoeffCond}, the bound on $|C_m|$ follows from the corresponding bound for $|W_m|$ and the estimates $\|\widetilde{G}_{m-1}\|\lesssim_C 1$, $|\det \widetilde{G}_{m-1}|\geq_C 1$, completing the proof of the induction step \eqref{eq:IndHyp}. Continuing the induction until we reach $m=k$ finally shows that $V=\mathrm{span}\{v_1,\ldots,v_k\}$ is isotropic, as claimed. Moreover, by the definition \eqref{eq:v_j} of the vectors $v_j$, and the control on the coefficients $c_{m,l}$ given by \eqref{eq:IndHyp} and since $|z_j|\sim_C 1$, we obtain
 \begin{equation}\label{eq:smallPerturb}
     |z_m - v_m| =\left| \sum_{l=1}^{m-1} c_{m,l}\, J z_l\right|\lesssim_C \varepsilon,\quad m=1,\ldots,k\end{equation}

Since  $|\det G_k|\gtrsim_C 1$ and  $|z_j|\sim_C 1$, $j=1,\ldots,k$, the estimate \eqref{eq:smallPerturb} implies by an analogous argument as above that $\{v_1,\ldots,v_k\}$ are linearly independent, provided that $\varepsilon$ is chosen small enough depending on $C$ and $k$.

    Now let $\mathbb{V} := V \times \{0\}$. We conclude that $V$ is isotropic and $\mathbb{V}$ is a horizontal $k$-plane.

\medskip

\textbf{Step 4: $\mathbb{V}$ is an intrinsic graph.}
It suffices to show that the rank of the projection $\pi_{\mathbb{V}_0}|_{\mathbb{V}}$ is $k$. Let us see why this is the case. First $\mathbb{V}=V\times \{0\}$, where $V$ is spanned by the linearly independent vectors $v_1,\ldots,v_k$. The vectors   $\{\pi_{V_0}(z_1), \dots, \pi_{V_0}(z_k)\}$ span $V_0$, and the argument in Step 3 shows that they are quantitatively linearly independent in the sense that their Gram determinant is bounded away from $0$ in terms of $C$.
 Since  $|\pi_{V_0}(z_i)-\pi_{V_0}(v_i)|\leq |z_i - v_i| \lesssim_C \varepsilon$ by \eqref{eq:smallPerturb}, it follows that for $\varepsilon$ small enough, depending on $C$ and  $k$, the vectors $\{\pi_{V_0}(v_1), \dots, \pi_{V_0}(v_k)\}$ are also linearly independent. 
 Thus, we can deduce that  $\pi_{\mathbb{V}_0}|_{\mathbb{V}}$ has rank $k$. Therefore, $\mathbb{V}$ is an intrinsic graph over $\mathbb{V}_0$, and the function $\mathbb{A}_{\mathbb{V}}$ is the inverse of $(\pi_{\mathbb{V}_0})|_{\mathbb{V}}$.

 \medskip

\textbf{Step 5: Estimating the distance of $q_i$, $i=1,\ldots,k$, from $\mathbb{V}$.}
We need to show $d(q_i, \mathbb{A}_{\mathbb{V}}(\pi_{\mathbb{V}_0}(q_i))) \lesssim_C \sqrt{\varepsilon}$. 
    As $\{\pi_{V_0}(v_i)\}_{i=1}^k$ are linearly independent and therefore span $V_0$, we can write
    \begin{equation}\label{eq:ProjForm}
    \pi_{V_0}(z_i) = \sum_{j=1}^k \lambda_{i,j} \pi_{V_0}(v_j),\quad i=1,\ldots,k,
    \end{equation}
    for suitable coefficients $\lambda_{i,j}$.
Therefore, for $i=1,\ldots,k$,
\begin{equation}\label{eq:AMapExp}
   \mathbb{A}_{\mathbb{V}}(\pi_{\mathbb{V}_0}(q_i))=
     \mathbb{A}_{\mathbb{V}}(\pi_{V_0}(z_i),0)=\mathbb{A}_{\mathbb{V}}\bigg( \sum_{j=1}^k \lambda_{i,j} \pi_{V_0}(v_j),0\bigg)=\bigg(\sum_{j=1}^k \lambda_{i,j} v_j, 0\bigg)=:(z_i^{\ast},0),
\end{equation}
and $\mathbb{A}_{\mathbb{V}}(\pi_{\V_0}(q_0))=0$. Clearly, $d(q_0,\mathbb{A}_{\mathbb{V}}(\pi_{\mathbb{V}_0}(q_0))=0$. For $i=1,\ldots,k$, we have $q_i=(z_i,t_i)$ and thus 
\begin{equation}\label{eq:DistPointMap}
      d\left(q_i, \mathbb{A}_{\mathbb{V}}(\pi_{\mathbb{V}_0}(q_i))\right) \sim |z_i - z_i^*| + \left| t_i - \tfrac{1}{2}\omega(z_i^*, z_i) \right|^{1/2}.
\end{equation}
To bound the expression $|z_i-z_i^{\ast}|$, we need to control the coefficients $\lambda_{i,j}$ in \eqref{eq:ProjForm}. We first notice that the vectors  $\{\pi_{V_0}(v_1), \dots, \pi_{V_0}(v_k)\}$ are quantitatively linearly independent, in terms of $C$. This follows from the fact that the same holds for $\{\pi_{V_0}(z_1), \dots, \pi_{V_0}(z_k)\}$ and from the estimate $|\pi_{V_0}(z_i)-\pi_{V_0}(v_i)|\leq |z_i - v_i| \lesssim_C \varepsilon$. Hence, we get the same conclusion for $\{\pi_{V_0}(v_1), \dots, \pi_{V_0}(v_k)\}$ by taking $\varepsilon$ small enough depending only on $k$ and $C$. Combining this with \eqref{eq:ProjForm} and using again the estimate $|\pi_{V_0}(z_i)-\pi_{V_0}(v_i)|\leq |z_i - v_i| \lesssim_C \varepsilon$, it follows that $\lambda_{i,j}=\delta_{ij}+O(\varepsilon)$. It then follows from \eqref{eq:AMapExp} and \eqref{eq:smallPerturb} that $|z_i-z_i^{\ast}|\lesssim_C \varepsilon.$
Also note that by \eqref{eq:t1_t2_bdd}, we have $|t_i|\lesssim_C \varepsilon$. Thus, the Korányi distance \eqref{eq:DistPointMap} is estimated by
    \begin{align*}
        d\left(q_i, \mathbb{A}_{\mathbb{V}}(\pi_{\mathbb{V}_0}(q_i))\right) \sim_C |z_i - z_i^*| + \left| t_i - \frac{1}{2}\omega(z_i^*-z_i, z_i) \right|^{1/2}
        \lesssim_C \varepsilon +\sqrt{\varepsilon}\lesssim_C \sqrt{\varepsilon},
    \end{align*}
    for $\varepsilon \ll 1$ small enough. This completes the proof. 
\end{proof}


\subsubsection{One-dimensional slices of intrinsic Lipschitz graphs}\label{ss:1Dslice}

Let $1<k\leq n$, $\mathbb{V}_0=\mathbb{R}^k \times \{0\}$ and $\mathbb{W}_0=\{0\}\times \mathbb{R}^{2n+1-k}$. Fix an intrinsic $L$-Lipschitz function 
\begin{displaymath}
\varphi:\mathbb{V}_0 \equiv \mathbb{R}^k \to \mathbb{W}_0\equiv \R^{2n+1-k},\quad \varphi(x)=(\varphi_1(x),\ldots,\varphi_{2n+1-k}(x)).
\end{displaymath}
We aim to use that $1$-dimensional ``slices'' of $\varphi$ are again intrinsic Lipschitz, with controlled Lipschitz constant. To this end, consider an arbitrary affine line in $\mathbb{R}^k$, parametrized as follows:
\begin{equation}\label{def_lae}
\ell_{a,e}(v)= v + \sum_{i=1}^{k-1}a_i e_i^{\bot}\quad v\in V_e:=\mathrm{span}\{e\}\subset\R^k,
\end{equation}
for fixed $e\in S^{k-1}$, $\{e_1^{\bot},\ldots,e_{k-1}^{\bot}\}$ an orthonormal basis of $e^{\bot}$ inside $\R^k$, and $a:=(a_1,\ldots,a_{k-1})\in \mathbb{R}^{k-1}$.
We identify the 1-dimensional horizontal subgroup $\mathbb{V}_e:= \mathrm{span}\{e\}\times \{0\}$ with the 1-dimensional subspace $V_e$ of $\R^k$, via the identification of $\V_0$ with $\R^k$. 
We also let $\mathbb{W}_e =\mathrm{span}\{e_1^{\bot},\ldots,e_{k-1}^{\bot}\}\times \mathbb{R}^{2n-k+1}$ to be the complementary vertical subgroup of $\V_e$.  
We finally consider the \emph{vertical projection} $\pi_{\mathbb{W}_e}$ associated to the splitting $\mathbb{H}^n= \mathbb{V}_e \ltimes \mathbb{W}_e$, that is
\begin{displaymath}
  \pi_{\mathbb{W}_e}: \mathbb{H}^n \to \mathbb{W}_e,\quad p=p_{\V_e}\cdot p_{\W_e} \mapsto p_{\W_e};
\end{displaymath}
see for instance \cite{MR2789472}.
\begin{lemma}[$1$-dimensional slices of intrinsic Lipschitz graphs]\label{l:SlicingLemma}
Let $e\in S^{k-1}$ and $a\in \R^{k-1}$. If $\varphi:\mathbb{V}_0 \to \mathbb{W}_0$ is intrinsic $L$-Lipschitz, then so is the function 
    \begin{displaymath}
    \psi:=\psi_{a,e}:\mathbb{V}_e\to \mathbb{W}_e,\quad  \psi(v):=
\pi_{\mathbb{W}_e}\left(\Phi(\ell_{a,e}(v))\right),
\end{displaymath}
where $\ell_{a,e}(v)$ is defined in \eqref{def_lae} and $\Phi$ denotes the graph map of $\varphi$.
\end{lemma}
\begin{proof}
Note that $\omega(e,e_i^{\bot})=0$ for $i=1,\ldots,k-1$ since $e,e_i^{\bot} \in \mathbb{V}_0$, and $\mathbb{V}_0$ is isotropic.
Therefore, 
\begin{equation}\label{eq:ProdAdd1}
    \ell_{a,e}(v)= (v,0)\cdot \left(\sum_{i=1}^{k-1}a_i e_i^{\bot} ,0\right),\quad v\in V_e.
\end{equation}
Then, for $\psi$ defined as in the statement of the lemma, we have
\begin{displaymath}
\psi(v)=
     \pi_{\mathbb{W}_e}\left(\Phi(\ell_{a,e}(v))\right)\overset{\eqref{eq:ProdAdd1}}{=}  \pi_{\mathbb{W}_e}\left(v\cdot \left(\sum_{i=1}^{k-1}a_i e_i^{\bot} ,0\right)\cdot \varphi(\ell_{a,e}(v))\right) =\left(\sum_{i=1}^{k-1}a_i e_i^{\bot} ,0\right)\cdot \varphi(\ell_{a,e}(v)).
\end{displaymath}
In the last step we used that $v\in \mathbb{V}_e$ (identified with $V_e$), while $(\sum_{i=1}^{k-1}a_i e_i^{\bot} ,0)$ and $\varphi(\ell_{a,e}(v))$ belong to $\W_e$. 
We denote by $\Psi$ the graph map of $\psi$, given by
\[\Psi(v):=v\cdot \psi(v)=v\cdot \left(\sum_{i=1}^{k-1}a_i e_i^{\bot} ,0\right)\cdot  \varphi(\ell_{a,e}(v))=\ell_{a,e}(v)\cdot \varphi(\ell_{a,e}(v))=\Phi(\ell_{a,e}(v)).\]
Therefore
\begin{align*}
    \|v'^{-1}\cdot v\cdot \Psi(v)^{-1}\cdot \Psi(v')\|&=  \|v'^{-1}\cdot v\cdot \Phi(\ell_{a,e}(v))^{-1}\cdot \Phi(\ell_{a,e}(v'))\|\\
    &=  \|(\ell_{a,e}(v'),0)^{-1}\cdot (\ell_{a,e}(v),0)\cdot \Phi(\ell_{a,e}(v))^{-1}\cdot \Phi(\ell_{a,e}(v'))\|,
\end{align*}
where we used again $\omega(e,e_i^{\bot})=0$ for $i=1,\ldots,k-1$ in the last step. Using Proposition \ref{1.2}, the last expression is bounded by $L|\ell_{a,e}(v)-\ell_{a,e}(v')|=L|v-v'|$, since $\varphi$ is intrinsic $L$-Lipschitz. This proves the $L$-Lipschitz property of $\psi$.
\end{proof}
Note that, using the notation of the previous proof,
\[\mathrm{gr}(\psi)=\{\Psi(v): v\in \V_e \}=\{\Phi(\ell_{a,e}(v)):v\in\V_e\}\subset\mathrm{gr}(\varphi).\] In particular, the intrinsic graph of $\psi$ is a 1-dimensional slice of the intrinsic graph of $\varphi$.
\medskip\\
Until the end of this section, we use notation inspired by \cite[Sections 2.2-2.4]{MR4345824}. More care needs to be taken because not every plane containing a horizontal line is itself horizontal. The following lemma allows to extend a $1$-dimensional intrinsic affine function in a canonical way to a higher-dimensional intrinsic affine function.

\begin{lemma}\label{l:1DSlicesAffine} Let $1<k\leq n$.
    For $e\in S^{k-1}$, let $\mathbb{V}_e=\mathrm{span}\{e\}\times \{0\}$ with complementary vertical subgroup $\mathbb{W}_e$. If $A_e:\mathbb{V}_e \to \mathbb{W}_e$ is intrinsic affine, then there exists an intrinsic affine map $A:\mathbb{V}_0=\mathbb{R}^k \times \{0\}\to \mathbb{W}_0=\{0\}\times \mathbb{R}^{2n+1-k}$ with the property that
    \begin{displaymath}
    \mathbb{A}_e(v)= \mathbb{A}(\Pi(\mathbb{A}_e(v))),\quad v\in\mathbb{\V}_e,
    \end{displaymath}
    where $\mathbb{A}$ and $\mathbb{A}_e$ denote the graph maps of $A$ and $A_e$ respectively, and $\Pi$ is the projection defined in \eqref{projtoRk}.
\end{lemma}

\begin{proof}
   \textbf{Step 1:} we first assume that $A_e:\V_e\to\W_e$ is intrinsic \textit{linear}, and we show the existence of an intrinsic \textit{linear} map $A:\V_0\to\W_0$ as in the statement.
    Since $A_e$ is intrinsic linear, then $\mathbb A_e(\V_e)$ is a 1-dimensional horizontal subgroup of $\Hn$. We now show that there exists a $k$-dimensional horizontal subgroup $\V=V\times \{0\}\subset\Hn$ which contains $\mathbb A_e(\V_e)$ with the additional property that $V$ is a Euclidean graph over $\R^k$ in $\R^{2n}$. We let $V=\text{span}_{\R^{2n}}(v_1,v_2,\dots, v_k)$, where \[v_k=\pi(\mathbb A_e(e)),\quad\text{and}\quad v_i= e_i^\perp+\omega(e_i^\perp, \pi(\mathbb A_e(e)))Je, \text{  for }1\leq i\leq k-1,\]
    where $J$ is the standard complex structure and $\{e_1^\perp,\dots, e_{k-1}^\perp\}$ is an orthonormal basis of $e^\perp$ inside $\R^k$. We now show that $\omega(v_i,v_j)=0$ for every $1\leq i,j\leq k$. For $j=k$ and $1\leq i\leq k-1$, this follows from
    \begin{align*}
    \omega(v_i,v_k)&=\omega(e_i^\perp,\pi(\mathbb A_e(e)))+\omega(e_i^\perp, \pi(\mathbb A_e(e)))\,\omega(Je, \pi(\mathbb A_e(e)))\\ &=\omega(e_i^\perp,\pi(\mathbb A_e(e)))-\omega(e_i^\perp, \pi(\mathbb A_e(e)))\,\langle e, \pi(\mathbb A_e(e))\rangle \\ &=\omega(e_i^\perp,\pi(\mathbb A_e(e)))-\omega(e_i^\perp, \pi(\mathbb A_e(e)))=0,
    \end{align*}
    where the penultimate equality follows from the fact that $\langle e, \pi(\mathbb A_e(e))\rangle =1$, which holds since $\mathbb A_e$ is a graph map over $\V_e$ and $e$ is a unit vector. This also implies that $\omega(v_k,v_j)=0$ for every $1\leq j\leq k-1$. Assume now that $1\leq i,j\leq k-1$. Then
    \begin{align*}
        \omega(v_i,v_j)=&\,\omega(e_i^\perp, e_j^\perp)+\omega(e_i^\perp, \pi(\mathbb A_e(e)))\,\omega(Je,e_j^\perp)+\omega(e_j^\perp, \pi(\mathbb A_e(e)))\,\omega(e_i^\perp,Je)\\ &+ \omega(e_i^\perp, \pi(\mathbb A_e(e)))\,\omega(e_j^\perp, \pi(\mathbb A_e(e)))\,\omega(Je,Je)\\=  &\,\langle J e_i^\perp, e_j^\perp\rangle-\omega(e_i^\perp, \pi(\mathbb A_e(e)))\,\langle e,e_j^\perp\rangle+\omega(e_j^\perp, \pi(\mathbb A_e(e)))\,\langle e_i^\perp,e\rangle=0.
    \end{align*}
     It follows that $V$ is an isotropic subspace of $\R^{2n}$. Note also that, by construction, $v_1,\dots, v_k$ are independent (since $e, e_1^\perp, \dots, e_{k-1}^\perp, Je$ are independent) and span a $k$-dimensional plane which is a graph over $\R^k$. 
    In particular, $V$ is the (Euclidean) graph of an isotropic linear map $A_0:\R^k\to\R^{2n-k}$.
    Hence, by Proposition \ref{ahp}, $A_0$ can be lifted to an intrinsic linear map $A: \V_0\to\W_0$ whose intrinsic graph is $\V=V\times\{0\}$. Thus, for every $p\in\V$, it holds $\mathbb A(\Pi(p))=p$. In particular,
    \[\mathbb{A}_e(v)= \mathbb{A}(\Pi(\mathbb{A}_e(v))),\quad v\in\mathbb{V}_e,\]
    since $\mathbb A_e(\mathbb V_e)\subset \V$.\\
    \textbf{Step 2:} assume now that $A_e:\V_e\to\W_e$ is intrinsic \textit{affine}. Let $p=\mathbb{A}_e(0)$, where $\mathbb{A}_e$ denotes the intrinsic graph map of $A_e$, and consider
 \begin{equation}\label{eq:B_e_def}
     \mathbb{B}_e:= p^{-1}\cdot \mathbb{A}_e.
 \end{equation}
We claim that $\mathbb{B}_e$ is the intrinsic graph map of an intrinsic \emph{linear} map $B_e:\mathbb{V}_e \to \mathbb{W}_e$. Indeed, $\mathbb{B}_e(\mathbb{V}_e)=p^{-1}\cdot\mathbb{A}_e(\mathbb{V}_e)$ is a horizontal subgroup and an intrinsic graph of a function $B_e:\mathbb{V}_e\to \mathbb{W}_e$, by the left invariance of these properties. Since moreover $0\in \mathbb{B}_e(\mathbb{V}_e) $, the function $B_e$ is intrinsic linear. By Step 1 of the proof, we can then find an intrinsic linear map $B:\V_0 \to \W_0$ such that its graph map $\mathbb{B}$ satisfies the relation
\begin{equation}\label{eq:B_iden}
    \mathbb{B}_e(v)=\mathbb{B}(\Pi(\mathbb{B}_e(v))),\quad v\in \V_e.
\end{equation}
Then the map $v\mapsto \mathbb{A}(v):=p\cdot \mathbb{B}(v+\Pi(p^{-1}))$ sends $\V_0$ to a horizontal $k$-plane that is the intrinsic graph of an (intrinsic affine) function $A:\V_0\to \W_0$. Moreover
\begin{displaymath}
    \mathbb{A}_e(v)= p\cdot \mathbb{B}_e(v)\overset{\eqref{eq:B_iden}}{=}p\cdot \mathbb{B}(\Pi(\mathbb{B}_e(v)))= p\cdot \mathbb{B}\left(\Pi(p^{-1})+\Pi(p\cdot \mathbb{B}_e(v))\right)\overset{\eqref{eq:B_e_def}}{=}\mathbb{A}(\Pi(\mathbb{A}_e(v))),\quad v\in \mathbb{V}_e,
\end{displaymath}
as desired. 
\end{proof}

\subsubsection{One-codimensional slices of intrinsic Lipschitz graphs}\label{ss:1cDslice}

Let $1<k\leq n$, $\mathbb{V}_0=\mathbb{R}^k \times \{0\}$ and $\mathbb{W}_0=\{0\}\times \mathbb{R}^{2n+1-k}$. Fix an intrinsic $L$-Lipschitz function 
\begin{displaymath}
\varphi:\mathbb{V}_0 \equiv \mathbb{R}^k \to \mathbb{W}_0\equiv \R^{2n+1-k},\quad \varphi(x)=(\varphi_1(x),\ldots,\varphi_{2n+1-k}(x))
\end{displaymath}
with graph map $\Phi(x)=x\cdot \varphi(x)$.
We aim to use that ($k-1$)-dimensional ``slices'' of $\varphi$ are again intrinsic Lipschitz, with controlled Lipschitz constant. To this end, consider an arbitrary affine ($k-1$)-dimensional plane $\pi_{a,e}$ in $\mathbb{R}^k$, parametrized as follows
\begin{equation}\label{eq:AffParam}
\pi_{a,e}(v)= v+ a e^{\bot}, \quad v\in  V_e:=\mathrm{span}\{e_1,\dots, e_{k-1}\}\subset \mathbb{R}^k,
\end{equation}
for fixed $e=\{e_1,\dots, e_{k-1}\}\subset S^{k-1}$ an orthonormal orientation of $\pi$, a unit vector $e^\bot$  orthogonal to $\pi$ inside $\R^k$, and $a\in \mathbb{R}$.

We consider the ($k-1$)-dimensional horizontal subgroup $\mathbb{V}_e:= V_e\times \{0\}$, and let $\mathbb{W}_e =\mathrm{span}\{e^{\bot}\}\times \mathbb{R}^{2n-k+1}$ be the complementary vertical subgroup. We also consider the vertical projection $\pi_{\mathbb{W}_e}$ associated to the splitting $\mathbb{H}^n= \mathbb{V}_e \ltimes \mathbb{W}_e$, defined as above Lemma \ref{l:SlicingLemma}. Through the identification of $\mathbb{V}_0$ with $\mathbb{R}^k$, we also identify the $(k-1)$-dimensional horizontal subgroup $\V_e$  of $\mathbb{H}^n$ with the $(k-1)$-dimensional subspace $V_e$ of $\mathbb{R}^k$.

\begin{lemma}[Codimension-$1$ slices of intrinsic Lipschitz graphs]\label{l:SlicingLemmak-1}
Let $e=\{e_1,\ldots,e_{k-1}\}\subset S^{k-1}$ and $a\in \mathbb{R}$ be arbitrary. 
 If $\varphi:\mathbb{V}_0 \to \mathbb{W}_0$ is intrinsic $L$-Lipschitz, then so is the function 
    \begin{displaymath}
    \psi:=\psi_{a,e}:\mathbb{V}_e\to \mathbb{W}_e,\quad \psi(v):= 
\pi_{\mathbb{W}_e}\left(\Phi(\pi_{a,e}(v))\right),
\end{displaymath}
where $\pi_{a,e}$ is defined in \eqref{eq:AffParam}.
Moreover, 
\begin{equation}\label{eq:SliceIncl}
    \mathrm{gr}(\psi)=\Phi(\pi_{a,e}(V_e))\subset \mathrm{gr}(\varphi).
\end{equation}
\end{lemma}
By \eqref{eq:SliceIncl},
the intrinsic graph of $\psi$ is indeed a ($k-1$)-dimensional slice of the intrinsic graph of $\varphi$.

\begin{proof}
First note that $\omega(e_i,e^{\bot})=0$ for $i=1,\ldots,k-1$ since $e_i,e^{\bot} \in \mathbb{V}_0$, and $\mathbb{V}_0$ is isotropic. Therefore, we have
\begin{equation}\label{eq:ProdAdd}
    (v,0)\cdot (ae^{\bot},0)=(v+ae^{\bot},0),\quad v\in V_e.
\end{equation}
Now, for $\psi$ defined as in the statement of the lemma, and with the identification of $V_e$ and $\V_e$, we have
\begin{displaymath}
\psi(v)=
     \pi_{\mathbb{W}_e}\left(\Phi(\pi_{a,e}(v))\right)\overset{\eqref{eq:ProdAdd}}{=}  \pi_{\mathbb{W}_e}\left(v\cdot (ae^{\bot},0)\cdot \varphi(\pi_{a,e}(v))\right) =(a e^\bot,0)\cdot \varphi(\pi_{a,e}(v)).
\end{displaymath}
Here we used in the last step that $v\in \mathbb{V}_e$, while $(ae^{\bot},0)$ and $\varphi(\pi_{a,e}(v))$ (and hence also their product) belong to $\W_e$. 

Next, denoting by $\Psi$ the graph map of $\psi$ associated to the splitting $\mathbb{H}^n = \mathbb{V}_e \cdot \mathbb{W}_e$, it follows
\begin{equation}\label{eq:FormPsi}
    \Psi(v)= v\cdot (a e^\bot,0)\cdot \varphi(\pi_{a,e}(v))
   \overset{\eqref{eq:ProdAdd}}{=} \pi_{a,e}(v)\cdot \varphi(\pi_{a,e}(v))=\Phi(\pi_{a,e}(v)),
\end{equation}
which proves that the intrinsic graph of $\psi$ is a slice of the intrinsic graph of $\varphi$, as stated in \eqref{eq:SliceIncl}. Finally, using \eqref{eq:FormPsi}, the intrinsic $L$-Lipschitz continuity of $\psi$ can be seen as follows
\begin{align*}
    \left\|{v'}^{-1}\cdot v\cdot \Psi\left(v\right)^{-1}\cdot \Psi\left(v'\right)\right\|&=  \left\|{v'}^{-1}\cdot v\cdot \Phi(\pi_{a,e}(v))^{-1}\cdot \Phi(\pi_{a,e}(v'))\right\|\\
    &=  \|(\pi_{a,e}(v'),0)^{-1}\cdot (\pi_{a,e}(v),0)\cdot \Phi(\pi_{a,e}(v))^{-1}\cdot \Phi(\pi_{a,e}(v'))\|,
\end{align*}
where we used again $\omega(e_i,e^{\bot})=0$ for $i=1,\ldots,k-1$ in the last step. The last expression is bounded by $L|\pi_{a,e}(v)-\pi_{a,e}(v')|=L|v-v'|$, since $\varphi$ is intrinsic $L$-Lipschitz, recall Proposition \ref{1.2}. This proves the $L$-Lipschitz property of $\psi$.
\end{proof}

\begin{lemma}[Extending $1$-codimensional intrinsic affine functions]\label{l:k-1SlicesAffine} Let $1<k\leq n$.
    For a fixed system $e=\{e_1,\dots, e_{k-1}\}\subset S^{k-1}$ of orthonormal vectors, consider the horizontal subgroup $\mathbb{V}_e=\mathrm{span}\{e_1,\dots, e_{k-1}\}\times \{0\}$ with complementary vertical subgroup $\mathbb{W}_e$. If $A_e:\mathbb{V}_e \to \mathbb{W}_e$ is intrinsic affine, then there exists an intrinsic affine map $A:\mathbb{V}_0=\mathbb{R}^k \times \{0\}\to \mathbb{W}_0=\{0\}\times \mathbb{R}^{2n+1-k}$ with the property that
    \begin{displaymath}
    \mathbb{A}_e(v)= \mathbb{A}(\Pi(\mathbb{A}_e(v))),\quad v\in \mathbb{V}_e,
    \end{displaymath}
    where $\mathbb{A}$ and $\mathbb{A}_e$ denote the graph maps of $A$ and $A_e$, respectively, and $\Pi:\mathbb{H}^n \to \mathbb{R}^k$ is the projection defined in \eqref{projtoRk}. 
    In addition, one can choose $A$ such that $\mathrm{Lip}(\mathbb A)\lesssim_{k,\mathrm{Lip}(\mathbb A_e)}1$.
\end{lemma}

\begin{proof}
    Throughout the proof, we will often identify elements in $\mathbb{V}_0$ with points in $\mathbb{R}^k$ without special mentioning. \\\textbf{Step 1.} We first assume that $A_e:\V_e\to\W_e$ is intrinsic \textit{linear}, and we show the existence of an intrinsic \textit{linear} map $A:\V_0\to\W_0$ as in the statement.
Since $A_e$ is intrinsic linear, then $\mathbb A_e(\V_e)$ is a ($k-1$)-dimensional horizontal subgroup of $\Hn$. We now show that there exists a $k$-dimensional horizontal subgroup $\V=V\times \{0\}\subset\Hn$ which contains $\mathbb A_e(\V_e)$ with the additional property that $V$ is a Euclidean graph over $\R^k$ in $\R^{2n}$.

We let $V=\text{span}_{\R^{2n}}(v_1,v_2,\dots, v_k)$ where
\begin{equation}\label{eq:vj_def}
v_i=\pi(\mathbb A_e(e_i))\text{ for }1\leq i\leq k-1\quad\text{and}\quad v_{k}= e^{\bot}+\sum_{j=1}^{k-1}\omega(e^{\bot},\pi(\mathbb A_e(e_j)))Je_j,
\end{equation}
where $J$ is the standard complex structure defined below \eqref{sym}, and $e^{\bot}$ is a unit vector orthogonal to $V_e=\mathrm{span}\{e_1,\ldots,e_{k-1}\}$ in $\mathbb{R}^k$. As usual, we identify here $v\in \mathbb{R}^k$ with $(v,0)\in \mathbb{R}^k \times \mathbb{R}^{2n-k}=\mathbb{R}^{2n}$, so that $e^{\bot}\in \mathbb{R}^k \times \{0\}\subset \mathbb{R}^{2n}$ and $Je_i\in \{0\}\times \mathbb{R}^n \subset \mathbb{R}^{2n}$.

We claim that $\omega(v_i,v_j)=0$ for every $1\leq i,j\leq k$. If $i,j\leq k-1$, this follows by the fact that $v_i$ and $v_j$ belong to  $\pi(\mathbb A_e(\mathbb V_e))$, which is an isotropic subspace of $\mathbb{R}^{2n}$. The other cases follow directly by construction of $v_k$. Indeed, for $1\leq i\leq k-1$, we have
\begin{align}\label{eq:VanForm}
    \omega(v_i,v_k)&=\omega(\pi(\mathbb A_e(e_i)),e^{\bot})+\sum_{j=1}^{k-1}\omega(e^{\bot},\pi(\mathbb A_e(e_j)))\omega(\pi(\mathbb A_e(e_i)),Je_j)\notag\\
    &=\omega(\pi(\mathbb A_e(e_i)),e^{\bot})+\sum_{j=1}^{k-1}\omega(e^{\bot},\pi(\mathbb A_e(e_j)))\langle\pi(\mathbb A_e(e_i)),e_j\rangle\notag\\
    &=\omega(\pi(\mathbb A_e(e_i)),e^{\bot})+\omega(e^{\bot},\pi(\mathbb A_e(e_i)))=0.
\end{align}
To see why the penultimate equation holds true, it suffices to consider $\Pi(\mathbb A_e(e_i))$, since $\langle v,e_j\rangle=0$ for $v\in \{0\}\times \mathbb{R}^{2n-k}$. 
Now, writing $\Pi(\mathbb A_e(e_i))$ in the orthonormal basis $\{e_1,\ldots,e_{k-1},e^{\bot}\}$ of $\mathbb{R}^k$
shows that $\langle\pi(\mathbb A_e(e_i)),e_j\rangle=0$ for $i\neq j$, and $\langle\pi(\mathbb A_e(e_i)),e_i\rangle=1$,
since $\Pi(\mathbb A_e(V_e))$ is a graph in $\mathbb{R}^k$ over $V_e=\mathrm{span}\{e_1,\ldots,e_{k-1}\}$. Thus 
 \eqref{eq:VanForm} holds.

Note also that $v_1,\dots, v_k$ are independent and span a $k$-dimensional plane $V$ which is a graph over $\R^k\times \{0\}$ in $\mathbb{R}^{2n}$. To verify the linear independence, consider the orthogonal projection $\pi_k:\mathbb{R}^{2n}\to \mathbb{R}^k\times \{0\}$. 
Then, by definition \eqref{eq:vj_def} of $v_j$, 
\begin{equation}\label{eq:CoordProjSpan}
    \pi_k\left(\sum_{j=1}^k c_j v_j\right)=\left[\sum_{j=1}^{k-1}c_j e_j\right]
+ \widetilde{c_k }e^{\bot},\quad\text{where }\widetilde{c}_k = c_k + \sum_{j=1}^{k-1}c_j\langle \pi(\mathbb{A}_e(e_j)),e^{\bot} \rangle.
\end{equation}
Thus, if $\sum_{j=1}^k c_j v_j=0$, then $\widetilde{c}_k=0$ and $c_j=0$ for $j=1,\ldots,{k-1}$, and consequently also $c_k=0$,  by linear independence of $\{e_1,\ldots,e_{k-1},e^{\bot}\}$ in $\mathbb{R}^k$. The identity \eqref{eq:CoordProjSpan} then also shows that $V$ is a graph over $\mathbb{R}^k\times \{0\}\subset \mathbb{R}^{2n}$.  In particular, $V$ is the (Euclidean) graph of an isotropic linear map $A_0:\R^k\to\R^{2n-k}$. 
    Hence, by Proposition \ref{ahp}, $A_0$ can be lifted to an intrinsic linear map $A: \V_0\to\W_0$ whose intrinsic graph is $\V=V\times \{0\}$. Thus, for every $p\in\V$, it holds $\mathbb A(\Pi(p))=p$. In particular,
    \[\mathbb{A}_e(v)= \mathbb{A}(\Pi(\mathbb{A}_e(v))),\quad v\in \mathbb{V}_e\]
since $\mathbb{A}_e(\mathbb{V}_e)\subset \mathbb{V}$.

Finally, we aim to control the Lipschitz constant of $\mathbb A:\mathbb{V}_0 \to \mathbb{H}^{n}$. 
By definition, we have
\begin{displaymath}
    \pi_k(v_i)=e_i+\langle \pi(\mathbb A_e(e_i)),e^{\bot}\rangle e^{\bot}\text{ for }i=1,\ldots,k-1,\quad\text{and}\quad \pi_k(v_k)=e^{\bot}.
\end{displaymath}
Note that $\{\pi_k(v_i):i=1,\dots, k\}$ form a basis of $\R^k$. Let $x$ be an arbitrary vector of $\R^k$ and write
\begin{align*}x=\sum_{i=1}^kx_i\pi_k(v_i)&=\sum_{i=1}^{k-1} x_i(e_i+\langle \pi(\mathbb A_e(e_i)),e^{\bot}\rangle e^{\bot})+x_ke^\perp\\ &=\sum_{i=1}^{k-1}x_ie_i+\left(x_k+\sum_{i=1}^{k-1}x_i\langle \pi(\mathbb A_e(e_i)),e^{\bot}\rangle\right)e^\perp.\end{align*}
Since $\{e_1,\dots, e_{k-1}, e^\perp\}$ is an orthonormal basis of $\R^k$,
\begin{equation}\label{orthonormal}\|x\|^2=\sum_{i=1}^{k-1}x_i^2+\left(x_k+\sum_{i=1}^{k-1}x_i\langle \pi(\mathbb A_e(e_i)),e^{\bot}\rangle\right)^2.\end{equation} Here and in the following, we denote by $\|\cdot\|$ the Kor\'{a}nyi norm, which agrees with the Euclidean norm when considering elements in $\mathbb R^{2n}\times \{0\}\subset \mathbb H^n$.
Notice that 
\[|\langle \pi(\mathbb A_e(e_i)),e^{\bot}\rangle|\leq \|\mathbb A_e(e_i)\|\leq \mathrm{Lip}(\mathbb A_e).\]
By \eqref{orthonormal}, it follows that
\[|x_i|\leq \|x\|\qquad i=1,\dots, k-1;\]
\[|x_k|\leq \|x\|+\mathrm{Lip}(\mathbb A_e)\sum_{i=1}^{k-1}|x_i|\leq (1+(k-1)\mathrm{Lip}(\mathbb A_e))\|x\|.\]
Hence
\begin{equation}\label{l1sum}\sum_{i=1}^k|x_i|\leq (k+(k-1)\mathrm{Lip}(\mathbb A_e))\|x\|.\end{equation}
By construction, $\mathbb A:\mathbb V_0\equiv \R^k\to \mathbb H^n\equiv \R^{2n+1}$ is a linear map. Hence
\[\mathbb A (x)=\mathbb A\left(\sum_{i=1}^kx_i\pi_k(v_i)\right)=\sum_{i=1}^kx_i\mathbb A(\pi_k (v_i)).\]
Finally, it suffices to prove that $\|\mathbb A(\pi_k(v_i))\|\lesssim_{k,\mathrm{Lip}(\mathbb A_e)}1$, since then we deduce that
\[\|\mathbb A(x)\|=\left\|\sum_{i=1}^kx_i\mathbb A(\pi_k (v_i))\right\|\leq \sum_{i=1}^k|x_i|\|\mathbb A(\pi_k(v_i))\|\stackrel{\eqref{l1sum}}\lesssim_{k,\mathrm{Lip} (\mathbb A_e)}\|x\|\]
and we conclude by linearity:
\[d(\mathbb A(x), \mathbb A(y))=\|\mathbb A(x)^{-1}\mathbb A(y)\|=\|\mathbb A(y)-\mathbb A(x)\|=\|\mathbb A(y-x)\|\lesssim_{k,\mathrm{Lip}(\mathbb A_e)}\|y-x\|.\]
Here we exploited the linearity of the map $\mathbb A$ and the fact that $\mathbb A(\mathbb V_0)=\mathbb V$ is a horizontal subgroup of $\mathbb H^n$. It remains to show that $\|\mathbb A(\pi_k(v_i))\|\lesssim_{k,\mathrm{Lip}(\mathbb A_e)}1$. First, notice that $\mathbb A(\pi_k(v_i))=v_i$. Then, by definition,
\[\|v_i\|=\|\pi(\mathbb A_e(e_i))\|\leq \|\mathbb A_e(e_i)\|\leq \mathrm{Lip}(\mathbb A_e)\qquad i=1,\dots, k-1;\]
\[\|v_k\|=\|e^{\bot}+\sum_{j=1}^{k-1}\omega(e^\perp, \pi(\mathbb A_e(e_j)))J e_j\|\leq 1+\sum_{j=1}^{k-1}\|\pi(\mathbb A_e(e_j)\|\lesssim_{k,\mathrm{Lip}(\mathbb A_e)}1.\]

 \textbf{Step 2:} If $A_e:\V_e\to\W_e$ is intrinsic \textit{affine}, then the conclusion follows from Step 1 by left-translation, exactly as in Step 2 of Lemma \ref{l:1DSlicesAffine}. The control on the Lipschitz constant of $\mathbb A$ follows by left-translation invariance of the Kor\'{a}nyi distance and by Step 1.
\end{proof}

\subsubsection{Proof of the geometric lemma for intermediate dimensions}\label{ss:ProofConclu}
Let $1<k\leq n$ and assume that $\varphi:\mathbb{V}_0=\mathbb{R}^k \times \{0\}\to \{0\}\times \mathbb{R}^{2n+1-k}=\W_0$ is intrinsic $L$-Lipschitz. Let $\mathcal{D}=\mathcal{D}^k$ be the family of standard dyadic cubes on $\R^k$, namely $\mathcal D^k=\cup_{j\in\mathbb Z} \mathcal D_j^k$, where
\[\mathcal D^k_j:=\left\{Q=\prod_{i=1}^k\big[a_i2^{-j}, (a_i+1)2^{-j}\big): (a_1,\dots, a_k)\in\mathbb Z^k\right\}.\] Throughout this section, for a parameter $\lambda>0$, the set $\lambda Q$ denotes the cube which is concentric with $Q$ and has diameter equal to $\lambda\,\mathrm{diam}(Q)$.
Suppressing the $\varphi$-dependence in the notation, for $Q\in\mathcal{D}$, we denote, for every $1\leq p<+\infty$, 
\begin{displaymath}
    \beta_p(Q):=\inf_A \left[\frac{1}{\mathrm{diam}(Q)^k}\int_Q \left(\frac{d(\varphi(x),A(x))}{\mathrm{diam}(Q)}\right)^p d\mathcal{H}^{k}(x)\right]^{1/p},
\end{displaymath}
\begin{displaymath}
    \beta_\infty(Q):=\inf_A\sup_{x\in Q}\frac{d(\varphi(x), A(x))}{\mathrm{diam} (Q)},
\end{displaymath}
where the infimum is taken over all intrinsic affine maps $A:\mathbb{V}_0 \to \mathbb{W}_0$. These are parametric and dyadic versions of the $\beta$-numbers in our main result, Theorem \ref{t:GlemBetanILGLk=2}.

As in Definition \ref{def_projbetas}, we denote by $A(k,m)$ the family of affine $m$-planes in $\mathbb{R}^k$, 
and by $\mathcal{A}_m(B)$ the subfamily of those elements in $ A(k,m)$ that intersect a given set $B\subset \mathbb{R}^k$, as in \cite{MR4345824}. The measure $\eta_m$ on $A(k,m)$ is defined as in \cite{MR4345824}. We use here the same notation, where $\mathbb{V}_0$ is identified with $\mathbb{R}^k$ via $(x,0)\in \mathbb{R}^k \times \{0\}\mapsto x\in \mathbb{R}^k$.

Let $Q\in \mathcal{D}$ be a standard dyadic cube in $\mathbb{R}^k$. For any $1\leq m\leq k$, any affine plane $V\in\mathcal{A}_m(Q)$, $1\le p<\infty$ and $1\le q <\infty$, 
we define
\begin{displaymath}\beta_p(Q,V):=\inf_A \left[\frac{1}{\mathrm{diam}(Q)^m}\int_{Q\cap V} \left(\frac{d(\varphi(x),A(x))}{\mathrm{diam}(Q)}\right)^p d\mathcal{H}^{m}(x)\right]^{1/p},
\end{displaymath}
\begin{displaymath}
\beta_{\infty}(Q,V):=   \inf_A\sup_{x\in Q\cap V}\frac{d(\varphi(x),A(x))}{\mathrm{diam}(Q)} ,
\end{displaymath}
where the infimum is taken over all intrinsic affine $A:\mathbb{V}_0 \to \mathbb{W}_0$.
Finally, we define 
\begin{displaymath}
    \beta^{m}_{p,q}(Q) :=\left[\fint_{\mathcal{A}_m(Q)}\beta_{p}(Q,V)^q\,d\eta_m(V)\right]^{1/q}.
\end{displaymath}

We stress that in these definitions, the maps $A$ are required to be \emph{intrinsic} affine, and the metric $d$ is the Kor\'anyi distance, so that the resulting coefficients are different from their Euclidean counterparts. Nonetheless, to ease the notation, we suppress the dependence on the Heisenberg geometry in the symbols. 

\begin{remark}\label{r:DiffBeta}
We will also use versions of these coefficients where ``$Q$'' is replaced by an enlarged cube ``$CQ$'' or by a ball ``$B$'' on both sides of the identity.
\end{remark}
We will repeatedly use that the relevant coefficients are bounded in our setting, as ensured by the following remark.

\begin{remark}\label{rmk:boundedness} 
   Let $1\leq k\leq n$ and assume that $\varphi:\mathbb{V}_0=\mathbb{R}^k \times \{0\}\to \{0\}\times \mathbb{R}^{2n+1-k}=\W_0$ is intrinsic $L$-Lipschitz.
   Then, for any cube $Q$ in $\V_0$, any $C\geq 1$, and if $k>1$, 
   for any affine $m$-plane $V$ with  $1\leq m\leq k-1$, it holds $\beta_p(CQ)\leq L$, $\beta_{p}(CQ,V)\leq L$ and $\beta_{p,q}^m(CQ)\leq L$ for all $1\leq p\leq \infty$, $1\leq q<\infty$.
    Indeed, let $\bar x$ be any point in $CQ$. Consider the intrinsic affine map $A:\V_0\to\W_0$ defined by 
    \[A(x):=x^{-1}\cdot \bar x\cdot\varphi(\bar x)\cdot \bar x^{-1}\cdot x.\]
    Notice that this map is well-defined since $\W_0$ is normal and that it is intrinsic affine since
    \[\gr(A)=\bar x\cdot\varphi(\bar x)\cdot \bar x^{-1}\cdot \mathbb V_0\]
    is a horizontal affine subspace of $\Hn$. Then, for any $x\in CQ$, by Proposition \ref{1.2}, \[d(\varphi(x), A(x))=\|x^{-1}\cdot \bar x\cdot\varphi(\bar x)^{-1}\cdot\bar x^{-1}\cdot x\cdot \varphi(x)\|\leq L\|\bar x^{-1}\cdot x\|\leq L\mathrm{diam}(CQ).\] The claim now follows immediately by definition of  $\beta_p(Q)$, $\beta_{p}(Q,V)$ and $\beta_{p,q}^m(Q)$.
\end{remark}

The geometric lemma for $1$-dimensional intrinsic Lipschitz graphs (Theorem \ref{t:GlemStratBetanILGLinftyk=2} in its parametric version) implies a geometric lemma for $k$-dimensional intrinsic Lipschitz functions and the integral-geometric $\beta_{\infty,4}^1$-numbers. The corresponding statement in the Euclidean case can be proven by a short argument, see  \cite[Lemma 2.5]{MR4345824}, but deriving the key inequality \eqref{eq:1DGlemRestrLine} in our setting requires additional work since 
 we need to consider all components of the relevant mappings simultaneously and resort to the 
slicing and extension theorems from Section \ref{ss:1Dslice}.

\begin{lemma}[$1$-dim.\ Dorronsoro implies $\mathrm{GLem}( \beta_{\infty,4}^1,4)$ for $k$-dimensional graphs]\label{l:TLem2.5} Let $1<k\leq n$, and set $\mathbb{V}_0=\mathbb{R}^k\times \{0\}$ as well as $\mathbb{W}_0=\{0\}\times \mathbb{R}^{2n+1-k}$.
    Let $\varphi:\mathbb{V}_0\to \mathbb{W}_0$ be intrinsic $L$-Lipschitz. Then, for any $C\geq 1$,
    \begin{displaymath}
   \sum_{Q\in \mathcal{D}^k,Q\subset Q_0}   \beta_{\infty,4}^1(CQ)^4\mathcal{H}^k(Q)\lesssim_{C,L,k}   \mathcal{H}^k(Q_0),\quad Q_0\in \mathcal{D}^k.
    \end{displaymath}
\end{lemma}

\begin{proof} Let $\ell:=\ell_{a,e}$ be an arbitrary affine line in $\mathbb{V}_0\equiv \mathbb{R}^k$, parametrized by \[\ell_{a,e}(v)= v + \sum_{i=1}^{k-1}a_i e_i^{\bot}\] for $ v\in V_e:=\mathrm{span}\{e\}\subset\R^k$, for some $e\in S^{k-1}$, $\{e_1^\perp,\dots, e_{k-1}^\perp\}$ an orthonormal basis of $e^\perp$ inside $\R^k$ and $a=(a_1,\ldots,a_{k-1})\in \mathbb{R}^{k-1}$.
By Lemma \ref{l:SlicingLemma}, $\Phi(\ell_{a,e})$ is an intrinsic $L$-Lipschitz graph  of the function 
\begin{displaymath}
\psi:\mathbb{V}_e \to \mathbb{W}_e,\quad \psi(v)=\pi_{\mathbb{W}_e}(\Phi(\ell_{a,e}(v))),
\end{displaymath}
where $\mathbb{V}_e:= V_e\times \{0\}$.
By Theorem \ref{t:GlemStratBetanILGLinftyk=2} (or,  rather, its parametric version), we  have 
\begin{equation}\label{eq:1DGlemRestrLine}
    \sum_{Q\subset Q_0,CQ\cap \ell \neq \emptyset} \beta_{\infty}(CQ,\ell)^4 \mathrm{diam}(Q)\lesssim_{C,L,k}\mathrm{diam}(Q_0),\quad Q_0\in \mathcal{D}^k
\end{equation}
with implicit constant independent of $Q_0$ and $\ell$. Let us now explain this argument in detail.

First, let $\psi:\V_e\to\W_e$ be as above. This is a 1-dimensional intrinsic $L$-Lipschitz map, so we can apply Theorem \ref{t:ControlStratifBetaByOmegaInftyk=1}. Actually, we need something more precise. Fix $Q,Q_0\in\mathcal D^k, Q\subset Q_0$ and assume $CQ\cap \ell\neq \emptyset$. 
Let $r:=\mathrm{diam}(CQ)$ and let $p$ denote any point in $CQ\cap\ell$. By the proof of Theorem \ref{t:ControlStratifBetaByOmegaInftyk=1}, for any $\varepsilon>0$ sufficiently small, we can find $A_e=(\Pi^\perp(A_e), A_{e,2n}):\V_e\to \W_e$ intrinsic affine such that 
\begin{equation*}I_\pi(\Pi^\perp(A_e)):=\sup_{v\in B_{\mathbb{V}_e}(\ell_{a,e}^{-1}(p),r)} \frac{|\Pi^{\bot}(\psi)(v)-\Pi^\perp(A_e)(v)|}{r}<\Omega_{\infty,\Pi^{\bot}(\psi)}(\ell_{a,e}^{-1}(p),r)+\varepsilon. \end{equation*}
Here, if $\W_e=W_e\times\R$ with $W_e\subset\R^{2n}$, we are denoting by $\Pi^\perp$ the standard projection on $W_e$, recall \eqref{eq:PiPerp1D}.
Since, in the supremum, we consider the Euclidean norm, we can assume without loss of generality that the first $k-1$ components of $A_e$ (by this we mean the coordinates of $A_e$ associated to $\{e_1^\perp, \dots, e_{k-1}^\perp\}$) are constant and coincide with the ones of $\psi$. Indeed, up to modifying $A_{e,2n}$ correspondingly, $A_e$ remains an intrinsic affine map and the quantity defining $I_\pi(\Pi^\perp(A_e))$ decreases. Hence, we can assume that the first $k-1$ components of $A_e$ are the same of those of $\psi$. This implies that $\Pi(\mathbb A_e(v))=\ell_{a,e}(v)$, where $\mathbb A_e$ denotes the graph map of $A_e$ and $\Pi$ is defined as in \eqref{projtoRk}. We now use Lemma \ref{l:1DSlicesAffine} to associate to $A_e$ an intrinsic affine map $ A:\V_0\to\W_0$ such that $\mathbb A_e(v)={\mathbb A}(\Pi(\mathbb A_e(v)))={\mathbb A}(\ell _{a,e}(v))$ for every $v\in\V_e$.
Let also $\Phi$ and $\Psi$ denote the graph maps of $\varphi$ and $\psi$ respectively. Then, if $x=\ell_{a,e}(v)$, \begin{equation}\label{section}d(\varphi(x), A(x))=d(\Phi(x), {\mathbb A}(x))= d(\Phi(\ell_{a,e}(v)), \mathbb A(\ell_{a,e}(v)))=d(\Psi(v),\mathbb A_e(v)).
\end{equation} By the proof of Theorem \ref{t:ControlStratifBetaByOmegaInftyk=1}, it follows that $\frac{d(\Psi(v),\mathbb A_e(v))}{r}\lesssim_L (\Omega_{\infty, \Pi^\perp(\psi)}(\ell_{a,e}^{-1}(p),r)+\varepsilon)^{1/2}$ for every $v\in B_{\mathbb V_e}(\ell_{a,e}^{-1}(p),r)$. Hence, by \eqref{section}, 
\[\sup_{x\in CQ\cap \ell_{a,e}}\frac{d(\varphi(x), A(x))}{\mathrm{diam}(CQ)}
\leq\sup_{v\in B_{\mathbb V_e}(\ell_{a,e}^{-1}(p),r)}\frac{d(\Psi(v),\mathbb A_e(v))}{r}\lesssim_L (\Omega_{\infty, \Pi^\perp(\psi)}(\ell_{a,e}^{-1}(p),r)+\varepsilon)^{1/2}.\]
Passing to the infimum among all intrinsic affine maps $ A:\V_0\to\W_0$, we get that
\[\beta_{\infty}(CQ,\ell_{a,e})= \inf_{ A}\sup_{x\in CQ\cap \ell_{a,e}}\frac{d(\varphi(x), A(x))}{\mathrm{diam}(CQ)}\lesssim_L (\Omega_{\infty, \Pi^\perp(\psi)}(\ell_{a,e}^{-1}(p),r))^{1/2}. \]
As a consequence, $\beta_\infty(CQ, \ell)^4\lesssim_L (\Omega_{\infty, \Pi^\perp(\psi)}(\ell_{a,e}^{-1}(p),\mathrm{diam}(CQ)))^{2}$ and \eqref{eq:1DGlemRestrLine} will follow by Theorem \ref{t:DorronsoroLinftyk=1} in the corresponding discrete version for dyadic cubes. Indeed, let $\overline{\mathcal D}^1=\cup_j\overline{\mathcal D_j}^1$ be a standard family of dyadic cubes on $\V_e$, induced by an isometric identification $\V_e\equiv \R$. For every $Q\in \mathcal D_j^k$ with $CQ\cap \ell\neq\emptyset$, there exists $\overline Q_Q\in\overline{\mathcal D}_j^1$ such that $B_{\V_e}(\ell_{a,e}^{-1}(p),\mathrm{diam}(CQ))\subseteq 3\sqrt kC\overline Q_Q$. Moreover, for any $\overline Q\in\overline{\mathcal D}_j^1$, the number of cubes $Q\in\mathcal D_j^k$ with $CQ\cap \ell\neq\emptyset$ and $B_{\V_e}(\ell_{a,e}^{-1}(p), \mathrm{diam}(CQ))\subseteq 3\sqrt kC\overline Q$ is bounded by a geometric constant $N$ depending only on $k$ and $C$. This allows us to estimate
\begin{align*}
   \sum_{Q\subset Q_0,CQ\cap \ell \neq \emptyset} \beta_{\infty}(CQ,\ell)^4 \mathrm{diam}(Q)&\lesssim_L \sum_{Q\subset Q_0,CQ\cap \ell \neq \emptyset} (\Omega_{\infty, \Pi^\perp(\psi)}(\ell_{a,e}^{-1}(p),\mathrm{diam}(CQ)))^{2}\mathrm{diam}(Q)\\ &\lesssim_{k,L}  \sum_{Q\subset Q_0,CQ\cap \ell \neq \emptyset} (\Omega_{\infty, \Pi^\perp(\psi)}(3\sqrt k C\overline Q_Q))^{2}\mathrm{diam}(\overline Q_Q)\\&\lesssim_{C,k,L}  \sum_{\overline Q\in \overline{\mathcal D}^1, \overline Q\subset M\overline Q_0} (\Omega_{\infty, \Pi^\perp(\psi)}(3\sqrt k C\overline Q))^{2}\mathrm{diam}(\overline Q)\\ &\lesssim_{C,k,L}\mathrm{diam}(\overline Q_0)\leq\mathrm{diam}(Q_0),\quad Q_0\in \mathcal{D}^k. 
\end{align*}
Here $\overline Q_0\in \overline{\mathcal D}^1$ is the cube associated to $Q_0\in\mathcal D^k$ as above and $M$ is a constant depending only on $k$. 
This completes the proof of \eqref{eq:1DGlemRestrLine}.

For fixed $Q_0\in \mathcal{D}^k$, we can now estimate as follows
\begin{align*}
    \sum_{Q\in \mathcal{D}^k,Q\subset Q_0}\beta_{\infty,4}^1(CQ)^4 \mathcal{H}^k(Q) &= \sum_{Q\in \mathcal{D}^k,Q\subset Q_0}\left[\fint_{\mathcal{A}_1(CQ)}\beta_{\infty}(CQ,\ell)^4\,d\eta_1(\ell)\right]\mathcal{H}^k(Q)\\
    &\lesssim_{C,k} \int_{\mathcal{A}_1(CQ_0)}\sum_{Q\subset Q_0,CQ\cap \ell \neq \emptyset}\beta_{\infty}(CQ,\ell)^4 \mathrm{diam}(Q)\,d\eta_1(\ell),
\end{align*}
where we have used $\eta_1(\mathcal{A}_1(CQ))\sim_{C,k} \mathrm{diam}(Q)^{k-1}$ (see the proof of Lemma 2.5 in \cite{MR4345824}) in the last step. The lemma then follows from \eqref{eq:1DGlemRestrLine}.
\end{proof}

The following serves as a counterpart for \cite[Lemma 2.6]{MR4345824}, and it is a $1$-codimensional version of the last argument in the proof of Lemma \ref{l:TLem2.5}.

\begin{lemma}[$(k-1)$-dim.\  Dorronsoro implies $\mathrm{GLem}( \beta_{p,q}^{k-1},q)$ for $k$-dimensional graphs]\label{l:HigherDimInduc} Let $1\leq p\leq \infty$ and $1\leq q<\infty$.
    Let $n\in \mathbb{N}$ and $1<k \le n$. Suppose that, for any $M\geq 1$, there is a constant $C=C_{k,n,p,q,L,M}>0$ such that every intrinsic $L$-Lipschitz function $\psi:\mathbb{R}^{k-1}\times \{0\}\to \{0\}\times \mathbb{R}^{2n+2-k}$ satisfies 
 \begin{displaymath}
         \sum_{Q\in \mathcal{D}^{k-1},Q\subset Q_0}\beta_{p}(MQ)^q \mathcal{H}^{k-1}(Q)\leq C \mathcal{H}^{k-1}(Q_0),\quad Q_0 \in \mathcal{D}^{k-1}.
    \end{displaymath}
Then every intrinsic $L$-Lipschitz function $\varphi:\mathbb{V}_0=\mathbb{R}^k \times \{0\}\to \W_0=\{0\}\times \mathbb{R}^{2n+1-k}$ satisfies, for any $K\geq 1$,
    \begin{displaymath}
         \sum_{Q\in \mathcal{D}^{k},Q\subset Q_0}\beta_{p,q}^{k-1}(KQ)^q \mathcal{H}^k(Q)\leq C' \mathcal{H}^k(Q_0),\quad Q_0 \in \mathcal{D}^k,
    \end{displaymath}
    with $C'$ depending only on $k,n,p,q,L,K$.
\end{lemma}

\begin{proof} Fix $K\geq 1$ and $Q_0\in \mathcal D^k$.
    We compute
    \begin{align}\label{eq:SlicingGlem}
        \sum_{Q\in \mathcal{D}^k,Q\subset Q_0}\beta_{p,q}^{k-1}(KQ)^q &\mathcal{H}^k(Q)=
         \sum_{Q\in \mathcal{D}^k,Q\subset Q_0}\left[\fint_{\mathcal{A}_{k-1}(KQ)}\beta_p(KQ,V)^q\,d\eta_{k-1}(V)\right]\mathcal{H}^k(Q)\notag\\
        & \sim_{k,K}  \sum_{Q\in \mathcal{D}^k,Q\subset Q_0}\left[\int_{\mathcal{A}_{k-1}(KQ)}\beta_p(KQ,V)^q\,d\eta_{k-1}(V)\right]\mathrm{diam}(Q)^{k-1}\notag\\
         &\leq \int_{\mathcal{A}_{k-1}(KQ_0)}\left[\sum_{Q\in \mathcal{D}^k,Q\subset Q_0,KQ\cap V\neq \emptyset}\beta_p(KQ,V)^q\mathrm{diam}
(Q)^{k-1} \right]\,d\eta_{k-1}(V).   \end{align}
At this point we want to apply the assumption for $(k-1)$-dimensional intrinsic Lipschitz graphs. To this end, we need essentially Lemma \ref{l:SlicingLemmak-1} and Lemma \ref{l:k-1SlicesAffine}. The idea is to consider ``$(k-1)$-dimensional slices $\psi$'' of the given $k$-dimensional intrinsic Lipschitz function $\varphi:\V_0 \to \W_0$.
This would immediately yield the desired conclusion if we knew that the intrinsic affine function $A_e$ minimizing the $\beta$ numbers associated to the slicing map $\psi$ from Lemma \ref{l:SlicingLemmak-1}  coincided with $\psi$ in the $e^{\bot}$-component, which is constant. However, due to the constraint for graph isotropic functions on more than $1$-dimensional domains, it is not clear to us whether this can be arranged in general. For this reason, our argument is slightly more complicated and passes via comparison with  an auxiliary function; see \eqref{eq:AuxFunct} below.

For a fixed $V\in \mathcal A_{k-1}(KQ_0)$, we will now estimate the expression inside the integral in \eqref{eq:SlicingGlem}. To this end, we assume that $V$ is 
parametrized by $\pi_{a,e}$ as in \eqref{eq:AffParam}, with $e=\{e_1,\dots,e_{k-1}\}\subset S^{k-1}$ an orthonormal orientation of $V$ and $a\in\R$.
By Lemma \ref{l:SlicingLemmak-1}, using the same notation, there exists an intrinsic $L$-Lipschitz map $\psi:\mathbb V_e\to\mathbb W_e$ such that  $\Psi(v)=\Phi(\pi_{a,e}(v))$ for every $v\in \V_e$. Let $\overline{ \mathcal D}^{k-1}=\cup_j \overline {\mathcal D_j}^{k-1}$ be a standard family of dyadic cubes on $\mathbb V_e\equiv \R^{k-1}$. For every $Q\in\mathcal D_j^k$ with $Q\subset Q_0$ and $KQ\cap V\neq \emptyset$, fix $p_Q\in KQ\cap V$. Then we choose $\overline{ Q}_Q\in\overline{\mathcal D_j}^{k-1}$ such that $B_{\V_e}(\pi_{a,e}^{-1}(p_Q), \mathrm{diam}(KQ))\subset 3\sqrt k K\overline{Q}_Q$. Notice that, for any $\overline Q\in\overline{\mathcal D}_j^{k-1}$, the number of cubes $Q\in\mathcal D_j^k$ with $KQ\cap V\neq\emptyset$ and $B_{\V_e}(\pi_{a,e}^{-1}(p_Q), \mathrm{diam}(KQ))\subset 3\sqrt{k}K\overline Q$ is bounded by a constant $N$ depending only on $k$ and $K$. From now on, we assume $1\leq p<+\infty$; the case $p=+\infty$ can be treated analogously, as explained below.
Let $\varepsilon>0$ and let $A_{\varepsilon,e}:\mathbb V_e\to\mathbb W_e$ be an intrinsic affine map such that 
\begin{equation}\label{choiceofAepse}
\frac{1}{\mathrm{diam}(3\sqrt kK\overline Q_Q)^{k-1}}\int_{3\sqrt kK\overline Q_Q}\left(\frac{d(\psi(v),A_{\varepsilon,e}(v))}{\mathrm{diam}(3\sqrt kK\overline Q_Q)}\right)^p\,d\mathcal H^{k-1}(v)\leq \beta_p^p(3\sqrt kK\overline Q_Q)+\varepsilon,\end{equation}
where $\beta_p$ are the $\beta$-numbers associated to $\psi:\V_e\to\W_e$.
By Lemma \ref{l:k-1SlicesAffine}, there exists an intrinsic affine map $A_\varepsilon:\V_0\to\W_0$ such that $\mathbb A_{\varepsilon,e}(v)=\mathbb A_\varepsilon(\Pi(\mathbb A_{\varepsilon,e}(v)))$ for every $v\in \mathbb V_e$ and such that $\mathrm{Lip}(\mathbb A_\varepsilon)\lesssim_{k,\mathrm{Lip}(\mathbb A_{\varepsilon,e})}1$. Hence, by the triangle inequality, for every $v\in \V_e$,
\begin{align}\label{eq:AuxFunct} d(\Phi(\pi_{a,e}(v)), \mathbb{A}_\varepsilon(\pi_{a,e}(v))&\leq
  d(\Phi(\pi_{a,e}(v)), \mathbb{A}_\varepsilon(\Pi(\mathbb{A}_{\varepsilon,e}(v))))+
d(\mathbb{A}_\varepsilon(\Pi(\mathbb{A}_{\varepsilon,e}(v))),\mathbb{A}_\varepsilon(\pi_{a,e}(v)))\notag\\
  &=d(\Psi(v),\mathbb{A}_{\varepsilon,e}(v))+ d(\mathbb{A}_\varepsilon(\Pi(\mathbb{A}_{\varepsilon,e}(v))),\mathbb{A}_\varepsilon(\pi_{a,e}(v))).\end{align}
  Notice that $d(\Psi(v),\mathbb A_{\varepsilon,e}(v))=d(\psi(v), A_{\varepsilon,e}(v))$, while
  \begin{align}\label{eq:estimatesecondterm}
d(\mathbb{A}_\varepsilon(\Pi(\mathbb{A}_{\varepsilon,e}(v))),\mathbb{A}_\varepsilon(\pi_{a,e}(v))) &\leq \mathrm{Lip}(\mathbb A_\varepsilon)\, d(\Pi(\mathbb A_{\varepsilon,e}(v)), \pi_{a,e}(v))\notag\\ &=  \mathrm{Lip}(\mathbb A_\varepsilon)\, d(\Pi(\mathbb A_{\varepsilon,e}(v)), \Pi(\Psi(v)))\notag\\ &\leq \mathrm{Lip}(\mathbb A_\varepsilon)\, d(\mathbb A_{\varepsilon,e}(v), \Psi(v))\notag\\ &=\mathrm{Lip}(\mathbb A_\varepsilon)\, d( A_{\varepsilon,e}(v), \psi(v)).
  \end{align}
  Here we exploited the fact that \[\Pi(\Psi(v))=\Pi(\Phi(\pi_{a,e}(v)))=\Pi(\pi_{a,e}(v)\cdot\varphi(\pi_{a,e}(v)))=\pi_{a,e}(v).\]
  Recall, by Lemma \ref{l:k-1SlicesAffine}, that the Lipschitz constant of $\mathbb A_\varepsilon$ is controlled by the one of $\mathbb A_{\varepsilon,e}$. 

  We want now to bound $\mathrm{Lip}(\mathbb A_{\varepsilon,e})$ by $\mathrm{Lip}(\Psi)$. First, if the intrinsic Lipschitz constant of $\psi$ is $\mathrm{Lip}(\psi)=0$, then $\psi$ is intrinsic affine and we can take $A_{\varepsilon,e}\equiv \psi$, and thus $\mathrm{Lip}(\mathbb A_{\varepsilon,e})=\mathrm{Lip}(\Psi)$.   Let us therefore assume that $\mathrm{Lip}(\psi)>0$ and let us
  now prove that $\mathrm{Lip}(\mathbb A_{\varepsilon,e})\lesssim_{k,n,p,K} \mathrm{Lip}(\Psi)$, provided that $\varepsilon$ is chosen small enough, for instance if $\varepsilon\leq (\mathrm{Lip}(\psi))^p$. Recall that we have chosen $A_{\varepsilon,e}$ in such a way that \eqref{choiceofAepse} holds. By Remark \ref{rmk:boundedness}, we know that $\beta_p(3\sqrt k K\overline Q_Q)\leq \mathrm{Lip}(\psi)$.  Hence,
  \begin{equation*}
\frac{1}{\mathrm{diam}(3\sqrt kK\overline Q_Q)^{k-1}}\int_{3\sqrt k K\overline Q_Q}\left(\frac{d(\psi(v),A_{\varepsilon,e}(v))}{\mathrm{diam}(3\sqrt kK\overline Q_Q)}\right)^p\,d\mathcal H^{k-1}(v)\leq 2(\mathrm{Lip}(\psi))^p.\end{equation*}
We want now to argue that the same conclusion holds with $\mathrm{Lip}(\psi)$ replaced by $\mathrm{Lip}(\Psi)$, up to a geometric constant $C=C(k,n)$, see \eqref{controlLip} below. In fact, for every $v,v'\in \mathbb V_e$, it holds
\begin{align}\label{lippsi/Psi}\|v'^{-1} \cdot v \cdot \psi(v)^{-1}&\cdot v^{-1} \cdot v' \cdot \psi(v')\|=\|(\Psi(v)^{-1}\cdot\Psi(v'))_{\W_e}\|\notag\\ &\leq C\|\Psi(v)^{-1}\cdot\Psi(v')\|\leq C\,\mathrm{Lip}(\Psi)\|v^{-1}\cdot v'\|.\end{align}
Here we are using the vertical projection on $\W_e$, defined by $p= p_{\V_e}\cdot p_{\W_e}$ with $p_{\V_e}\in\V_e$ and $p_{\W_e}\in\W_e$. Specifically,
\begin{equation*}\Psi(v)^{-1}\cdot\Psi(v')=\underbrace{v^{-1}\cdot v'}_{\in \V_e}\cdot \underbrace{v'^{-1} \cdot v \cdot \psi(v)^{-1}\cdot v^{-1} \cdot v' \cdot \psi(v')}_{\in\W_e} \end{equation*}
The first inequality follows instead by \cite[Proposition 4.8]{MR3587666}. There, the constant $C$ depends a priori on the fixed splitting $\mathbb H^n=\V_e\cdot \W_e$. However, in our arguments, $\V_e$ is always a $(k-1)$-dimensional horizontal subspace with orthogonal complement $\W_e$ in $\R^{2n+1}$. One can then use the rotations from Section \ref{ss:RedStd} to show that $C=C_e$ is independent of $e$ for such splittings. More precisely, for any two choices $e$ and $e'$, there exists $U\in U(n)$ such that $R_U(\V_e)=\V_{e'}$, and consequently also $R_U(\W_e)=\W_{e'}$. Then $p_{\V_{e'}}=R_U([R_U^{-1}(p)]_{\V_e})$ and  $p_{\W_{e'}}=R_U([R_U^{-1}(p)]_{\W_e})$ since these points belong to $\V_{e'}$ and $\W_{e'}$, respectively, and their product equals
\begin{displaymath}
R_U([R_U^{-1}(p)]_{\V_e})\cdot R_U([R_U^{-1}(p)]_{\W_e})
= R_U([R_U^{-1}(p)]_{\V_e}\cdot [R_U^{-1}(p)]_{\W_e})
= R_U(R_U^{-1}(p))=p.
\end{displaymath}
Thus, since $R_U$ is an isometry,
\begin{displaymath}
    \|p_{\W_{e'}}\|=\|R_U([R_U^{-1}(p)]_{\W_e})\|=
    \|[R_U^{-1}(p)]_{\W_e}\|\leq C_e \|R_U^{-1}(p)\|=C_e\|p\|,
\end{displaymath}
which shows that we can take $C_{e'}=C_e$. Inequality \eqref{lippsi/Psi} and Proposition \ref{1.2} show that $\mathrm{Lip}(\psi)\leq C\mathrm{Lip}(\Psi)$. Therefore, using the left invariance of the distance $d$, we deduce that 
\begin{equation}\label{controlLip}
\frac{1}{\mathrm{diam}(3\sqrt kK\overline Q_Q)^{k-1}}\int_{3\sqrt kK\overline Q_Q}\left(\frac{d(\Psi(v),\mathbb A_{\varepsilon,e}(v))}{\mathrm{diam}(3\sqrt kK\overline Q_Q)}\right)^p\,d\mathcal H^{k-1}(v)\leq 2C^p(\mathrm{Lip}(\Psi))^p.\end{equation}
Our goal is now to use a variation of Lemma \ref{lemma_lipconst} in order to conclude that $\mathrm{Lip}(\mathbb A_{\varepsilon, e})\lesssim_{k,n,p,K} \mathrm{Lip}(\Psi)$. 
The argument is quite similar, so we do not repeat it. We only focus on the suitable modifications. First, the exponent $2$ in that lemma can be replaced by any $p\geq 1$ by considering in the proof the $p-$th root instead of the square root. Moreover, the coefficient 2 appearing in the right-hand side of Lemma \ref{lemma_lipconst} can be replaced by any positive number $A>0$ (in our case $2C^p$) by replacing the constant $4$ appearing in the proof under the square root with $2A$. Similarly, also the fact that in Lemma \ref{lemma_lipconst} we consider balls $B(x,r)$ is not crucial. They can be replaced by any other family of sets, each of which contains a cube of comparable measure and diameter (in particular it applies to $3\sqrt k K\overline Q_Q$).

The main difference is that in \eqref{controlLip} we consider $\mathbb H^n$-valued maps, while in the original formulation of Lemma \ref{lemma_lipconst} we deal with Euclidean values. In particular, the Lipschitz function $f:\R^k\to\R^m$ is replaced by the Lipschitz map $\Psi: \V_e\equiv \R^{k-1}\to \mathbb H^n$. Since the latter is still a metric Lipschitz map, then all the (metric) estimates appearing in the original proof work also here without modifications. The last variation with respect to the original statement of Lemma \ref{lemma_lipconst} is that the affine map $A:\mathbb{R}^k\to\R^m$ is replaced by the graph map $\mathbb A_{\varepsilon,e}:\V_e\equiv \R^{k-1}\to \mathbb H^n$. However, this map still shares similar properties: there exists a direction $\bar e\in \mathbb S^{k-2}$ of maximal stretch for $\mathbb A_{\varepsilon, e}$, meaning that \begin{equation}\label{maximalstretch}d(\mathbb A_{\varepsilon,e}(y),\mathbb A_{\varepsilon,e}(y+h\bar e))=\mathrm {Lip}(\mathbb A_{\varepsilon,e})|h|\end{equation} for every $y\in\R^{k-1}$, $h\in\R$ (here we are identifying $\V_e$ with $\R^{k-1}$ by using an orthonormal basis). The fact that, for fixed $\bar e\in\mathbb S^{k-2}$, the equality does not depend on $y\in \R^{k-1}$ and $h\in \R$ follows from the property that $A_{\varepsilon,e}$ is intrinsic affine: this implies that $\mathbb A_{\varepsilon,e}(\V_e)$ is an affine $k-1$-dimensional horizontal plane, so that $d(\mathbb A_{\varepsilon,e}(y),\mathbb A_{\varepsilon,e}(y+h\bar e))=d_{\mathrm{Eucl}}(\pi(\mathbb A_{\varepsilon,e}(y)),\pi(\mathbb A_{\varepsilon,e}(y+h\bar e)))$, where $\pi:\mathbb H^n\to\R^{2n}$ denotes the standard coordinate projection. Then we can use the fact that $\pi\circ \mathbb A_{\varepsilon,e}$ is a standard affine map (see also Proposition \ref{ahp} 2.), so that the ratio 
 \[\frac{d_{\mathrm{Eucl}}(\pi(\mathbb A_{\varepsilon,e}(y)),\pi(\mathbb A_{\varepsilon,e}(y+h\bar e)))}{|h|}\]
 is constant, for fixed $\bar e\in\mathbb S^{k-2}$. By similar arguments, also $\mathrm{Lip}(\mathbb A_{\varepsilon,e})=\mathrm{Lip}_{\mathrm {Eucl}}(\pi\circ\mathbb A_{\varepsilon, e})$ and the conclusion for $\mathbb A_{\varepsilon,e}$ follows from the corresponding one for $\pi\circ \mathbb A_{\varepsilon,e}$.
 
With these modifications, the proof follows the same lines as the original one (by replacing the Euclidean distance with the Kor\'{a}nyi metric). In particular, the constant $N$ obtained by this strategy depends not only on $k$, but also on $n,p,K$. Hence, $\mathrm{Lip}(\mathbb A_{\varepsilon,e})\lesssim_{k,n,p,K} \mathrm{Lip}(\Psi)$.

By construction, it also holds that $\mathrm{Lip}(\Psi)\leq \mathrm{Lip}(\Phi)$ since, for every $v,v'\in\V_e$,
\[d(\Psi(v),\Psi(v'))=d(\Phi(\pi_{a,e}(v)), \Phi(\pi_{a,e}(v')))\leq \mathrm{Lip}(\Phi)|\pi_{a,e}(v)-\pi_{a,e}(v')|=\mathrm{Lip}(\Phi)|v-v'|.\] Finally, recalling that $\varphi:\V_0\to\W_0$ is intrinsic $L$-Lipschitz, Proposition \ref{1.2} gives 
\[\|\Phi(v)^{-1}\Phi(v')\|\leq \|v^{-1}\cdot v'\|+\|v'^{-1} \cdot v \cdot \varphi(v)^{-1}\cdot v^{-1} \cdot v' \cdot \varphi(v')\|\leq (L+1)\|v^{-1}\cdot v'\|.\] This implies $\mathrm{Lip}(\Phi)\leq L+1$. Combining all the previous estimates on the Lipschitz constants, we get (using also the explicit expression hidden in $\mathrm{Lip}(\mathbb A_\varepsilon)\lesssim_{k,\mathrm{Lip}(\mathbb A_{\varepsilon,e})}1$, given by the proof of Lemma \ref{l:k-1SlicesAffine}) that $\mathrm{Lip}(\mathbb A_{\varepsilon})\lesssim_{k,n,p,L,K} 1$.
We now conclude as follows.
By choice of $\overline Q_Q$, we have
\[\pi_{a,e}^{-1}(KQ\cap V)\subset B_{\V_e}(\pi_{a,e}^{-1}(p_Q), \mathrm{diam}(KQ))\subset 3\sqrt kK\overline Q_Q.\] Hence,
\begin{align*}\beta_p^p(KQ,V)&\leq \frac{1}{\mathrm{diam}(KQ)^{k-1}}\int_{KQ\cap V} \left(\frac{d(\varphi(x),A_\varepsilon(x))}{\mathrm{diam}(KQ)}\right)^p d\mathcal{H}^{k-1}(x)\\ &=\frac{1}{\mathrm{diam}(KQ)^{k-1}}\int_{KQ\cap V} \left(\frac{d(\Phi(x),\mathbb A_\varepsilon(x))}{\mathrm{diam}(KQ)}\right)^p d\mathcal{H}^{k-1}(x)\\ &=\frac{1}{\mathrm{diam}(KQ)^{k-1}}\int_{\pi_{a,e}^{-1}(KQ\cap V)} \left(\frac{d(\Phi(\pi_{a,e}(v)),\mathbb A_\varepsilon(\pi_{a,e}(v)))}{\mathrm{diam}(KQ)}\right)^p d\mathcal{H}^{k-1}(v)\\ &\lesssim_{k,p}\frac{1}{\mathrm{diam}(3\sqrt kK\overline Q_Q)^{k-1}}\int_{3\sqrt kK\overline Q_Q} \left(\frac{d(\Phi(\pi_{a,e}(v)),\mathbb A_\varepsilon(\pi_{a,e}(v)))}{\mathrm{diam}(3\sqrt k K\overline Q_Q)}\right)^p d\mathcal{H}^{k-1}(v)\\ &\stackrel{\eqref{eq:AuxFunct},\eqref{eq:estimatesecondterm}}\lesssim_{k,n,p,L,K}\frac{1}{\mathrm{diam}(3\sqrt kK\overline Q_Q)^{k-1}}\int_{3\sqrt kK\overline Q_Q} \left(\frac{d(\psi(v), A_{\varepsilon,e}(v))}{\mathrm{diam}(3\sqrt k K\overline Q_Q)}\right)^p d\mathcal{H}^{k-1}(v)\\ &\stackrel{\eqref{choiceofAepse}}\leq \beta_p^p(3\sqrt k K\overline Q_Q)+\varepsilon.\end{align*}
By arbitrariness of $\varepsilon$, we get that $\beta_p(KQ,V)\lesssim_{k,n,p,L,K} \beta_p (3\sqrt kK\overline Q_Q)$, where the latter coefficients are associated to the function $\psi$. 

Let us just mention the main modifications for the case $p=+\infty. $ In that case, for $0<\varepsilon\leq \mathrm{Lip}(\psi)$, choose $A_{\varepsilon, e}:\V_e\to\W_e$ such that
\[\sup_{v\in 3\sqrt k K\overline Q_Q} \frac{d(\psi(v),A_{\varepsilon,e}(v))}{\mathrm{diam}(3\sqrt kK\overline Q_Q)}\leq \beta_\infty(3\sqrt kK\overline Q_Q)+\varepsilon.\]
By proceeding exactly as above, one gets
\[\sup_{v\in 3\sqrt k K\overline Q_Q} \frac{d(\Psi(v),\mathbb A_{\varepsilon,e}(v))}{\mathrm{diam}(3\sqrt kK\overline Q_Q)}\leq 2C \mathrm{Lip}(\Psi).\]
This implies that $\mathrm{Lip}(\mathbb A_{\varepsilon,e})\lesssim_{k,n,K}\mathrm{Lip}(\Psi)$. In fact, if $\bar e\in \mathbb S^{k-2}$ is a direction of maximal stretch for $\mathbb A_{\varepsilon, e}$ as above, by considering  $z=y+h\bar e$, with $y\in 3\sqrt k K\overline Q_Q$, $h\sim \mathrm{diam}(\overline Q _Q)$ and  $y+h\bar e\in 3\sqrt k K\overline Q_Q$, then 
\[\begin{split}\mathrm{Lip}(\mathbb A_{\varepsilon,e})|z-y|&=d(\mathbb A_{\varepsilon,e}(z),\mathbb A_{\varepsilon,e}(y))\leq d(\mathbb A_{\varepsilon,e}(z), \Psi(z))+d(\Psi(z), \Psi(y))+d(\Psi(y), \mathbb A_{\varepsilon,e}(y))\\ &\leq \mathrm{Lip}(\Psi)|z-y|+4C\mathrm{Lip}(\Psi)\mathrm{diam}(3\sqrt kK \overline Q_Q) \sim_{k, n, K} \mathrm{Lip}(\Psi)|z-y|.\end{split}\]
Hence, $\mathrm{Lip}(\mathbb A_{\varepsilon,e})\lesssim_{k,n,K}\mathrm{Lip}(\Psi)$ and, exactly as above, this implies $\mathrm{Lip}(\mathbb A_\varepsilon)\lesssim_{k,n,L,K}1$. Concerning the final estimate $\beta_\infty(KQ,V)\lesssim_{k,n,L,K}\beta_\infty(3\sqrt k K \overline Q_Q)$ it is enough to run the same computation as above, but using the definitions for the infinity-based $\beta$-numbers.

Coming back to \eqref{eq:SlicingGlem}, we can now estimate
\begin{align*}
   \sum_{Q\in \mathcal{D}^k,Q\subset Q_0,KQ\cap V\neq \emptyset}\beta_p(KQ,V)^q\mathrm{diam}
(Q)^{k-1}& \lesssim_{k,n,p,L,K} \sum_{\overline Q\in\overline{\mathcal D}^{k-1}, \overline Q\subset M\overline Q_{Q_0}}\beta_p(3\sqrt kK\overline Q)^q\mathrm{diam}(\overline Q)^{k-1}\\ &\lesssim_{k,n,p,q,L,K} \mathrm{diam}(\overline Q_{Q_0})^{k-1}\leq \mathrm{diam}(Q_0)^{k-1},
\end{align*}
where we applied our inductive hypothesis on $\psi$ and proceeded as in  Lemma \ref{l:TLem2.5}. Substituting into the final estimate of \eqref{eq:SlicingGlem},
we get the conclusion.
\end{proof}

In Lemma \ref{l:IntGeomBdd} below, we will prove for $k$-dimensional intrinsic Lipschitz graphs an estimate of the form
\begin{equation}\label{eq:GoalBetaBound}
    \beta_2(cQ)^{8q}\lesssim \beta_{\infty,q}^{k-1}(Q)^q+\beta_{\infty,4}^1(Q)^4,
\end{equation}
for cubes in $\mathbb{R}^k$.
This will yield a geometric lemma for such graphs since the terms on the right-hand side can be controlled by Lemma \ref{l:TLem2.5} and \ref{l:HigherDimInduc} respectively. 

To prove \eqref{eq:GoalBetaBound},  we will study, for each $Q\in \mathcal{D}^k$, the intrinsic Lipschitz function $\varphi$ on a simplex $\triangle$ associated to $Q$. Then we will use the values of its graph map $\Phi$ at the vertex points of $\triangle$ to build a well-approximating intrinsic affine function $A_{\mathbb{V}}$ with the help of Lemma \ref{l:orp21,2.15}. A crucial idea,
borrowed from \cite{MR4345824}, is that $A_{\mathbb{V}}$ should be well approximating because $\Phi$ is well approximated by intrinsic affine functions along the $(k-1)$-dimensional faces of $\triangle$ (and along line segments inside $\triangle$).  While intuitively plausible, this requires some work since we only have information about the integral-geometric $\beta_{\infty,q}^{k-1}$ numbers, rather than approximation along any fixed $(k-1)$-dimensional plane. This requires modifying $\triangle$ so that its faces are in a generic position, which is the content of Lemma \ref{l:orp21,2.15} below. To state it precisely, we recall \cite[Definition 2.12]{MR4345824}:
\begin{definition}
    A family $\mathcal{V}\subset A(k,k-1)$ is said to be \emph{$\tau$-transversal}, for some $\tau>0$, if for any selection of pairwise distinct planes $V_1,\ldots,V_k\in \mathcal{V}$, the determinant of the normal vectors to $V_1,\ldots,V_k$ is at least $\tau$.
\end{definition}
Transversal families have useful geometric properties that will be applied in the proof of Lemma \ref{l:orp21,2.15}. In particular, if to each plane in a transversal family we associate a nearby plane, then the new family will also be transversal, quantitatively. The notion of ``nearby'' is made precise with the help of the following metric $\tilde{d}$ on $A(k,k-1)$.

 Any $V \in A(k,k-1)$ can be written as $V = \{x : x \cdot e = t\}$, where $e \in S^{k-1}\subset\R^k\cong\V_0$ is normal to $V$, and $t \in \mathbb{R}$. The pair $(e,t)$ is unique up to sign. If $V_1, V_2$ are then associated to $(e_1,t_1)$ and $(e_2,t_2)$, respectively, write
\begin{equation}\label{eq:DistPlane}
\tilde{d}(V_1, V_2) := \min\{|(e_1, t_1) - (e_2, t_2)|, |(e_1, t_1) + (e_2, t_2)|\},
\end{equation}
where $|\cdot|$ refers to the Euclidean metric on $S^{k-1} \times \mathbb{R} \subset \mathbb{R}^{k+1}$.

The following is an ($L^{\infty}$-based) counterpart of \cite[Lemma 2.15]{MR4345824}. Since the domain $\mathbb{V}_0$ is isometric to Euclidean $\mathbb{R}^k$, many of the arguments can be used in our case without changes. The final part of the argument is simpler in our situation since we only need to consider $\beta_{\infty}$ instead of $\beta_2$. On the other hand, the exponent $2$ appearing in $\beta_{\infty,2}^{k-1}$ can also be replaced by other exponents. We state Lemma \ref{l:orp21,2.15} only for the $L^2$-based integral geometric numbers in order to keep the formulation close to the original one. This is sufficient for our purposes and does not require any additional difficulty in the proof.

\begin{lemma}\label{l:orp21,2.15} 
    Let $Q=[0,1]^k$ and $\tau > 0$. Fix an intrinsic $L$-Lipschitz map $\varphi:\mathbb{V}_0\cong\R^k \to\mathbb{W}_0\cong\R^{2n-k+1}$, and a collection of $k+1$ planes $V_1, \dots, V_{k+1} \in \mathcal{A}_{k-1}(\frac{1}{2}Q)$ that is $\tau$-transversal. Provided $C \ge 1$ is sufficiently large and $\varepsilon > 0$ is sufficiently small (depending on $\tau$), there exist planes $V'_1, \dots, V'_{k+1} \in \mathcal{A}_{k-1}(Q)$ satisfying the following properties:
    \begin{enumerate}
        \item[(a)] $\tilde{d}(V_j, V'_j) \le \varepsilon$ for all $1 \le j \le k+1$.
        \item[(b)] $\beta_\infty(CQ, V'_j) \lesssim_{C,\varepsilon} \beta_{\infty,2}^{k-1}(CQ)$ for all $1 \le j \le k+1$.
        \item[(c)] For each $1 \le j \le k+1$, let $A'_j$ be an \textbf{intrinsic} affine quasi minimizer for $\beta_\infty(CQ, V'_j)$. If $1 \le i_1 < \dots < i_k \le k+1$ and $x$ is the unique point in the intersection $V'_{i_1} \cap \dots \cap V'_{i_k}$, then $x \in CQ$ and
        $$d(\varphi(x),A'_{i_j}(x)) \lesssim_{C,\varepsilon} \beta_{\infty,2}^{k-1}(CQ), \quad 1 \le j \le k.$$
    \end{enumerate}
    Here, all beta numbers are associated with the intrinsic Lipschitz map $\varphi$.
\end{lemma}
\begin{proof}
Fix $j\in \{1,\ldots,k+1\}$ and consider the $(k-1)$-plane $V_j$. Let $e_j\in S^{k-1}$ be a direction normal to $V_j$. The plane $V_j'$ will ultimately be of the form
\begin{displaymath}
    V_j'=\{x\in \mathbb{R}^k:\, \langle x,e_j'\rangle=s_j+t\}
\end{displaymath}
for suitable $e_j'\in S^{k-1}$ with $|e_j-e_j'|\leq \varepsilon$, $s_j\in \mathbb{R}$, and $t\in [0,\varepsilon]$. 
To simplify the notation, for $e_j'\in S^{k-1}$ with $|e_j-e_j'|\leq \varepsilon$,  we write
\begin{displaymath}
    V_j'(s):=\{x\in \mathbb{R}^k:\, \langle x,e_j'\rangle =s\}.
\end{displaymath}
Then let $s_j\in \mathbb{R}$ be a parameter minimizing $s\mapsto \tilde d(V_j,V_j'(s))$. So far, everything is exactly as in the proof of \cite[Lemma 2.15]{MR4345824}, and in exactly the same way we obtain
\begin{equation}\label{eq:Orp2.16}
\tilde{d}(V_j,V_j'(s_j+t))\lesssim\varepsilon\quad\text{and}\quad V_j'(s_j+t)\cap Q\neq \emptyset,\quad t\in [0,\varepsilon]
\end{equation}
if $\varepsilon>0$ is sufficiently small and $\tilde{d}$ is the metric on $A(k,k-1)$ as defined in \eqref{eq:DistPlane}.

Now we explain how to choose $e_j'$, again following \cite{MR4345824}, but using $\beta_{\infty}$ instead of $\beta_2$. By definition of $\beta_{\infty,2}^{k-1}(CQ)$ and the measure $\eta_{k-1}$, Chebyshev's inequality guarantees the existence of $e_j'\in S^{k-1}$ with 
\begin{equation}\label{eq:Orp2.17}
    |e_j-e_j'|\leq \varepsilon\quad\text{and}\quad
    \int_0^{\varepsilon}\beta_{\infty}(CQ,V_j'(s_j+t))^2\,dt\lesssim_{C,\varepsilon}\beta_{\infty,2}^{k-1}(CQ)^2.
\end{equation}
Applying Chebyshev's inequality again to \eqref{eq:Orp2.17}
and using \eqref{eq:Orp2.16}, we finally find $t\in [0,\varepsilon]$ such that $V_j':=V_j'(s_j+t)$ satisfies conditions (a) and (b) in the statement of Lemma \ref{l:orp21,2.15}. Then (c) is just a consequence of (b) (concerning the uniqueness of the point $x$ in (c), see Example 2.13 and Remark 2.20 in \cite{MR4345824}). 
\end{proof}

\begin{lemma}\label{l:IntGeomBdd}
    Let $1 < k \leq n$. Suppose that $\varphi:\mathbb{V}_0\cong\mathbb{R}^k \to \mathbb{W}_0\cong \mathbb{R}^{2n-k+1}$ is an intrinsic $L$-Lipschitz map. Then, there exists a sufficiently small dimensional constant  $c > 0$ such that for every cube $Q \subset \mathbb{R}^k$, we have
    \begin{equation}\label{eq:betaBound}
         \beta_2(cQ)^{8} \lesssim_{{k,n, L}}\beta_{\infty,4}^1(Q)^4 +\beta_{\infty,2}^{k-1}(Q).
    \end{equation}
    
\end{lemma}

We emphasize that we will prove the statement for any cube, not just for dyadic cubes, which helps with some rescaling arguments.
We follow the approach in \cite[Section 2.4.3]{MR4345824}, but work with the ``full'' mapping $\varphi$ instead of its real-valued components, and the Heisenberg distance in the computation of the $\beta$-numbers. In the end we apply Lemma \ref{l:almostHoriz_k} to find an approximating \emph{horizontal} $k$-plane instead of a general affine $k$-plane.

\begin{proof}[Proof of Lemma \ref{l:IntGeomBdd}]
 We will actually establish the following modified bound
    \begin{equation}\label{eq:betaBoundModified}
        \beta_2(cQ)^{8} \lesssim_{k,n,{L}} \beta^{1}_{\infty,4}(CQ)^4+\beta^{k-1}_{\infty,2}(CQ) 
    \end{equation}
    for any cube $Q \subset \mathbb{R}^k$, where $C \gg 1$ is a large  constant depending on $k$. 
    This modified bound implies the desired result. Indeed, for any arbitrary cube $\tilde{Q}$, setting $Q = C^{-1}\tilde{Q}$ (so that $\tilde{Q} = CQ$) in \eqref{eq:betaBoundModified} yields
    $$
    \beta_2\Big(\frac{c}{C}\tilde{Q}\Big)^{8} \lesssim_{k,n,{L}}\beta^{1}_{\infty,4}(\tilde{Q})^4 +\beta^{k-1}_{\infty,2}(\tilde{Q}),
    $$
    which is precisely \eqref{eq:betaBound} with the small constant $c$ replaced by $c/C$.

    Having established this reduction, we can now focus on proving \eqref{eq:betaBoundModified}. By applying appropriate Heisenberg dilations and left translations, we may and will assume without loss of generality that $Q = [0,1]^k$.
Indeed, the families of intrinsic $L$-Lipschitz functions and intrinsic affine functions are preserved by these operations, as well as the beta-coefficients appearing in \eqref{eq:betaBoundModified} (when evaluated on the corresponding cubes). For the same reason, by applying a left translation, we can also assume that $\varphi(0)=0$.

For a small constant $0<c<1$ which can be determined depending only on $k$, let $\Delta_0 \subset \mathbb{V}_0$ be a closed simplex bounded by $k+1$ faces with $10cQ \subset \Delta_0 \subset \frac{1}{2}Q$. 
Let $V_1, \dots, V_{k+1} \in \mathcal{A}_{k-1}(\frac{1}{2}Q)$ be the planes containing the faces of $\Delta_0$. 
    Applying Lemma \ref{l:orp21,2.15}, we can find $V_1', \dots, V_{k+1}' \in \mathcal{A}_{k-1}(Q)$ satisfying (a)--(c), where $C\gg 1$ is at least as large and $\varepsilon$ at least as small as required for the application of  Lemma \ref{l:orp21,2.15}.
Let $A_j$ be the essential minimizer for $\beta_\infty(CQ, V_j')$. Let $\Delta$ be the simplex whose $(k-1)$-faces are contained in the planes $V_1', \dots, V_{k+1}'$; denote these faces by $\Delta_j$. 
By choosing $\varepsilon > 0$ small enough, it still holds that $5cQ \subset \Delta$. Note that $\varepsilon$ can be chosen depending only on $k$, so we will refrain from writing the $\varepsilon$-dependence explicitly in the following. 

Since $A_j$ is intrinsic affine, its image $\mathbb{A}_j(V_j')$ is a $(k-1)$-dimensional horizontal plane, which we shall denote by $\mathbb{U}_j$. Let $p_0, \dots, p_k$ be the vertices of the simplex $\Delta$, and define $q_j \coloneqq \Phi(p_j)$ for $j = 0, \dots, k$. By Lemma \ref{l:orp21,2.15}(c), it follows that $\Delta\subset CQ$. Because the projections $\pi_{\mathbb{V}_0}(q_j) = p_j$ exactly form the vertices of $\Delta$, condition (1) of Lemma \ref{l:almostHoriz_k} is satisfied. 

To verify condition (2), observe that
\begin{equation*}
    |p_i - p_j| \le d(q_i, q_j) = d(\Phi(p_i), \Phi(p_j)) \lesssim_L|p_i - p_j|.
\end{equation*}
Thus, we may choose a sufficiently large constant $C_L \ge C$ (depending also on $L$, in addition to $k$) such that $1/C_L \leq d(q_i,q_j) \leq C_L$ (using also Remark \ref{lem:geometric}). This confirms condition (2) of Lemma \ref{l:almostHoriz_k}.

Regarding condition (3), recall that $\beta_2(cQ)$ is bounded by Remark \ref{rmk:boundedness}. So, we can suppose that $\beta_{\infty, 2}^{k-1}(CQ)$ is as small as we want, since otherwise \eqref{eq:betaBound} would be immediately satisfied. Let us therefore assume that 
 $\beta_{\infty, 2}^{k-1}(CQ)\ll 1$.
 Then, if $p_j$ belongs to $V_m'$ and $q_j = \Phi(p_j)$, by (c) of Lemma \ref{l:orp21,2.15} we have
$$
d(q_j, \mathbb{U}_m) \le d(q_j, \mathbb{A}_m(p_j)) \lesssim_k \beta_\infty(CQ, V_m') \lesssim_{C,k} \beta_{\infty, 2}^{k-1}(CQ).
$$ Thus, recalling that $C$ depends only on $k$, part  (3) of Lemma \ref{l:almostHoriz_k} is verified with an error parameter $\lesssim_{k} \beta_{\infty, 2}^{k-1}(CQ)$. 

Now we apply Lemma \ref{l:almostHoriz_k} with the enlarged constant $C_L\geq C$ and we conclude that there exist a $k$-dimensional horizontal plane $\mathbb{V}$ and an associated intrinsic affine map $A_{\mathbb{V}}$ with graph map $\mathbb{A}_{\mathbb{V}}$ such that
\begin{equation}\label{eq:EpsSquareEst}
d(q_i, \mathbb{A}_{\mathbb{V}}(\pi_{\mathbb{V}_0}(q_i))) \lesssim_{k,L} \beta_{\infty, 2}^{k-1}(CQ)^{1/2}
\end{equation}
for all $i = 0, 1, \dots, k$ (note how the dependence on $C_L$ is now absorbed into the implicit constant $\lesssim_{C,k,L}$).
Consequently, for all vertices $p_i$ of the face $\Delta_j$, we have
\begin{align}\label{eq:FromVertexToInteriorEst}
d(A_{\mathbb{V}}(p_i), A_j(p_i)) &\le d(A_{\mathbb{V}}(p_i), \varphi(p_i)) + d(\varphi(p_i), A_j(p_i)) \notag\\
&\stackrel{\text{Lem. \ref{l:orp21,2.15} (c)}}{\lesssim_k} d(A_{\mathbb{V}}(p_i), \varphi(p_i)) + \beta_{\infty, 2}^{k-1}(CQ) \notag\\
&\stackrel{\text{\eqref{eq:EpsSquareEst}}}{\lesssim_{{k,L}}} \beta_{\infty, 2}^{k-1}(CQ)^{1/2} + \beta_{\infty, 2}^{k-1}(CQ) \notag\\
&\lesssim_{{k,L}} \beta_{\infty, 2}^{k-1}(CQ)^{1/2}.
\end{align}

Note that $\Delta_j$ is a $(k-1)$-simplex, containing $k$ corners of $\Delta$. In particular, the (Euclidean or Heisenberg) distance of the points $p_i$ is $\sim_c 1$. The estimate \eqref{eq:FromVertexToInteriorEst}, which is valid at the vertex points $p_i\in \Delta_j$ implies an estimate for all points $p\in \Delta_j$. To see this, we use the following comparison between Euclidean metric $|\cdot|$ and Kor\'{a}nyi metric $d$:
\begin{equation}\label{eq:MetricComp}
\frac{1}{C_0} |y-y'|\leq d(y,y')\leq  C_0|y-y'|^{1/2},\quad y,y'\in B(0,R),
\end{equation}
for some constant $C_0=C_0(R,n)\geq 1$ (see \cite[Proposition 2.7 and Proposition 2.15]{MR3587666}). Since $Q=[0,1]^k$ by assumption and we are assuming that $\varphi(0)=0$, we only have to deal with points that stay inside the domain  $B(0,R(L,k,n))$, where the estimate \eqref{eq:MetricComp} is valid.
Thus, by \eqref{eq:FromVertexToInteriorEst}, we have the following estimate for the vertex points
\begin{displaymath}
    |\mathbb{A}_{\mathbb{V}}(p_i)-\mathbb{A}_j(p_i)|\lesssim_{{k,L,n}}\beta_{\infty, 2}^{k-1}(CQ)^{1/2}.
\end{displaymath}
Now $\mathbb{A}_{\mathbb{V}}(V_j')$ and $\mathbb{A}_j(V_j')$ are two horizontal $(k-1)$-planes, therefore
\begin{displaymath}
    \||\mathbb{A}_{\mathbb{V}}-\mathbb{A}_j|\|_{L^{\infty}(\Delta_j)}
    \lesssim_{{k,L,n}} \beta_{\infty, 2}^{k-1}(CQ)^{1/2}.
\end{displaymath}
Then, applying again \eqref{eq:MetricComp},
\begin{equation}\label{eq:DedInnerEst}
d(A_{\mathbb{V}}(x), A_j(x)) \lesssim_{{k,L,n}} \beta_{\infty, 2}^{k-1}(CQ)^{1/4}, \quad x \in \Delta_j.
\end{equation}
We have shown $A_j$ and $A_{\mathbb{V}}$ are close to each other on the faces of $\Delta$. Next, we show that $A_\V$ approximates well $\varphi$ inside $\Delta$. 
By the definition of $\beta^1_{\infty, 4}(CQ)$ and of the measure $\eta_1$ (see also \cite{MR4345824}), we can find $e \in S^{k-1} \subset \mathbb{V}_0$ such that the following holds. Let $\pi := \pi_e$ be the orthogonal projection to $V^e := e^\perp$. Write $\ell^\xi := \pi^{-1}(\xi)$ for $\xi \in V^e$. Then we may assume
\begin{equation}\label{eq:ReprIntHigher}
\int_{\pi(cQ)} \beta_\infty(CQ, \ell^\xi)^4 \, d\mathcal{H}^{k-1}(\xi)\lesssim_k \beta^1_{\infty, 4}(CQ)^4.
\end{equation}

The line $\ell^\xi$ for $\xi \in \pi(cQ)$ meets $cQ$ and $\partial \Delta$. Let $\partial \Delta \cap \ell^\xi = \{x^\xi, y^\xi\}$. Assume $x^\xi \in V_i'$ and $y^\xi \in V_j'$. We have $|x^\xi - y^\xi| \sim_c 1$ whenever $\xi \in \pi(cQ)$, since $cQ\subset 5cQ\subset \Delta$ as observed at the beginning of the proof. Let $A^\xi$ be the (essential) minimizer for $\beta_\infty(CQ, \ell^\xi)$. Then
\begin{align*}
d(A_{\mathbb{V}}(x^\xi), A^\xi(x^\xi)) &\le d(A_{\mathbb{V}}(x^\xi), A_i(x^\xi)) + d(A_i(x^\xi), \varphi(x^\xi)) + d(\varphi(x^\xi), A^\xi(x^\xi)) \\
&\stackrel{\eqref{eq:DedInnerEst}\,+\,\text{Lem. \ref{l:orp21,2.15} (b)}}{\lesssim_{{k,L,n}}} \beta_{\infty, 2}^{k-1}(CQ)^{1/4} + \beta_{\infty, 2}^{k-1}(CQ) + \beta_\infty(CQ, \ell^\xi) \\
&\lesssim_{{k,L,n}} \beta_{\infty, 2}^{k-1}(CQ)^{1/4} + \beta_\infty(CQ, \ell^\xi).
\end{align*}
Analogously,
$$
d(A_{\mathbb{V}}(y^\xi), A^\xi(y^\xi)) \lesssim_{{k,L,n}} \beta_{\infty, 2}^{k-1}(CQ)^{1/4} + \beta_\infty(CQ, \ell^\xi).
$$
Again, arguing as in the proof of \eqref{eq:DedInnerEst},  we have,
\begin{equation}\label{eq:DistAVAxi}
d(A_{\mathbb{V}}(x), A^\xi(x)) \lesssim_{{k,L,n}} \beta_{\infty, 2}^{k-1}(CQ)^{1/8} + \beta_\infty(CQ, \ell^\xi)^{1/2}, \quad  x \in \ell^\xi \cap cQ.
\end{equation}
Thus, for each $x\in \ell^{\xi}\cap cQ$, we can estimate as follows
\begin{align*}
    d(\varphi(x),A_{\mathbb{V}}(x))&\leq d(\varphi(x),A^{\xi}(x))+d(A^{\xi}(x),A_{\mathbb{V}}(x))\\
    &\lesssim_{k,L,n} \beta_{\infty}(CQ,\ell^{\xi})+ \beta_{\infty, 2}^{k-1}(CQ)^{1/8} + \beta_\infty(CQ, \ell^\xi)^{1/2}.
\end{align*}

Then, by Jensen's inequality and Fubini's theorem,
\begin{align*}
    \beta_2(cQ)^{8} \lesssim_c \int_{cQ} d(\varphi(x),A_{\mathbb{V}}(x))^{8}&\,d\mathcal{L}^k(x)=\int_{\pi(cQ)}\int_{\ell^{\xi}\cap cQ}d(\varphi(x),A_{\mathbb{V}}(x))^{8}\,d\mathcal{H}^1(x) d\mathcal{H}^{k-1}(\xi)\\
    &{\lesssim_{c,{k,L,n}}} \int_{\pi(cQ)} \beta_{\infty}(CQ,\ell^\xi)^{8}+ [\beta_{\infty,2}^{k-1}(CQ)+\beta_{\infty}(CQ,\ell^\xi)^{4}]\, d\mathcal{H}^{k-1}(\xi)\\& \lesssim_{{k,L,n}} \int_{\pi(cQ)}\beta_{\infty}(CQ,\ell^\xi)^4\, d\mathcal{H}^{k-1}(\xi) + \beta_{\infty,2}^{k-1}(CQ).
\end{align*}
By \eqref{eq:ReprIntHigher}, we can further estimate as follows:
\begin{displaymath}
    \beta_2(cQ)^{8} \lesssim_{{k,L,n}}\beta_{\infty,4}^1(CQ)^4 +\beta_{\infty,2}^{k-1}(CQ).
\end{displaymath}
This establishes \eqref{eq:betaBoundModified} and thus concludes the proof.
\end{proof}

With these preparations in place, we are ready to run the inductive argument for the proof of our main result.
This is inspired by the argument in \cite[Section 2.4.2]{MR4345824}, but the difference for $k>2$ is that we pass via $L^{\infty}$-based integral geometric $\beta$-numbers. Recall that this allowed us to simplify some steps in the proof of Lemma \ref{l:IntGeomBdd} compared to Orponen's argument in Euclidean spaces, where it was important to recover the sharp exponent $2$ in the geometric lemma. On the other hand, we now need an additional step to bound the $L^{\infty}$-based parametrized $\beta$-numbers of an intrinsic Lipschitz graph by the $L^2$-based ones with a worse exponent.

\begin{theorem}[Dorronsoro for $k$-dim.\ intrinsic Lipschitz graphs]\label{t:kDDorronsoro} Let $n\geq 2$ and $1<k\leq n$. There exists $q^*={q^*(k)}$ such that the following holds.
    Let  $\mathbb{V}_0=\mathbb{R}^k\times \{0\}$ as well as $\mathbb{W}_0=\{0\}\times \mathbb{R}^{2n+1-k}$.
    Let $\varphi:\mathbb{V}_0\to \mathbb{W}_0$ be intrinsic $L$-Lipschitz. Then, for any $C\geq 1$, 
    \begin{equation}\label{eq:ParamGLemCube}
        \sum_{Q\in \mathcal{D}^k,Q\subset Q_0}\beta_{2}(CQ)^{q^*}\mathcal{H}^k(Q) \lesssim_{{k,n,L,C}} \mathcal{H}^k(Q_0),\quad Q_0\in \mathcal{D}^k.
    \end{equation}
    For $k=2$, we can choose $q^*=32$.
\end{theorem}

\begin{proof}[Proof of Theorem \ref{t:kDDorronsoro}] We fix $n$ and proceed by induction on $k$.

We start with the base case $k=2$ of the induction. Let $\varphi:\V_0=\mathbb{R}^2 \times \{0\}\to \W_0=\{0\}\times \mathbb{R}^{2n-1}$ be an intrinsic Lipschitz function. 
For $k=2$, Lemma \ref{l:IntGeomBdd} yields, for any cube $Q$ in $\mathbb{R}^2$, that
\begin{displaymath}
    \beta_2(cQ)^{8} \lesssim_{{n, L}}\beta_{\infty,4}^1(Q)^4 +\beta_{\infty,2}^{1}(Q).
\end{displaymath}
Since $\beta_{\infty,2}^1\leq \beta_{\infty,4}^1$ we get, raising both sides to the fourth power and using Remark \ref{rmk:boundedness},
\begin{equation}\label{eq:BetaComp}
    \beta_2(cQ)^{32} \lesssim_{{n, L}}\beta_{\infty,4}^1(Q)^4. 
\end{equation}
Since
Lemma \ref{l:TLem2.5} shows, for any $C\geq 1$, that
\begin{displaymath}
    \sum_{Q\in \mathcal{D}^2,Q\subset Q_0}\beta_{\infty,4}^1(CQ)^4 \mathcal{H}^2(Q)\lesssim_{C,L}\mathcal{H}^2(Q_0),\quad Q_0 \in \mathcal{D}^2,
\end{displaymath}
we conclude the case $k=2$ of Theorem \ref{t:kDDorronsoro} with $q^*(2)=32$ 
from \eqref{eq:BetaComp}.

Let us now assume that the statement of Theorem \ref{t:kDDorronsoro}  holds for $k-1$ with an exponent $q^*(k-1)$, and deduce the version for $k$-dimensional intrinsic Lipschitz graphs. First, if the statement holds for $k-1$ and the $\beta_2$-numbers, then it also holds for the $\beta_{\infty}$-numbers with a possibly worse exponent $q=q(k-1)$.
To see this, we will prove
for intrinsic $L$-Lipschitz functions $\varphi:\mathbb{R}^{k-1}\times \{0\}\to \{0\}\times \mathbb{R}^{2n+2-k}$ and  $M\geq 1$ that
\begin{equation}\label{eq:Beta_inftyBeta_2}
    \beta_\infty(MQ)\lesssim_{L,k,n,M} \beta_2(2MQ)^{\frac{2}{2+k-1}},\quad Q\in \mathcal{D}^{k-1}.
\end{equation} Without loss of generality, we assume that $\varphi$ is not intrinsic affine on $MQ$.
The proof of \eqref{eq:Beta_inftyBeta_2} is similar to the comparison for standard $\beta$-numbers \cite[p.27]{MR1251061}.
Indeed, let $A$ be an intrinsic affine function (essentially) realizing the infimum in the definition of $\beta_2(2MQ)$, and let $z\in MQ$ be a point (essentially) realizing $ \mathrm{arg max}_{x\in MQ}d(\varphi(x),A(x))$. 
Finally, we set
\begin{displaymath}
    D:= d(\varphi(z),A(z))=d(\Phi(z),\mathbb{A}(z)).
\end{displaymath}
By Remark \ref{rmk:boundedness}  and Lemma \ref{lemma_lipconst} (with the variations explained in the proof of Lemma \ref{l:HigherDimInduc}), and since $\varphi$ is intrinsic $L$-Lipschitz, we know that
\begin{displaymath}
    \mathrm{Lip}(\mathbb{A})\lesssim_{k,n,M}\mathrm{Lip}(\Phi),
\end{displaymath}
where $\mathbb{A}$ and $\Phi$ denote the graph maps of $A$ and $\varphi$, respectively. Applying the triangle inequality as follows:
\begin{align*}
d(\varphi(x),A(x))&=d(\Phi(x),\mathbb{A}(x))\geq 
d(\Phi(z),\mathbb{A}(z))-d(\Phi(z),\Phi(x))-d(\mathbb{A}(z),\mathbb{A}(x))\\
&\geq D- [\mathrm{Lip}(\Phi)+\mathrm{Lip}(\mathbb{A})]d(z,x),
\end{align*}
we conclude that $d(\varphi(x),A(x))\geq D/2$ for all $x\in 2MQ$ with 
\begin{equation}\label{eq:suffclose}
    d(z,x)\leq \frac{D}{2[\mathrm{Lip}(\Phi)+\mathrm{Lip}(\mathbb{A})]}.
\end{equation}
Therefore, if $D\leq \mathrm{diam}(MQ)$, then $d(\varphi(x),A(x))\geq D/2$ holds for points $x$ in a ball of radius $\sim_{L,k,n,M} D$. Thus, in this case
\begin{align*}
 \beta_2(2MQ)^2 &= 
 \frac{1}{\mathrm{diam}(2MQ)^{k-1}}\int_{2MQ}\left(\frac{d(\varphi(x),A(x))}{\mathrm{diam}(2MQ)}\right)^2\,d\mathcal{H}^{k-1}(x)\gtrsim_{L,k,n,M}\left(\frac{D}{\mathrm{diam}(MQ)}\right)^{2+k-1},
\end{align*}
and therefore
\begin{equation}\label{eq:CompInfty2}
    \beta_{\infty}(MQ)\leq \frac{2D}{\mathrm{diam}(MQ)}\lesssim_{L,k,n,M} \beta_2(2MQ)^{\frac{2}{2+k-1}},
\end{equation}
where the first inequality can be ensured by the choice of $A$ and $z$.
This concludes the proof of  \eqref{eq:CompInfty2} if $D\leq \mathrm{diam}(MQ)$. Otherwise, we use \eqref{eq:suffclose} to conclude that $d(\varphi(x),A(x))\geq \mathrm{diam}(MQ)/2$ for points $x$ in a ball of radius $\sim_{L,k,n}\mathrm{diam}(MQ)$, which yields $\beta_2(2MQ)\gtrsim_{L,k,n,M}1$, and we conclude simply by 
 using the boundedness of $\beta_{\infty}(MQ)$ ensured by Remark \ref{rmk:boundedness}.

With  estimate \eqref{eq:CompInfty2} in place, Lemma \ref{l:HigherDimInduc} yields the following geometric lemma for $k$-dimensional intrinsic Lipschitz graphs and  the integral geometric $\beta^{k-1}_{\infty,q}$-numbers, for $q=q(k-1)\geq 2$:
\begin{equation}\label{eq:Glem1needed}
    \sum_{Q\in \mathcal{D}^k,Q\subset Q_0}\beta_{\infty,q}^{k-1}(KQ)^{q}\mathcal{H}^k(Q)\lesssim_{{k,n,q,L,K}}\mathcal{H}^k(Q_0),\quad Q_0\in\mathcal D^k.
\end{equation}
Now, by Lemma \ref{l:IntGeomBdd} applied to $\varphi$,
\begin{displaymath}
    \beta_2(cQ)^{8} \lesssim_{{k,n, L}}\beta_{\infty,4}^1(Q)^4 +\beta_{\infty,2}^{k-1}(Q).
\end{displaymath}
Since $\beta_{\infty,2}^{k-1}\leq \beta_{\infty,q}^{k-1}$ we get, raising both sides to the power $q$ and using Remark \ref{rmk:boundedness},
\begin{equation}\label{eq:Glem2needed}
   \beta_2(cQ)^{8q}\lesssim_{k,n,q,{L}}\beta^{k-1}_{\infty,q}(Q)^q+\beta^{1}_{\infty,4}(Q)^4  
\end{equation}
for any cube $Q$ in $\mathbb{R}^k$.

Finally,
 Lemma \ref{l:TLem2.5}, \eqref{eq:Glem1needed} and \eqref{eq:Glem2needed} will imply the statement in   Theorem \ref{t:kDDorronsoro} for $k$-dimensional intrinsic Lipschitz graphs.
\end{proof}

We finally prove the main result of this section, a geometric lemma for $k$-dimensional intrinsic Lipschitz graphs in $\mathbb{H}^n$ if $1<k<n$.

\begin{proof}[Proof of Theorem \ref{t:GlemBetanILGLk=2}]
   Proceeding as in the first part of the proof of Theorem \ref{t:ControlStratifBetaByOmega}, for every $\V\in\mathcal V_k$, one has
   \[\beta_{2,\mathcal V_k}^E(x,r)\lesssim_{\mathrm{Lip}(\varphi_i)}\left[\fint_{B_{\R^k}(\Pi(x),r)}\left(\frac{d(\Phi(v),\V)}{r}\right)^2dv\right]^{1/2}.\]
   In particular, for every intrinsic affine map $A:\V_0\to\W_0$, if $\mathbb A$ denotes the corresponding graph map, then
   \[\beta_{2,\mathcal V_k}^E(x,r)\lesssim_{\mathrm{Lip}(\varphi_i)}\left[\fint_{B_{\R^k}(\Pi(x),r)}\left(\frac{d(\Phi(v),\mathbb A(v))}{r}\right)^2dv\right]^{\frac{1}{2}}=\left[\fint_{B_{\R^k}(\Pi(x),r)}\left(\frac{d(\varphi(v), A(v))}{r}\right)^2dv\right]^\frac{1}{2}.\]
   This means that $\beta_{2,\mathcal V_k}^E(x,r)\lesssim_{\mathrm{Lip}(\varphi_i)}\beta_2(B_{\R^k}(\Pi(x),r))$ with the notation introduced at the beginning of Section \ref{ss:ProofConclu}, recall Remark \ref{r:DiffBeta}. Hence, Theorem \ref{t:kDDorronsoro} implies, together with the area formula (Theorem \ref{af}) and the estimate on the Jacobian (Lemma \ref{rem_jacobian}), the desired Carleson condition also on the $\beta_{2,\mathcal V_k}^E$-numbers. To be precise, for this argument we need a Carleson-integral version of  Theorem \ref{t:kDDorronsoro} stated in terms of balls in $\mathbb{R}^k$ instead of the discrete version with dyadic cubes. However, Theorem \ref{t:kDDorronsoro}  is easily seen to imply such a version, see for instance similar arguments in the proof of \cite[Lemma 3.18]{arXiv:2601.03837}. Here the argument is even simpler in the sense that the domain is $\mathbb{R}^k$ with a standard dyadic system instead of a more general $k$-regular set.
\end{proof}

\subsection{Consequences of the geometric lemmas}\label{ss:ConseqGLem}
As immediate corollaries of the geometric lemmas proven in the previous sections, we also obtain corresponding
 \emph{weak geometric lemmas} in the sense of Definition \ref{def:WGLballs}, which we state here in terms of horizontal $\beta$-numbers.

\begin{corollary}\label{cor:WGLforILG} Let $n\in \mathbb{N}$ and $k\in \{1,\cdots,n\}$.
    Let $\mathbb{V}$ be a $k$-dimensional horizontal subgroup of $\mathbb{H}^n$ with complementary vertical subgroup $\mathbb{W}$. If $E\subset \mathbb{H}^n$ is the intrinsic graph of an intrinsic Lipschitz function $\varphi:\mathbb{V}\to \mathbb{W}$, then $E\in \mathrm{WGL}(\beta^E_{\infty,\mathcal V_k},C(\cdot))$ for a function $\varepsilon \mapsto C(\varepsilon)$ depending only on $k$, $n$, and the intrinsic Lipschitz constant of $\varphi$.
\end{corollary}

\begin{proof} 
We first recall that in all the considered cases, we have 
$E\in \mathrm{GLem}(\beta^E_{p,\mathcal{V}_k},q,C)$ for some $p=p_{k,n}\in [2,\infty]$ and some exponent $q=q_{k,n}\geq 4$, with $C$ depending only on $k$, $n$, and the intrinsic Lipschitz constant of $\varphi$.

In the  case $k=n$, the geometric lemma follows from Theorem \ref{t:GlemStratBetanILG}, recalling from Definitions \ref{d:HorizBeta} and \ref{d:StratifBeta} that $\beta^E_{2,\mathcal{V}_n}\leq \widehat{\beta}^E_{2,\mathcal{V}_n}$. Therefore, Theorem \ref{t:GlemStratBetanILG} yields that $E\in \mathrm{GLem}(\beta^E_{2,\mathcal{V}_n},q)$ for $q=4$. 
If $k=1$, Theorem \ref{t:GlemStratBetanILGLinftyk=2} yields similarly that  $E\in \mathrm{GLem}(\beta^E_{\infty,\mathcal{V}_1},4)$, which immediately implies that $E\in \mathrm{WGL}(\beta^E_{\infty,\mathcal V_1})$.
Finally, if $1<k<n$, we have  $E\in\mathrm{GLem}(\beta^E_{2,\mathcal{V}_k},q^*) $ by Theorem \ref{t:GlemBetanILGLk=2}. In each case, the quantitative control on the constant in the  geometric lemma is ensured by the precise formulation of the respective theorem.

 Since $E$ is $k$-regular (\cite[Theorem 3.9]{zbMATH06607318}), for each $k\in \{1,\ldots,n\}$, the respective geometric lemma implies the stated weak geometric lemma by 
 the same argument as in Euclidean spaces, which works
   irrespective of the integrability exponents; see \cite[pag. 27]{zbMATH00051391}.
\end{proof}

Low-dimensional intrinsic Lipschitz graphs in $\mathbb{H}^n$ also satisfy the geometric lemma for projection $\beta$-numbers (as in Definition \ref{def_projbetas}) with exponent $2$. This is a technical condition which appeared first in \cite{arXiv:2601.03837} and which will be used, in combination with Corollary \ref{cor:WGLforILG}, to complete the characterization of sets with corona decompositions by intrinsic Lipschitz graphs  (see Theorem \ref{c:ILG-C Char} in Section \ref{ss:BigPieceAppl}).

\begin{theorem}\label{t:GlemAffProjBetanILG}
 Let $\mathbb{V}$ be a $k$-dimensional horizontal subgroup of $\mathbb{H}^n$, $1\leq k\leq n$, with complementary vertical subgroup $\mathbb{W}$. If $E\subset \mathbb{H}^n$ is the intrinsic graph of an intrinsic Lipschitz function $\varphi:\mathbb{V}\to \mathbb{W}$, then
\begin{equation*}
    \int_{E\cap B(y,R)} \int_0^R {\beta^E_{2,\pi,A(2n,k)}}(x, r)^2 \frac{dr}{r} d\mathcal{H}^k(x) \lesssim_{k,n, \mathrm{Lip}(\varphi_1),\dots,\mathrm{Lip}(\varphi_{2n-k})} R^k,\qquad y\in E,\,R>0.
\end{equation*}
\end{theorem}
For $k=n$, we have applied the isotropic Dorronsoro theorem (Proposition \ref{aiso}) to prove a stronger result for horizontal projection $\beta$-numbers instead of projection $\beta$-numbers. This result, formulated as Theorem \ref{t:GlemProjBetanILG}, immediately implies the case $k=n$ of Theorem \ref{t:GlemAffProjBetanILG}. Similarly, Theorem \ref{t:GlemStratBetanILGLinftyk=2} implies that $1$-dimensional intrinsic Lipschitz graphs in $\mathbb{H}^n$ satisfy $\mathrm{GLem}(\beta_{\infty,\pi,\mathcal V_1},2)$, which yields the case $k=1$ of Theorem \ref{t:GlemAffProjBetanILG}.
Below we give a proof of Theorem \ref{t:GlemAffProjBetanILG} that works for all cases $k\in \{1,\ldots,n\}$.
\begin{proof}[Proof of Theorem \ref{t:GlemAffProjBetanILG}]

By Lemma \ref{l:RotRed} and Lemma \ref{l:RotInvCoeff}, we may  assume that
$ \V=\mathbb{R}^k\times \{0\}$ and $\W=\{0\}\times \mathbb{R}^{2n+1-k}$.
Proceeding as in the proof of Theorem \ref{t:ControlStratifBetaByOmega}, we find that 
\begin{displaymath}
\beta^E_{2,\pi,A(2n,k)}(x,r)^2\lesssim_{\mathrm{Lip}(\varphi_1),\dots,\mathrm{Lip}(\varphi_{2n-k})}\inf_A I_{\pi,2}(A),
\end{displaymath}
where 
\begin{displaymath}
I_{\pi,2}(A):=   \fint_{B_{\mathbb{R}^k}(\Pi(x),r)} \left(\frac{|\Pi^{\bot}(\varphi)(v)-A(v)|}{r}\right)^2 dv
\end{displaymath} 
and $A$ runs over \emph{all} affine maps  $\R^k\to\R^{2n-k}$. Here, 
\begin{equation*}
 \Pi^{\bot}: \Hn \to \mathbb{R}^{2n-k},\quad    \Pi^\bot(w_1, \dots, w_{2n}, t) = (w_{k+1}, \dots, w_{2n}).\end{equation*}
Since $\inf_A I_{\pi,2}(A)=\Omega_{2,\Pi^\bot(\varphi)}(\Pi(x),r)^2$, then
\[\beta^E_{{2},\pi, A(2n,k)}(x,r)^{2}\lesssim_{k,n,\mathrm{Lip}(\varphi_1),\dots,\mathrm{Lip}(\varphi_{2n-k})} \Omega_{2,\Pi^{\bot}(\varphi)}(\Pi(x),r)^2,\quad x\in E,\,r>0\]
and the conclusion follows from Theorem \ref{t:DorronsoroSpecial} and a localization argument, which is the standard (and easier) counterpart of Proposition \ref{aiso} for $\Omega_{2,f}(x,r)$ instead of $\Omega_{2,f}^{iso}(x,r)$.
\end{proof}

Combined with an observation from \cite{arXiv:2601.03837}, the results of this section imply that low-dimensional intrinsic Lipschitz graphs in $\mathbb{H}^n$ admit certain corona decompositions defined in
 \cite{arXiv:2601.03837}. 
 The definitions are stated in terms of systems of dyadic cubes, whose definition we will recall for AD-regular sets in abstract metric spaces in Definition \ref{d:DyadicSystem}. For now, the reader may think of $\mathbb{D}(E)$ to be a generalization of the standard dyadic cubes $\mathcal{D}$ from $\mathbb{R}^k$ to $k$-regular sets in $\mathbb{H}^n$.

 \begin{definition}[Coronization]\label{d:coronization}
 Let $n\in \mathbb{N}$, $k\in \{1,\ldots,n\}$, and $E\subset \mathbb{H}^n$ be $k$-regular with a dyadic system $\mathbb{D}(E)$.  A \emph{coronization of $E$ with constant $C>0$} is a decomposition $\mathbb{D}(E)=\mathcal{G} \dot{\cup}\mathcal{B}$ of the dyadic cubes into a  \emph{good set} $\mathcal G$ and a \emph{bad set}
 $\mathcal B$ with the following properties:
    \begin{itemize}
        \item for every $R\in\mathbb{D}(E)$ \begin{equation*}
            \sum_{Q\in\mathcal B, Q\subseteq R} \mathcal{H}^k(Q)\leq C\mathcal{H}^k(R);
        \end{equation*}
        \item the good set $\mathcal G$ can be partitioned into a family $\mathcal F$ of disjoint \emph{trees} $\mathcal S$ (also called \emph{stopping time regions}) such that:
        \begin{itemize}
           \item each $\mathcal S\in\mathcal F$ is \textit{coherent}: it has a (unique) maximal element, denoted by $Q(\mathcal S)$, that contains all other elements of $\mathcal{S}$ as subsets, has the property that if 
        $Q\in \mathcal S, Q'\in\mathbb{D}(E)$ with $Q\subseteq Q'\subseteq Q(\mathcal S)$, then $Q'\in\mathcal S$, and finally is such that if $Q\in \mathcal{S}$, then either all of the children of $Q$ lie in $\mathcal{S}$ or none of them do;
         \item for every $R\in\mathbb{D}(E)$ \begin{equation}\label{carlesonpackingcond}
            \sum_{\mathcal S\in\mathcal F, Q(\mathcal S)\subseteq R} \mathcal{H}^k(Q(\mathcal S))\leq C\mathcal{H}^k(R).
        \end{equation}
    \end{itemize}
 \end{itemize}
\end{definition}

\begin{definition}[Corona decomposition by intrinsic Lipschitz graphs (ILG-C)]\label{d:ILG-C}
Let $n\in \mathbb{N}$, $k\in \{1,\ldots,n\}$, and $E\subset \mathbb{H}^n$ be $k$-regular.  We say that $E$ admits a \textit{corona decomposition by intrinsic Lipschitz graphs} (ILG-C) if, for every $\eta>0$, there exists a constant $C=C(\eta)>0$ such that $E$ admits a coronization $\mathbb{D}(E)=\mathcal{G}\dot{\cup}\mathcal{B}$ with constant $C$  
where,  for each $\mathcal S\in\mathcal F$, there exists a $k$-dimensional intrinsic Lipschitz graph $\Gamma=\Gamma_\mathcal S$ with intrinsic Lipschitz constant bounded by $\eta$ so that, for all $Q\in\mathcal S$, it holds that
     \begin{equation}\label{eq:ILGcorona}
     \mathrm{dist}(x,\Gamma)\leq \eta\,\mathrm{diam} ( Q)\quad\text{if }x\in E, \,\mathrm{dist}(x,Q)\leq \mathrm{diam}(Q).
     \end{equation}
\end{definition}

We next show that low-dimensional intrinsic Lipschitz graphs $E\subset \mathbb{H}^n$ satisfy (ILG-C). This may sound like a tautology at first, but the crux of the matter is that (ILG-C) requires the conditions in Definition \ref{d:ILG-C} to hold with an arbitrarily small choice of $\eta$. Assuming that $E$ is not flat,  for small enough $\eta$, the graph $\Gamma_\mathcal S$ in \eqref{eq:ILGcorona} is therefore required to have smaller Lipschitz constant than the given graph $E$ (as a graph over a, possibly different, $k$-dimensional subgroup).

\begin{corollary} Let $n\in \mathbb{N}$ and $k\in \{1,\cdots,n\}$.
  Let $\mathbb{V}$ be a $k$-dimensional horizontal subgroup of $\mathbb{H}^n$ with complementary vertical subgroup $\mathbb{W}$. If $E\subset \mathbb{H}^n$ is the intrinsic graph of an intrinsic Lipschitz function $\varphi:\mathbb{V}\to \mathbb{W}$, then $E$ admits a corona decomposition by intrinsic Lipschitz graphs. 
\end{corollary}
\begin{proof}
By Corollary \ref{cor:WGLforILG} and Theorem \ref{t:GlemAffProjBetanILG}, it follows that
\begin{displaymath}
E\in \mathrm{WGL}(\beta_{\infty,\mathcal V_k})\quad\text{and}\quad E\in \mathrm{GLem}(\beta_{1,\pi, A(2n,k)},2).
\end{displaymath}Applying \cite[Corollary 1.7]{arXiv:2601.03837}, we deduce that $E$ has a corona decomposition by intrinsic Lipschitz graphs with small constant (ILG-C).
\end{proof}

\section{The ``big pieces'' approach in metric spaces and consequences}\label{s:BP}

The geometric lemmas we obtained in Section \ref{s:GLem} for intrinsic Lipschitz graphs propagate to sets that have merely big pieces, or big pieces squared, of intrinsic Lipschitz graphs. The purpose of this section is to make that statement rigorous, following the approach in \cite{MR4485846}, and derive some consequences from it. This is most conveniently done using dyadic systems for AD-regular sets, to which we already referred in the previous section, for instance in Definition \ref{d:coronization}. We begin by stating the precise definitions.

\subsection{Dyadic systems}

Let $(X,d)$ be a metric space. Recall that a set $E\subset X$ is \emph{$s$-regular} with constant $C$, denoted $E\in \mathrm{Reg}_s(C)$ if it is closed and satisfies condition 
\eqref{def_Ahlfors}.

\begin{definitiontheorem}[Dyadic system]\label{d:DyadicSystem}
 Let $(X,d)$ be a metric space and $0<s<\infty$.    Assume that $E\in \mathrm{Reg}_s(C_E)$. 
 There exists a constant $D \in (1, \infty)$, depending only on $C_E$ and $s$, and a collection $\mathbb{D}(E) = \cup_{j\in\mathbb{J}}\mathbb{D}_j$ where, for every $j \in \mathbb{J}$, $\mathbb{D}_j$ is a family of pairwise disjoint Borel sets and
\begin{enumerate}
    \item[(1)] For each $j \in \mathbb{J}$, $E = \bigcup_{Q\in\mathbb{D}_j} Q$.
    \item[(2)] If $Q_1, Q_2 \in \mathbb{D}(E)$ and $Q_1 \cap Q_2 \neq \emptyset$, then $Q_1 \subseteq Q_2$ or $Q_2 \subset Q_1$.
    \item[(3)] If $j \in \mathbb{J}$ and $Q \in \mathbb{}{D}_j$, then $\operatorname{diam}(Q) \leq D2^{-j}$.
    \item[(4)] For each $j \in \mathbb{J}$ and $Q \in \mathbb{D}_j$, there exists $x_Q \in E$ such that $B(x_Q, D^{-1}2^{-j}) \cap E \subset Q$.
    \item[(5)] $\mathcal{H}^s(\{x \in Q : d(x, E \setminus Q) \leq \eta2^{-j}\}) \leq D\eta^{\frac{1}{D}}\mathcal{H}^s(Q)$ for all $j \in \mathbb{J}$, $Q \in \mathbb{D}_j$, $\eta > 0$.
\end{enumerate}
Here $\mathbb{J} = \mathbb{Z}$ if $E$ is unbounded, and $\mathbb{J} = \{j \in \mathbb{Z} : j \geq J_0\}$ where $J_0 \in \mathbb{Z}$ is such that $D^{-1}\, 2^{-(J_0+1)} \leq \operatorname{diam}(E) < D^{-1}\,2^{-J_0}$ otherwise.
We also define $\ell(Q):=2^{-j}$ for $Q\in \mathbb{D}_j$,
$$ \mathbb{D}_{Q_0}(E) := \{Q \in \mathbb{D}(E) : Q \subset Q_0\}, \quad Q_0 \in \mathbb{D}(E), $$
and for a given constant $\lambda > 1$, we set
$$ \lambda Q := \{x \in E : d(x, Q) \leq (\lambda - 1)\operatorname{diam}(Q)\}. $$
\end{definitiontheorem}

This construction is essentially due to Christ \cite[Theorem 11]{MR1096400} (for spaces of homogeneous type, with cubes covering up to null sets and $1/2$ replaced by some constant $\varrho\in (0,1)$); see \cite[Proposition 2.12]{MR3589162} and, e.g., the comments in \cite[\S 2.6]{2023arXiv230612933B}, \cite[\S 2.2]{MR4485846}, or  \cite[\S 3.2.1]{arXiv:2601.03837} for further references and various modifications. Note that, if $E$ is bounded, condition (4) ensures that $\mathbb{D}_{J_0}=\{E\}$ for our choice of $J_0$ (since $D^{-1}2^{-J_0}>\mathrm{diam}(E)$ and therefore  $B(x_Q, D^{-1}2^{-J_0}) \cap E =E$).

\subsection{Big pieces}

We will employ and generalize parts of the ``big pieces'' framework developed in \cite{MR4485846} for abstract metric spaces.

\begin{definition}[Big pieces]\label{d:BP}
    Let $(X,d)$ be a separable metric space, $0<k<\infty$, and let $C,C'>1$. Assume that $\mathcal{E}\subset \mathrm{Reg}_k(C)$. Then a set $E\in \mathrm{Reg}_k(C')$ \emph{has big pieces of $\mathcal{E}$} if there exists a constant $\theta>0$ such that for every $x\in E$ and $0<r<\infty$, $r\leq \mathrm{diam}(E)$, there exists $\Gamma \in \mathcal{E}$ such that
    \begin{equation}\label{eq:MeasuLowerBound}
        \mathcal{H}^k(\Gamma \cap E\cap B(x,r))\geq \theta r^k.
    \end{equation}
    We also write $E\in BP(\mathcal{E})(\theta,C')$ in this case.
\end{definition}

It is implicit from Definition \ref{d:BP} that the diameters of the elements in $\mathcal{E}$ can be quantitatively controlled in the following sense. If $\Gamma\in \mathcal{E}$ satisfies \eqref{eq:MeasuLowerBound} for some $x\in E$ and $0<r<\infty$, $r\leq \mathrm{diam}E$, then $\mathrm{diam}(\Gamma)\gtrsim r$, where the implicit constants depend only on $\theta,k$, and $C$.

\begin{example}
    If there exists a constant $L>0$ such that a $k$-regular set $E\subset \mathbb{H}^n$ satisfies the condition in Definition \ref{d:BP} for the family $\mathcal{E}^{\mathrm{iLG}}_L$ of $k$-dimensional intrinsic $L$-Lipschitz graphs, then we say that $E$ has \emph{big pieces of intrinsic Lipschitz graphs}.
\end{example}

\begin{definition}[Big pieces squared]
      Let $(X,d)$ be a separable metric space, $0<k<\infty$, and let $C,C',C''>1$. Assume that $\mathcal{E}\subset \mathrm{Reg}_k(C)$. Then a set $E\in \mathrm{Reg}_k(C'')$ \emph{has big pieces squared of $\mathcal{E}$} if there exist  constants $\theta,\theta'>0$ such that for every $x\in E$ and $0<r<\infty$, $r\leq \mathrm{diam}(E)$, there exists $\Gamma \in \mathrm{Reg}_k(C')\cap BP(\mathcal{E})(\theta,C')$ such that
    \begin{displaymath}
        \mathcal{H}^k(\Gamma \cap E\cap B(x,r))\geq \theta' r^k.
    \end{displaymath}
    We also write $E\in BP\left(BP(\mathcal{E})(\theta,C')\right)(\theta',C'')$ in this case.
\end{definition}

\subsection{Stability of abstract geometric lemmas under the big pieces functor}\label{s:StabAbstrGlem}
Stability of geometric lemmas under the ``big pieces functor'' is a general principle, which was noticed and frequently used for classical geometric lemmas in Euclidean spaces already by David and Semmes \cite{MR1251061}. In a more abstract setting, it was formulated in 
\cite[Section 2.6 and Appendix A]{MR4485846}, on which the present section is heavily based. However, not all the geometric lemmas needed for the applications in our paper are of the form considered in \cite{MR4485846}. Therefore, we generalize  the stability result in \cite[Proposition 2.23]{MR4485846} so that it covers the relevant $\beta$-numbers (as well as some other flatness coefficients from the literature beyond the scope of 
\cite{MR4485846}). The results in \cite{MR4485846} are formulated for sets that are AD-regular with respect to a fixed Borel regular measure $\mu$ on the ambient metric space; for us, $\mu$ will be a Hausdorff measure on a separable metric space.
At the same time, we extend the result to general, possibly unbounded, AD-regular sets (cf.  \cite[Remark 2.12]{MR4485846}).

Let $(X,d)$ be a separable metric space and let $\mathscr{H}$ be a family of abstract coefficient functions  
$h^E_q=h_q:\mathbb{D}(E)\to [0,\infty)$, one for each $k$-regular set $E$ in $X$, and each $q\in[1,\infty]$. We assume that the abstract coefficient functions are uniformly bounded in the sense that, for all $1\leq C<\infty$, we have 
\begin{equation}\label{eq:BoundH}
  \sup_{E\in \mathrm{Reg}_k(C)}\sup_{Q\in \mathbb{D}(E)}h^E_q(Q)<\infty.
\end{equation}
Throughout the section, we consider a fixed family
$\mathscr{H}$, and constants are allowed to depend on the supremum in \eqref{eq:BoundH} without explicit mentioning.

\begin{definition}[Geometric lemma (abstract version with cubes)]\label{d:GLemAbstr}
    Let $p\in [1,\infty)$, $q\in [1,\infty]$, and $0<k<\infty$. We say that a $k$-regular set $E\subset X$ satisfies the \emph{$p$-geometric Lemma with respect to the coefficient function $h_q$}, denoted $E\in \mathrm{GLem}(h_q,p)$, if there exists a constant $M>0$ such that 
    \begin{equation}\label{eq:GLemCond}
        \sum_{Q\in \mathbb{D}_R(E)}h_q(Q)^p\,\mathcal{H}^k(Q)\leq M\mathcal{H}^k(R),\quad R\in \mathbb{D}(E).
    \end{equation}
    To emphasize the dependence on the constant, we also write $E\in \mathrm{GLem}(h_q,p,M)$.
\end{definition}

The reader may be familiar with a formulation of the geometric lemma where  \eqref{eq:GLemCond} appears with a coefficient function of the form $h_q(KQ)$ for some $K>1$, say $K=2$. This is just a matter of notational convention. 
Following the notation in \cite{MR4485846}, we include the enlargement of the cubes in the definition of the coefficient functions; see Definition  \ref{d:AbstractCoeffFunction} for an example.

\begin{definition}[Weak geometric lemma (abstract version with cubes)]\label{d:WGLAbsrt}
      Let  $q\in [1,\infty]$ and $0<k<\infty$. We say that a $k$-regular set $E\subset X$ satisfies the \emph{weak geometric lemma with respect to the coefficient function $h_q$}, denoted $E\in \mathrm{WGL}(h_q)$, if there exists a function $C:(0,\infty)\to(0,\infty)$ such that, for all $\varepsilon>0$,
    \begin{displaymath}
        \sum_{\substack{Q\in \mathbb{D}_R(E)\\h_q(Q)>\varepsilon}}\mathcal{H}^k(Q)\leq C(\varepsilon)\mathcal{H}^k(R),\quad R\in \mathbb{D}(E).
    \end{displaymath}
    To emphasize the dependence on the function, we also write $E\in \mathrm{WGL}(h_q,C(\cdot))$.
\end{definition}

Assume now that $E\in\mathrm{Reg}_k(C)\cap BP(\mathcal E)(\theta, C)$, where $\mathcal E\subset \mathrm {Reg}_k(C)$. 
The ``big pieces'' property from Definition \ref{d:BP} implies a corresponding property for dyadic cubes. To see this, let $R\in\mathbb D(E)$ be a dyadic cube in $E$.
By property (4) of dyadic cubes (Definition and Theorem \ref{d:DyadicSystem}) there exists a constant $D$, depending only on $C$ and $k$, such that $B(x_R,D^{-1}\ell(R))\cap E \subset R$. Moreover, since $E\in\mathrm{Reg}_k(C)\cap BP(\mathcal E)(\theta, C)$, there exists $\widetilde{E} \in \mathcal{E}$ such that
$\mathcal H^k(\widetilde E\cap E\cap B(x_R,D^{-1}\ell(R)))\geq \theta [D^{-1}\ell(R)]^k$.

In conclusion, to every $R\in\mathbb D(E)$ we can assign $\widetilde E\in\mathcal E$ such that $\mathcal H^k(\widetilde E\cap R)\geq c\theta \mathcal{H}^k(R)$, where $c=c(k,C,\theta)$; 
cf.\ \cite[Lemma 2.14]{MR4485846}. 
Then we define the family
\begin{displaymath}
    \mathbb{D}_{R,\widetilde{E}}(E):=\{Q\in \mathbb{D}_R(E):\,  Q\cap \widetilde{E}\neq \emptyset\}.
\end{displaymath}
For every $Q\in    \mathbb{D}_{R,\widetilde{E}}(E)$, we find $\widetilde Q=\widetilde Q(Q)\in \mathbb D(\widetilde E)$ with $Q\cap \widetilde Q\neq \emptyset$,
\begin{equation}\label{eq:DiamComp}
    \mathrm{diam}(\widetilde{Q})\sim_{C,k,\theta} \mathrm{diam}(Q),
\end{equation}
and
\begin{equation}\label{eq:CubeIncl}
 \{z\in \widetilde{E}:\, d(z,Q)\leq 5 \mathrm{diam}(Q) \}  
 \subset 2\widetilde{Q}.
\end{equation}
To see this, first consider $Q\in    \mathbb{D}_{R,\widetilde{E}}(E)$ with $\mathrm{diam}(\widetilde{E})\geq 10\, \mathrm{diam}Q$. Then we can choose  $\widetilde Q=\widetilde Q(Q)\in \mathbb D(\widetilde E)$ with $Q\cap \widetilde Q\neq \emptyset$ such that
\begin{equation}\label{eq:DiamCompPrecise}
    10 \,\mathrm{diam}(Q)\leq \mathrm{diam}(\widetilde{Q})\lesssim_{C,k} \mathrm{diam}(Q).
\end{equation}
By assumption, there exists $x_0 \in Q\cap \widetilde Q$. Then,  for all
$z$ in the set on the left-hand side of \eqref{eq:CubeIncl}, we have $d(z,x_0)\leq 6\,\mathrm{diam}Q$ and consequently, by the lower bound in \eqref{eq:DiamCompPrecise}, it follows that $d(z,\widetilde{Q})\leq \mathrm{diam}(\widetilde{Q})$, which proves  \eqref{eq:CubeIncl}.

On the other hand, if $Q\in    \mathbb{D}_{R,\widetilde{E}}(E)$ is such that $\mathrm{diam}(\widetilde{E})< 10\, \mathrm{diam}Q$ (thus, in particular, $\widetilde{E}$ is bounded), then we simply take $\widetilde{Q}=\widetilde{Q}(Q)=\widetilde{E}\in \mathbb{D}(\widetilde{E})$ equal to the top cube of $\widetilde{E}$. This trivially satisfies \eqref{eq:CubeIncl}. The condition \eqref{eq:DiamComp} follows since the reasoning below  Definition \ref{d:BP},  ensures that $\mathrm{diam}(\widetilde{E})\gtrsim \mathrm{diam}(R)(\geq \mathrm{diam}(Q))$ with implicit constants depending only on $C,\theta$, and $k$.

For $Q\in \mathbb{D}_R(E)$, we also define
  \begin{equation}\label{eq:Iq(Q)def}I_q(Q):=\left(\mathcal H^k(Q)^{-1}\int_{2Q\cap(E\setminus \widetilde{E})\cap \widetilde E(2\mathrm{diam}(Q))}(d(y,\widetilde E)(\mathrm{diam}(Q))^{-1})^q\,d\mathcal H^k(y)\right)^\frac{1}{q},\quad q<\infty\end{equation}
   and
\begin{displaymath}
    I_{\infty}(Q)=\sup_{y\in 2Q\cap (E\setminus \widetilde{E})\cap \widetilde{E}(2\mathrm{diam}(Q))}\frac{d(y,\widetilde{E})}{\mathrm{diam}(Q)},
\end{displaymath}where $\widetilde E(2\mathrm{diam}(Q))$ denotes the 
$2\mathrm{diam}(Q)$-neighborhood of $\widetilde{E}$,

We are now ready to state our generalization of \cite[Proposition 2.23]{MR4485846}. The main difference is that we consider general coefficient functions $h_q^E$ satisfying a certain axiomatic condition \eqref{eq:AssLemmaA2}, rather than $\beta_q$-numbers  defined by measuring distances from approximating sets as in \cite[Definition 2.15]{MR4485846}. For convenience, we restrict ourselves to the case $q\geq 1$; presumably one could obtain results in the full range $q>0$ as in \cite{MR4485846}, by introducing $q$-dependent constants in suitable places, but this generality is not needed for the applications we have in mind.

\begin{theorem}[Stability of geometric lemma under big pieces functor]\label{thm: BHHGNgeneralizedAbstract}
    Let $C_0>0,C>1, \theta>0, M>0,p\in [1,\infty), q\in[1,\infty], k\in\mathbb N$ be fixed such that
    \begin{equation}\label{eq:exponentRange}
    \frac{1}{q}-\frac{1}{p}+\frac{1}{k}>0.
    \end{equation}
   Assume that $E\in\mathrm{Reg}_k(C)\cap BP(\mathcal E)(\theta, C)$, where $\mathcal E\subset \mathrm {Reg}_k(C)$.
   Suppose that  $\mathscr{H}$ is a family of abstract coefficient functions with the property that, for every $R\in \mathbb{D}(E)$,
 \begin{equation}\label{eq:AssLemmaA2}
          h_{q}(Q)\leq C_0(\widetilde h_{ q}(\widetilde Q)+ I_q(Q)),\qquad Q\in \mathbb{D}_{R,\widetilde{E}}(E),
      \end{equation}
      where the notation is as stated before the theorem, and $\widetilde{h}_q=h_q^{\widetilde{E}}$ is the coefficient function associated to $\widetilde{E}$.
Then,   if $\mathcal E\subset \mathrm{GLem}(h_{q},p,M)$, it also holds that
   \[E\in\mathrm{GLem}(h_{q},p,M')\] for a positive constant $M'=M'(C_0,C,\theta, M, p,q,k)$.
\end{theorem}
\begin{proof}
With assumption \eqref{eq:AssLemmaA2} in place, the proof reduces to the same steps as in \cite[p.2070]{MR4485846} to deduce \cite[Proposition 2.23]{MR4485846} from \cite[Lemma A.2]{MR4485846}.

In view of assumption \eqref{eq:AssLemmaA2}, for a fixed $R\in \mathbb{D}(E)$ and $Q\in \mathbb{D}_{R,\widetilde{E}}(E)$,
one can estimate $  h_{q}(Q)$ by $C_0(\widetilde h_{q}(\widetilde Q)+ I_q(Q))$. The Carleson sum of the quantities $\widetilde h_{q}(\widetilde Q)$ can be controlled since $\widetilde E\in \mathrm{GLem}(h_{ q},p,M)$, while for the quantity $I_q(Q)$ we apply Lemma A.5 in \cite{MR4485846} (which requires \eqref{eq:exponentRange}) without modifications. More precisely, combining these observations one gets
\[\int_{R\cap \widetilde E}\left(\sum_{x\in Q,Q\in \mathbb{D}_R(E)}h_{q}(Q)^p\right)\,d\mathcal H^k(x)\leq \sum_{Q\in \mathbb{D}_{R,\widetilde{E}}(E)}h_q(Q)^p\,\mathcal{H}^k(Q)\lesssim \mathcal H^k(R),\]
where the implicit constant depends on $C_0$, $M$, $C$, $p$, $q$, $k$,  $\theta$. Here we also used  \eqref{eq:DiamComp} to ensure that each $Q'\in \mathbb{D}(\widetilde{E})$ can occur as $\widetilde{Q}=\widetilde{Q}(Q)$ only for $\sim_{C,k,\theta} 1$ many $Q\in \mathbb{D}_R(E)$ (cf.\ \cite[(A.1)]{MR4485846}).

Applying Chebyshev's inequality and proceeding as in \cite{MR4485846}, this implies that there exists $N>0$ such that 
\[\mathcal H^k\left(\left\{x\in R: \sum_{x\in Q,Q\in \mathbb{D}_R(E)}h_{q}(Q)^p<N\right\}\right)\gtrsim\mathcal H^k(R),\quad R\in\mathbb D(E).\]
Now the claim follows by \cite[Lemma A.1]{MR4485846}, which is general and applies for $\alpha=h_{q}^p$.
\end{proof}
We next give sufficient conditions for coefficient functions to satisfy the assumption \eqref{eq:AssLemmaA2} in Theorem \ref{thm: BHHGNgeneralizedAbstract}. This generalizes \cite[Lemma A.2.]{MR4485846} and is the crucial result for our generalization of \cite[Proposition 2.23]{MR4485846}.

\begin{definition}\label{d:AbstractCoeffFunction} Let $(X,d)$ be a separable metric space and let $\mathcal{B}(X)$ be its Borel $\sigma$-algebra.
    Let $\mathscr{P}=(\rho_A)_{A\in\mathcal A}$ be a collection of nonnegative $[\mathcal{B}(X)\otimes \cdots \otimes \mathcal{B}(X)]$-measurable functions on $X^m=X\times \cdots \times X$, indexed by a family $\mathcal A$. 
    Given a $k$-regular subset $E$ of $X$, $m\in \mathbb{N}$, a dyadic cube $Q\in \mathbb{D}(E)$, $\mu=(\mathcal{H}^k)^m=\mathcal{H}^k\otimes \cdots \otimes \mathcal{H}^k$, $q\in [1,\infty]$, we define the coefficients $h_{\mathscr{P}, q}$ as 
    \[h_{\mathscr{P}, q}(Q):=\inf_{A\in\mathcal A}\left(\fint_{[2Q]^m}\left(\frac{\rho_A(y)}{\mathrm{diam}(Q)}\right)^q\,d\mu(y)\right)^\frac{1}{q}\quad\text{if }q<\infty,\]
    and
    \begin{displaymath}
        h_{\mathscr{P}, q}(Q):=\inf_{A\in\mathcal A}
        \sup_{y\in [2Q]^m}
       \frac{\rho_A(y)}{\mathrm{diam}(Q)}\quad\text{if }q=\infty.
    \end{displaymath}
\end{definition}
Here, $\mu$ denotes the product measure of the corresponding Hausdorff measures. In particular, if $A_1,\ldots,A_m$ are $\mathcal{H}^k$-measurable subsets of $X$, then $\mu(A_1\times \cdots \times A_m)=\mathcal{H}^k(A_1)\cdots \mathcal{H}^k(A_m)$. Definition \eqref{d:AbstractCoeffFunction} can be seen as a generalization of the $\beta$-numbers in \cite[Definition 2.15]{MR4485846}. In addition to the multi-variable version and the more general function $\rho_A$, our coefficients have the additional slight difference that we normalize by the measure of $2Q$ rather than $Q$. This is a notational convenience, but means that, compared to \cite{MR4485846}, we have an additional dependence on AD-regularity constants in some of our statements.

\begin{lemma}\label{lemmaA2}
 Let  $L>0$, $C>1, \theta>0, q\in[1,\infty], m\in\mathbb N$, $0<k<\infty$.
   Suppose that $\mathscr P=(\rho_A)_{A\in\mathcal A}$ is a family of nonnegative $[\mathcal{B}(X)\otimes \cdots \otimes \mathcal{B}(X)]$-measurable functions defined on $X^m= X\times \cdots \times X$   with the property that
   \begin{equation}\label{eq:AxiomRho}\rho_A(y)\leq L\left( \rho_A(z)+ \sum_{j=1}^m d(y_j,z_j)\right),\quad y=(y_1,\ldots,y_m)
   ,\, z=(z_1,\ldots,z_m)\in X^m.\end{equation}
  Then, if $E\in \mathrm{Reg}_k(C)\cap BP(\mathcal E)(\theta, C)$, where $\mathcal E\subset \mathrm {Reg}_k(C)$, it holds
   for every $R\in \mathbb{D}(E)$,
 \begin{equation}\label{eq:ObtainedAssLemmaA2}
          h_{\mathscr{P},q}(Q)\leq C(q,k,C,\theta,L,m)(\widetilde h_{\mathscr{P}, q}(\widetilde Q)+ I_q(Q))\quad Q\in \mathbb{D}_{R,\widetilde{E}}(E),
      \end{equation}
      where the notation is as stated before Theorem \ref{thm: BHHGNgeneralizedAbstract}.
\end{lemma}
\begin{remark}\label{r:ApplGenBigPiece}
    We are particularly interested in Lemma \ref{lemmaA2} for the case $L=1$, $m=1$, $X=\mathbb H^n$, $k\in \{1,\ldots,n\}$,  $\mathcal A=\mathcal A(2n,k)$, $q=2$, $\rho_A(y):=d_{\mathrm {Eucl}}(\pi(y), A)$; see Theorem \ref{c:ILG-C Char}.
\end{remark}
\begin{remark}
    If $m=1$ and $\mathcal A$ is a family of subsets of $X$ and $\rho_A:=d(\cdot,A)$, then  Lemma \ref{lemmaA2} essentially corresponds to \cite[Lemma A.2]{MR4485846}.
\end{remark}

We present two other examples of flatness coefficients from the literature that satisfy the assumptions of Lemma \ref{lemmaA2}, and to which the big pieces stability result in Theorem \ref{thm: BHHGNgeneralizedAbstract} therefore applies. 

\begin{example}[$\kappa$-coefficients]
 For  a set $E\in \mathrm{Reg}_1(C)$ in a metric space $(X,d)$, we
define
\begin{equation}\label{eq:MetricBeta}
\kappa(Q):= \frac{1}{\mathcal{H}^1(2Q)^3}\int_{2Q}\int_ {2Q}\int_{2Q}
\frac{\partial(\{x_1,x_2,x_3\})}{\mathrm{diam}(Q)}\,d\mathcal{H}^1\lfloor_E(x_1)d\mathcal{H}^1\lfloor_E(x_2)d\mathcal{H}^1\lfloor_E(x_3),
\end{equation}
for $Q\in\mathbb{D}(E)$. Here, $\partial(\{x_1,x_2,x_3\})$ denotes the \emph{triangular excess} of
 three points
$x_1,x_2,x_3\in X$ defined by
\begin{equation*}
\begin{split}
\partial (\{x_1,x_2,x_3\})&\coloneqq \inf_{\sigma \in S_3} \left\{\partial_1(x_{\sigma(1)},x_{\sigma(2)},x_{\sigma(3)})\right\}\\
&\coloneqq \inf_{\sigma \in S_3} \left\{d(x_{\sigma(1)},x_{\sigma(2)})+d(x_{\sigma(2)},x_{\sigma(3)})-d(x_{\sigma(1)},x_{\sigma(3)})\right\},
\end{split}
\end{equation*}
where $S_3$ is the group of permutations of $\{1,2,3\}$. Thus, in the notation of Definition \ref{d:AbstractCoeffFunction}, we have $\kappa=h_{\mathscr{P},1}$ with $\mathscr{P}=\{\partial (\cdot)\}$ and $m=3$.

\medskip

Let $y=(y_1,y_2,y_3),z=(z_1,z_2,z_3)\in X^3$, and assume that $\sigma \in S_3$ is such that
\begin{displaymath}
    \partial(\{z_1,z_2,z_3\})  = d(z_{\sigma(1)},z_{\sigma(2)})+d(z_{\sigma(2)},z_{\sigma(3)})-d(z_{\sigma(1)},z_{\sigma(3)}),
\end{displaymath}
then we have by triangle inequality
\begin{align*}
    \partial(\{y_1,y_2,y_3\}) \leq &d(y_{\sigma(1)},y_{\sigma(2)})+d(y_{\sigma(2)},y_{\sigma(3)})-d(y_{\sigma(1)},y_{\sigma(3)})\\
    \leq & d(y_{\sigma(1)},z_{\sigma(1)})+ d(z_{\sigma(1)},z_{\sigma(2)})+ d(z_{\sigma(2)},y_{\sigma(2)})\\
    & + d(y_{\sigma(2)},z_{\sigma(2)})+ d(z_{\sigma(2)},z_{\sigma(3)})+ d(z_{\sigma(3)},y_{\sigma(3)})\\
    &+ d(y_{\sigma(1)},z_{\sigma(1)})- d(z_{\sigma(1)},z_{\sigma(3)})+ d(z_{\sigma(3)},y_{\sigma(3)})\\
    \leq & \partial(\{z_1,z_2,z_3\})+ 2\sum_{i\in \{1,2,3\}} d(y_i,z_i),
\end{align*}
which verifies the assumption of Lemma \ref{lemmaA2}.

\medskip
Geometric
lemmas for the coefficients $\kappa$ and related rectifiability
results were studied by Hahlomaa and Schul
\cite{MR2163108,MR2297880,MR2456269,MR2337487,MR2342818,2007arXiv0706.2517S} and recently by the second author and Violo \cite{MR4963423}. See also \cite{2026arXiv260517680C} for an application related to singular integrals. 
\end{example}

\begin{example}[$\iota$-coefficients]\label{ex:iota}
Let $1\leq k\leq n$ and assume that $E\in \mathrm{Reg}_1(C)$ is a set in $(\mathbb{H}^n,d)$. For every $\mathbb{V}\in \mathcal{V}_k$, 
 $q\in[1,\infty)$ and every $Q\in
\mathbb{D}(E)$,  we  define
\[
\iota_{q,\V}(Q)\coloneqq \left(\frac{1}{\mathcal{H}^k(2Q)^2}\int_{2Q}\int_{2Q}
\left[\frac{|d(x_1,x_2)-d(\pi_{\mathbb{V}}(x_1),\pi_{\mathbb{V}}(x_2))|}{\mathrm{diam}(2Q)}\right]^q
d\mathcal{H}^k(x_1) d\mathcal{H}^k(x_2) \right)^\frac1q
\]
and
\begin{equation}\label{d:newbetas proj}
    \iota_{q,\mathcal V_k}(Q)=\inf_{\V\in \mathcal V_k}\iota_{q,\V}(Q).
\end{equation}
Here $\pi_{\mathbb{V}}$ denotes the horizontal Heisenberg projection onto the affine horizontal plane $\mathbb{V}$, see \cite{Hahlomaa,FV2,arXiv:2601.03837}, and Definition \ref{d:HorizProj} for the case of a horizontal subgroup $\mathbb{V}$. Thus, in the notation of Definition \ref{d:AbstractCoeffFunction}, we have $\iota_{q,\mathcal{V}_k}=h_{\mathscr{P},q}$ with $m=2$ and
\begin{displaymath}
\mathscr{P}=\{(x_1,x_2)\mapsto |d(x_1,x_2)-d(\pi_{\mathbb{V}}(x_1),\pi_{\mathbb{V}}(x_2))|\}_{\mathbb{V}\in \mathcal{V}_k}.
\end{displaymath}
Now, for arbitrary $(y_1,y_2),(z_1,z_2)\in \mathbb{H}^n \times \mathbb{H}^n$, we obtain by the triangle inequality and the $1$-Lipschitz continuity of $\pi_{\mathbb{V}}$ that 
\begin{align*}
    |d(y_1,y_2)-d(\pi_{\mathbb{V}}(y_1),\pi_{\mathbb{V}}(y_2))|= &
    d(y_1,y_2)-d(\pi_{\mathbb{V}}(y_1),\pi_{\mathbb{V}}(y_2))\\
    \leq & d(y_1,z_1)+d(z_1,z_2)+d(z_2,y_2)\\
    &+ d(\pi_{\mathbb{V}}(y_1),\pi_{\mathbb{V}}(z_1))-
    d(\pi_{\mathbb{V}}(z_1),\pi_{\mathbb{V}}(z_2))
    + d(\pi_{\mathbb{V}}(z_2),\pi_{\mathbb{V}}(y_2))\\
    \leq & |d(z_1,z_2)-d(\pi_{\mathbb{V}}(z_1),\pi_{\mathbb{V}}(z_2))|+ 2\sum_{i=1}^2d(y_i,z_i).
\end{align*}
This shows that the coefficients $\iota_{q,V}$ satisfy the assumptions of Lemma  \ref{lemmaA2}.
Geometric lemmas for  $\iota_{q,\V}$  were studied in \cite{FV2,arXiv:2601.03837}.
\end{example}

We now discuss the proof of Lemma \ref{lemmaA2}. This follows closely the proof strategy of \cite[Lemma A.2.]{MR4485846} (for $q<\infty$) and \cite[Lemma A.3]{MR4485846} (for $q=\infty$), and we will refer to \cite{MR4485846} for some details that work exactly in the same way. Our  task is to verify that \eqref{eq:AxiomRho} is the correct axiomatic assumption for this kind of statement, which is rather immediate for $m=1$. The main challenge is the multi-variable case $m>1$.

\begin{proof}[Proof of Lemma \ref{lemmaA2}]
Fix $m\in \mathbb{N}$. All implicit constants in the following will be allowed to depend on $m$ without special mentioning.

We first discuss the proof for $q<\infty$.
  Fix $\eta>0$, $R\in \mathbb{D}(E)$ with associated $\widetilde{E}$, and $Q\in \mathbb{D}_{R,\widetilde{E}}(E)$. We assign $\widetilde{Q}=\widetilde{Q}(Q)\in \mathbb{D}(\widetilde{E})$ to the cube $Q$ as explained around \eqref{eq:DiamComp}.
By definition of $\widetilde h_{\mathscr{P}, q}(\widetilde Q)$, there exists $A\in\mathcal A$ such that
 \begin{equation}\label{eq:ChoiceA}\left(\fint_{[2\widetilde Q]^m}\rho_A(y)^q\,d\mu(y)\right)^\frac{1}{q}\leq [\mathrm{diam}(\widetilde Q)]\widetilde h_{\mathscr{P}, q}(\widetilde Q)+\eta,\end{equation}
 where $\mu=(\mathcal{H}^k)^m=\mathcal{H}^k\otimes \cdots \otimes \mathcal{H}^k$.
 For this choice of $A$ we also have, by definition,
 \[[\mathrm{diam}( Q)]h_{\mathscr{P}, q}(Q)\leq\left(\fint_{[2Q]^m}\rho_A(y)^q\,d\mathcal \mu(y)\right)^\frac{1}{q}.\]
By the triangle inequality, as in \cite[(A.7)]{MR4485846} we combine the previous inequalities and get
\begin{align}\label{eq:A7}[\mathrm{diam}&(Q)]h_{\mathscr{P}, q}(Q)\\&\leq \left(\mu([2Q]^m)^{-1}\int_{[2Q]^m\cap \widetilde{E}^m}\rho_A(y)^q\,d\mu(y)\right)^\frac{1}{q}+\left(\mu([2Q]^m)^{-1}\int_{[2Q]^m\setminus \widetilde{E}^m}\rho_A(y)^q\,d\mu(y)\right)^\frac{1}{q}\notag
\\&\lesssim_{C,k,q,\theta} ( [\mathrm{diam}(\widetilde Q)]\widetilde h_{\mathscr{P}, q}(\widetilde Q)+\eta)+ \left(\mu([2Q]^m)^{-1}\int_{[2Q]^m\setminus \widetilde{E}^m}\rho_A(y)^q\,d\mu(y)\right)^\frac{1}{q}.\notag\end{align}
Here we used the following two  facts. First, $2Q\cap \widetilde{E}\subset 2 \widetilde{Q}$ by \eqref{eq:CubeIncl}, and hence $[2Q]^m\cap \widetilde{E}^m\subset [2\widetilde{Q}]^m$. Second, $\mathcal{H}^k(2Q)\sim \mathcal{H}^k(Q)\sim \mathcal{H}^k(\widetilde{Q})\sim \mathcal{H}^k(2\widetilde{Q})$ with constants depending only on $\theta$, $k$ and $C$ by $k$-regularity and \eqref{eq:DiamComp}.

To complete the proof of the lemma we  need to estimate the last term. Indeed, returning to \eqref{eq:A7}, we see that it suffices to show
\begin{equation}\label{eq:Goal}
\mu(Q^m)^{-1}\int_{[2Q]^m\setminus \widetilde{E}^m}\rho_A(y)^q\,d\mu(y) \lesssim [\mathrm{diam}(\widetilde{Q})]^q\left(\widetilde{h}_{\mathscr{P}, q}(\widetilde{Q})^q + I_q(Q)^q\right) + \eta^q,
\end{equation}
with implicit constant depending at most on $q,k,C,\theta,L$ (and $m$).
To prove \eqref{eq:Goal}, we proceed similarly as in \cite{MR4485846}, from (A.8) to (A.12), but some care needs to be taken in the case $m>1$. We define
\begin{displaymath}
    \Omega_i:=\{y=(y_1,\ldots,y_m)\in[2Q]^m\setminus \widetilde{E}^m:\, d(y_i,\widetilde{E})=\max_{1\leq j\leq m}d(y_j,\widetilde{E})>0\},\quad i=1,\ldots,m,
\end{displaymath}
that is, points in $\Omega_i$ can have some components in $\widetilde{E}$ but the $i$-th component lies in $2Q\setminus \widetilde{E}$ and
\begin{equation}\label{eq:DistEst}
d(y_i,\widetilde{E})\geq d(y_j,\widetilde{E})\quad\text{for all }j=1,\ldots,m.
\end{equation}

Then we can write $[2Q]^m \setminus \widetilde{E}^m = \bigcup_{i=1}^m \Omega_i$. The inclusion ``$\subseteq$'' holds since for every $y\in [2Q]^m \setminus \widetilde{E}^m$ at least one of the components $y_i$
does not belong to $\widetilde{E}$. The sets $\Omega_i$ are not necessarily disjoint, but this is not a problem since we only need the upper bound
\begin{equation}\label{eq:Omega_iUpper}
 \mu(Q^m)^{-1}\int_{[2Q]^m \setminus \widetilde{E}^m}\rho_A(y)^q \, d\mu(y) \leq \sum_{i=1}^m \mu(Q^m)^{-1}\int_{\Omega_i}\rho_A(y)^q d\mu(y).   
\end{equation}
By our assumption \eqref{eq:AxiomRho} on $\rho_A$, we have
\begin{equation}\label{eq:rho_q_ineq}
   \rho_A(y)^q\lesssim_{L,q} \rho_A(z)^q+ \sum_{j=1}^m d(y_j,z_j)^q,\quad y=(y_1,\ldots,y_m),\, z=(z_1,\ldots,z_m)\in X^m.
\end{equation}

To obtain \eqref{eq:Goal}, by \eqref{eq:Omega_iUpper}, we need to integrate the left-hand side of \eqref{eq:rho_q_ineq} over $\Omega_i$ for $i=1,\ldots,m$.
Thus, let us fix $i\in \{1,\ldots,m\}$.
Inspired by \cite{MR4485846}, we will first integrate with respect to $z$ over a suitable domain, and then integrate with respect to $y$ over $\Omega_i$. For this purpose we need to introduce some notation.

By definition of $\Omega_i$, we have $d(y_i,\widetilde{E})>0$ for $y\in \Omega_i$. Following \cite[(A.8),(A.11)]{MR4485846}, we  use the notation
\begin{equation}\label{eq:G(y_i)}
    G(y_i):=\{v\in \widetilde{E}:\,v\in B(y_i,2d(y_i,\widetilde{E})),\,M(v,2d(y_i,\widetilde{E}))\leq K_1\},\quad y\in \Omega_i,
\end{equation}
where
\begin{displaymath}
    M(v,s):=\int_{\{w\in E\setminus \widetilde{E}:\,d(w,\widetilde{E})\leq s,d(v,w)\leq 2 d(w,\widetilde{E})\}}d(w,\widetilde{E})^{-k}d\mathcal{H}^k(w),\quad v\in \widetilde{E},
\end{displaymath}
and $K_1>0$, depending only on $k$, $C$, and $\theta$, is chosen 
large enough such that
\begin{equation}\label{eq:LowerG(yi)}
    \mathcal{H}^k(G(y_i))\gtrsim_{C,k,\theta} d(y_i,\widetilde{E})^k,\quad y_i\in 2Q\setminus \widetilde{E}.
\end{equation}
The existence of $K_1$ follows by Fubini's theorem, Chebyshev's inequality and $k$-regularity of $\widetilde{E}$ as in the one-variable case
\cite[(A.12)]{MR4485846} since the number $m\geq 1$ is not involved in this statement. To be precise, in the case of bounded $\widetilde{E}$, in order to
apply the $k$-regularity of 
$\widetilde{E}$, we also need the fact that if $y_i \in 2Q\setminus \widetilde{E}$ for $Q\in \mathbb{D}_{R,\widetilde{E}}(E)$, then
\begin{equation}\label{eq:RadLargeEnough}
d(y_i,\widetilde{E})\leq 2 \mathrm{diam}(Q)\leq 
2 \mathrm{diam}(R)\lesssim_{C,k,\theta}\mathrm{diam}(\widetilde{E})
\end{equation}
by the reasoning below Definition \ref{d:BP}.

Now we fix $y\in \Omega_i$ and integrate \eqref{eq:rho_q_ineq} with respect to $z$ over 
\begin{displaymath}
    B_i^{\ast}(y):=\{z=(z_1,\ldots,z_m):\, z_i\in G(y_i)\text{ and }z_j\in B(y_j,2d(y_i,\widetilde{E}))\cap \widetilde{E} \text{ for }j\neq i \}.
\end{displaymath}
If $m=1$ (as for the case considered in \cite{MR4485846}), then simply $y=y_1$ and $B_1^{\ast}(y)=G(y)$. In general,
the set $B_i^{\ast}(y)$ is a direct product where
all  factors are contained in balls of radius $2d(y_i,\widetilde{E})$ and specifically
the $i$-th factor agrees with the set $G(y_i)$ defined in \eqref{eq:G(y_i)}. 
This produces
\begin{displaymath}
\rho_A(y)^q    \mu(B_i^{\ast}(y))
\lesssim_{L,q} \int_{B_i^{\ast}(y)}\rho_A(z)^q \,d\mu(z) +  d(y_i,\widetilde{E})^q\mu(B_i^{\ast}(y)),
\end{displaymath}
where we used that $d(y_j,z_j)\leq 2 d(y_i,\widetilde{E})$ for $z_j\in B(y_j,2d(y_i,\widetilde{E}))$ and all $j=1,\ldots,m$.
Dividing by the measure of $B_i^{\ast}(y)$, we obtain
\begin{equation}\label{eq:AfterZInt}
\rho_A(y)^q   
\lesssim_{L,q}  \mu(B_i^{\ast}(y))^{-1} \int_{B_i^{\ast}(y)}\rho_A(z)^q \,d\mu(z) +d(y_i,\widetilde{E})^q,\quad y\in \Omega_i.
\end{equation}
The measure $\mu(B_i^{\ast}(y))$ can be bounded from below as follows.
By $k$-regularity of $\widetilde{E}$, we have
\begin{equation}\label{eq:MeasBall}
    \mathcal{H}^k\left(B\left(y_j,2d(y_i,\widetilde{E})\right)\cap \widetilde{E}\right)\gtrsim_{C,k,\theta} d(y_i,\widetilde{E})^k.
\end{equation}
Here, we used again \eqref{eq:RadLargeEnough}, and the observation  that $B\left(y_j,2d(y_i,\widetilde{E}\right)$ contains a ball of comparable radius centered at a point in $\widetilde{E}$ since \eqref{eq:DistEst} holds. Therefore, by \eqref{eq:LowerG(yi)} and \eqref{eq:MeasBall}, we obtain
\begin{displaymath}\label{eq:G(y)}
 \mu(B_i^{\ast}(y))\gtrsim_{C,k,\theta} d(y_i,\widetilde{E})^{km}.
\end{displaymath}
Returning to \eqref{eq:AfterZInt}, we have found that
\begin{align*}
    \rho_A(y)^q   
&\lesssim_{L,q,C,k,\theta} \frac{1}{ d(y_i,\widetilde{E})^{km}} \int_{B_i^{\ast}(y)}\rho_A(z)^q \,d\mu(z) + d(y_i,\widetilde{E})^q\\
&=\frac{1}{ d(y_i,\widetilde{E})^{km}} \int_{\widetilde{E}^m}\chi_{B_i^{\ast}(y)}(z)\rho_A(z)^q \,d\mu(z) + d(y_i,\widetilde{E})^q,\quad y\in \Omega_i.
\end{align*}
Integrating with respect to $\mu$ over $\Omega_i$, and multiplying by $\mu(Q^m)^{-1}$ yields
\begin{align}\label{eq:mu_rho_Int_interm}
&\mu(Q^m)^{-1} \int_{\Omega_i} \rho_A(y)^q   \,d\mu(y)\lesssim_{m,L,q,C,k,\theta}\\& \mu(Q^m)^{-1}  \int_{E^m}\frac{1}{ d(y_i,\widetilde{E})^{km}} \int_{\widetilde{E}^m}\chi_{B_i^{\ast}(y)}(z)\rho_A(z)^q \chi_{\Omega_i}(y)\,d\mu(z)\,d\mu(y) + \mu(Q^m)^{-1}\int_{\Omega_i} d(y_i,\widetilde{E})^q\,d\mu(y).\notag
\end{align}
The second term is already familiar. Indeed, we recall
\begin{displaymath}
\mathrm{diam}(Q)I_q(Q)=\left(\mathcal H^k(Q)^{-1}\int_{[2Q\cap \widetilde E(2\mathrm{diam}(Q))]\setminus \widetilde{E}}d(v,\widetilde E)^q\,d\mathcal H^k(v)\right)^\frac{1}{q}.
\end{displaymath}
If $m=1$, we can apply this definition directly. If $m>1$, denoting by $\hat y$ the point in $X^{m-1}$ which coincides with $y\in X^m$ except that the $i$-th coordinate is deleted, we obtain 
\begin{align*}
   \mu(Q^m)^{-1}&\int_{\Omega_i} d(y_i,\widetilde{E})^q\,d\mu(y) \\& \leq  \mu(Q^m)^{-1} \int_{[2Q\cap \widetilde E(2\mathrm{diam}(Q))]\setminus \widetilde{E}} d(y_i,\widetilde E)^q\int_{[2Q]^{m-1}} d[\mathcal{H}^k\otimes \cdots \otimes \mathcal{H}^k ](\hat y) \, d\mathcal{H}^k(y_i)\\
     &\sim_{C,k} \mathcal{H}^k(Q)^{-1} \int_{[2Q\cap \widetilde E(2\mathrm{diam}(Q))]\setminus \widetilde{E}} d(y_i,\widetilde E)^q \,d\mathcal{H}^k(y_i)\\
     &\sim_{C,k,\theta} [\mathrm{diam}(\widetilde{Q}) I_q(Q)]^q.
\end{align*}
The first inequality holds since $y\in \Omega_i$ implies that $y_i\in 2Q$, where by assumption $Q\cap \widetilde{E}\neq \emptyset$, thus also $y_i \in \widetilde{E}(2\mathrm{diam}(Q))$, while $y_i\notin \widetilde{E}$. Then we also used that $\mu(Q^m)\sim_{C,k} \mathcal{H}^k(Q)^m$, and \eqref{eq:DiamComp} by the definition of $\widetilde{Q}=\widetilde{Q}(Q)$.

Summing over $i$ as in \eqref{eq:Omega_iUpper}, this yields one of the desired bounds in \eqref{eq:Goal}. It remains to bound the first term in \eqref{eq:mu_rho_Int_interm}. That is,
\begin{align}\label{eq:DoubleIntEst}
 &  \mu(Q^m)^{-1}  \int_{E^m}\frac{1}{ d(y_i,\widetilde{E})^{km}} \int_{\widetilde{E}^m}\chi_{B_i^{\ast}(y)}(z)\rho_A(z)^q \chi_{\Omega_i}(y)\,d\mu(z)\,d\mu(y)\notag\\&= 
    \mu(Q^m)^{-1} \int_{\widetilde{E}^m}  \left[\int_{E^m}\frac{1}{ d(y_i,\widetilde{E})^{km}} \chi_{B_i^{\ast}(y)}(z) \chi_{\Omega_i}(y)\,d\mu(y)\right]\,\rho_A(z)^q\,d\mu(z)\notag\\
    &= \mu(Q^m)^{-1} \int_{[2\widetilde{Q}]^m}  \left[\int_{E^m}\frac{1}{ d(y_i,\widetilde{E})^{km}} \chi_{B_i^{\ast}(y)}(z) \chi_{\Omega_i}(y)\,d\mu(y)\right]\,\rho_A(z)^q\,d\mu(z).
\end{align}
The justification for the last inequality is the following. The integrand, as a function of $z$, can be non-zero  only if there is $y\in \Omega_i \subset [2Q]^m$ such that $z\in B_i^{\ast}(y)\subset \widetilde{E}^m$. This means in particular that, for all $j\in \{1,\ldots,m\}$, we have  $y_j \in 2Q$ and $d(z_j,y_j)\leq 2 d(y_i,\widetilde{E})$. Recall that since $y_i \in 2Q$ and $Q\cap \widetilde{E}\neq \emptyset$, we have $d(y_i,\widetilde{E})\leq 2 \mathrm{diam}(Q)$. Therefore, we know that the integrand in \eqref{eq:DoubleIntEst} vanishes unless possibly if
\begin{displaymath}
    z_j \in  B(y_j, 2d(y_i,\widetilde{E}))\cap \widetilde{E} \subset 
    B(y_j, 4\mathrm{diam}(Q))\cap \widetilde{E}
    \subset
    \{z\in\widetilde{E}:\,d(z,Q)\leq 5 \mathrm{diam}(Q)\}
    \overset{\eqref{eq:CubeIncl}}{\subset} 2 \widetilde{Q}.
\end{displaymath}
Here the second but last inclusion follows from the fact that $y_j\in 2Q$, so $d(y_j,Q)\leq \mathrm{diam}(Q)$ and $v\in  B(y_j, 4\mathrm{diam}(Q))$ satisfies $d(v,Q)\leq 5\mathrm{diam}(Q)$.

Thus, for $z\in [2\widetilde{Q}]^m$, it remains to control the inner integral in \eqref{eq:DoubleIntEst}. This can be written as
 \begin{align*}
   &\int_{E^m}\frac{1}{ d(y_i,\widetilde{E})^{km}} \chi_{B_i^{\ast}(y)}(z) \chi_{\Omega_i}(y)\,d\mu(y)\\
   &\leq \int_{\{y_i\in 2Q\setminus\widetilde{E}:\,z_i\in G(y_i)\}}\frac{1}{ d(y_i,\widetilde{E})^{km}}\left[\int_{\{\hat y\in [2Q]^{m-1}:\, d(z_j,y_j)\leq 2 d(y_i,\widetilde{E})\text{ for }j\neq i\}}\,d[\mathcal{H}^k\otimes \cdots \otimes d\mathcal{H}^k ](\hat y)\right]\,d\mathcal{H}^k(y_i).
 \end{align*}
 Since $d(y_i,\widetilde{E})\lesssim \mathrm{diam}(Q)$ (recall \eqref{eq:RadLargeEnough}) and $z_j \in 2\widetilde{Q}$ with $Q\cap \widetilde{Q}\neq \emptyset$ and $\mathrm{diam}(\widetilde{Q})\sim_{C,k,\theta}\mathrm{diam}(Q)$ by \eqref{eq:DiamComp}, we obtain by $k$-regularity of $E$ that
 \begin{displaymath}
   \int_{E^m}\frac{1}{ d(y_i,\widetilde{E})^{km}} \chi_{B_i^{\ast}(y)}(z) \chi_{\Omega_i}(y)\,d\mu(y)
   \lesssim_{C,k,\theta} \int_{\{y_i\in 2Q\setminus\widetilde{E}:\,z_i\in G(y_i)\}}\frac{1}{ d(y_i,\widetilde{E})^{k}}\,d\mathcal{H}^k(y_i).
   \end{displaymath}
   Returning to \eqref{eq:DoubleIntEst}, we have found that
   \begin{align*}
    &  \mu(Q^m)^{-1}  \int_{E^m}\frac{1}{ d(y_i,\widetilde{E})^{km}} \int_{\widetilde{E}^m}\chi_{B_i^{\ast}(y)}(z)\rho_A(z)^q \chi_{\Omega_i}(y)\,d\mu(z)\,d\mu(y)\\
    & \lesssim_{C,k,\theta} \mu(Q^m)^{-1} \int_{[2\widetilde{Q}]^m}  \left[\int_{\{y_i \in 2Q\setminus\widetilde{E}:\,z_i\in G(y_i)\}}\frac{1}{ d(y_i,\widetilde{E})^{k}}\,d\mathcal{H}^k(y_i)\right]\,\rho_A(z)^q\,d\mu(z).
   \end{align*}
   The inner integral is exactly the same expression as in  \cite[(A.18)]{MR4485846}, and it can be bounded in the same way by $K_1$, using the definitions of $G$ and $M$ in \eqref{eq:G(y_i)}. Thus, we obtain 
   \begin{align*}
    &  \mu(Q^m)^{-1}  \int_{E^m}\frac{1}{ d(y_i,\widetilde{E})^{km}} \int_{\widetilde{E}^m}\chi_{B_i^{\ast}(y)}(z)\rho_A(z)^q \chi_{\Omega_i}(y)\,d\mu(z)\,d\mu(y)\\
    & \lesssim_{C,k,\theta} \mu(Q^m)^{-1} \int_{[2\widetilde{Q}]^m}  \,\rho_A(z)^q\,d\mu(z) \lesssim [\mathrm{diam}(\widetilde{Q}) \widetilde{h}_{\mathscr{P},q}(\widetilde{Q})+\eta]^q,
    \end{align*}
by our choice of $A$ in \eqref{eq:ChoiceA} at the beginning of the proof.  Inserting this in \eqref{eq:mu_rho_Int_interm}, and recalling \eqref{eq:Omega_iUpper}, it yields \eqref{eq:Goal}.

\medskip

The case $q=\infty$ is proven
with a more direct application of the assumption \eqref{eq:AxiomRho},
thus adapting \cite[Lemma A.3]{MR4485846}  (instead of  \cite[Lemma A.2]{MR4485846}). This  is easier than the case $q<\infty$ since no product measure is involved. In fact, the argument follows closely the beginning of the proof of \cite[Lemma A.4]{MR4485846}, with the main difference that we assign to $y\in [2Q]^m$  the point $z=(z_1,\ldots,z_m)\in \widetilde{E}^m$ defined by setting $z_j=y_j$ if $y_j\in \widetilde{E}$ and $z_j\in \widetilde{E}$ with $d(y_j,z_j)<2d(y_j,\widetilde{E})$ otherwise. This replaces the choice of $\widetilde{y}$ above \cite[(A.26)]{MR4485846} in the case $m=1$. 
 \end{proof}

\subsection{Applications of the big pieces functor}\label{ss:BigPieceAppl}
In this section we apply the stability result for abstract geometric lemmas under the big pieces functor (Theorem \ref{thm: BHHGNgeneralizedAbstract}) to the geometric lemmas in Section \ref{s:GLem}. The careful reader may note that the geometric and weak geometric lemmas in Section \ref{s:GLem} are formulated with integrals as in Definitions \ref{def:qgeomlemballs}--\ref{def:WGLballs}, whereas the results in Section 
\ref{s:StabAbstrGlem} are stated with Carleson sums over systems of dyadic cubes (Definitions \ref{d:GLemAbstr}--\ref{d:WGLAbsrt}) since in this abstract setting the relevant functions might not exhibit the measurability needed for integration. However, for all weak and strong geometric lemmas in our applications, the integral and dyadic sum versions are equivalent by standard arguments; see \cite[Lemma 3.18 and Remark 3.21]{arXiv:2601.03837}. Therefore, we henceforward use the symbol ``$\mathrm{GLem}$'' interchangeably for the integral and dyadic sum versions, and analogously for ``$\mathrm{WGL}$".

 Combining results from previous sections, we obtain
the characterization of the condition (ILG-C) (corona decomposition by intrinsic Lipschitz graphs)
 stated in the introduction as Corollary \ref{c:introChar}.
 To do so, we will use the fact that if a $k$-regular set $E\subset \mathbb{H}^n$ satisfies (ILG-C) in the sense of  Definition \ref{d:ILG-C}, then, for every $\eta>0$, it admits a $(\eta,2)$-coronization in the sense of \cite[Definition 2.20]{MR4485846} with respect to the family $\mathcal{E}=\mathcal{E}^{\mathrm{iLG}}_{\eta}$ of intrinsic $\eta$-Lipschitz graphs.

\begin{theorem}\label{c:ILG-C Char} Let $n\in \mathbb{N}$ and $k\in \{1,\cdots,n\}$.
    Let $E\subset\mathbb H^n$ be $k$-regular. Then \[E\in\mathrm{(ILG-C)}\Longleftrightarrow E\in\mathrm{GLem}(\beta_{2,\pi,A(2n,k)},2) \text{  and } E\in\mathrm{WGL}(\beta_{\infty,\mathcal V_k}).\]
\end{theorem}

\begin{proof}
    The  implication ``$\Leftarrow$'' is \cite[Corollary 1.7]{arXiv:2601.03837}.

    To prove the converse implication ``$\Rightarrow$'', let $E\in \mathrm{Reg}_k(C_E)$ satisfy (ILG-C) with respect to a dyadic system $\mathbb{D}(E)$. We first assume that $E$ is unbounded.  Then, by  \cite[Theorem 1.1]{MR4485846}, $E$ has big pieces squared of intrinsic Lipschitz graphs. More precisely,  for every $\eta>0$, take for instance $\eta=1$, we have $E\in BP(BP(\mathcal{E}^{\mathrm{iLG}}_{\eta})(\theta,C'))(\theta',C'')$ with constants $\theta,\theta',C',C''$ depending only on $C_E,k,\eta$ and the constant $C(\eta)$ from the coronization. Here, for the application of \cite[Theorem 1.1]{MR4485846}, we used the fact that $k$-dimensional intrinsic $\eta$-Lipschitz graphs in $\mathbb{H}^n$ are $k$-regular with a constant depending only on $\eta,k$, and $n$. 
Now, we
     recall that, for the given range of dimensions, $k$-dimensional intrinsic Lipschitz graphs satisfy $\mathrm{GLem}(\beta_{2,\pi,A(2n,k)},2)$ and $\mathrm{WGL}(\beta_{\infty,\mathcal V_k})$, with quantitatively controlled constants, by Theorem \ref{t:GlemAffProjBetanILG} and Corollary \ref{cor:WGLforILG}. It then follows from the big pieces stability (Theorem \ref{thm: BHHGNgeneralizedAbstract} for $q=p=2$, Remark \ref{r:ApplGenBigPiece},  and \cite[Proposition 2.24]{MR4485846}), that $E$ also satisfies $\mathrm{GLem}(\beta_{2,\pi,A(2n,k)},2)$ and $\mathrm{WGL}(\beta_{\infty,\mathcal V_k})$.

In the case where the given set $E\in \mathrm{Reg}_k(C_E)$ with (ILG-C) is bounded, analogously as in \cite[Remark 2.12]{MR4485846}, we fix a $k$-dimensional horizontal plane $\mathbb{V}$ whose distance to $E$ is comparable to the diameter of $E$ and we  consider the set $E_{\ast}=E\cup \mathbb{V}$. Since horizontal planes also have (ILG-C), the set $E_{\ast}$ will be an unbounded set with (ILG-C) by construction. Following the above reasoning for the unbounded case, it follows that $E_{\ast}\in \mathrm{GLem}(\beta_{2,\pi,A(2n,k)},2)$ and $E_{\ast}\in \mathrm{WGL}(\beta_{\infty,\mathcal V_k})$. Since $E$ is a $k$-regular subset of $E_{\ast}$, we conclude that $E$ itself also has the corresponding properties. 
\end{proof}

We conclude with another application of the results  in Section \ref{s:GLem} and stability of weak geometric lemmas under the big pieces functor. For this purpose, we need the following definition which builds on the notion of horizontal Heisenberg projections (cf.\ Definition \ref{d:HorizProj}).

\begin{definition}\label{d:BPHP}
    Let $n\in \mathbb{N}$ and $k\in \{1,\ldots,n\}$. We say that a $k$-regular set $E\subset \mathbb{H}^n$ has \emph{big horizontal projections} if there exists a constant $c>0$ such that for all $x\in E$ and all $0<r<\mathrm{diam}(E)$ there is a horizontal subgroup $\mathbb{V}=\mathbb{V}_{x,r}\in \mathcal{V}_k^0$ with the property that
    \begin{displaymath}
        \mathcal{H}^k(\pi_{\mathbb{V}}(E\cap B(x,r)))\geq c r^k,
    \end{displaymath}
    where $\pi_{\mathbb{V}}:\mathbb{H}^n \to \mathbb{V}$, $\pi_{\mathbb{V}}(z,t)=(\pi_V(z),0)$.
\end{definition}

We can now prove Theorem \ref{t:BPiLGIntro} from the introduction, which we re-state here for the reader's convenience.

\begin{theorem}\label{t:BPiLG} Let $n\in \mathbb{N}$ and {$k\in \{1,\cdots,n\}$}.
    Let $E\subset\mathbb H^n$ be $k$-regular. Assume that $E$ has big pieces of intrinsic Lipschitz graphs. Then $E\in \mathrm{WGL}(\beta_{\infty,\mathcal V_k})$ and $E$ has big horizontal projections.
\end{theorem}

\begin{proof}
   The fact that $E\in \mathrm{WGL}(\beta_{\infty,\mathcal V_k})$ follows by Corollary \ref{cor:WGLforILG} combined with the `big pieces' stability result in \cite[Proposition 2.24]{MR4485846}. Literally, the latter only applies if $E$ is unbounded, but in the bounded case we can deduce the result by considering a union of $E$ with a horizontal plane, analogously as in the proof of Theorem \ref{c:ILG-C Char}.

   The property that $E$ has big horizontal projections is standard: let $x\in E$, $r>0$. Since $E$ has big pieces of intrinsic Lipschitz graphs, there exists $\varphi:\V\to \W$ intrinsic $L$-Lipschitz, where $\V$ is a  horizontal subgroup of dimension $k$ and $\W$ is the corresponding complementary vertical subgroup such that
   \[\mathcal H^k(E\cap B(x,r)\cap \Phi(\V))\geq Cr^k,\] where $\Phi$ denotes the graph map associated to $\varphi$, and $L, C$ do not depend on $x,r$. Recall that $\Phi$ is metric $\tilde L$-Lipschitz, where $\tilde L$ depends only on $L$ (see e.g., \cite[Corollary 4.62 (i)]{MR3587666}). Hence, we get
   \[\mathcal H^k(\pi_{\V}(E\,\cap B(x,r)))\geq \frac{1}{\tilde L^k}\mathcal H^k(E\,\cap B(x,r)\cap \Phi(V))\geq \frac{C}{\tilde L^k}r^k. \qedhere\]
  
\end{proof}

\begin{remark}\label{r:BPiLG}  As stated in the introduction, in the codimension $1$ case, it is known that $(2n+1)$-regular sets in $\mathbb{H}^n$ have big pieces of intrinsic Lipschitz graphs over vertical hyperplanes if and only if they admit big vertical projections and satisfy a weak geometric lemma for vertical $\beta$-numbers.
    This was first proven for  $\mathbb{H}^1$  \cite{zbMATH07106894}. The generalization to $\mathbb{H}^n$ is a combination of several results:  \cite[Theorem 7.1]{zbMATH07222190} establishes  the ``if'' part. 
The ``only if'' part  follows essentially  from \cite[Propositions 4.6 and 6.3]{zbMATH07222190}  since entire codimension-$1$ intrinsic Lipschitz graphs are Semmes surfaces by \cite[Section 3.3]{2019arXiv190406904R}. The generalization from ``intrinsic Lipschitz graphs'' to ``sets with big pieces of intrinsic Lipschitz graphs'' also uses the stability result in  
\cite[Proposition 2.24]{MR4485846} and a higher-dimensional version of \cite[Remark 4.21]{zbMATH07106894}. 
\end{remark}

\printbibliography
\end{document}